\documentclass[11pt]{amsart}
\usepackage[left=1.5in, right=1.5in]{geometry}
\usepackage[latin1]{inputenc}
\usepackage{amsfonts}
\usepackage[toc,page,title,titletoc,header]{appendix}
\usepackage{graphicx,psfrag,epsfig,multirow,caption,subcaption}
\usepackage{amssymb,amsmath,amscd,amsthm,amssymb,verbatim,setspace}
\numberwithin{equation}{section}
\usepackage{mathrsfs,wrapfig}
\usepackage{indentfirst}
\usepackage{extarrows}
\usepackage[colorlinks=true]{hyperref}

\newtheorem{theo}{Theorem}[section]
\newtheorem{lem}[theo]{Lemma}

\newtheorem{pro}[theo]{Proposition}
\newtheorem{cor}[theo]{Corollary}
\newtheorem{defi}[theo]{Definition}
\newtheorem{Rk}[theo]{Remark}
\newtheorem{Not}[theo]{Notation}

\newtheorem{Conjecture}[theo]{Conjecture}
\newtheorem{Opro}[theo]{Open Problem}

\def\({\left(}
\def\){\right)}
\def\l{\left}
\def\r{\right}
\def\bdf{\begin{defi}}
\def\edf{\end{defi}}
\def\bnt{\begin{Not}}
\def\ent{\end{Not}}
\def\bth{\begin{theo}}
\def\eth{\end{theo}}
\def\bpp{\begin{pro}}
\def\epp{\end{pro}}
\def\bpf{\begin{proof}}
\def\epf{\end{proof}}
\def\blm{\begin{lem}}
\def\elm{\end{lem}}
\def\bcor{\begin{cor}}
\def\ecor{\end{cor}}
\def\brk{\begin{Rk}}
\def\erk{\end{Rk}}
\def\bcs{\begin{cases}}
\def\ecs{\end{cases}}
\def\bfig{\begin{picture}}
\def\efig{\end{picture}}
\def\beq{\begin{equation}}
\def\eeq{\end{equation}}
\def\ben{\begin{enumerate}}
\def\een{\end{enumerate}}
\def\bal{\begin{aligned}}
\def\eal{\end{aligned}}
\def\bmt{\left[\begin{array}}
\def\emt{\end{array}\right]}
\def\beqa{\begin{equation}\begin{aligned}}
\def\eeqa{\end{aligned}\end{equation}}
\def\-{\setminus}
\def\0{\emptyset}

\renewcommand{\th}{\theta}

\newcommand{\al}{\alpha}
\newcommand{\da}{\dagger}

\newcommand{\cA}{\mathcal{A}}
\newcommand{\cM}{\mathcal{M}}
\newcommand{\cR}{\mathcal{R}}
\newcommand{\cG}{\mathcal{G}}

\newcommand{\cN}{\mathcal{N}}

\newcommand{\R}{\mathbb{R}}

\newcommand{\Z}{\mathbb{Z}}
\newcommand{\C}{\mathbb{C}}
\newcommand{\cC}{\mathcal{C}}
\newcommand{\bk}{\mathbf{k}}

\newcommand{\bK}{\mathbf{K}}

\newcommand{\sH}{\mathsf{H}}

\newcommand{\sa}{\mathsf{a}}
\newcommand{\sC}{\mathsf{c}}

\newcommand{\sP}{\mathsf{P}}
\newcommand{\sR}{\mathsf{R}}
\newcommand{\sY}{\mathsf{Y}}

\newcommand{\sV}{\mathsf{V}}

\newcommand{\sA}{\mathsf{A}}
\newcommand{\sG}{\mathsf{G}}

\newcommand{\sx}{\mathsf{x}}
\newcommand{\sy}{\mathsf{y}}
\newcommand{\sk}{\mathsf{k}}
\newcommand{\sK}{\mathsf{K}}

\newcommand{\N}{\mathbb{N}}
\newcommand{\om}{\omega}

\newcommand{\bt}{\beta}

\newcommand{\dt}{\delta}
\newcommand{\lb}{\lambda}

\newcommand{\eps}{\varepsilon}
\newcommand{\T}{\mathbb{T}}

\def\bs{\boldsymbol}

\title[Arnold diffusion]{Arnold diffusion in nearly integrable Hamiltonian systems of arbitrary degrees of freedom}
\author{Chong-Qing Cheng and Jinxin Xue}
\address{Department of mathematics, Nanjing Univerisity, Nanjing 210093, China}
\email{chengcq@nju.edu.cn.}
\address{Department of Mathematics \& Yau Mathematical Sciences Center, Tsinghua University, Beijing 100084, China}
\email{jxue@mail.tsinghua.edu.cn}
\begin{document}
\maketitle
\begin{abstract} In this paper Arnold diffusion is proved to be a generic phenomenon in the smooth categroy for nearly integrable convex Hamiltonian systems with arbitrarily many degrees of freedom:
$$
H(x,y)=h(y)+\eps P(x,y), \qquad x\in\mathbb{T}^n,\ y\in\mathbb{R}^n,\quad n\geq 3.
$$
Under typical perturbation $\eps P$, the system admits ``connecting" orbit that passes through any finitely many prescribed small balls in the same energy level $H^{-1}(E)$ provided $E>\min h$.
\end{abstract}
\begin{spacing}{0.5}
\tableofcontents
\end{spacing}
\renewcommand\contentsname{Index}
\section{ Introduction}
\setcounter{equation}{0}
In this paper, we consider nearly integrable Hamiltonian systems of the form
\begin{equation}\label{EqHam}
H(x,y)=h(y)+\eps P(x,y),\quad (x,y)\in T^*\T^n,\quad n\geq 3.
\end{equation}
where $h$ is strictly convex, namely, the Hessian matrix $\frac{\partial^2h}{\partial y^2}$ is positive definite. It is also assumed that both $h$ and $P$  are $C^r$-function with $7\leq r\leq \infty$ and $\min h=0$.

The problem of studying the (in)stability of the above system $H$ was considered to be the fundamental problem of Hamiltonian dynamics by Poincar\'e. According to the celebrated KAM theorem, there exists a large measure Cantor set of Lagrangian tori on which the dynamics is conjugate to irrational rotations and the oscillation of the slow variable (or called action variable) $y$ is at most $O(\sqrt\eps)$. The KAM theorem also excludes the possibility of large oscillation of $y$ in the case of $n=2$ since each energy level, which is three dimensional, is laminated by two dimensional KAM tori and each orbit either stays on a KAM torus or is confined between two tori.

For $n\geq 3$, there does not exist topological obstruction for the slow variables $y$ to have $O(1)$ oscillation. Arnold was the first one who had realized such instability \cite{A63} and constructed the first example in \cite{A64} half a century ago
\beq\label{EqArnold}
H(I,\th,y,x,t)=\dfrac{I^2}{2}+\dfrac{y^2}{2}+\eps(\cos x+1)(1+\mu(\cos\th+\sin t)),
\eeq
where there are orbits giving rise to large oscillations of the action variable $I$. Although the perturbation is far from being typical, Arnold still proposed

\begin{Conjecture}[\cite{A66}] The ``general case" for a Hamiltonian system \eqref{EqHam} with $n\ge 3$ is represented by the situation that for an arbitrary pair of neighborhood of tori $y=y'$, $y=y''$, in one component of the level set $h(y)=h(y')$ there exists, for sufficiently small $\eps$, an orbit intersecting both neighborhoods.
\end{Conjecture}

In this paper, we prove the conjecture in the smooth category in the sense of cusp-residual genericity for nearly integrable convex Hamiltonian systems of $n\geq 3$ degrees of freedom. To state our result, let us introduction some notations and definitions.

By adding a constant to $H$ and introducing a translation $y\to y+y_0$, one can assume $\min h(y)=h(0)=0$. For $E>0$, let $H^{-1}(E)=\{(x,y):H(x,y)=E\}$ denote the energy level set, and $B\subset\mathbb{R}^n$ denote a ball in $\mathbb{R}^n$ such that $\bigcup_{E'\le E+1}h^{-1}(E')\subset B$. Let $\mathfrak{S}_a,\mathfrak{B}_a\subset C^r(\mathbb{T}^n\times B)$ denote a  sphere and a ball with radius $a>0$ respectively: $F\in\mathfrak{S}_a$ if and only $\|F\|_{C^r}=a$ and $F\in\mathfrak{B}_a$ if and only $\|F\|_{C^r}\le a$. They inherit the topology from $C^r(\mathbb{T}^n\times B)$. For a perturbation $P$ independent of $y$ (for instance, in classical mechanical systems), we use the same notation $\mathfrak{S}_a,\mathfrak{B}_a\subset C^r(\mathbb{T}^n)$ to denote a sphere and a ball with radius $a>0$.

\begin{defi}
Let $\mathfrak{R}_a$ be a set open-dense in $\mathfrak{S}_a$, each $P\in\mathfrak{R}_a$ is associated with a set $R_P$ residual in the interval $[0,a_P]$ with $a_P\le a$. A set $\mathfrak{C}_a$ is said \emph{cusp-residual} in $\mathfrak{B}_a$ if
$$
\mathfrak{C}_a=\{\lambda P:P\in\mathfrak{R}_a,\lambda\in R_P\}.
$$
\end{defi}
Let $\Phi_H^t$ denote the Hamiltonian flow determined by $H$. Given an initial value $(x,y)$, $\Phi_H^t(x,y)$ generates an orbit of the Hamiltonian flow $(x(t),y(t))$.  An orbit $(x(t),y(t))$ is said to {\it visit} $B_{\varrho}(y_0)\subset\mathbb{R}^n$ if there exists $t\in\mathbb{R}$ such that $y(t)\in B_{\varrho}(y_0)$ a ball centered at $y_0$ with radius $\varrho$. Our main theorem is as follows.
\begin{theo}
\label{ThmMain} Given any small $\varrho>0$, there exists $\eps_0$, such that given finitely many small balls $B_{\varrho}(y_i)\subset \mathbb{R}^n$, where $y_i\in h^{-1}(E)$ with $E>\min h$, there exists a cusp-residual set $\mathfrak{C}_{\eps_0}\subset C^r(\mathbb{T}^n\times B)$ with $7\leq r\leq\infty$ such that for each $\eps P\in\mathfrak{C}_{\eps_0}$, the Hamiltonian flow $\Phi_H^t$ admits orbits which visit the balls $B_{\varrho}(y_i)$ in any prescribed order. Moreover, the theorem still holds if we replace the function space $C^r(\mathbb{T}^n\times B)$ by $C^r(\T^n).$
\end{theo}

In recent years, it has become clear that Arnold diffusion is a typical phenomenon in so called {\it a priori} unstable systems, which are small perturbations of compound pendulum-single rotator system. There are many works studying this problem based on two streams of methods: the variational method (c.f. \cite{Be2,CY1,CY2,LC}) and the geometric method (c.f. \cite{DLS06,DLS13,Tr}). With the variational method, the genericity of perturbations was established in \cite{CY1,CY2}, which relies on the existence of a normally hyperbolic invariant cylinder (NHIC) of dimension two given by the {\it a priori} unstable condition, as well as a parametrization of all weak KAM solutions into a H\"older family. 

Nearly integrable Hamiltonian systems like \eqref{EqHam} are also called {\it a priori} stable systems. 
It was known to Arnold \cite{A66} that the first thing to do is to study the dynamics around double resonances. 
\begin{defi}A frequency $\omega(y)=\frac{\partial h}{\partial y}\neq 0$ is said to admit \emph{a resonance relation}, if there exists an integer vector $\bk\in \Z^n\setminus\{0\}$ such that $\langle\bk,\omega(y)\rangle=0$ at the point $y$. The number of linearly independent resonance relations is called the \emph{multiplicity} of the resonance. A nonzero frequency is called a \emph{complete resonance point} if the multiplicity is $n-1$.
\end{defi}
Away from strong double resonance, certain normally hyperbolic invariant cylinders can be found so that diffusing orbit can be constructed as in the \emph{a priori} unstable case. See \cite{BKZ} for the existence of NHICs $O(\eps^{1/4})$-away from double resonances and \cite{CZ1,C17a} for the NHICs $o(\sqrt\eps)$ away from double resonances. Without studying how to pass through these neighborhoods, it would be impossible to construct orbits which can drift for large scale. In \cite{C17b}, the first author analyzes dynamics around strong double resonances in details and discovers a mechanism of skirting around the strong double resonance (see Figure \ref{fig11}), hence proves the Arnold diffusion conjecture in the smooth category in the sense of cusp-residual genericity for nearly integrable convex Hamiltonian systems of three degrees of freedom (c.f. \cite{C17b,C19}). There is another mechanism suggested by Mather, the phase space dynamics of which in our understanding is to move along the NHIC with single homology class to the zeroth energy level, next to jump along a heteroclinic orbit to the hyperbolic fixed point corresponding to the double resonance, and then to jump along another heteroclinic orbit to another NHIC with opposite homology class. For NHIC with compound type homology class, an extra jump to a NHIC of single homology class is needed on energy levels slightly above zero. There are two groups of people working on details of this approach, for which we refer the readers to the preprints \cite{KZ1,GM,Mo}.

In the case of $n>3$, the main theme of this paper, we again find NHICs $\sqrt\eps$-away from the complete resonance (resonance with multiplicity $n-1$) and study the dynamics within a $\sqrt\eps$-neighborhood of the complete resonance. 
First, away from the complete resonance, using a scheme of reduction of order, we find two dimensional NHICs restricted to which the time-1 map of the system is a twist map and construct diffusing orbits as in {\it a priori} unstable systems. The scheme of order reduction shares some similarity with \cite{KZ3} appeared earlier than us, though our construction is straightforward and explicit. The idea of the scheme is to consider frequency path along which there are at least $(n-2)$ linearly independent resonant integer vectors $\bk',\bk'',\ldots,\bk^{(n-2)} \in \Z^n$ forming a hierarchy $|\bk^{(i)}|\ll|\bk^{(i+1)}| ,\  i = 1,\ldots,n-3$ except that for finitely many points, there are $(n-1)$ linearly independent resonant integer vectors forming such a hierarchy $|\bk^i|\ll|\bk^{i+1}| ,\  i\neq j$, with two vectors having comparable lengths. We show that for any two balls in the frequency space of a given energy level, there is such a frequency curve with a hierarchy structure shadowing a path with Diophantine property (Lemma \ref{LmAllDiop}).
 It is a folklore that Fourier modes in the span$_\Z\{\bk\}$ appears in the KAM normal form if $\bk$ is a resonance relation to the frequency $\omega(y)$ (Section \ref{sct:NormalForm}). In other words, ``resonance produces pendulum".
 The hierarchy structure allows us to treat the Fourier modes in span$_\Z\{\bk',\bk'',\ldots,\bk^{(i+1)}\}$ as a small perturbation of the subsystem depending only on Fourier modes in span$_\Z\{\bk',\bk'',\ldots,\bk^{(i)}\}$. NHICs can always be found in the latter by looking for the unstable equilibrium of the pendulum. With the persistence and  symplecticity of the NHICs (\cite{DLS08}), we restrict the Hamiltonian to the NHICs to get a system of less degrees of freedom. By repeated reduction of order utilizing the hierarchy structure, we eventually obtain two dimensional NHICs.
To construct diffusing orbit along the NHICs, we adapt the Arnold mechanism to the hierarchy structure to allow ``incomplete intersections". Namely, the (un)stable ``manifolds" of these Aubry-Mather sets do not always need to intersect {\it transversally} in order to implement Arnold's mechanism. Instead, sometimes it is enough for them to split along some but not all directions (see Appendix \ref{SVariation} for detailed formulations and proofs).

Second, the dynamics near the complete resonance, without knowing the existence of NHICs, is much more delicate. In particular, repeated order reductions are not allowed near complete resonance due to the lack of regularity of the NHICs after the first step of order reduction (generically only $C^{1+}$ by the theorem \ref{NHIM} of NHIM). The mechanisms of crossing the double resonance in the $n=3$ case are not sufficient to cross the complete resonance here. Indeed, when viewed in the space of cohomological classes, the two channels corresponding to two NHICs in the phase space that we would like to find orbits to connect typically have a misalignment in the extra dimensions so that they cannot be connected by the paths constructed in \cite{C17b} (Figure \ref{fig1}). To overcome this difficulty, we find a mechanism, which is essentially an autonomous version of Arnold's mechanism (Lemma \ref{ladder} and Remark \ref{RkLadder}), to bridge the channels complementary to the paths obtained by the mechanism of \cite{C17b} (the blue path of Figure \ref{fig1}). To implement the genericity argument of \cite{CY1,CY2}, we need to parametrize all the weak KAM solutions of a subsystem of two degrees of freedom on a fixed energy level into a H\"older family, which is done in a separate paper \cite{CX}.

\begin{figure}[htp] 
  \centering
  \includegraphics[width=7.8cm,height=4cm]{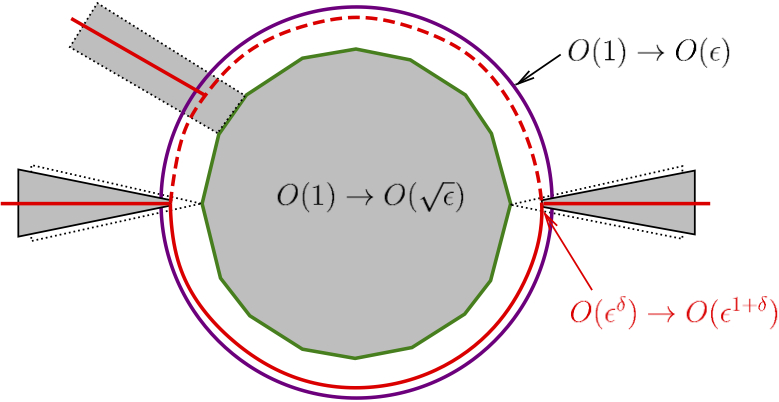}
  \caption{The $n=3$ case, (red) curves of cohomology equivalence}
  \label{fig11}
\end{figure}
\begin{figure}[htp] 
  \centering
  \includegraphics[width=8.8cm,height=3.5cm]{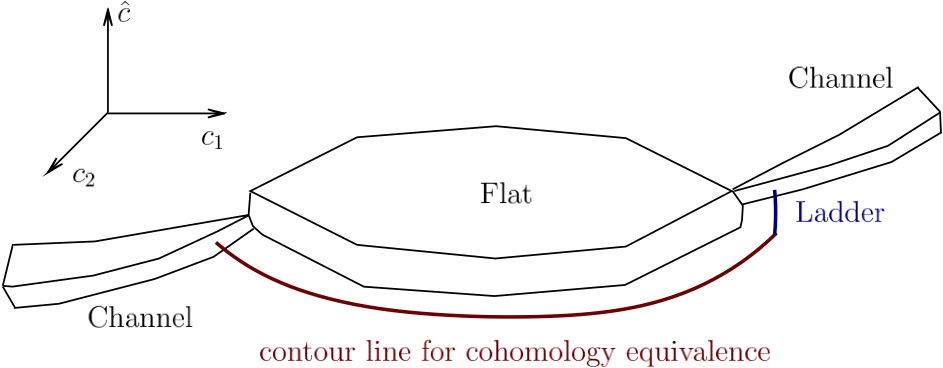}
  \caption{The $n>3$ case, the pizza and the ladder climbing}
  \label{fig1}
\end{figure}

To summarize, we have invoked and introduced the following mechanisms to construct diffusing orbits in this paper:
\begin{enumerate}
\item A mixture of Arnold (\cite{A64}) and Mather (\cite{M91}) mechanisms in the presence of NHIC as in {\it a priori} unstable systems whose genericity is established in \cite{CY1,CY2},
\item the mechanism of $c$-equivalence introduced by the first author in \cite{C17b} near strong double resonance,
\item a new mechanism which is autonomous version of (1) (see Lemma \ref{ladder} and Remark \ref{RkLadder}) for multiple resonances complementary to (2). 
\end{enumerate}

A main disadvantage of our proof is that the speed of the diffusing orbits gets slowed down as the dimension $n$ gets larger, which is unnatural considering statistical physics. This is because the argument relies crucially on the complete understanding of the Aubry-Mather sets in the two dimensional case and most part of our diffusing orbit shadows Aubry-Mather sets for twist maps. So we propose the following:

\begin{Opro}Find an effective proof of Arnold diffusion, which does not rely on the Aubry-Mather theory for twist maps and gives more abundant and faster diffusing orbits as $n$ gets larger. 
\end{Opro}

The paper is organized as follows.
\begin{enumerate}
\item In Section \ref{SOutline}, we prove the main theorem by introducing the main abstract framework and a main technical theorem.
\item In Section \ref{sct:NormalForm}, we explain the first approximation of the frequency path and prove a general KAM normal form.
\item In Section \ref{SSingle}, we perform the reduction of order in the single resonance regime to obtain a system of one less degree of freedom.
\item In Section \ref{SDouble}, we study dynamics around the strong double resonance. This part is a higher dimensional generalization of the results of \cite{C17a,C17b} to our setting. We show the existence of NHICs and a path of cohomological equivalence. Techniques such as shear coordinates transform and center straightening are introduced to facilitate the generalization and to make preparation for further order reduction.
\item In Section \ref{STriple}, we study the dynamics around triple resonance. This involves the second approximation of the frequency path and the second step of order reduction. The channel misalignment issue mentioned above already appears here, so we introduce our mechanism to overcome the issue. The result in this section completes the proof in the $n=4$ case. This section is the heart of the whole paper.
\item In Section \ref{SHierarchy}, we perform induction to generalize the constructions in the previous three sections to the case of multiple resonances. We show how to construct orbits crossing the complete resonance by repeating the argument in Section \ref{STriple}. 
\item In Section \ref{SSwitchMain} the proof of the main technical theorem \ref{ThmMainNHIC} is given.
\end{enumerate}
Finally, we have five appendices.
\begin{enumerate}
\item In Appendix \ref{AppMather}, we give a brief introduction to Mather theory and weak KAM theory.
\item In Appendix \ref{AppDiop}, we give the proof of Lemma \ref{LmAllDiop} on the existence of a Diophantine frequency path to be shadowed by a true frequency path.
\item In Appendix \ref{SNHICSpecial}, we present the proof of a special NHIM theorem adapted to our needs in this paper. 
\item In Appendix \ref{SVariation}, we give the variational construction of local and global connecting orbit shadowing generalized transition chain constructed in the main body of the paper. The new ingredients in this appendix are ``incomplete intersections" of the stable and unstable manifolds.
\item In Appendix \ref{AppGenericity}, we give the genericity argument.
\end{enumerate}
\section{Proof of the main theorem}\label{SOutline}
In this section, we prove the main theorem based on some propositions.

\subsection{Conventions and notations}
We first fix some standing conventions of notations for the rest of the paper. Please refer to Appendix \ref{AppMather} for a brief introduction of the Mather theory where more notations are given including Tonelli Hamiltonian $H$ and Lagrangian $L$, minimal measures, cohomology class $c$, rotation vector $h$, Mather set $\widetilde\cM(c)$ and $\widetilde\cM_h$, Aubry set $\widetilde \cA(c)$, Ma\~n\'e set $\widetilde \cN(c)$, $\al$ and $\beta$ functions, weak KAM solutions $u^\pm_c$ and barrier functions $B_c$, etc. 
\begin{defi}[$c$-minimal curve and  $c$-minimal orbit] Given cohomology class $c\in H^1(M,\R)$ where $M$ is a closed manifold, a curve $\gamma$: $\mathbb{R}\to M$ is called \emph{$c$-minimal} if for any curve $\xi$: $\mathbb{R}\to M$ and for any $t_0,t_1,t_1'\in\mathbb{R}$ with $t_1'=t_1\mod 1$ one has
$$
\int_{t_0}^{t_1}(L(\gamma(t),\dot\gamma(t),t)-\langle c,\dot\gamma(t)\rangle+\alpha(c))\,dt\le \int_{t_0}^{t'_1}(L(\xi(t),\dot\xi(t),t)-\langle c,\dot\xi(t)\rangle+\alpha(c))\,dt,
$$
where the Tonelli Lagrangian $L$ is assumed time-1-periodic: $L(\cdot,t)=L(\cdot,t+1)$. If a curve $\gamma$ is $c$-minimal, then $d\gamma:=(\gamma,\dot\gamma)$  is called a \emph{$c$-minimal orbit}.
\end{defi}

\begin{defi}[$\lambda g$-minimal periodic curve and $\lambda g$-minimal periodic orbit] Consider a Tonelli Lagrangian $L(x,\dot x)$ independent of time defined on $TM$. A periodic curve $\gamma$: $[0,\lambda^{-1}]\to M$ is associated with a class $[\gamma]=g\in H_1(M,\mathbb{Z})\setminus\{0\}$. It is called a \emph{$\lambda g$-minimal periodic curve} if one has
$$
\int_{0}^{\frac 1{\lambda}}L(\gamma(t),\dot\gamma(t))dt=\inf_{[\xi]=g}\int_{0}^{\frac 1{\lambda}} L(\xi(t),\dot\xi(t))dt.
$$
In this case $d\gamma=(\gamma,\dot\gamma)$  is called a \emph{ $\lambda g$-minimal orbit.}
\end{defi}

\begin{Not} We list some conventions and notations.
\begin{itemize}
\item $(${\it The vector norm}$)$ Our convention of using $|\cdot|$ as follows.
\begin{enumerate}
\item[$\ast$] It is the usual absolute value when applied to real or complex numbers.
\item[$\ast$] It is the $\ell_1$ norm when applied to an integer vector $\bk\in \Z^n$ which is a \emph{row vector}.
\item[$\ast$] It is the $\ell_\infty$ norm when applied to a frequency $\omega\in \R^n$ which is a \emph{column vector}.
\end{enumerate}
So we can write estimate $|\langle\bk,\omega\rangle|\leq |\bk|\cdot|\om|$.\\
We use $\|\cdot\|$ to denote the Euclidean norm. 

\item $(${\it the $hat$ notation}$)$ We fix the meaning of the $hat$ notation throughout the paper. For a vector $v=(v_1,\ldots,v_n)\in \R^n$, we use $\hat v_{n-i}$ to denote the vector $(v_{i+1},v_{i+2},\ldots,v_n) $ in $\R^{n-i}$ for $1<i<n$.
\item  $(${\it the $tilde$ notation}$)$ Dual to the $hat$ notation, we introduce the $tilde$ notation. For a vector $v=(v_1,\ldots,v_n)\in \R^n$, we use $\tilde v_i$ to denote the vector $(v_1,\ldots,v_i)\in \R^i,\ 3<i<n$. We omit the subscript $i$ if $i=2$.

\end{itemize}
\end{Not}

\subsection{The choice of the frequency path}\label{SSFrequency0}

\begin{defi} We say that a vector $v\in \R^{d},\ d\geq 1,$ is \emph{Diophantine}, if there exist $\al,\tau>0$ such that
\begin{equation}\label{EqDiop}
\l|\l\langle v, \bk\r\rangle\r|\geq \dfrac{\al}{\l| \bk\r|^{\tau}},\quad \forall\ \bk\in \Z^d\setminus\{0\}.
\end{equation}
We denote $v\in \mathrm{DC}(d,\al,\tau).$
\end{defi}

The next lemma gives us $n(M-1)$ segments of frequency paths with certain special Diophantine property which will be shadowed by the diffusing orbit.
\begin{lem}\label{LmAllDiop}
Given any $\varrho>0,\tau>n$ and any finitely many frequency vectors $\bs\omega_1,\ldots, \bs\omega_M\in \partial h(h^{-1}(E))$,\ $E>\min h,\ M>1$, there exist constant $\al>0$ and vectors
$$
\bs\omega_i^{*}=(\omega_{i,1}^{*},\ldots,\omega_{i,n}^{*})\in \partial h(h^{-1}(E))
$$
satisfying 
$
\l|\bs\omega_{i}-\bs\omega^{*}_i\r|<\varrho, \ i=1,\ldots,M,
$
and $$\bs\omega^{*}_{i,[j]}:=(\omega_{i+1,1}^{*},\ldots, \omega_{i+1,j}^{*},\omega_{i,j+1}^{*},\ldots,\omega_{i,n}^{*})\in\mathrm{DC}(n,\al,\tau)$$ for all $i=1,\ldots,M-1$ and $j=0,1,2,\ldots,n.$
\end{lem}
The proof of this lemma is given in Appendix B.

From the Diophantine vectors $\bs\omega^{*}_{i,[j]}$, we construct $n(M-1)$ frequency segments
$$\Omega_{i,[j]}(t)=\rho_{i,[j]}(t)\left(\omega_{i+1,1}^{*},\ldots, \omega_{i+1,j-1}^{*},t,\omega_{i,j+1}^{*},\ldots,\omega_{i,n}^{*}\right),\ t\in [\omega_{i,j}^{*},\omega_{i+1,j}^{*}],$$
$j=1,\ldots,n,\ i=1,\ldots,M-1,$ where the scalar multiple $\rho_{i,[j]}(t)$ is determined by requiring that the segment $\Omega_{i,[j]}$ lies on $ \partial h(h^{-1}(E))$. By the construction, the end point  of $\Omega_{i,[j]}$ agrees with the starting point of $\Omega_{i,[j+1]}$ (for $j<n$) and the end point of $\Omega_{i,[n]}$ agrees with the starting point of $\Omega_{i+1,[1]}$, for all $i=1,\ldots,M-1$.  So the segments concatenate into a connected curve in $ \partial h(h^{-1}(E))$ connecting $\bs\omega^*_1$ to $\bs\omega^*_M$ and passing by the points $\bs\omega^*_i,\ i=1,\ldots,M$ with the given order. We remark that it may happen that $\omega_{i,j}^{*}>\omega_{i+1,j}^{*}$ so the interval $[\omega_{i,j}^{*},\omega_{i+1,j}^{*}]$ is empty. If that happens, we use the interval $[\omega_{i+1,j}^{*},\omega_{i,j}^{*}]$ instead. 

The diffusing orbit will be constructed to shadow these frequency segments when projected to the frequency space.
\subsection{The abstract variational framework}
Roughly speaking, a generalized transition chain is such a path $\Gamma$: $[0,1]\to H^1(M,\mathbb{R})$ that, for any $s, s'\in[0,1]$ with $|s-s'|\ll 1$, the Aubry sets $\widetilde{\mathcal{A}}(\Gamma(s))$ and $\widetilde{\mathcal{A}}(\Gamma(s'))$ are connected by an orbit. An orbit $(\gamma,\dot\gamma)$ of the Euler-Lagrange flow $\phi^t_L$ is said to \emph{connect two Aubry sets} if the $\alpha$-limit set of the orbit is contained in one Aubry set and the $\omega$-limit set is contained in the other.

Let us formulate the definition of generalized transition chain for autonomous Hamiltonian $H:T^*M\to\mathbb{R}$ where $M=\mathbb{T}^n$ with $n\ge 3$.

\begin{defi}[Generalized transition chain: the autonomous case]\label{chaindef1}
Two cohomology classes $c,c'\in H^1(M,\mathbb{R})$ are said to be joined by a \emph{generalized transition chain} if there exists a continuous path $\Gamma$: $[0,1]\to H^1(M,\mathbb{R})$ such that $\Gamma(0)=c$, $\Gamma(1)=c'$, $\alpha(\Gamma(s))\equiv E>\min\alpha$ and for each $s\in [0,1]$ at least one of the following cases takes place:
\begin{enumerate}
\item[(H1)] In some finite covering manifold: $\check{\pi}:\check{M}\to M$ the Aubry set $\mathcal{A}(\Gamma(s))$ consists of two classes $\mathcal{A}_1(\Gamma(s))$ and $\mathcal{A}_2(\Gamma(s))$. There are two open domains $N_1$ and $N_{2}$ with $\bar N_1\cap\bar N_{2}=\varnothing$, a decomposition $\check M=M_1\times\mathbb{T}^{\ell}$, $(n-\ell-1)$-dimensional disks $\{O_m\subset M_1\}$ with $\bar O_m\cap\bar O_{m'}=\varnothing$, an $(n-1)$-dimensional disk $D_{s}$ and two small numbers $\delta_s,\delta'_s>0$ such that
\begin{enumerate}
    \item[(i)] the Aubry sets $\mathcal{A}_1(\Gamma(s))\subset N_1$, $\mathcal{A}_2(\Gamma(s))\subset N_2$ and $\mathcal{A}(\Gamma(s'))\subset (N_1\cup N_2)$ for each $|s'-s|<\delta_s$,
    \item[(ii)] $\check\pi\mathcal{N} (\Gamma(s),\check M)|_{D_{s}}\backslash (\mathcal{A}(\Gamma(s))+\delta'_s)$ is non-empty, of which each connected component is contained in $O_m\times\mathbb{T}^{\ell}$,
    \item[(iii)] $\langle\Gamma(s')-\Gamma(s),g\rangle=0$ holds for each $g\in H_1(\check M,M_1,\mathbb{R})$;
\end{enumerate}
\item[(H2)] For each $s'\in (s-\delta_s,s+\delta_s)$, the cohomology class $\Gamma(s')$ is equivalent to $\Gamma(s)$: some section $\Sigma_s$ and some small neighborhood $U$ of $\mathcal{N}(\Gamma(s))\cap \Sigma_{s}$ exist such that  $\langle\Gamma(s')-\Gamma(s),g\rangle=0$ holds for each $g\in H_1(U,\mathbb{Z})$. 
\end{enumerate}
\end{defi}
\begin{Rk}
 Item $(H1)$ with $\ell=0$ is a variational reformulation of Arnold's mechanism and $(H2)$ is also called the cohomological equivalence which was first introduced by Mather for nonautonomous systems in \cite{M93} and introduced by the first author and Li for autonomous systems in \cite{LC}.

If $\ell=0$, the assumption $(H1)\mathrm{(iii)}$ on the class $\Gamma(s')$ turns out to be trivial. The case $(H1)$ with $\ell>0$ is a generalization of Arnold's mechanism by allowing the stable and unstable sets of the Aubry sets to have incomplete intersection in the sense that the stable and unstable sets are allowed to merge in the $\T^\ell$ components and are only required to intersect transversally in the $M_1$ component.
\end{Rk}

Once such a generalized transition chain exists, one can construct diffusion orbits by variational method  (see Appendix \ref{SGlobalConn}). 
\begin{theo}[\cite{LC, C17b}, Appendix \ref{SGlobalConn}]\label{constructionthm1}
If $c$ is connected to $c'$ by a generalized transition chain $\Gamma$ as in Definition \ref{chaindef1}, then
\begin{enumerate}
\item there exists an orbit of the Lagrange flow  $d\gamma:=(\gamma,\dot\gamma)$: $\mathbb{R}\to TM$ which connects the Aubry set $\widetilde{\mathcal{A}}(c)$ to $\widetilde{\mathcal{A}}(c')$, namely, the $\al$-limit set $\alpha(d\gamma)\subseteq\widetilde{\mathcal{A}}(c)$ and the $\omega$-limit set $\omega(d\gamma)\subseteq\widetilde{\mathcal{A}}(c')$;
\item for any $c_1,c_2,\cdots, c_k\in\Gamma$ and arbitrarily small $\delta>0$, there exist times $t_1<t_2<\cdots<t_k$ such that the orbit $(\gamma,\dot\gamma)$  passes through $\delta$-neighborhood of the Aubry set $\widetilde{\mathcal{A}}(c_i)$ at the time $t=t_i$.
\end{enumerate}
\end{theo}
\subsection{Existence of the generalized transition chain}
We have the following more elaborate statement on the existence of generalized transition chain. The proof occupies the main body of the paper and is completed in Section \ref{SSProofMain}.
\begin{theo}\label{ThmMainNHIC} Let the Hamitonian system $H=h+\eps P\in C^r(T^*\T^n,\R),\ 7\leq r\leq \infty,$ be as in \eqref{EqHam} restricted to the energy level $E>\min h$. For any $\varrho>0$,  and any $M$ open balls $B_1,\ldots,$$B_M$ of radius $\varrho$ centered on $h^{-1}(E)$, there exist some $\eps_0>0$ and an open-dense set $\mathfrak R\subset \mathfrak S_1$, such that for each $P\in \mathfrak R$ there exist $\eps_P$ and a residual set $R_P\subset(0,\min\{\eps_P,\eps_0\})$ such that for all $\eps\in R_P$ the following hold.
\begin{enumerate}
\item   There exists a continuous frequency path $\omega(t)$ with $\partial\beta(\omega(t))\in \al^{-1}(E),\ t\in[1,M]$, lying in a $\varrho$-neighborhood of the union of frequency segments $\Omega_{i,[j]}\subset \partial h(h^{-1}(E))$, $i=1,\ldots,M-1,\ j=0,\ldots,n$, and satisfying
\begin{enumerate}
\item $(\partial h)^{-1}(\omega(i))\cap B_i\neq\emptyset,\ i=1,2,\ldots,M.$
\item Each point $\omega(t)$ is resonant with multiplicity at least $n-2$. There are finitely many marked points on $\omega(t)$ denoted by $\omega_1,\ldots, \omega_m$, where $m$ is independent of $\eps$, that are resonant with multiplicity $n-1$.
\end{enumerate}
\item  On the energy level $E$ there are finitely many disjoint $C^r$ normally hyperbolic weakly invariant cylinders $($wNHICs, see Section  \ref{SSDLS}$)$ homeomorphic to $T^*\T\times \T$.
\item For each $\omega_i,\ i=1,\ldots,m,$ there exists $\lambda_i>0$ such that \begin{enumerate}
\item  the Mather sets of rotation vectors $\omega(t)$ with $|\omega(t)-\omega_i|\geq \lambda_i\sqrt\eps$ for all $i=1,2,\ldots, m$, lie in the wNHICs;
\item any continuous curve lying in the interior of $\{\partial\beta(\omega(t))\ | \ |\omega(t)-\omega_i|\geq \lambda_i\sqrt\eps\}\subset \al^{-1}(E)$ is a generalized transition chain satisfying $(H1)$; and
 \item the two neighboring connected components $\{\partial\beta(\omega(t))\ | \ |\omega(t)-\omega_i|\geq \lambda_i\sqrt\eps\}$ $\subset \al^{-1}(E)$ near $\partial \beta(\omega_i)$ are joined by a generalized transition chain.
 \end{enumerate}

\end{enumerate}
\end{theo}
Next we explain how the main Theorem \ref{ThmMain} follows from this theorem. Indeed, given balls $B_1,\ldots,B_M$ of radius $\varrho$ centered on $h^{-1}(E)$, we first construct a frequency path $\omega(t),\ t\in[0,M]$ as stated. By item (3.b) and (3.c) of the above theorem, there exists a continuous curve of generalized transition chain visiting small neighborhoods of $\partial \beta(\omega(i))\subset \al^{-1}(E),\ i=1,\ldots,M.$ By Theorem \ref{constructionthm1}, we see that once a generalized transition chain is known to exist, an orbit can be constructed shadowing Aubry sets whose cohomology classes are on the chain. By item (1.a) such an orbit necessarily visit the two balls $B_1,\ldots,B_M$ as ordered. This proves Theorem \ref{ThmMain}.

The remaining part of the paper is devoted to proving Theorem \ref{ThmMainNHIC}. The proof is completed in Section \ref{SSProofMain}. The proof consists of mainly two parts. In the first part, we establish the existence of wNHICs away from complete resonances (part (2)). The generalized transition chain along the wNHICs (part (3.a), (3.b)) are constructed following the standard procedure for {\it a priori} unstable systems (\cite{CY1, CY2}, see Appendix \ref{AppGenericity}). In the second part, we construct generalized transition chains passing through the complete resonances and connecting nearby wNHICs (part (3.c)). 

\section{The frequency segment and the KAM normal form}\label{sct:NormalForm}
\setcounter{equation}{0}

In this section, we construct the frequency vectors with special number theoretic properties and derive Hamiltonian normal forms associated to such frequency vectors.

\subsection{Number theoretic properties of the frequency line}\label{ssct:Number}
\subsubsection{Single resonance} For given $\varrho,\tau >0$, let $\al$ be as in Lemma \ref{LmAllDiop}. We first study how to move along one frequency segment. 
Consider frequency segment $\omega_a\in\mathbb{R}^n$ of the form
\begin{equation}\label{EqFreq}
\omega_a=\rho_a\Big(a,\frac{P}{Q}\omega_2^*,\dfrac{p}{q}\omega_2^*,\hat\omega_{n-3}^*\Big)^t,\ P,Q,p,q\in \Z,\ a\in [\omega_1^{*i}-\varrho,\omega_1^{*f}+\varrho],\
\end{equation}
$\hat\omega_{n-3}^*=(\omega_3^*,\ldots,\omega_n^*)\in \mathrm{DC}(n-3,\al,\tau)$ and $\hat\omega_{n-2}^*=(\omega_2^*,\hat\omega_{n-3}^*)\in \mathrm{DC}(n-2,\al,\tau).$ We choose $\frac PQ$ and $\frac pq$ such that $|\frac PQ-1|<\varrho/2$ and $|\frac pq\omega_2^*-\omega_3^*|<\varrho/2 $ and in addition g.c.d.$(pQ,Pq)=1$.
The scalar $\rho_a$ does not influence the resonance relations. Since we know that $y$ lies on an energy level $E$ and since the energy hyper surface $h^{-1}(E)$ encloses a convex set containing the origin, the equation $h(\omega^{-1}(\omega_a))=E$, $\omega(y)=\partial h(y)$, determines uniquely $\rho_a$. For example, when $h(y)=\frac{1}{2}\Vert y\Vert^2$, we see easily that $\rho_a=\frac{\sqrt{2E}}{\l\Vert \l(a,\frac{P}{Q}\omega_2^*,\frac{p}{q}\omega_2^*,\hat\omega_{n-3}^*\r)\r\Vert}.$

Since we assume $\hat\omega_{n-2}^*\in$ DC$(n-2,\al,\tau)$, we have at most two resonances as $a$ varies in an interval. We always have a first resonance given by the integer vector
$$
\bk'=(0,Qp,-qP,\hat 0_{n-3}).
$$
The g.c.d. of all the components of $\bk'$ is 1. Then we have
\begin{equation}\label{EqSympZ}
\left[\begin{array}{llll}
1&0&0& \hat 0_{n-3}\\
0 &Qp&-qP& \hat 0_{n-3}\\
0&r&s&\hat 0_{n-3}\\
\hat 0_{n-3}& \hat 0_{n-3}& \hat 0_{n-3}&\mathrm{id}_{n-3}
\end{array}\right]\left[\begin{array}{c}
a\\
\frac{P}{Q}\omega_2^*\\
\frac{p}{q}\omega_2^*\\
\hat\omega^*_{n-3}\\
\end{array}\right]=\left[\begin{array}{c}
a\\
0\\
\frac{1}{qQ}\omega_2^*\\
\hat\omega^*_{n-3}\\
\end{array}\right]
\end{equation}
where $r,s$ are such that $s Qp+rqP=1.$ We denote the $n\times n$ matrix by $M'\in\mathrm{SL}(n,\Z)$.
\subsubsection{Double resonance, away from triple or more resonances}
In this section, we consider that the vector \eqref{EqFreq} at double resonance.
We fix some large number $K$ and define $\Z^n_K=\{\bk\in \Z^n\ |\ |\bk|< K\}$. As $a$ varies in an interval, we may encounter {\it double resonant} points
\[
\left\{\omega_a\ |\ \l\langle\bk,\omega_a\r\rangle=0,\ \mathrm{\ for\ some\ } \bk\in \Z_K^n\setminus\mathrm{span}_\Z\l\{\bk'\r\}\right\}.
\]
There are finitely many such double resonant points, whose number depends only on $K$.

In this paper, we consider only those resonant integer vectors that are irreducible. 
\begin{defi}
An integer vector $\bk\in \Z^n\setminus\{0\}$ is called \emph{irreducible} if its entries have no common divisor except 1.
\end{defi}

The next lemma shows that for fixed $K$, points along the frequency line $\omega_a$ are uniformly bounded away from triple or more resonances.
\begin{lem}\label{LmDouble}
Let an irreducible vector $\bk^o\in \Z_K^n\setminus\mathrm{span}_\Z\left\{\bk'\right\}$ be the second resonance of $\omega_a$, i.e.
$
\l\langle \bk^o,\omega_a\r\rangle=0
$
 at some point $a=a^o$.  Then for all
$
\bk\in \Z_K^n\setminus\mathrm{span}_\Z\left\{ \bk' ,\bk^o \right\},
$
we have the estimate
\begin{equation}\label{EqDiop1}
|\langle\bk,\omega_{a^o}\rangle|\geq\dfrac{\al\cdot \inf_a\rho_a}{2^\tau(qQ )^{\tau+1} (\Vert M'\Vert_\infty K)^{2\tau+1}}.
\end{equation}
\end{lem}
\begin{proof}
We use the linear transformation \eqref{EqSympZ} to convert $\omega_a$ to the vector
$$
\omega_a'=M'\omega_a=\rho_a\Big(a,0,\frac{1}{qQ}\omega_2^*,\hat\omega_{n-3}^*\Big)^t.
$$
Denote by $\tilde \bk^o=(\tilde k^o_{1},\tilde k^o_{2},\ldots,\tilde k^o_{n}):=\bk^oM'^{-1}$ so that we have
$$
0=\langle \bk^o,\omega_{a^o}\rangle=\langle \bk^oM'^{-1},M'\omega_{a^o}\rangle:=\langle \tilde\bk^o,\omega_{a^o}'\rangle.
$$
We have that $\tilde{k}_1^o\neq 0$ since otherwise $\langle \tilde\bk^o,\omega_a'\rangle=0$ for all $a$, which is impossible considering that $\hat\omega_{n-2}^*$ is Diophantine. We want to bound
$|\langle\bk,\omega_{a^o}\rangle|$ from below for all
$$
\bk=(k_1,k_2,\ldots,k_n)\in \Z_K^n\setminus\mathrm{span}_\Z\left\{ \bk' ,\bk^o \right\}.
$$
We denote $\tilde \bk=(\tilde k_1,\tilde k_2,\ldots,\tilde k_n):= \bk M'^{-1}$ to get
$\langle\bk,\omega_a\rangle=\l\langle \bk M'^{-1},M'\omega_a\r\rangle=\l\langle\tilde\bk,\omega_a'\r\rangle.$

We introduce a new vector $
\bar \bk=\tilde \bk-\dfrac{\tilde k_1}{ \tilde k^o_{1}}\tilde \bk^o=\dfrac{1}{\tilde k^o_{1}}(\tilde k^o_{1}\tilde \bk- \tilde k_1 \tilde \bk^o).
$
The new vector
\[
\tilde k^o_{1}\bar \bk=\tilde k^o_{1}\tilde \bk-\tilde k_1\tilde \bk^o:=\l(0,\bar k_2,\bar k_3,\hat{\bar \bk}_{n-3}\r)\in \Z^n
\]
has zero first entry. We introduce further a new vector
$
\bar{\bar \bk}=\left(0,\bar k_2, \bar k_3,qQ\hat{\bar \bk}_{n-3}\right)\in \Z^n.
$
We estimate the norm of $\bar{\bar \bk}$ as
\[
|\bar{\bar \bk}|\leq qQ|\tilde k^o_{1}\tilde \bk-\tilde k_1\tilde \bk^o|\leq 2qQ|\tilde \bk^o|\cdot|\tilde \bk|  \leq 2qQ (\Vert M'\Vert_\infty\cdot K)^2
\]
Using the Diophantine conditions and the fact that $\omega_a'$ has zero second entry, we have
\begin{equation}
\begin{aligned}
\l|\l\langle\bk,\omega_{a^o}\r\rangle\r|&=\l|\l\langle\tilde\bk,\omega_{a^o}'\r\rangle\r|=\l|\l\langle\bar\bk, \omega_{a^o}'\r\rangle\r|=\left|\dfrac{1}{\tilde k^o_{1}}\l\langle(\tilde k^o_{1}\tilde \bk-\tilde k_1 \tilde \bk^o),\omega_{a^o}'\r\rangle\right|\\
&=\dfrac{\rho_a}{\tilde k^o_{1} qQ}\l|\l\langle\bar{\bar \bk},(0,0,\omega_2^*,\hat\omega_{n-3}^*) \r\rangle\r| \geq\dfrac{\inf_a\rho_a}{\tilde k^o_{1}qQ}\dfrac{\al}{\l|\bar{\bar \bk}\r|^\tau}\\
&\geq\dfrac{\al \inf_a\rho_a}{2^\tau(qQ )^{\tau+1} (\Vert M'\Vert_\infty K)^{2\tau+1}}.
\end{aligned}
\end{equation}
\end{proof}
Finally, we have the following fact.
\begin{lem}\label{LmM''}
Let $\bk^o$ and $\omega_{a^o}$ be as in Lemma \ref{LmDouble}. Then there exists a matrix $M^o \in \mathrm{SL}(n,\Z)$ such that $M'':=M^oM'\in \mathrm{SL}(n,\Z)$ has the first row $\bk^o$ and the second row $\bk'$.
\end{lem}
\bpf
Denote $\om'_{a^o}=M'\om_{a^o}$ and $\tilde\bk^o=\bk^oM'^{-1}$. We have $\langle \bk^o,\om_{a^o}\rangle=\langle \bk^oM'^{-1}, \om'_{a^o}\rangle=0.$ We set the second entry of $\tilde \bk^o$ to be zero and treat it as a vector in $\Z^{n-1}$. We claim that we can find $n-2$ integer vectors in $\Z^{n-1}$ spanning unit volume together with $\tilde\bk^o$. Indeed, suppose without loss of generality, the first two entries $k_1,k_2$ of $\tilde \bk^o$ are nonzero and have common divisor 1. This is always possible after permutation of entries. Then using Euclidean algorithm, we find two numbers $s_1,s_2$ such that $k_1s_2-k_2s_1=1$. Extending $s_1,s_2$ by adding zeros to a vector in $\Z^{n-1}$ as the second row of the matrix and for the remaining rows, we put $1$'s on the diagonal and zeros off diagonal. This gives the desired matrix.

By adding a number 0 as their second entries, we extend these vectors to be $n$-dimensional and  put these vectors together to get an $n\times n$ matrix $M^o$ whose first row is $\tilde \bk^o:=\bk^o(M')^{-1}$, and second row is $(0,1,0,\ldots,0)$, and it satisfies the properties stated in the lemma.
\epf
\subsection{Resonant submanifolds and their neighborhoods}\label{SSNbd}
Let $\omega_a,\, \bk',\, \bk^o$ be as in Section \ref{ssct:Number}.
\begin{defi}
\begin{enumerate}
\item We define the \emph{single resonant sub-manifold} associated to the vector $\bk'$
\beq\label{EqSigma1}
\Sigma(\bk'):=\l\{y\in h^{-1}(E)\ |\ \l\langle\bk',\omega(y)\r\rangle=0\r\}.
\eeq

\item In the single resonant sub-manifold we define the \emph{double resonant sub-manifold} for the resonant vectors $\bk',\bk^o$
\begin{equation}\label{EqSigma2}
\Sigma(\bk',\bk^o):=\l\{y\in h^{-1}(E)\ |\ \l\langle\bk',\omega(y)\r\rangle=\l\langle\bk^o,\omega(y)\r\rangle=0\r\}.
\end{equation}
\end{enumerate}
\end{defi}
Next, we find a number $\mu$ as the size of the neighborhood of the single resonant manifold to apply the KAM normal forms. 
\begin{Not}
We use the notation $B(a,r)$ to denote a ball of radius $r$ centered at $a$ and the notation $B(A,r):=\cup_{a\in A}B(a,r)$ to denote the $r$-neighborhood of a set $A$.
\end{Not}
We denote $a^o_1,\ a^o_2,\ldots,a^o_m$ the list of points such that the corresponding frequency vector $\omega_{a^o}$ admits a second resonant vector $\bk^o_{a^o_i},\ i=1,2,\ldots,m.$ The total number of such $a^o$'s is bounded  if we require $|\bk^o|\leq K$.

\blm\label{Lm:nbd}
Let $\omega_a,\, K,\, \bk',\, \bk^o_{a^o_i},\ i=1,2,\ldots,m$ be as above. Let $(\bk^o_{a^o_i})^\perp$ be the $(n-1)$-dimensional space orthogonal to the vector $\bk^o_{a^o_i}$. Then there exists $\mu=\mu(K)$ such that
\ben
\item
for all $\omega$ in the neighborhood
$
B(\omega_a,\mu)\setminus \bigcup_{i}B\l(\omega_{a_i^o}+(\bk^o_{a^o_i})^\perp,\eps^{1/3}\r),
$
and for sufficiently small $\eps$ we have
\[|
\langle \bk,\omega\rangle|> \eps^{1/3},\quad \forall\ \bk\in \Z_K^n\setminus\mathrm{span}_\Z\{\bk'\}.
\]
\item for all $\omega$ in $B(\omega_a,\mu)\bigcap B\l(\omega_{a_i^o}+(\bk^o_{a^o_i})^\perp,\eps^{1/3}\r)$, and for all
    $
    \bk\in \Z_K^n\setminus\mathrm{span}_\Z\left\{ \bk' ,\bk_{a^o_i}^o \right\},
    $ $i=1,\ldots,m$,
we have
\begin{equation}\label{EqDiop1/2}
|\langle\bk,\omega\rangle|\geq nK\mu.
\end{equation}
\een
\elm
\bpf
Part (1). We consider two cases depending on if $\bk$ in the assumption is one of the double resonant vector $\bk^o_{a^o_i}$ or not.

First we suppose $\bk=\bk^o_{a^o_i}$ for some $i$, then we get
\beqa\nonumber
|\langle \bk,\omega\rangle|=|\langle \bk,\omega_{a_i^o}\rangle+\langle \bk,\omega-\omega_{a_i^o}\rangle|=|\langle \bk,\omega-\omega_{a_i^o}\rangle|.
\eeqa
By the assumption, the projection of $\omega-\omega_{a_i^o}$ to the vector $\bk^o_{a^o_i}$ has length at least $\eps^{1/3}.$
This completes the proof in the case $\bk=\bk^o_{a^o_i}$ for some $i$ since $|\bk|\geq 1$.

We define
\begin{equation}\label{EqKmu}
\mu=\frac{1}{2nK}\frac{\al\cdot \inf_a\rho_a}{2^\tau(qQ )^{\tau+1} (\Vert M'\Vert_\infty K)^{2\tau+1}}.
\end{equation}

Next, suppose $\bk\neq \bk^o_{a^o_i}$, $\forall\ i.$ Consider the case where the first entry $k_1$ of $\bk$ is 0. We have that the vector $\bk M'^{-1}$ has zero first entry and $M'\omega_a=(a,0,\frac{1}{qQ}\omega_2^*,\hat\omega_{n-3}^*)$ has zero second entry. We have the estimate
\begin{equation}
|\langle \bk, \omega_a\rangle|=|\langle \bk M'^{-1},M'\omega_a\rangle|\geq \dfrac{\al\cdot \inf_a\rho_a} {(qQ)^{\tau+1}(\|M'\|_\infty K)^{2\tau+1}}
\end{equation}
using the Diophantine property of $(\omega_2^*,\hat\omega_{n-3}^*)$. 
We get
\beqa
|\langle \bk, \omega^\dagger\rangle|&\geq |\langle \bk, \omega_a\rangle|-|\langle \bk, \omega^\dagger-\omega_a\rangle|\\
&\geq \dfrac{\al\cdot \inf_a\rho_a}{(qQ )^{\tau+1} (\|M'\|_\infty K)^{2\tau+1}}-nK\mu\\
&\geq \dfrac{\al\cdot \inf_a\rho_a}{2(qQ )^{\tau+1} (\|M'\|_\infty K)^{2\tau+1}}\gg \eps^{1/3}.
\eeqa
Next consider the case $k_1\neq 0$. We change the first entry $a$ of $\omega_a$ to $a^o:=a-\frac{\langle \bk,\omega_a\rangle}{k_1}$ to get another frequency vector $\omega_{a^o}$. We have by definition $\langle \bk,\omega_{a^o}\rangle=0.$ This contradicts to the assumption that $\bk\neq \bk^o_{a^o_i}$, $\forall\ i.$

Part (2). For given $\omega$ as assumed, we have $|\omega-\omega_{a_i^o}|\leq \mu$. As $\bk\in \Z_K^n\setminus\mathrm{span}_\Z\{\bk' ,\bk_{a_i^o}^o\}$, we have the following estimate
\beqa\nonumber
|\langle \bk,\omega\rangle|&=|\langle \bk,\omega_{a^o_i}\rangle+\langle \bk,\omega-\omega_{a_i^o}\rangle|\\
&\geq |\langle \bk,\omega_{a_i^o}\rangle|-|\langle \bk,\omega-\omega_{a_i^o}\rangle|\\
&\geq \dfrac{\al\cdot \inf_a\rho_a}{2^{\tau}(qQ )^{\tau+1} (\Vert M'\Vert_\infty K)^{2\tau+1}}-nK\mu\\
&\geq \dfrac{\al\cdot \inf_a\rho_a}{2^{\tau+1}(qQ )^{\tau+1} (\Vert M'\Vert_\infty K)^{2\tau+1}}
\eeqa
where in the second inequality, we apply Lemma \ref{LmDouble} and in the third inequality, we apply the definition of $\mu.$
\epf

\subsection{Homogenization}
We first introduce the $C^r$-norm as follows, $r\in \N$.
\begin{defi}
\begin{enumerate}
\item For a function $f(x,y)$ defined on a domain $\mathcal D\times \T^n$, we define the $C^r$ norm as
\[|f|_{C^r}:=\sup_{y\in \mathcal D}\left(\sum_{ \bk\in \Z^n} \sum_{|\al|+|\beta|\leq r}\left|\dfrac{\partial^{|\al|} f^\bk}{\partial y^\al}(y)\right|\(\l|k^\beta\r|+1\r)\right)\]
where $f^\bk$ is the $\bk$-th Fourier coefficient and we use the multi-index notation $x^\al=x_1^{\al_1}\cdots x_n^{\al_n},$ etc. for $\al=(\al_1,\al_2,\ldots,\al_n),\ \beta=(\bt_1,\bt_2,\ldots,\bt_n)\in \Z^n,\ \al_i,\ \bt_i\geq 0,\ i=1,2,\ldots,n$.
\item For a function $f(x)$ defined on a domain $\T^n$, the $C^r$ norm is defined by setting $\al=0$ in the previous item. Namely, 
\[|f|_{C^r}:=\sum_{ \bk\in \Z^n} \sum_{|\beta|\leq r}\left| f^\bk\right|\(\l|k^\beta\r|+1\r).\]
\item For a function $f$ defined on a domain $\mathcal D\subset \R^n$, the $C^r$ norm is standard
$$|f|_{C^r}:=\sup_{y\in \mathcal D}\sum_{|\beta|\leq r}\left|\dfrac{\partial^{|\beta|} f}{\partial y^\beta}(y)\right|.$$
\end{enumerate}
\end{defi}

\subsubsection{Covering a $\mu$-neighborhood $B(\omega_a,\mu)$ of the frequency line $\omega_a$}

Consider the $\mu$-neighborhood $B(\omega_a,\mu)$ of the frequency line $\omega_a$. In the space of action variables, its preimage under the frequency map $\omega$ is $\omega^{-1}(B(\omega_a,\mu))$. We fix a large constant $\Lambda>0$ and cover the set $\omega^{-1}(B(\omega_a,\mu))$ by balls of radius $\Lambda \sqrt\eps$. We choose the covering to be locally finite and the Lebesgue number of the covering to be $0.1 \Lambda\sqrt\eps$ so that any ball of radius $1/20 \Lambda\sqrt\eps$ lies entirely in the $\Lambda \sqrt\eps$-ball that it intersects.

\subsubsection{Homogenization}
Fix $y^\star\in h^{-1}(E)$. We introduce the homogenization operator
\beq\label{EqHom0}
\mathfrak{H}:\quad y-y^{\star}:=\sqrt{\eps}Y, \quad t=\tau/\sqrt{\eps},\quad H(x,y)=\eps\sH(x,Y),
\eeq
where $Y,\tau,\sH$ are the homogenized action variable, time and Hamiltonian respectively. The homogenization is done in the region $\|y-y^\star\|<\sqrt\eps\Lambda$ so that $\|Y\|<\Lambda$. The Hamiltonian becomes
\begin{align}\label{EqHomog1}
\sH(x,Y)=\dfrac{h(y^{\star})}{\eps}+\dfrac{1}{\sqrt\eps}\langle\omega^\star,Y \rangle+\dfrac{1}{2}\langle \sA Y,Y\rangle+
\sV(x)+\sP(x,\sqrt\eps Y),
\end{align}
where
\ben
\item $\frac{h(y^{\star})}{\eps}+\frac{1}{\sqrt\eps}\langle\omega^{\star},Y\rangle+\frac{1}{2}\langle AY,Y\rangle$ is the first three terms of the Taylor expansion of $h(y)$ around $y^{\star}$;
\item  $\omega^\star=\frac{\partial h}{\partial y}(y^{\star})$; 
\item $\sA=\frac{\partial^2 h}{\partial y^2}(y^{\star})$ is a positive definite constant matrix;
\item $\sV(x)=P(x,y^\star)$;
\item The term $\sP$ has a decomposition $\sP=\sP_I+\sP_{II}$ where
      \beqa\label{EqP1}
      \sP_I=&\frac 1{\eps}\Big(h(y^{\star}+\sqrt{\eps}Y)-h(y^{\star})-\sqrt{\eps}\langle\omega,Y\rangle -\frac {\eps}2\langle AY,Y\rangle\Big),\\
      =&\dfrac{\sqrt\eps}{6} \sum_{1\leq i,j,k\leq n}Y_iY_jY_k\int_0^1 \dfrac{\partial^3 h}{\partial y_i\partial y_j\partial y_k}(t\sqrt\eps Y+y^\star)t^2\,dt,\\
      \sP_{II}=&P(x,y^{\star}+\sqrt{\eps}Y)-P(x,y^{\star})\\
     =& \sqrt\eps \l\langle Y,\int_0^1\dfrac{\partial P}{\partial y}(x,t\sqrt\eps Y+y^\star)\,dt\r\rangle.
      \eeqa
\een

We have the following estimates
      \begin{equation}\label{EqOepsilon}
      \left|\dfrac{\partial^{|\al|+|\beta|}\sP_{II}}{\partial x^\al\partial Y^\beta}\right| \le C_{\beta,\Lambda} |P|_{C^r}\sqrt{\eps}^{|\beta|+1},\quad 0\leq |\beta|\leq r-1,
      \end{equation}
       \beq\label{EqPI} \l|\frac{\partial^\beta \sP_I}{\partial Y^\beta}\r|\leq C_{\beta,\Lambda}|h|_{C^r}\sqrt\eps^{|\beta|+1},\quad 0\leq |\beta|\leq r-3.\eeq
In the following, we assume that $|P|_{C^r}\leq 1$ and $|h|_{C^r}\leq 1$. 
\begin{Not}
We use the notation $|\cdot |_r$ to denote the $C^r$ norms with respect the variables $x,Y$ in the homogenized system. So we get $|\sP|_{r-3}\leq C_{r,\Lambda}(|P|_{C^r}+|h|_{C^r})\sqrt\eps.$
\end{Not}
\subsection{The KAM normal form}\label{SSNormalForm1}

In this section, we work out a general normal form. 
\begin{Not}
\begin{enumerate}
\item Given a collection of linearly independent irreducible integer vectors $\bk_1,\ldots,\bk_m\in \Z^n$, $m< n,$ and a function $f\in C^r(\T^n)$, we denote by $\Pi_{\bk_1,\ldots,\bk_m} f$ the function consisting of Fourier modes of $f$ in $\mathrm{span}_\Z\{\bk_1,\ldots,\bk_m\}$.
    \item We denote by $\Pi_{\bk_1,\ldots,\bk_m} C^r(\T^n)$ the space of $C^r$ functions on $\T^n$ consisting of Fourier modes in $\mathrm{span}_\Z\{\bk_1,\ldots,\bk_m\}$. Similarly for $\Pi_{\bk_1,\ldots,\bk_m} C^r(T^*\T^n)$.
    \end{enumerate}
\end{Not}
\begin{pro}\label{prop: codim1}
Let $\bk_1,\ldots,\bk_m$ be $m(<n)$ linearly independent irreducible integer vectors.
Given any small $\dt$, there exists $\eps_0=\eps_0(\dt,\Lambda)$ such that for all $\eps<\eps_0$, the following holds. Let $\omega^\star=\partial h(y^\star)$ satisfy the following,
\beq\label{EqFreqSingle}
|\langle  \bk,\omega^\star\rangle |>\eps^{1/3},\ \forall \ \bk\in \Z_K^n\setminus\mathrm{span}_\Z\{\bk_1,\ldots,\bk_m\},\  K=(\dt/3)^{-\frac{1}{2}}.
\eeq
Then there exists a symplectic transformation $\phi$ defined on $B(0,\Lambda)\times \T^n$ satisfying $|\phi-\mathrm{id}|_{{r}}=O(\eps^{1/6})$ and sending the Hamiltonian $\sH$ in equation \eqref{EqHomog1} to the following form
\begin{equation}\label{EqNormalForm}
\begin{aligned}
\sH\circ \phi(x,Y)=&\dfrac{1}{\sqrt{\eps}}\langle \omega^\star, Y\rangle+\dfrac{1}{2}\langle \sA Y,Y\rangle +\Pi_{\bk_1,\ldots,\bk_m}\sV+\dt \sR(x,Y)
\end{aligned}
\end{equation}
where
\begin{enumerate}
\item the remainder $\dt \sR(x,Y)=\dt \sR_I(x)+\dt \sR_{II}(x,Y),$ and $\dt\sR_I$ consists of all the Fourier modes of $\sV$ not in the set $\mathrm{span}_\Z\{\bk_1,\ldots,\bk_m\}\cup \Z^n_K$;
\item the remainders $\sR_I,\, \sR_{II}$ satisfy $|\sR_I|_{ r-2}\leq 1,\ |\sR_{II}|_{r-5}\leq 1$.
\end{enumerate}
\end{pro}

\begin{proof}
We decompose the Hamiltonian \eqref{EqHomog1} as follows
\begin{align*}
\sH=&\frac{1}{\sqrt{\eps}}\langle \omega^\star, Y\rangle+\dfrac{1}{2}\langle \sA Y,Y\rangle+\Pi_{\bk_1,\ldots,\bk_m}\sV+R_\leq ( x)+ R_>( x)+\sP(x,\sqrt{\eps}Y),
\end{align*}
where \begin{enumerate}
\item $R_\leq (x)+R_>(x)$ consists of all the Fourier modes of $\sV(x)$ in $\Z^n\setminus $span$_\Z\l\{  \bk_1,\ldots,\bk_m\r\}$. 
\item The Fourier modes with $\bk\in \Z_K^n\setminus\mathrm{span}_\Z\{\bk_1,\ldots,\bk_m\}$ are put in $R_\leq$ and those in $\Z^n\setminus ($span$_\Z\l\{  \bk_1,\ldots,\bk_m\r\}\cup \Z_K^N)$ are put in $R_>$. 
\end{enumerate}
We have the estimate $|R_>|_{{r-2}}\leq \dt$ since we have $ K=(\dt/3)^{-1/2}$. 

Only one step of KAM iteration is enough. We use a new Hamiltonian $\sqrt{\eps}F$ whose induced time-1 map $\phi_{\sqrt{\eps}F}^1$ gives rise to a symplectic transformation
\begin{align*}
{\sH}\circ \phi_{\sqrt{\eps}F}^1=&{\sH}+\sqrt{\eps}\{{\sH},F\}+\frac{\eps}{2}\int_0^1(1-t)\{\{{\sH}, F\},F\}(\Phi_{\sqrt{\eps}F}^t)\,dt\\
=&\dfrac{1}{\sqrt{\eps}}\langle\omega^\star, Y\rangle+\dfrac{1}{2}\langle \sA Y,Y\rangle+ \Pi_{\bk_1,\ldots,\bk_m}\sV+\Big\langle \omega^\star,\dfrac{\partial F}{\partial x}\Big\rangle\\
&+R_\leq (x)+ R_>(x)+\sP(x,\sqrt{\eps}Y)\\
&+\sqrt{\eps}\Big\langle \sA Y+\frac{\partial\sP}{\partial Y},\frac{\partial F}{\partial x}\Big\rangle+\dfrac{\eps}{2}\int_0^1(1-t)\{\{{\sH}, F\},F\}(\Phi_{\sqrt{\eps}F}^t)\,dt,
\end{align*}
where $F$ solves the cohomological equation
$ R_\leq ( x)+ \left\langle \omega^\star,\frac{\partial F}{\partial  x}\right\rangle=0.$

Notice that $F$ is a function of only $x$. Notice also $|P|_{C^r}\leq 1$ and $\sV(x)=P(x,y^\star)$, so we get
$\sqrt\eps|F|_{{r}}\leq \eps^{1/6}$ by solving the cohomological equation under the assumption \eqref{EqFreqSingle}.

Let $\delta\sR_I=R_>$, so we have $|\sR_I|_{r-2}\le 1$. Let
\begin{align*}
\dt\sR_{II}&=\sP(x,\sqrt{\eps}Y)+\sqrt{\eps}\Big\langle \sA Y+\frac{\partial\sP}{\partial Y},\frac{\partial F}{\partial x}\Big\rangle+\dfrac{\eps}{2}\int_0^1(1-t)\{\{{\sH}, F\},F\}\l(\phi_{\sqrt{\eps}F}^t\r)\,dt.
\end{align*}
We have
\begin{enumerate}
   \item $|\sP|_{r-3}\leq |\sP_{I}|_{r-3}+ |\sP_{II}|_{r-3}\le C_r\eps^{1/2}$ from formula \eqref{EqOepsilon} and \eqref{EqPI}.
   \item Using the derivative estimates of $F$ and the fact that $\|Y\|\leq \Lambda$ we find
   $$
   \Big|\sqrt{\eps}\Big\langle \sA Y+\frac{\partial\sP}{\partial Y},\frac{\partial F}{\partial x}\Big\rangle\Big|_{{r-4}}=O(\eps^{1/6}).
   $$
   \item Since $\{{\sH}, F\}=\Big\{ \frac{1}{\sqrt{\eps}}\langle\omega^\star, Y\rangle+\dfrac{1}{2}\langle \sA Y,Y\rangle+ \sV(x)+\sP,F\Big\}$, we find
   \[
   \l|\dfrac{\eps}{2}\int_0^1(1-t)\{\{{\sH}, F\},F\}\l(\phi_{\sqrt{\eps}F}^t\r)\,dt\r|_{{r-5}} =O(\eps^{1/3}).
   \]
\end{enumerate}
Therefore, we have $|\dt\sR_{II}|_{r-5}= O(\eps^{1/6})$ and can make the term $\dt\sR_{II}$ less than $\dt$ in the $C^{r-5}$ norm by decreasing $\eps$. The proof is now complete.
\end{proof}

\section{The reduction of order for single resonances}\label{SSingle}
In this section, we perform the reduction of order in the single resonance regime.
\setcounter{equation}{0}
\subsection{Normally hyperbolic invariant manifold for Hamiltonian system}\label{SSDLS}
In this section, we introduce the theory of normally hyperbolic invariant manifold (NHIM). 
We introduce the definition of the normally hyperbolic invariant manifold following \cite{DLS00, DLS08}.

\bdf\label{DefNHIM} Let $f: M \to M$ be a $C^r$-diffeomorphism on a smooth manifold $M$ with $r>1$.  Let $N \subset M$ be a submanifold invariant under $f$, $f(N) = N$. We say that $N$ is a \emph{normally hyperbolic invariant manifold(NHIM)} if there exist a constant $C > 0$, rates $0 <\lb <\mu^{-1} < 1$ and a splitting $T_x M = E_x^s \oplus E_x^u \oplus T_xN$ for every $x \in N$
in such a way that
\beqa\nonumber
v\in E_x^s\quad &\Leftrightarrow\quad  |Df^k(x)v| \leq C\lb^k |v|, \quad k \geq 0,\\
v\in E_x^u\quad &\Leftrightarrow \quad |Df^k(x)v| \leq C\lb^{|k|} |v|, \ \ k \leq 0,\\
v\in T_xN\quad &\Leftrightarrow \quad |Df^k(x)v| \leq C\mu^k |v|, \quad k \in \Z.
\eeqa
\edf
\begin{Not}\label{NotUniformNormal} In the following, we use the phrase ``with uniform normal hyperbolicity independent of $\eps$", which means that neither the normal Lyapunov exponents nor the splitting angle between $E^s$ and $E^u$ depends on $\eps$. 
\end{Not}

\begin{theo}[Theorem A.14 of \cite{DLS00}]\label{NHIM} Let $N_X\subset M$ - not necessarily compact - be normally hyperbolic invariant for the map $f_X$ generated by the vector field $X$, which is uniformly $C^r$ in a neighborhood $U$ of $N_X$ such that dist$(M \setminus U,N_X) > 0$. Let $f_Y$ be the $C^r$-map generated by another vector field $Y$ which is sufficiently close to $X$ in the $C^1$-topology. Then, we can find a manifold $N_Y$ which is normally hyperbolic for $Y$ and close to $N_X$ in the $C^{\min\{r,\left|\frac{\ln\lambda}{\ln\mu}\right|-\eps\}}$ topology, for any small $\eps$. The Lyapunov exponents for $N_Y$ are arbitrarily close to those of $N_X$ if $Y$ is sufficiently close to $X$ in the $C^1$ topology. The manifold $N_Y$ is the only $C^1$ manifold close to $N_X$ in the $C^0$ topology, and invariant under the flow of $Y$.
\end{theo}
We give a proof of the result in Appendix \ref{SNHICSpecial} in a special setting adapted to the need of the paper. 

When the normally hyperbolic flow is Hamiltonian, we have the following theorem saying that the restriction of the Hamiltonian system to the central manifold is also Hamiltonian with less number of degrees of freedom.
\bth[Theorem 23 and 26 of \cite{DLS08}]\label{NHIMHam} Suppose $M$ is endowed with a $($an exact$)$ symplectic form $\omega$
Let $f_\eps :M \to M$ be a $C^r$ family of Hamiltomorphisms, $r \geq 2$ preserving $\omega$. Assume that $N \subset M$ is
a normally hyperbolic invariant manifold for $f_0$ with rate $\lb,\mu$.
\ben
\item Then for sufficiently small $\eps,$ there exist $C^\ell$-families of diffeomorphisms $k_\eps,\ r_\eps$ with $\ell\leq \min\l\{r,\l|\frac{\ln\lb}{\ln\mu}\r|\r\}$,  satisfying $f_\eps\circ k_\eps=k_\eps\circ r_\eps$ where $k_\eps$ is the map such that $k_\eps (N)=N_\eps$ and  $r_\eps: N\to N$ is the restricted map on $N$ .
 \item
We denote by $\mathcal R_\eps$ the generating vector field corresponding to $r_\eps$ defined by $\frac{d}{d\eps}r_\eps=\mathcal R_\eps\circ r_\eps$. Then we have
 \ben
\item[$\bullet$] $k_\eps^*\omega=\omega_N$ is a $($an exact$)$ symplectic form on $N$. It is independent of $\eps$.
\item[$\bullet$] The vector field $\mathcal R_\eps$ is $($exactly$)$ Hamiltonian vector field with respect to $\omega_N$. Moreover, its $($global$)$ Hamiltonian is
$R_\eps =F_\eps\circ k_\eps$ where $F_\eps$ is the Hamiltonian for $f_\eps.$
\een
\een
\eth

In this paper, we will deal with submanifolds that might not be invariant. We introduce the following definition.
\begin{defi}
A piece of submanifold $N$ of $M$ is called \emph{weakly invariant} under the Hamiltonian flow of $H$, if the Hamiltonian flow is tangent to the manifold at each point of $N$. We use the abbreviation wNHIC to mean a weakly invariant normally hyperbolic cylinder.
\end{defi}

\subsection{Normally hyperbolic invariant cylinder (NHIC) around single resonances}\label{SSSingle}

We apply the normal form Proposition \ref{prop: codim1} to the case (1) of Lemma \ref{Lm:nbd}.
\begin{lem}\label{LmNormalForm1} Let $\omega_a,\mu(K)$ be as in Lemma \ref{Lm:nbd} where $ K=(\dt/3)^{-1/2}$ for a small $\dt$. Then there exists $\eps_1=\eps_1(\dt,\Lambda)$ such that for $\eps<\eps_1$ the following holds.  Let $\omega^\star\in B(\omega_a,\mu( K))\setminus \bigcup_{i}B\l(\omega_{a_i^o}+(\bk^o_{a^o_i})^\perp,\eps^{1/3}\r)$ be as in case $(1)$ of Lemma \ref{Lm:nbd}. Then there exists a symplectic transform $\phi$ defined on $B(0,\Lambda)\times \T^n$ that is $o_{\eps\to 0}(1)$ close to identity in the $C^{r}$ norm, such that
\begin{equation}\label{EqNormalForm1}
\begin{aligned}
\sH\circ  \phi(x,Y)=&\dfrac{1}{\sqrt{\eps}}\langle \omega^\star, Y\rangle+\dfrac{1}{2}\langle \sA Y,Y\rangle+ V(\langle \bk',x\rangle)+\dt\sR(x,Y),
\end{aligned}
\end{equation}
where
\begin{enumerate}
\item $V(\langle \bk',x\rangle)= \Pi_{\bk'}\sV$;
\item $\dt\sR(x,Y)=\dt\sR_I(x)+\dt\sR_{II}(x,Y)$, where $\sR_I$ consists of Fourier modes of $\sV$ not in $\mathrm{span}_\Z\{\bk'\}\cup \Z^n_K$, and we have $|\sR_I|_{{r-2}}\leq 1$ and $|\sR_{II}|_{{r-5}}\leq 1$.
\end{enumerate}
\end{lem}

Using Formula \eqref{EqSympZ}, we introduce a linear symplectic transformation denoted by $\mathfrak{M}':\ T^*\T^n\to T^*\T^n$,
$$
\mathfrak{M}'(x,Y)=(M'  x,(M' )^{-t}Y):=(x',Y').
$$

In \eqref{EqNormalForm1}, we choose $y^\star\in \Sigma(\bk')$ such that $\omega'^\star=M'\omega^\star$ has zero as the second entry.  Applying the symplectic transformation $\mathfrak M'$ to the normal form \eqref{EqNormalForm1}, we get the following system up to an additive constant
\begin{equation}\label{EqH'}
\begin{aligned}
\sH'_\dt:=\mathfrak M'^{-1*}\sH\circ \phi=\dfrac{1}{\sqrt{\eps}}\langle \omega'^\star, Y'\rangle+\dfrac{1}{2}\langle A Y',Y'\rangle +V\l(x'_2\r)+\dt R(x',Y'),
\end{aligned}
\end{equation}
where $A=M'\sA M'^t$ and $R(x',Y')=\mathfrak M'^{-1*}\sR(x,Y)$. 

We next cite a result from \cite{CZ1} in order to find NHICs in the system $\sH'_\dt$.
\begin{pro}[Theorem 3.1 of \cite{CZ1}] \label{PropCZ}

Let $F_\zeta\in C^r(\T^1,\mathbb R)$ with $r\ge4,\ \zeta \in [0,1],$ and $F_\zeta$ be Lipschitz in the parameter $\zeta$. Then, there exists an open-dense set $\mathfrak V\subset C^r(\T^1, \R)$ so that for each $V\in \mathfrak V$, it holds simultaneously for all $\zeta\in [0, 1]$ that the global max of $F_\zeta+V$ is non-degenerate. Moreover, given $V \in \mathfrak V$ there are finitely many $\zeta_i\in [0,1]$ such that $F_\zeta+V$ has only one global max for $\zeta\neq\zeta_i$ and has two global max if $\zeta=\zeta_i$.

\end{pro}


The next result establishes the existence of wNHICs. 
\begin{pro}\label{NHICSingle}There exists an open dense set $\mathcal O_1=\mathcal O_1(\bk')\subset \Pi_{\bk'}C^r(T^*\T^n)$, $r\geq 7$, such that for each $P\in C^r(T^*\T^n)$ with $\Pi_{\bk'}P\in \mathcal O_1$, there exists $\dt_1=\dt_1(\Pi_{\bk'}P)$ such that for all $0<\dt<\dt_1$, the system \eqref{EqNormalForm1} based at a point $y^\star\in \Sigma(\bk')$ and defined on $B(0,\Lambda)\times \T^n$ 
\begin{enumerate}
\item admits a $C^r$ wNHIC $\mathcal C(\bk') $ homeomorphic to $T^*\T^{n-1}$ with uniform normal hyperbolicity, independent of $\dt$ or $\eps$;
\item Mather sets with rotation vectors in $\{\eps^{-1/2}\omega(y^\star+\sqrt\eps Y),\quad \|Y\|\leq 0.9\Lambda\}$ and perpendicular to $\bk'$ lie inside $\cC(\bk')$.
\end{enumerate}
\end{pro}
\begin{proof}
We first apply Proposition \ref{PropCZ} to the function $\Pi_{\bk'}P$ along the segment $y\in \omega^{-1}(\omega_a)$ to get an open dense set $\mathcal O_1$ in $\Pi_{\bk'}C^r(T^*\T^n)$ such that that for each $y$, the function $\Pi_{\bk'}P(y,\cdot)\in \mathcal O_1$ admits a nondegenerate global max up to finitely many bifurcations where there are two nondegenerate global max. Let us now choose a $P$ with $\Pi_{\bk'}P\in \mathcal O_1$ and determine $V(x'_2)$ from $\Pi_{\bk'}P$ by applying the homogenization and Lemma \ref{LmNormalForm1} so that $V$ has nondegenerate global max.
 
In \eqref{EqH'}, we neglect the remainder $\dt R$ to get that the remaining system $$\sH'_0:=\frac{1}{\sqrt{\eps}}\langle \omega'^\star, Y'\rangle+\dfrac{1}{2}\langle A Y',Y'\rangle +V\l(x'_2\r)$$
admits a NHIC given by \beq\label{EqNHIC}\l\{\dot Y'_2=\frac{\partial V}{\partial x_2'}=0,\quad\dot x'_2=\frac{1}{2}\frac{\partial\langle AY',Y'\rangle }{\partial Y'_2}=\sum_{i=1}^n A_{2i} Y'_i=0\r\}. \eeq
The normal hyperbolicity depends only on $A$ and the second order derivative of $V$ at the global max, hence does not depend on $\eps$ or $\dt$. Restricted to the NHIC we get a system with one less degrees of freedom due to Theorem \ref{NHIMHam}. 

Let us now make preparation for the application of the theorem of NHIM. The system $\sH$ (without the linear transformation) is defined in a $\Lambda$-ball in the $Y$ variables since the homogenization is done in a $\Lambda\sqrt\eps$ ball. We introduce a $C^\infty$ bump function $\chi$ supported in $B(0,\Lambda)$ satisfying $\chi(Y)=1$ if $\|Y\|<0.95\Lambda$ and is zero for $\|Y\|>0.98\Lambda$. To apply the NHIM theorem, we replace the remainder $\dt \sR$ in \eqref{EqNormalForm1} by $\chi(Y)(\dt \sR)$. The modification vanishes the perturbation for in the region $\{\|Y\|\geq 0.98\Lambda\}$ so that the dynamics therein is integrable when restricted to the NHIC which is the unperturbed NHIC. We will show below how to apply the theorem of NHIM to obtain a NHIC for the modified system. Since the modified system agrees with the original system on $\{\|Y\|\leq 0.95\Lambda\}$, the NHIC for the modified system is indeed a wHNIC for the original system in the region $\{\|Y\|\leq 0.95\Lambda\}$.

We next apply the NHIM theorems \ref{NHIM} and \ref{NHIMHam}. However, there is a subtle point. In the Hamiltonian equations, the vector field in the center is fast $\dot{ x}=\frac{\omega^\star}{\sqrt \eps}+O(1)$. This is a nonstandard setting where the NHIM theorems are applicable. We present the statement and proof in Appendix \ref{SNHICSpecial}. The conclusions of the NHIM theorems still hold since the large term $\frac{\omega^\star}{\sqrt \eps}$ is constant, which does not contribute to the derivatives of the Hamiltonian flow, hence the normal hyperbolicity.  The perturbation $\dt\sR$ is $\dt$-small in the $C^{n-5}$ norm, so its perturbation to the Hamiltonian vector field is $\dt$-small in the $C^{n-6}$ norm. By assumption $r\geq 7$, and applying the NHIM theorem (Theorem \ref{NHICSpecial}) we get a NHIC which is $C^r$ and is $\dt$-close to the unperturbed one in the $C^{r-6}$-topology as the center Lyapunov exponents are zero. 

In this case, we apply Theorem \ref{NHIMHam} to restrict the system to the NHIC to get a Hamiltonian system with one degree of freedom less. Note that here the $\dt_1$ is determined by the normal hyperbolicity which comes from the second order derivative of $V$ at the global max, hence $\dt_1$ is determined by $\Pi_{\bk'}P.$

Finally, we study the oscillation of the action variables of orbits in the Mather set. First we know that for the modified system, all the Mather sets with cohomology classes $\|c\|\leq \Lambda$ and with rotation vectors perpendicular to $\bk'$ lie inside the NHIC, since these Mather sets necessarily lie in a small neighborhood of the NHIC if $\dt$ is small and if a Mather set does not lie on the NHIC, the normal hyperbolicity will push it away from the NHIC violating the invariance of Mather sets. We next show that within the NHIC, the action variables of orbits in the Mather sets has $O(\sqrt\dt)$ oscillation. Write the Lagrangian as $$\mathsf L_c(x,\dot x)=\frac{1}{2}\langle A^{-1}(\dot x-\eps^{-1/2}\omega^\star-c),(\dot x-\eps^{-1/2}\omega^\star-c)\rangle-V(x_2)-\dt \chi R-\frac{1}{2}\langle Ac,c\rangle+\al(c),$$ where $\mathsf L_c(x,\dot x):=\mathsf L(x,\dot x)-\langle c,\dot x\rangle+\al(c)$ and $$\mathsf L(x,\dot x):=\frac{1}{2}\langle A^{-1}(\dot x-\eps^{-1/2}\omega^\star),(\dot x-\eps^{-1/2}\omega^\star)\rangle-V(x_2)-\dt \chi R.$$
 Let $\mu$ be a measure in the Mather set of cohomology class $c$. Fix a large number $C$ and decompose $\mu=\mu_1+\mu_2$ such that supp$ \mu_1\subset \{\|\dot x-\eps^{-1/2}\omega^\star-c\|\leq C\sqrt\dt\}\times \T^n$ and supp$\mu_2$ lies in the complement. Denote by $\hat\mu_i=\frac{1}{m_i}\mu_i$, the normalization of $\mu_i$, where $m_i=\int d\mu_i$, $i=1,2$. So we get the action $A_c(\mu):=\int \mathsf L_c\,d\mu=m_1\int\mathsf  L_c\,d\hat \mu_1+m_2\int\mathsf  L_c\,d\hat \mu_2$.
 We always have $\int\mathsf  L_c\,d\hat \mu_1\geq 0$. For the second term, we have $$\frac{1}{2}\left\langle A^{-1}(\dot x-\eps^{-1/2}\omega^\star-c),(\dot x-\eps^{-1/2}\omega^\star-c)\right\rangle>C^2\|A\|^{-1}\dt/2$$ by the definition of $\mu_2$ and $|V(x_2)|_{\mathrm{supp}\mu_2}|\leq \ell\dt^2$ for some constant $\ell$, since the Mather set lies on the NHIC and the NHIC undergoes a $O(\dt)$ perturbation from the unperturbed one given by $x_2^*$, a nondegenerate global max of $V$. We denote by $\mu_0$ the Haar measure supported on the torus $\{\dot x=\eps^{-1/2}\omega^\star+c\}\times \{x_2=x_2^*\}$ and we have $A_c(\mu_0)=\int \mathsf L_c\,d\mu_0= -\frac{1}{2}\langle Ac,c\rangle+\al(c)+O(\dt)\geq 0$.  We also have $\sup |R|\leq 1$, so we conclude $$\int \mathsf L_c\,d\hat \mu_2-A_c(\mu_0)\geq \frac{1}{2} C^2\|A\|^{-1}\dt-\ell\dt^2-\dt.$$ Choose $C$ large and $\dt$ small, we find that $A_c(\mu)\geq m_2 A_c(\hat\mu_2)>0$ violating the definition of minimal measure. Part (2) of the proposition is proved since Mather sets intersecting the region $\{\|Y\|\leq 0.9\Lambda\}$ have to stay in $\{\|Y\|\leq 0.95\Lambda\}$ where the modified system agrees with the original system.

\end{proof}

\section{Dynamics around strong double resonances}\label{SDouble}
The number of double resonances depends on $\dt$. However, most of the double resonances are weak and can be treated as single resonances. The number of strong double resonances is independent of $\delta,\eps$.

\subsection{Distinguishing weak and strong double resonances}
We apply the normal form Proposition \ref{prop: codim1} to the case (2) of Lemma \ref{Lm:nbd} to obtain the following.
\begin{lem}\label{LmNormalForm2}
Let $\omega_a$ and $\mu(K)$ be as in Lemma \ref{Lm:nbd}, where $ K=(\dt/3)^{-1/2}$ for a small $\dt$. Then there exists $\eps_2=\eps_2(\dt,\Lambda)$ such that for $\eps<\eps_2$, the following holds. Let $\omega^\star\in B(\omega_a,\mu( K))\bigcap B\l(\omega_{a_i^o}+(\bk^o_{a^o_i})^\perp,\eps^{1/3}\r)$ be as in case $(2)$ of Lemma \ref{Lm:nbd}. Then there exists a symplectic transform $\phi$ defined on $\{|Y|\leq \Lambda\}\times \T^n$ that is $o_{\eps\to 0}(1)$ close to identity in the $C^{r}$ norm, such that
\begin{equation}\label{EqNormalForm2}
\begin{aligned}
\sH\circ  \phi(x,Y)=&\dfrac{1}{\sqrt{\eps}}\langle \omega^\star, Y\rangle+\dfrac{1}{2}\langle \sA Y,Y\rangle+V\l(\l\langle  \bk', x\r\rangle,\l\langle \bk^o_{a^o_i}, x\r\rangle\r)+\dt \sR(x,Y),
\end{aligned}
\end{equation} where
\begin{enumerate}
\item $V\l(\l\langle  \bk', x\r\rangle,\l\langle \bk^o_{a^o_i}, x\r\rangle\r)=\Pi_{\bk',\bk_{a^o_i}^o}\sV$;
\item $\dt\sR(x,Y)=\dt\sR_I(x)+\dt\sR_{II}(x,Y)$, where $\sR_I$ consists of Fourier modes of $\sV$ not in $\mathrm{span}_\Z\{\bk',\bk^o_{a^o_i}\}\cup \Z^n_K$, and we have $|\sR_I|_{{r-2}}\leq 1$ and $|\sR_{II}|_{{r-5}}\leq 1$.
\end{enumerate}
\end{lem}

We next give a criteria to distinguish weak and strong double resonances, where the former will be treated as single resonance while the latter will need special care.

Consider a double resonance associated to the vector $\bk^o=\bk^o_{a^o_i}.$ We decompose $\Pi_{\bk',\bk^o}\sV(x)$ in \eqref{EqNormalForm2} in Proposition \ref{LmNormalForm2} as
\begin{equation}\label{EqCriteria}
\Pi_{\bk',\bk^o}\sV(x)= Z'(\langle \bk',x\rangle)+ Z''(\langle \bk',x\rangle,\langle \bk^o,x\rangle)
\end{equation}
where $Z'$ includes all the Fourier harmonics in the span$\{\bk'\}$ and $Z''$ contains the rest.

Notice $Z''$ must contain at least one term with $\bk^o$. Since $\bk'$ does not depend on $\dt$, we get $|Z''|_{C^{r-2}}\leq \frac{C}{|\bk^o|^{2}}$ for some constant $C$ independent of $\dt$. We first treat $ Z''+\dt R$ as a perturbation to the truncated Hamiltonian $\frac{1}{\sqrt{\eps}}\langle \omega^\star, Y\rangle+\frac{1}{2}\langle \sA Y,Y\rangle+Z'(\langle \bk',x\rangle)$, which has a wNHIC following from exactly the same reasoning as Proposition \ref{NHICSingle}. There is a threshold denoted by $\bs\dt$ that is the maximal allowable $C^1$ norm of the perturbation for applying the NHIM Theorem (Appendix \ref{SNHICSpecial}) to the NHIC in the truncated Hamiltonian. The threshold $\bs\dt$ does not depend on $\dt,\eps$ so we get when $\bs\dt> 2\frac{C}{|\bk^o|^{2}}$, we treat the corresponding double resonance point as a single resonance, otherwise we call the point a {\bf strong double resonance point} and will focus on it in the next. The total number of strong double resonance points are bounded by $\l(\frac{2C}{\bs\dt}\r)^{n/r'}$ which is
independent of $\eps,\dt$ for given $P\in\mathcal O_1$.
\subsection{The shear transformation for strong double resonances}\label{SSSShear}
In the following, we work on the strong double resonances. 
\begin{Not}
We denote by $\Sigma_!(\bk',\bk^o)$ the double resonance submanifold determined by a \emph{strong} double resonance. 
\end{Not}

In the homogenization and Lemma \ref{LmNormalForm2}, we choose the base point $y^\star\in \Sigma(\bk')$ so that $\omega^\star=\omega(y^\star)\in \bk'^\perp$.  We introduce the matrix $M''\in \mathrm{SL}(n,\Z)$ in Lemma \ref{LmM''} whose first two rows are $\bk^o$ and $\bk'$ respectively, and introduce the linear symplectic transformation 
\beq\label{Eqx''Y''}
\mathfrak{M}'': T^*\T^n\to T^*\T^n,\quad(x,Y)\mapsto (M''x, M''^{-t}Y):=(x'',Y'').
\eeq
We also keep track of the frequency vector $\omega''_a=M'' \omega_a=(\nu(a),0,*,\ldots,*)$ where $\nu(a)$ satisfies $\nu(a^o)=0$, where $a^o$ is such that $\omega_{a^o}\in \Sigma_!(\bk',\bk^o)$. We get a Hamiltonian system
\begin{equation}\label{EqH''}
\begin{aligned}
\sH''_\dt:=&(\mathfrak{M}''^{-1})^*\sH\circ \phi=\dfrac{1}{\sqrt{\eps}}\langle \omega''_a, Y''\rangle+\dfrac{1}{2}\langle A''Y'',Y''\rangle+ V( x''_1, x''_2)+\dt \sR''
\end{aligned}
\end{equation}
by applying $\mathfrak{M}''$ term by term to \eqref{EqNormalForm2}. 

In the next lemma, we are going to introduce a linear symplectic transformation induced by a matrix in $\mathrm{SL}(2n,\R)$ but not in $\mathrm{SL}(2n,\Z)$ so that it is not a symplectic transformation on $T^*\T^n$. We introduce the following notation. 
\begin{Not}
Given a matrix $S\in \mathrm{SL}(n,\R)$, we denote by $\T^{n}_S$ the torus $\R^n/(S\Z^n)$ where $S\Z^n=\{S \bk\ |\ \bk\in \Z^n\}$.
\end{Not}

\blm\label{shear}
There is a linear symplectic transformation from $T^*\T^n\to T^*\T^n_S$, defined by $$\mathfrak S\mathfrak{M}'':\ (x,y)\mapsto (SM''x,(SM'')^{-t}y):=(\sx,\sy)\in T^*\T^n_S$$ where $S\in \mathrm{SL}(n,\R)$ is in \eqref{EqShear}, such that the Hamiltonian system $\sH\circ \phi$ in Proposition \ref{LmNormalForm2} is reduced to the following Hamiltonian defined on $(SM'')^{-t}B(0,\Lambda) \times \T_S^n\subset T^*\T_S^n$, up to an additive constant
\beqa\label{EqHamShear}
\begin{aligned}
\sH_{S,\dt}:=&(\mathfrak S\mathfrak{M}'')^{-1*}\sH\circ \phi =\tilde{\sG}(\tilde{\sx},\tilde{\sy})+\hat{\sG}(\hat{\sy}_{n-2})+\dt R(\sx,\sy),
\end{aligned}
\eeqa
where
\begin{equation}\label{EqHamHom2}
\begin{aligned}
\tilde{\sG}(\tilde{\sx},\tilde{\sy})=&\eps^{-1/2} \omega_{S,1}\sy_1+\dfrac{1}{2}\Big\langle \tilde A \tilde \sy,\tilde \sy\Big\rangle+V(\tilde \sx):\ T^*\T^2\to \R,\\
\hat{\sG}(\hat{\sy}_{n-2})=&\dfrac{1}{2}\l\langle   \hat{\sy}_{n-2},B \hat{\sy}_{n-2}\r\rangle
+\frac{1}{\sqrt{\eps}}\Big\langle\hat{\omega}_{S,n-2},\hat \sy_{n-2}\Big\rangle,
\end{aligned}
\end{equation}
where 
\begin{enumerate}
\item $\omega_S=SM''\omega_{a^o}=(\tilde\omega_{S},\hat{\omega}_{S,n-2})$ with $\omega_{S,2}=0$ since $y^\star\in \Sigma(\bk')$, and $\tilde\omega_S=(\omega_{S,1},\omega_{S,2})=0$ if $y^\star\in \Sigma_!(\bk',\bk^o)$. 
\item The two matrices $\tilde A$ and
$B=(\hat A-\breve A^t\tilde A^{-1}\breve A)$ are positive definite, where $\tilde A,\, \breve A,\, \hat A$ in $\R^{2^2},\, \R^{2\times(n-2)},\, \R^{(n-2)^2}$ respectively form the matrix
\begin{equation}\label{matrixA}
A=\left(\begin{matrix}
\tilde A & \breve A\\
\breve A^t & \hat A
\end{matrix}\right).
\end{equation}
\item
The remainder $R(\sx,\sy)=(\mathfrak S\mathfrak{M}'')^{-1*}\sR$ satisfies $|R|_{{r-5}}<C$ where the constant $C$ is determined by $M''$ and $S$ hence is independent of $\eps$ or $\dt$.
\end{enumerate}
\elm
\bpf

In the proof, for simplicity of notations and without causing confusion, we also remove the $''$ in \eqref{Eqx''Y''}.
Let us denote
\begin{equation}\label{EqG}
\sG(Y,x)=\dfrac{1}{\sqrt\eps}\langle\omega''_{a^o},Y\rangle+\dfrac{1}{2}\langle AY,Y\rangle+V(x_1,x_2).
\end{equation}
We write the matrix $A$ in block form of (\ref{matrixA}). We also denote $\tilde v=(v_1,v_2)$ as the first two entries of a vector $v\in \R^n.$
Next we have the following formal derivations
\begin{equation}\label{fanren}
\begin{aligned}
\sG(Y,x)=&\frac{1}{\sqrt{\eps}}\langle\omega,Y\rangle+\frac 12\langle AY,Y\rangle+V(\tilde x)\\
=&\dfrac{1}{2}\langle \tilde A \tilde Y,\tilde Y\rangle+\langle \tilde Y,\breve A\hat Y_{n-2}\rangle+V(\tilde x)+\frac{1}{\sqrt{\eps}}\langle\tilde\omega,\tilde Y\rangle\\
&+\dfrac{1}{2}\langle \hat{A}\hat{Y}_{n-2},\hat{Y}_{n-2}\rangle+\frac{1}{\sqrt{\eps}}\langle\hat\omega_{n-2},\hat Y_{n-2}\rangle\\
=&\dfrac{1}{2}\langle \tilde A (\tilde Y+\tilde A^{-1}\breve A \hat{Y}_{n-2}),(\tilde Y+\tilde A^{-1}\breve A \hat{Y}_{n-2})\rangle+V(\tilde x)+\frac{1}{\sqrt{\eps}}\langle\tilde\omega,\tilde Y\rangle \\
&-\dfrac{1}{2}\langle  \breve A \hat{Y}_{n-2},\tilde A^{-1}\breve A \hat{Y}_{n-2}\rangle
 +\dfrac{1}{2}\langle \hat{A}\hat{Y}_{n-2},\hat{Y}_{n-2}\rangle+\frac{1}{\sqrt{\eps}}\langle\hat\omega_{n-2},\hat Y_{n-2}\rangle\\
=&\dfrac{1}{2}\langle \tilde A (\tilde Y+\tilde A^{-1}\breve A \hat{Y}_{n-2}),(\tilde Y+\tilde A^{-1}\breve A \hat{Y}_{n-2})\rangle+V(\tilde x)+\frac{1}{\sqrt{\eps}}\langle\tilde\omega,\tilde Y\rangle\\
&+\dfrac{1}{2}\langle   \hat{Y}_{n-2},(\hat A-\breve A^t\tilde A^{-1}\breve A) \hat{Y}_{n-2}\rangle
+\eps^{-1/2}\langle\hat\omega_{n-2},\hat Y_{n-2}\rangle.
\end{aligned}
\end{equation}
We perform the following linear shear symplectic transformation denoted by $\mathfrak{S}$,
\begin{equation}\label{EqSympSepa}
\begin{aligned}
&\left[\begin{matrix}\tilde \sy\\
\hat \sy_{n-2}\end{matrix}\right]
=\left[\begin{matrix}
\mathrm{id}_2& \tilde A^{-1}\breve A\\0& \mathrm{id}_{n-2}\end{matrix}\right]
 \left[\begin{matrix}\tilde Y\\ \hat Y_{n-2}\end{matrix}\right],\quad\left[\begin{matrix}\tilde \sx\\ \hat \sx_{n-2}\end{matrix}\right]
=\left[\begin{matrix}
\mathrm{id}_2& 0\\-\breve A^t\tilde A^{-t}&\mathrm{id}_{n-2}
\end{matrix}\right]
\left[\begin{matrix}\tilde x\\ \hat x_{n-2}\end{matrix}\right]
\end{aligned}
\end{equation}
so that the homogenized system in the new coordinates is written in the form $\sG=\tilde \sG+\hat\sG$ stated in the lemma. Here the variables $\sx$ are local coordinates on $\T^n_S$ and can be viewed as global coordinates on the universal covering space $\R^n$.

We denote
\beq\label{EqShear}
S=\bmt{cc}
\mathrm{id}_2& 0\\
-\breve A^t\tilde A^{-t}&\mathrm{id}_{n-2}
\emt,\quad S^{-t}=\bmt{cc}
\mathrm{id}_2& \tilde A^{-1}\breve A\\
0&\mathrm{id}_{n-2}
\emt
\eeq
so that the above symplectic transformation $\mathfrak{S}$ simplifies to $\sx=Sx,\ \sy=S^{-t}Y$.

Since $A$ is positive definite and the linear symplectic transformation $\mathfrak{S}$ does not change the signature so we get both $\tilde A$ and $B=(\hat A-\breve A^t\tilde A^{-1}\breve A) $ are positive definite.

Notice the above matrix $S$ is identity in the $\tilde x$ component, hence the Hamiltonian $\tilde\sG$ depends on $\tilde\sx$ $\Z^2$-periodically. So $\tilde\sG$ is a Hamiltonian defined on $T^*\T^2. $
\epf
\begin{Rk}
This lemma implies that configuration space dynamics on $\T^n$ of the system $\sH''_\dt,\ \dt=0,$ in \eqref{EqH''} has a skew product structure. The base dynamics is given by the configuration space dynamics on $\T^2$ of $\tilde \sG:\ T^*\T^2\to \R$. Each fiber is a $\T^{n-2}$. The dynamics on each fiber at the point $\tilde x$ depends on the base point $\tilde x$ by equation \eqref{EqSympSepa}.
\end{Rk}

For $\omega^\star$ satisfying the assumption of Lemma \ref{LmNormalForm2}, we again distinguish two cases depending on if $\omega^\star$ is in $ B(\omega_a,\mu( K))\bigcap B\l(\omega_{a_i^o}+(\bk^o_{a^o_i})^\perp,\Lambda\eps^{1/2}\r)$ or not. If $\omega^\star$ lies in the set, then when choosing the covering defining the homogenization, we require $y^\star\in \Sigma_!(\bk',\bk^o)$ so that $\omega^\star=\omega(y^\star)$ is at strong double resonance. In the following, we will focus mainly on this case. The other case is easy and will be studied in Section \ref{SSSHighEnergy}.

\subsection{Hamiltonian systems of two degrees of freedom}\label{SSNHICDouble}
Suppose $y^\star\in \Sigma_!(\bk',\bk^o)$ so in \eqref{EqHamHom2}, the frequency $\tilde\omega_{S}=0$ and 
we have obtained a mechanical system \begin{equation}\label{EqtsG}\tilde{\sG}(\tilde{\sx},\tilde{\sy})=\frac{1}{2}\langle \tilde A \tilde \sy,\tilde \sy\rangle+V(\tilde \sx),\quad (\tilde \sx,\tilde \sy)\in T^*\T^2.\end{equation} We normalize $V$ such that $\max V=0.$
In this section, we give the main properties of this system quoted from \cite{CZ1,C17a,C17b}.
\bth[Proposition 2.1 of \cite{C17b}]\label{flatthm1}
Let $H:\ T^*\T^n\to \R$ be a Tonelli Hamiltonian.
Given a class $c_0\in H^1(\mathbb{T}^n,\mathbb{R})$, if the minimal measure is supported on a hyperbolic fixed point, then there exists an $n$-dimensional convex flat $\mathbb{F}_0\subset H^1(\mathbb{T}^n,\mathbb{R})$ containing $c_0$ such that this fixed point supports a $c$-minimal measure for all $c\in\mathbb{F}_0$.
\eth
\begin{figure}[htp] 
  \centering
  \includegraphics[width=7.5cm,height=4cm]{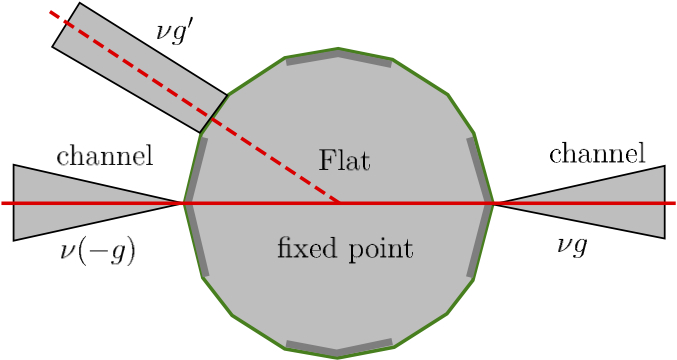}
  \caption{Two ways that the flat $\mathbb F_0$ connects to the channels}
 \label{fig 7}
\end{figure}
In the following, we specialize to the case of $n=2$.
\subsubsection{The NHIC in the low energy region}\label{SSSNHIC}

We cite the following result from \cite{CZ1}, which gives the existence of NHIC for low energy levels for Tonelli Hamiltonian systems of two degrees of freedom.
\bth[Theorem 2.1 of \cite{CZ1}] \label{ThmCZ}
Consider a $C^r,\ r\geq 5, $ Tonelli Hamiltonian $H:\ T^*\T^2\to \R$ normalized such that $\min\al_H=0$ by adding a constant. Given a class $g \in H_1(\T^2,\Z)$ and a closed interval $[E_-,E_+] \subset (0,\infty)$, there exists an open-dense set $\tilde{\mathcal O}_2(E_-,E_+) \subset C^r(\T^2)/\R,$ such that for each $V \in \tilde{\mathcal O}_2(E_-,E_+)$ normalized by adding a constant such that $\min\al_H=\min\al_{H+V}=0$, it holds simultaneously for all $E \in [E_-, E_+]$ that the Mather set $\widetilde\cM(E,g)$ on the energy level $E$ having the homology class $g$ for $H + V$ consists of hyperbolic periodic orbits. Moreover, except for finitely many $E_j \in [E_-,E_+]$ where the Mather set consists of two hyperbolic periodic orbits, for all other $E \in [E_-, E_+]$ the Mather set is exactly one hyperbolic periodic orbit.
\eth

In this paper, we will only apply this result to the homology class $g=(1,0)$. We denote by $\nu_\pm(1,0)$ the rotation vector of the Mather set on the energy levels $E_\pm$ with homology class $(1,0)\in H_1(\T^2,\Z)$. 

The next lemma shows that each hyperbolic periodic orbit corresponds to a one-dimensional flat in $H^1(\T^2,\R)$.
\begin{lem}\label{LmFlatPer}
Let $H(\tilde \sx,\tilde \sy):\ T^*\T^2\to \R$ be a Tonelli Hamiltonian and  $c^*\in H^1(\T^2,\R)$. We assume that the Mather set $\widetilde\cM(c^*)$ is supported on a hyperbolic periodic orbit with rotation vector $\nu g$ for $g\in H_1(\T^2,\Z)$ and $\nu\neq 0$. Then, the set $\partial\beta_{H}(\nu g)$ is an interval $\{c^*+s c_g\ |\ s_-\leq s\leq s_+\}\subset H^1(\mathbb{T}^2,\mathbb{R})$ with $s_-<s_+$, $s_-\leq 0\leq s_+$, $c_g\perp g$, and $\|c_g\|=1$ such that for each $c\in\{c^*+s c_g\ |\ s_-<s<s+\}$ we have
$
\widetilde \cA(c)=\widetilde\cM(c^*).
$
\end{lem}
\begin{proof}The proof is a variant of Proposition 2.1 of \cite{C17b}. 
As the system is autonomous with two degrees of freedom, $\partial\beta_{H}(\nu g)$ is either an interval or a point since $\partial\beta_{H}(\nu g)$ lies on an energy level $\al^{-1}(E)$, which is a closed curve. In the case of interval, some $c_g\in H^1(\mathbb{T}^2,\mathbb{R})$ exists such that $\partial\beta_{H}(\nu g)=\{c^*+s c_g\ |\ s_-\leq s\leq s_+\}$. It follows from \cite{Ms} that for all classes in the set $\{c^*+s c_g\ |\ s_-< s< s_+\}$, the Aubry sets $\widetilde \cA(c)$ are the same. Let us show that $s_-<s_+$ and $\widetilde \cA(c)=\widetilde\cM(c^*)$ for $c\in\{c^*+s c_g\ |\ s_-< s< s_+\}$.

Given any absolutely continuous curve $\gamma$, its Lagrange action is defined as follows
$$
A_c(\gamma)=\int L_{H}(\dot\gamma,\gamma)-\eta_c+\al_{H}(c)\,dt,\quad [\eta_c]=c.
$$
Denote by $\gamma_0$ the hyperbolic periodic orbit, we consider {\it minimal} homoclinic orbits to $\gamma_0$, which is located in the intersection of the stable and unstable manifolds of $(\dot\gamma_0,\gamma_0)$. A homoclinic orbit $(\dot\gamma,\gamma)$ is called {\it minimal} if the lift of $\gamma$, $\check{\gamma}$: $\mathbb{R} \to\check{M}$ is semi-static for the class $c^*$, where $\check{M}$ is the largest covering space of $\mathbb{T}^2$ so that $\pi_1(\check{M})=\pi_1(U)$ holds for each open neighborhood of $\mathcal{M}(c^*)$. Because of the topology of $\mathbb{T}^2$, there are only two types of minimal homoclinic orbits, denoted by $(\dot\gamma^{\pm},\gamma^{\pm})$. Given a point $x\in\gamma_0$, there are four sequences of time $t_{i,\pm}^{\pm}$ such that $\gamma^{-}(t_{i,-}^{\pm})\to x$ as $t_{i,-}^{\pm}\to\pm\infty$ and $\gamma^{+}(t_{i,+}^{\pm})\to x$ as $t_{i,+}^{\pm}\to\pm\infty$ and $t_{i,-}^{\pm}\to\pm\infty$ as $i\to\infty$. We define
\begin{align*}
A_{c}(\gamma^-,x)=&\liminf_{i\to\infty}\int_{t_{i,-}^{-}}^{t_{i,-}^{+}}\Big(L_{H}(\dot\gamma^-,\gamma^-) -\langle c,\dot\gamma^-\rangle +\al_{H}(c)\Big)dt\\
A_{c}(\gamma^+,x)=&\liminf_{i\to\infty}\int_{t_{i,+}^{-}}^{t_{i,+}^{+}}\Big(L_{H}(\dot\gamma^+,\gamma^+) -\langle c,\dot\gamma^+\rangle +\al_{H}(c)\Big)dt
\end{align*}
We obviously have $A_{c^*}(\gamma^\pm,x)\geq 0$. Next, we claim that
$$
A_{c^*}(\gamma^+,x)+A_{c^*}(\gamma^-,x)>0.
$$
Otherwise, we would have $A_{c^*}(\gamma^\pm)= 0$ for both $\pm$, which implies that $\gamma^\pm\subset \widetilde\cA(c^*)$. However, this violates the graph property of the Aubry set since in the first relative homology group $H_1(\mathbb{T}^2,\gamma_0,\mathbb{Z})$ we have $[\gamma^+]\neq [\gamma^-]$, when lifted to $\R^2$, the two curves $\gamma^\pm$ lying in the same strip bounded by two neighboring lifts of $\gamma_0$ hence the projections of $\gamma^\pm$ on $\T^2$ must intersect. The contradiction proves our claim. Let us assume $A_{c^*}(\gamma^+)>0$ without loss of generality.

Pick $\Delta c$ small enough and satisfying $$\langle\Delta c,[\gamma_0]\rangle=0,\  \langle\Delta c,[\gamma^{+}]\rangle>0\ \mathrm{and}\ A_{c^*}(\gamma^+)-\langle\Delta c,[\gamma^{+}]\rangle>0.$$ 
According to the upper semi-continuity of Ma\~n\'e set in cohomology class, any minimal measure $\mu_c$ is supported by a set lying in a small neighborhood of these homoclinic orbits if $c=c^*+\Delta c$ and $|\Delta c|$ is very small. By assumption $A_{c^*}(\gamma^+)>0$, it can only happen that $\mu_c$ is supported in a neighborhood of $\gamma^-\cup \gamma_0$.

We claim that the minimal measure $\mu_c$ for $c=c^*+\Delta c$ is still supported on the periodic orbit $\gamma_0$. First we show that $\rho(\mu_c)\parallel \rho(\mu_{c^*})\perp \Delta c$. Otherwise, since supp$(\mu_c)$ lies in the small neighborhood of $\gamma^-$, it follows that $-\langle\Delta c,\rho(\mu_c)\rangle>0$. On the other hand, as the $c^*$-minimal measure is uniquely supported on the periodic orbits, the $\beta$-function is strictly convex at $\rho(\mu_{c^*})$ hence the $\alpha$-function is differentiable at $c^*$ and $\rho(\mu_{c^*})=\nu[\gamma_0]$ hold for certain number $\nu\neq 0$. Therefore, we have $\alpha_{H}(c^*+\Delta c)-\alpha_{H}(c^*)=o(|\Delta c|)$. Consequently, we obtain from the definition that
\begin{align*}
A_{c}(\mu_c)=&\int (L_{H}-\eta_{c^*})d\mu_c+\alpha_{H}(c^*+\Delta c)-\langle\Delta c,\rho(\mu_c)\rangle\\
=&\int (L_{H}-\eta_{c^*})d\mu_c+\alpha_{H}(c^*)-\langle\Delta c,\rho(\mu_c)\rangle+o(|\Delta c|),
\end{align*}
from which we have $A_{c}(\mu_c)>0$ as $A_{c^*}(\mu_{c^*})\geq 0$, $-\langle\Delta c,\rho(\mu_c)\rangle>0$ and $o(|\Delta c|)$ is a higher order term of $|\Delta c|$. The contradiction implies that $\rho(\mu_c)\perp \Delta c$. Next, by the convexity of $\al$, we have $$\al(c)-\al(c^*)\geq \langle \Delta c,\rho(\mu_{c^*})\rangle=0,\ \mathrm{and\ }\al(c^*)-\al(c)\geq \langle -\Delta c,\rho(\mu_{c})\rangle=0,$$
so we have $\al(c)=\al(c^*)$. We get that the interval $c^*+s\Delta c,\ s\in[0,1]$ lies entirely on the energy level $\al(c^*)$, on which the Mather set in the homology class $g\in H_1(\T^2,\Z)$ is known to be the unique hyperbolic periodic orbit hence the rotation vector is constant for $c$ in the interval. Finally, from the proof we see that the curves $\gamma^\pm$ appears in the Aubry set only when the cohomology class lies on the endpoints of the interval. Otherwise, the Aubry set agrees with the Mather set being the periodic orbit. This completes the proof. 

\end{proof}

\begin{pro}\label{NHICDouble} Let $y^\star\in \Sigma_!(\bk',\bk^o)$ so that $\omega^\star=\omega(y^\star)$ is at strong double resonance with integer vectors $\bk'$ and $\bk^o$. Then for any $\lambda>0$,
there is an open dense set $\mathcal O_2=\mathcal O_2(\bk',\bk^o;\lambda,\Lambda)\subset \Pi_{\bk',\bk^o}C^r(\T^n)/\R,\ r\geq 7,$ such that for each $P$ with $\Pi_{\bk',\bk^o}P(x,y^\star)\in \mathcal O_2$ normalized such that $\max \Pi_{\bk',\bk^o}P(x,y^\star)=0$, 
there exists $\dt_2=\dt_2(\Pi_{\bk',\bk^o}P(x,y^\star),\lambda)>0$ such that for all $0<\dt<\dt_2$ and all $0<\eps<\eps_2(\dt,\Lambda)$ as in Proposition \ref{LmNormalForm2},  the following holds
\begin{enumerate}
\item the Hamiltonian system \eqref{EqNormalForm2} admits a wNHIC $\cC(\bk')$ homeomorphic to $T^*\T^{n-1}$, up to finitely many bifurcations,   entering a $\lambda$-neighborhood of $\Sigma_!(\bk',\bk^o)\times \T^n$;
\item the wNHIC has uniform normal hyperbolicity, independent of $\dt$ or $\eps$;
\item Mather sets lying in $B(0,0.9\Lambda)\times \T^n$ and with rotation vectors perpendicular to $\bk'$ and of distance $\lambda$-away from $-\eps^{-1/2}\omega^\star+(\bk^o)^\perp$, are contained in $\cC(\bk')$.
\end{enumerate}
\end{pro}
\begin{proof}
In system $\sH_{S,\dt} $\eqref{EqHamShear}, we first discard the $\dt$-perturbation and consider the system \eqref{EqG}
$\sH_{S,0}=\tilde \sG(\hat\sx,\hat\sy)+\hat \sG(\hat\sy_{n-2}):\ T^*\T^{n}_S\to \R$.

First, the system $\tilde \sG(\hat\sx,\hat\sy)$ admits a NHIC by Theorem \ref{ThmCZ} with homology class $(1,0)\in H_1(\T^2,\Z)$ for $V$ chosen in an open dense subset $\mathcal O_2(E_-,E_+)$ of $C^r(\T^2)/\R$, $r\geq 5$. Here we choose $E_-=\al_{\tilde\sG}(\partial\beta_{\tilde\sG}(\lambda(1,0)))$ and $E_+$ to be the highest possible energy level for $\|Y\|\leq \Lambda$. This gives the open-dense set $\mathcal O_2(\bk^o,\bk';\lambda,\Lambda)$, since $V$ is obtained from $\Pi_{\bk^o,\bk'}P(y^\star,x)$ after a linear transform. We next show that the system $\sH_{S,0}$ admits a wNHIC.  Indeed,
given a periodic orbit $\tilde \gamma=(\tilde\sx_E(t),\tilde\sy_E(t))$ of the system $\tilde\sG$ in the Mather set $\widetilde\cM(E)$, it gives rise to an orbit of the system $\sH_{S,0}$ $$(\tilde\sx_E(t), \hat\sx(0)+(\eps^{-1/2}\hat\omega^\star+B\hat \sy(0)) t,\tilde\sy_E(t),\hat \sy(0))\subset T^*\T^n_S,\ t\in \R.$$
Taking union over all the periodic orbits and all initial conditions $\hat \sx(0)\in (-\breve A^t\tilde A^{-t}\tilde \sx+\T^{n-2})$ and $\|\hat \sy(0)\|\leq \Lambda$, we get a NHIC for the system $\sH_{S,0}$ that is homeomorphic to $T^*\T^{n-1}_{\bar S}$ where $\bar S$ is obtained from $S$ by removing the second row and second column. Going back to the system \eqref{EqNormalForm2} with $\dt=0$ by inverting the symplectic transform $\mathfrak S\mathfrak M''$, we get a NHIC homeomorphic to $T^*\T^{n-1}$.

Due to the uniform hyperbolcity, when the $\dt$-perturbation in \eqref{EqHamShear} is turned on, we get the persistence of the wNHIC as we did in the proof of Theorem \ref{NHICSingle}. Here the modification of the $\dt\sR$ should be done as follows in addition to that used in the proof of Theorem \ref{NHICSingle} in order to smoothen the Hamiltonian in the region of $0\leq\tilde \sG(\hat\sx,\hat\sy)<E_-$. We introduce a $C^\infty$ monotone cut-off function $\rho:\ [0,\infty)\to [0,1]$ satisfying $\rho(x)=0$ for $x\leq 1/3$ and $\rho(x)=1$ if $x>2/3$. We next modify $\dt \sR$ to $\rho(\frac{\tilde \sG(\hat\sx,\hat\sy)}{\al_{\tilde\sG}(\partial\beta_{\tilde\sG}(\lambda(1,0)))}))\chi(\|Y\|/\Lambda)\dt\sR$.
Now the Proposition follows from the same argument as Proposition \ref{NHICSingle}.
\end{proof}

\subsection{The high energy regime}\label{SSSHighEnergy}

In this section, we give the existence of NHIC in $B(y^\star,\Lambda\sqrt\eps)\times \T^n$ for $y^\star\in \Sigma(\bk')$ and\begin{equation}\label{Eqomegastar}\omega^\star=\omega(y^\star)\in  B(\omega_a,\mu)\bigcap \left(B\l(\omega_{a_i^o}+(\bk^o_{a^o_i})^\perp,\eps^{1/3}\r)\setminus B\l(\omega_{a_i^o}+(\bk^o_{a^o_i})^\perp,\Lambda\eps^{1/2}\r)\right)\end{equation}
that is $\Lambda\eps^{1/2}$-away from but $\eps^{1/3}$-close to strong double resonance. 

We first cite a result from \cite{C17a} concerning the high energy regime of the system $\tilde\sG$ in \eqref{EqtsG}. Without loss of generality we fix the homology class $g=(1,0)\in H_1(\T^2,\R)$. In the system $\tilde \sG$ in \eqref{EqtsG}, we define $$[V](\sx_2)=\int_{\T^1}V(\sx_1,\sx_2)\,d\sx_1.$$ Suppose $[V]$ has a unique nondegenerate global max at a point denoted by $\sx_2^*$, which is a $C^2$ open-dense condition.

\begin{theo}[Theorem 3.1 and Proposition 3.1 of \cite{C17a}]\label{LmHighEnergy}Suppose the potential $V$ of the system $\tilde\sG$ in \eqref{EqtsG} satisfies that $[V]$ has a unique nondegenerate global max at $\sx_2^*$.
Then there exists $E_*>0$, such that
\begin{enumerate}
\item the action minimizing periodic orbits in the homology class $g=(1,0)\in H_1(\T^2,\Z)$ on the energy levels $\{E>E_*\}$, form a unique $C^r$ NHIC homeomorphic to $T^*\T$ with uniform normal hyperbolicity.
\item As $E\to\infty$, the periodic orbit $\{(\tilde \sx_E(t),\tilde \sy_E(t))\} $ on the energy level $E$  has the following uniform convergence: $\sx_{2,E}(t)\to \sx_2^*$, and $\dot{\sx}_{2,E}(t)\to 0$. 
\end{enumerate}
\end{theo}
By the reversibility of the system $\tilde\sG$, the same conclusion holds for the homology class $g=(-1,0)$. In fact the periodic orbit in the Mather set $\widetilde\cM_{\nu (-1,0)}$ is the time reversal of $\widetilde\cM_{\nu(1,0)}$.

Here we only sketch the proof and the details can be found in Theorem 3.1 and Proposition 3.1 of \cite{C17a}.
 \begin{proof}[Sketch of proof] The main idea of the proof is that on the high energy level the fast oscillation in the $\sx_1$ component will effectively averge out the dependence on $\sx_1$ in $V$, so the Hamiltonian system is effectively $\frac{1}{2}\langle \tilde\sy,\tilde A\tilde\sy\rangle+ [V](\sx_2)$ as $E\to \infty$. So we get that the normal hyperbolicity is determined by $\tilde A$ and the second order derivative $[V]''(\sx_2^*)$ hence is independent of the energy levels. The genericity assumption on $V$ is to guarantee that $[V](\sx_2)$ has a nondegenerate global max. 
\end{proof}

With this theorem, we obtain the following existence of NHICs in the high energy regime. 

\begin{pro}\label{NHICHighEnergy}

Let $P\in \mathcal O_1$ and $\dt_1$ be as in Proposition \ref{NHICSingle}. Then there exists $\Lambda_*$, such that for all $\Lambda>\Lambda_*$ and $y^\star$ be such that $\omega^\star=\omega(y^\star)$ is as in \eqref{Eqomegastar}, all $0<\dt<\dt_1$ and all $0<\eps<\eps_2(\dt,\Lambda)$, the Hamiltonian system \eqref {EqNormalForm2}  defined in $B(0,\Lambda)\times \T^n$ admits a $C^r$ wNHIC $\cC(\bk')$ homeomorphic to $T^*\T^{n-1}$ with the following properties:
\begin{enumerate}
\item the normal hyperbolicity is uniform, independent of $\Lambda,\dt$ or $\eps$;
\item Mather sets lying in $\{\|Y\|\leq 0.9\Lambda\}\times \T^n$ with rotation vectors perpendicular to $\bk'$  lie inside $\cC(\bk')$.
\end{enumerate}

\end{pro}

\begin{proof}
We first show how to convert this case to the previous result on the high energy regime. In \eqref{EqtsG}, we consider $\tilde\sy^\star$ and $\Lambda_*$ such that $\tilde A\tilde \sy^\star=\nu(1,0)$ for some large $\nu$ with $\nu \|\tilde A^{-1}(1,0)\|>\Lambda_*$ and $\frac{1}{2}\|\tilde A^{-1}\|^{-1}\Lambda_*^2+\min V\geq E_*$, and introduce $\tilde\sy-\tilde\sy^\star=\tilde \sY$. In the coordinates $(\tilde\sx,\tilde\sY)$, the Hamiltonian becomes $$\tilde \sG(\tilde\sx,\tilde\sY)=\frac{1}{2}\langle  \tilde A\tilde\sy^\star,\tilde\sy^\star\rangle+\nu\tilde\sY_1+\frac{1}{2}\langle  \tilde A\tilde\sY,\tilde\sY\rangle+V(\tilde\sx).$$
This means that the Hamiltonian $\tilde\sG$ in \eqref{EqHamHom2} with a linear term in $Y_1$ with large $\omega_{S,1}$ can be considered as the high energy regime of the Hamiltonian $\tilde\sG$ in \eqref{EqtsG}. By the previous Theorem \ref{LmHighEnergy}, we get the existence of NHIC in the system $\tilde\sG$ in \eqref{EqHamHom2}. By the same argument as the proof of Proposition \ref{NHICDouble}, we get the existence of the wNHIC in the original system \eqref{EqNormalForm2}. 

The assumption in Theorem \ref{LmHighEnergy} on the nondegeneracy of $[V]$ turns out to be the nondegeneracy of the global max of $\Pi_{\bk'}P(x,y^\star)$ which is guaranteed by $\Pi_{\bk'}P\in \mathcal O_1\subset \Pi_{\bk'}C^r(T^*\T^n)$. The remaining statements are proved in the same way as 
Proposition \ref{NHICSingle}.

\end{proof}

This result tells us that this high energy regime can be treated in the same way as the single resonance regime in Proposition \ref{NHICSingle}. 
\subsection{Cohomology equivalence around strong double resonances}
In this section, we first recall the main result of \cite{C17b} on the existence of cohomological equivalence for Hamiltonian systems of two degrees of freedom near the zero energy level. Next, we generalize it to the full system to build a piece of transition chain.
\subsubsection{Cohomological equivalence for the subsystem of two degrees of freedom}
The following theorem is one of the main result in \cite{C17b}.
\begin{theo}[Theorem 3.1 of \cite{C17b}]\label{beltthm1}
 There is an open-dense set $\tilde{\mathcal O}_3\subset C^r(\T^2)/\R,\ r\geq 2,$  such that for each $V\in \tilde{\mathcal O}_3$ normalized by $\max V=0$,  for each $c\in\partial\tilde{\mathbb{F}}_0$, where $\tilde{\mathbb{F}}_0=\al^{-1}(\min \al)$ is the flat of the $\al$-function for $\tilde\sG$, the Ma\~n\'e set  $\cN(c)$ does not cover the whole configuration space $\mathbb{T}^2$, i.e. $\cN(c)\subsetneq \T^2$.
\end{theo}
\begin{Rk} Theorem 3.1 of \cite{C17b} gives only a residual set. The openness of the set follows immediately applying the upper-semi-continuity of the Ma\~n\'e set. 
\end{Rk}
This theorem allows us to construct orbit connecting two Aubry sets $\widetilde\cA(c)$ and $\widetilde\cA(c')$ for any $c$ and $c'$ in $\partial \tilde{\mathbb{F}}_0$. To state the result, we need the following notion of $c$-equivalence defined for Tonelli Hamiltonian $H:\ T^*M\to \R$.

\begin{defi} We call $\Sigma$ non-degenerately embedded  $(n-1)$-dimensional torus by assuming a smooth injection $\varphi$: $\mathbb{T}^{n-1}\to M$ such that $\Sigma$ is the image of $\varphi$, and the induced map $\varphi_*$: $H_1(\mathbb{T}^{n-1}, \mathbb{Z})\hookrightarrow H_1(M,\mathbb{Z})$ is an injection.
\end{defi}

For each class $c\in H^1(M,\R) $, we assume that there exists a non-degenerate embedded $(n-1)$-dimensional torus $\Sigma_c\subset M$ such that each $c$-semi static curve $\gamma$ transversally intersects $\Sigma_c$. Let
$$
\mathbb{V}_{c}=\bigcap_U\{i_{U*}H_1(U,\mathbb{R}): U\, \text{\rm is a neighborhood of}\, \mathcal {N}(c) \cap\Sigma_c\mathrm{\ in\ }\T^n\},
$$
where $i_U$: $U\to M$ denotes inclusion map. Define $\mathbb{V}_{c}^{\bot}$ to be the annihilator of $\mathbb{V}_{c}$, i.e. if $c'\in H^1(M,\mathbb{R})$, then $c'\in \mathbb{V}_{c}^{\bot}$ if and only if $\langle c',h \rangle =0$ for all $h\in \mathbb{V}_c$. Clearly,
$$
\mathbb{V}_{c}^{\bot}=\bigcup_U\{\text{\rm ker}\, i_{U}^*: U\, \text{\rm is a neighborhood of}\, \mathcal {N}(c) \cap\Sigma_c\mathrm{\ in\ }\T^n\}.
$$
Note that there exists a neighborhood $U$ of $\mathcal {N}(c)\cap\Sigma_c$ such that $\mathbb{V}_c=i_{U*}H_1(U,\mathbb{R})$ and $\mathbb{V}_{c}^{\bot}=\text{\rm ker}i^*_U$ (see \cite{M93}).
\begin{defi}[$c$-equivalence]
We say that $c,c'\in H^1(M,\mathbb{R})$ are \emph{cohomologically equivalent} if there exists a continuous curve $\Gamma$: $[0,1]\to H^1(M,\R)$ such that 
\begin{enumerate}
\item $\Gamma(0)=c$, $\Gamma(1)=c'$, 
\item $\alpha(\Gamma(s))$ keeps constant for all $s\in [0,1]$, and 
\item for each $s_0\in [0,1]$ there exists $\epsilon>0$ such that $\Gamma(s)-\Gamma(s_0)\in \mathbb{V}_{{\Gamma}(s_0)}^{\bot}$ whenever $s\in [0,1]$ and $|s-s_0|<\epsilon$.
\end{enumerate}
\end{defi}

Guaranteed by the upper semi-continuity of the Ma\~{n}\'{e} set, we obtain the description of the structure of the Ma\~{n}\'{e} set extends to energy levels slightly higher than $\min\alpha$.
\begin{pro}[Theorem 1.1 and 3.2 of \cite{C17b}]\label{PropManeBroken}

Given $V\in \tilde{\mathcal O}_3$ normalized by $\max V=0$, there exists some positive numbers $\tilde\Delta_0=\tilde\Delta_0(V)>0$ such that for each $E\in (0,\tilde\Delta_0)$ and each $c\in \alpha^{-1}(E)$ there exists a circle $\Sigma_{c}\subset\mathbb{T}^2$ such that all $c$-semi-static curves of the system $\tilde \sG$ pass through that circle transversally and
\begin{equation*}
\mathcal{N}(c)\cap\Sigma_{c}\subset\bigcup I_{c,i}
\end{equation*}
where $I_{c,i}\subset\Sigma_{c}$ are finitely many disjoint open intervals. Therefore any two cohomology classes $c$ and $c'$ in $\alpha^{-1}(E)$ are $c$-equivalent.
\end{pro}

\subsubsection{Cohomological equivalence for the full system}
We next construct a generalized transition chain using the $c$-equivalent mechanism in the full system $\sH_{S,\dt}:\ T^*\T_S^n\to \R$ near strong double resonance. We assume $y^\star\in \Sigma(\bk',\bk^o)$. Such a generalized transition chain will give rise to one for the original system $\sH$ after the linear symplectic transformations. We first study the $\al$-function for $\sH_{S,\dt}$. Note that Mather theory is defined for Tonelli systems on $T^*M$ for a general closed manifold $M$. Here we have $H^1(\T^n_S,\R)$ isomorphic to $H^1(\T^n,\R)$ with the basis vectors transformed by $S^{-t}$ in the same way as the $\sy$ variables in \eqref{EqSympSepa}.

\begin{lem}\label{LmAlphaH''}
The $\al$-functions of $\sH_{S,\dt}$ satisfies $\|\al_{\sH_{S,\dt}}-\al_{\sH_{S,0}}\|_{C^0}<\dt$, where $$\al_{\sH_{S,0}}(c)=\al_{\tilde\sG}(\tilde c)+\frac{1}{\sqrt\eps}\langle \hat\omega_{S,n-2},\hat c\rangle+\frac{1}{2}\langle B\hat c,\hat c\rangle,\ c\in H^1(\T^n_S,\R)=\R^n.$$
\end{lem}
\begin{proof}
For the $\dt$-estimate of the difference, we denote by $\mathsf L_\dt$ and $\mathsf L_0$ the Lagrangian corresponding to $\sH_{S,\dt}$ and $\sH_{S,0}$ respectively. Then we have $\|\mathsf L_\dt-\mathsf L_0\|_{C^0}\leq \dt$. Given cohomology class $c$, we denote by $\mu_\dt$ and $\mu_0$ the $c$-minimal measure for $\mathsf L_\dt$ and $\mathsf L_0$ respectively. Choose a closed one-form $\eta_c$ with $[\eta_c]=c$, then we get $$-\al_{\sH_{S,\dt}}=\int\mathsf  L_\dt-\eta_c \,d\mu_\dt\leq \int\mathsf  L_\dt-\eta_c \,d\mu_0,\quad -\al_{\sH_{S,0}}=\int\mathsf  L_0 -\eta_c\,d\mu_0\leq \int\mathsf  L_0-\eta_c \,d\mu_\dt.$$
The $\dt$-estimate of the difference follows by taking difference.

To determine the form of the $\al$ function for $\sH_{S,0}$, let us consider an invariant measure $\mu$ in the Mather set with cohomology class $c=(\tilde c,\hat c)$ of the system $\sH_{S,0}$. Denote by $\tilde \mu$ the corresponding invariant measure in the cohomology class of $\tilde c$ of the subsystem $\tilde \sG$.
By Mather's graph theorem, we know that $\mu$ is a graph from a subset of $\T^n_S$ to $\R^n$ and $\tilde\mu$ is a graph from a subset of $\T^2$ to $\R^2$. Next we know that $\mu$ has a skew product structure: for each $\tilde \sx\in \T^2$, there is a measure $\hat\mu_{\tilde \sx}$ supported on the torus $\mathrm{Graph}\tilde\mu(\tilde \sx)\times (-\breve A^t\tilde A^{-t}\tilde \sx+\T^{n-2})\times \{\hat c\}$ using the transformation $\mathfrak S$ in \eqref{EqSympSepa} as well as the fact $\dot{\hat \sy}=0$. So the integration with respect to $d\mu$ disintegrates into $d\mu(\sx)=d\hat\mu_{\tilde \sx}(\hat \sx)d\tilde \mu(\tilde \sx)$. When doing the inner integral with the integrand being the Lagrangian of $\sH_{S,0}$, note that the Lagrangian does not depend on $\hat \sx$, so the integration with respect to $d\tilde\mu_{\hat \sx}(\tilde \sx)$ is effectively the integration with respect to a Haar measure supported on the above torus containing the support of $\hat\mu_{\tilde\sx}$. In particular, in the $\hat \sy$ component, the measure is Dirac-$\dt$ supported on $\{\hat \sy=\hat c\}$. This gives the term $\frac{1}{\sqrt\eps}\langle \hat\omega_{S,n-2},\hat c\rangle+\frac{1}{2}\langle B\hat c,\hat c\rangle$. Finally, the outer integral with respect to $d\tilde\mu$ gives the term $\al_{\tilde\sG}(\tilde c)$.

\end{proof}
The next lemma shows that in the system $\tilde\sG$, the NHIC overlaps the region of c-equivalence.

\begin{lem}\label{LmOverlap}There exists an open-dense subset $ \hat{\mathcal O}_3\subset C^r(\T^2)/\R$, $r\geq 5$, such that for each $V\in \hat{\mathcal O}_3$ with $\max V=0$, there exists $\lambda>0$ such that the following holds for the system $\tilde\sG$:
\begin{enumerate}
\item the system $\tilde\sG$ admits a NHIC on the energy interval $[\al_{\tilde\sG}(\partial \beta_{\tilde\sG}((\lambda,0))), \infty)$, foliated by hyperbolic periodic orbits in the Mather sets with rotation vectors $\nu(1,0),\ |\nu|>\lambda,$ up to finitely many bifurcations and 
\item each curve $\al_{\tilde\sG}^{-1}(E),\ E/\al_{\tilde\sG}(\partial \beta_{\tilde\sG}((\lambda,0)))\in [1,2),$ is a curve of $c$-equivalence.
\end{enumerate}

\end{lem}
\begin{proof}
By Proposition \ref{PropManeBroken}, there exists an open-dense subset $\tilde{\mathcal O}_3$ in $C^r(\T^2)/\R$ such that for each $V\in \tilde{\mathcal O}_3$ with $\max V=0$ there exists $\tilde \Delta_0(V)>0$ such that each curve $\al_{\tilde\sG}^{-1}(E),\ E\in [0,\tilde\Delta_0),$ is a curve of $c$-equivalence. We introduce a sequence of open sets $\tilde{\mathcal O}_{3,\ell},\ \ell\in \N $ satisfying $ \tilde{\mathcal O}_{3,\ell}\subset \tilde{\mathcal O}_{3,\ell+1}$ and $\tilde{\mathcal O}_{3}=\cup_{\ell\in \N} \tilde{\mathcal O}_{3,\ell}$ where $$ \tilde{\mathcal O}_{3,\ell}:=\{ V\in \tilde{\mathcal O}_{3}\ |\ \tilde\Delta_0(V)>2/\ell\}.$$
Each set $ \tilde{\mathcal O}_{3,\ell}$ is open due to the upper-semi-continuity of the Ma\~n\'e set.  Indeed, suppose $V_*\in \tilde{\mathcal O}_{3,\ell}$ with $\tilde\Delta_0(V_*)>2/\ell$, so for all $c$ with $\al_{\tilde\sG}(c)<2/\ell $ the Ma\~n\'e sets $\cN(c)$ are broken in the sense of the conclusion of Proposition \ref{PropManeBroken}. By the upper-semi-continuity of the Ma\~n\'e set with respect to the Lagrangian, the same is true for any potential $V$ that is $C^2$ sufficiently close to $V_*$, so $\tilde \Delta_0(V)\geq \tilde\Delta_0(V_*)>2/\ell$. This means that there is a $C^r$-ball , $r\geq 2$, centered at $V_*$ contained in $\tilde{\mathcal O}_{3,\ell}$.

Next we fix large $E_+=E_*$ (see Theorem \ref{LmHighEnergy}) and choose $E_-=1/\ell$, we introduce an open-dense set $\tilde{\mathcal O}_{2,\ell}:=\tilde{\mathcal O}_{2}(1/\ell, E_+)\subset C^r(\T^2)/\R$ as in Theorem \ref{ThmCZ}. Now the intersection $\tilde{\mathcal O}_{3,\ell}\cap \tilde{\mathcal O}_{2,\ell}$ is open in $ C^r(\T^2)/\R$ and the union $\hat{\mathcal O}_3:=\cup_\ell (\tilde{\mathcal O}_{3,\ell}\cap \tilde{\mathcal O}_{2,\ell})$ is open-dense in $ C^r(\T^2)/\R$. To get the statement, it is enough to set $1/\ell= \al_{\tilde\sG}(\partial \beta_{\tilde\sG}((\lambda,0)))$ if $V\in (\tilde{\mathcal O}_{3,\ell}\cap \tilde{\mathcal O}_{2,\ell}).$ 
\end{proof}
Going back to the original system, we have the following. 
\begin{pro}\label{PropAlpha} 
Let $y^\star\in \Sigma_!(\bk',\bk^o)$ so that $\omega^\star=\omega(y^\star)$ is at strong double resonance with integer vectors $\bk'$ and $\bk^o$. Then there exists an open-dense set $\mathcal O_{3}=\mathcal O_{3}(\bk',\bk^o)\subset \Pi_{\bk',\bk^o}C^r(\T^n)/\R $, $r\geq 7$, such that for any $P$ with $\Pi_{\bk',\bk^o}P(x,y^\star)\in \mathcal O_3$ normalized by $\max \Pi_{\bk',\bk^o}P(x,y^\star)=0$,  there exist $\lambda=\lambda(\Pi_{\bk',\bk^o}P(x,y^\star))$ and $\dt_3=\dt_3( \Pi_{\bk',\bk^o}P,\lambda)$ such that for all $0<\dt<\dt_3$, the following holds. Suppose $c_*=(\tilde c_*,\hat c_*)\in \R^2\times \R^{n-2}=H^1(\T^n,\R)$ and $\sC_*:=(\tilde\sC_*,\hat\sC_*)=S^{-t} c_*$ satisfy $\al_{\tilde\sG}(\partial\beta_{\tilde\sG}((\lambda,0)))<\al_{\tilde\sG}(\tilde \sC_*)<2\al_{\tilde\sG}(\partial\beta_{\tilde\sG}((\lambda,0)))$ and $\|\hat\sC_*\|\leq \Lambda$. Then 
\begin{enumerate}
\item the path $\Gamma_\dt( \sC_*):=\{(\tilde \sC, \hat \sC_*)\ |\  \al_{\sH_{S,\dt}}(\tilde \sC, \hat \sC_*)=\al_{\sH_{S,\dt}}(\sC_*)\}$ is a path of $c$-equivalence for the system $\sH_{S,\dt}$ in \eqref{EqHamShear};
\item the path $\Gamma_\dt( \sC_*)$ lies in a $\dt$-neighborhood of the curve $\Gamma_0(\sC_*):=(\al_{\tilde \sG}^{-1}(\al_{\tilde\sG}(\tilde \sC_*)),\hat\sC_* ) $; 
\item the path $(SM'')^{t}\Gamma_\dt( \sC_*)$ is a path of $c$-equivalence for the original system \eqref{EqHomog1}. 
\end{enumerate}

\end{pro}
\begin{proof}
The open-dense set $\mathcal O_3$ is obtained by transforming the open-dense set $\hat{\mathcal O}_3$ in Lemma \ref{LmOverlap} by the linear transform $M''$. 
 Let us now go back to the system $\tilde \sG$ for which we choose $V\in \hat{\mathcal O}_3$ which determines $\lambda$.

We denote $\tilde\Gamma(\tilde \sC_*)=\al_{\tilde\sG}^{-1}(\al_{\tilde\sG}(\tilde \sC_*))$ for  given $\tilde \sC_*$ satisfying $\al_{\tilde\sG}(\tilde \sC_*)/\al_{\tilde\sG}(\partial\beta_{\tilde\sG}((\lambda,0)))\in (1,2)$.
The coordinates change $S$ does not change the $\tilde\sx$ components, so for each $\hat \sC$, denoting $\Gamma_0(\tilde \sC_*,\hat \sC)=(\tilde\Gamma(\tilde \sC_*),\hat \sC)$, the Ma\~n\'e set $\widetilde \cN_{\sG}(\sC(t))$ for the system $\sG$ in \eqref{EqG} and cohomology class $\sC(t)\in \Gamma_0(\tilde \sC_*,\hat \sC)$, when projected to the first $\T^2$ factor, coincides with the Ma\~n\'e set $\cN_{\tilde\sG}(\tilde \sC(t))$ for the system $\tilde\sG$. So by Proposition \ref{PropManeBroken}, for each $\sC(t)\in \Gamma_0(\tilde \sC_*,\hat \sC)$, there exist  a circle $\Sigma_{ \sC(t)}\subset\mathbb{T}^2$, and disjoint open intervals $I_{\sC(t),i}$, so that all $c$-semi-static curves of the system $\sG$ pass through that circle transversally and
\begin{equation*}
\mathcal{N}_{\sG}(\sC(t))\cap S(\Sigma_{\sC(t)}\times \T^{n-2})\subset\bigcup S(I_{\sC(t),i}\times \T^{n-2}),
\end{equation*}
whose homology is in the set $\{(0,0)\}\times \R^{n-2}$. For each $\hat\sC$ with $\|\hat\sC\|<\Lambda$, the curve $(\tilde\Gamma(\tilde \sC_*),\hat \sC)$ is a curve of cohomological equivalence for the system $\sH_{S,0}$ since for two points $\sC(t)$ and $\sC(t')$ on the curve, the difference $\sC(t)-\sC(t')=(*,\hat 0)$ is perpendicular to the subspace $\{(0,0)\}\times \R^{n-2}$ which contains the homology of $\mathcal{N}_{\sG}(\sC(t))\cap S(\Sigma_{\sC(t)}\times \T^{n-2})$. 

Next we show that the level set
$\{(\tilde \sC, \hat \sC_*)\ |\  \al_{\sH_{S,\dt}}(\tilde \sC, \hat \sC_*)=\al_{\sH_{S,\dt}}(\sC_*)\}$ is $O(\dt)$-close to that of the case $\dt=0$ which is $\Gamma_0(\tilde \sC_*,\hat \sC)$. This follows from the following fact about convex functions: given two convex functions $\al_\dt$ and $\al_0$ with $|\al_\dt-\al_0|_{C^0}\leq \dt$ and $\|D\al_0\|\geq C>0$ on the level set $\{\al_0(c)=E\}$, then the level sets $\{\al_\dt(c)=E\}$ and $\{\al_0(c)=E\}$ are $O(\dt)$-close to each other. To prove this fact, it is enough to measure the distance of the intersection points of the two level sets with each radial line. Since the subdifferential $D\al$ is bounded away from zero, to maintain constant $E$, the distance can at most be $O(\dt)$.

By the upper semi-continuity of the Ma\~n\'e set, since the Hamiltonians and the cohomology paths are $O(\dt)$-close, when we consider the system $\sH_{S,\dt}$ with $\dt$ small enough, the same conclusion holds.

\end{proof}
\subsection{The generalized transition chain mapped to the frequency space}\label{SSChainFreq2}
Our construction of the generalized transition chain applying the mechanism of $c$-equivalence (Proposition \ref{PropManeBroken}) is done in the space of cohomology classes dual to the frequency space. In this section, we describe the corresponding path in the frequency space.

Our goal is to move a frequency $\omega^i\in \bk'^\perp\cap \partial\al(\al^{-1}(E))$ to $\omega^f\in \bk'^\perp\cap \partial\al(\al^{-1}(E))$ separated by $\bk'^\perp\cap (\bk^o)^\perp\cap \partial\al(\al^{-1}(E))$. Mather sets with rotation vectors in $\bk'^\perp\cap \partial\al(\al^{-1}(E))$ and outside a $\lambda$-neighborhood of $\bk'^\perp\cap (\bk^o)^\perp\cap \partial\al(\al^{-1}(E))$ lie on wNHICs $\cC(\bk')$ in the phase space (Proposition \ref{NHICDouble}). However,
it is not clear if it is possible to cross $\bk'^\perp\cap (\bk^o)^\perp\cap \partial\al(\al^{-1}(E))$ inside $\bk'^\perp\cap \partial\al(\al^{-1}(E))$ due to the lack of wNHICs, so our strategy (Proposition \ref{PropAlpha}) is to take a detour outside $\bk'^\perp\cap \partial\al(\al^{-1}(E))$ to turn around $\bk'^\perp\cap (\bk^o)^\perp\cap \partial\al(\al^{-1}(E))$.

Let $\omega^\star\in \bk'^\perp\cap (\bk^o)^\perp\cap \partial\al(\al^{-1}(E)) $ be the strong double resonance, then after the coordinate change induced by $SM''$, the frequency $SM''\omega^\star$ has 0 as the first two entries. In the subsystem $\tilde\sG$, the Legendre transform of the $c$-equivalent path given by Proposition \ref{PropManeBroken} is a closed convex curve enclosing 0  and of diameter $<2\lambda$ on the plane $H_1(\T^2,\R)=\R^2$. This plane $\R^2$ is the first two coordinates in the frequency space.  Going back to the coordinates system before the linear symplectic transform $SM''$, we get a loop enclosing 0 on the plane $(SM'')^{-1}\mathrm{span}\{e_1,e_2\}$, where $e_1=(1,0,\ldots,0)$ and $e_2=(0,1,0,\ldots,0)$. Let us call the loop $\ell(\bk',\bk^o)$. 

We claim that 
{\it on the loop $ \omega^\star+\ell(\bk',\bk^o)$ there are two points in  $\bk'^\perp\cap \partial\al(\al^{-1}(E))$ separated by $\bk'^\perp\cap (\bk^o)^\perp\cap \partial\al(\al^{-1}(E))$. }

Indeed, it is enough to find two points on $ \omega^\star+\ell(\bk',\bk^o)$ orthogonal to $\bk'$. It is known that $\omega^\star\perp \bk'$, and $\ell(\bk',\bk^o)$ is a loop enclosing 0 on the plane $(SM'')^{-1}\mathrm{span}\{e_1,e_2\}$, so we can project $\bk'$ to the plane $(SM'')^{-1}\mathrm{span}\{e_1,e_2\}$ and find exactly two points orthogonal to the projection, provided $\bk'$ is not perpendicular to the plane, which can be verified directly.

\subsection{Center straightening}\label{ssctDouble}
Let the Tonelli Hamiltonian $H:\ T^*\T^2\to\R$, the homology class $g\in H_1(\T^2,\Z)$, the energy interval $[E_-,E_+]$ and the potential $V\in \tilde{ \mathcal O}_2(E_-,E_+)$ be as in Proposition \ref{ThmCZ}. Then we get at most finitely many pieces of NHICs foliated by hyperbolic periodic orbits.

\begin{pro}\label{PropSymp}
Let the Tonelli Hamiltonian $H:\ T^*\T^2\to\R$, the homology class $g=(1,0)\in H_1(\T^2,\Z)$, the energy interval $[E_-,E_+]$ and the potential $V\in\tilde{ \mathcal O}_2(E_-,E_+)$ be as in Proposition \ref{ThmCZ}. Suppose on this energy interval $H$ admits a NHIC $N$ foliated by hyperbolic periodic orbits in the Mather set of with rotation vectors $\nu g,\ \nu\in [\nu_-,\nu_+]\subset (0,\infty)$. Then 
\begin{enumerate}
\item 
restricted on the cylinder $N$, there exist two numbers $0<I_-<I_+$ and a symplectic change of variables $\Phi: (I,\varphi)\in [I_-,I_+]\times\mathbb{T}\to (x,y)|_{N}$, such that the Hamiltonian $H$ can be written as
$
\Phi^*H=H\circ\Phi=\tilde h(I),
$
where $\tilde h$ is as smooth as $H$ and satisfies
\[\tilde h(I_\pm)=E_\pm,\ \tilde h'(I_\pm)=\nu_\pm, \mathrm{\ and\ } \tilde h'(I)>0,\quad \tilde h''(I)>0,\quad \forall \ I\in[I_-,I_+].\]
\item There is a neighborhood $\mathcal U$ of the $c_1$ line in $H^1(\T^2,\R)$, such that for each $c=(c_1,c_2)\in \mathcal U$ with $ c_1\in [I_-,I_+],$ we have
$\al_{H}(c)=\tilde h(c_1).$
\item Assume furthermore that $H$ is reversible, i.e. $H(x,y)=H(x,-y)$, then the Mather set of with rotation vector $-\nu g,\ \nu\in [\nu_-,\nu_+]$ is the time reversal of that of $\nu g$. On the NHIC foliated by Mather sets with rotation vectors $-\nu g,\ \nu\in [\nu_-,\nu_+]$, the restricted Hamiltonian system $\bar h: [-I_+,-I_-]\times \T\to \R$ of one degree of freedom 
satisfies $\bar h(I)=\tilde h(-I)$.
 \end{enumerate}
\end{pro}
\begin{proof}
The normal hyperbolicity gives rise to the following decomposition of the symplectic form (Equation (63) of \cite{DLS08})
\begin{equation}\label{EqSympForm}
\Omega=\left[\begin{array}{c|c|c}
0&\Omega^{su}&0\\
\hline
-\Omega^{su}&0&0\\
\hline
0&0&\Omega|_{E^c}
\end{array}\right],
\end{equation}
with respect to the splitting of the tangent space $T_xM=E^s_x\oplus E^u_x\oplus E^c_x,\ x\in N$ (see Definition \ref{DefNHIM}). 
In particular, the symplectic form $\Omega$ restricted to the cylinder is still a symplectic form.

Let $\Omega_{N}$ be the restriction of standard symplectic form $\Omega$ on the cylinder. Denoted by $\mathbb{T}\times[I_-,I_+]$ the standard cylinder where $I_\pm$ are to be determined later, and let $\Psi_0$: $\mathbb{T}\times[I_-,I_+]\to N$ be a diffeomorphic map. Then the pull back $\Psi_0^*\Omega_{N}$ of $\Omega_{N}$ is a symplectic form on the standard cylinder $\mathbb{T}\times[I_-,I_+]$. As the second de Rham cohomology group of cylinder $\mathbb{T}\times[I_-,I_+]$ is trivial, Moser's argument on the isotopy of symplectic forms shows that certain diffeomorphism $\Psi$: $\mathbb{T}\times[I_-,I_+]\to\mathbb{T}\times[I_-,I_+]$ exists such that
$$
\Psi^*\Psi_0^*\Omega_{N}=dI\wedge d\varphi.
$$
The Hamiltonian $H$ induces a Hamiltonian defined on $\mathbb{T}\times[I_-,I_+]$: $H\Psi_0\Psi(I,\varphi)$.

Restricted to $N$, the Hamiltonian system has one degree of freedom hence is integrable. We have a standard method of introducing action-angle coordinates (c.f. Section 50 B and C of \cite{A89}). Namely, the action variable $I$ is defined as integrating the Poincar\'e-Cartan one form $ydx$ along the periodic orbits, and an angular variable $\varphi$ is introduced as symplectic conjugate of $I$. In action-angle coordinates, the Hamiltonian depends only on $I$, so we denote it by $\tilde h(I)$. We define $I_\pm$ by $\tilde h(I_\pm)=E_\pm$ and $\tilde h'(I_\pm)=\nu_\pm$.

It remains to show the twist. We use a result of Carneiro \cite{Car} saying that Mather's $\beta$ function is differentiable in the radial direction for autonomous systems. Now $\tilde h(I)$ is actually Mather's $\alpha$ function since Mather set is exactly the periodic orbit $\gamma_\nu$. The direction of $\nu g$ is the radial direction as $\nu$ varies. The $\alpha$ function is strictly convex $\frac{d^2\tilde h(I)}{dI^2}>0,\ $a.e. in order that $\beta$ is differentiable.
\[\dfrac{d\tilde h(I)}{dI}=\dfrac{d\tilde h(I_-)}{dI}+\int_{I_-}^I \dfrac{d^2\tilde h(t)}{dI^2}\,dt=\int_{I_-}^I \dfrac{d^2\tilde h(t)}{dI^2}\,dt>0.\]
Since the symplectic transformation is explicit, we get that $\tilde h$ is as smooth as $H$.

By Lemma \ref{LmFlatPer}, we get that for each rotation vector $\nu(1,0)$, $\nu\in [\nu_-,\nu_+]$, its Legendre transform is a line segment perpendicular to the homology class $(1,0)$. Taking union over all the line segments, we get a two-dimensional strip in $H^1(\T^2,\R)$ as the $\mathcal U$ in the statement.  It remains to locate $\mathcal U$. Note that integrating a closed one-form $\eta$ with cohomology class $c$ along a loop of homology class $(1,0)$ will pick out the first entry of $c$.  For Hamiltonian system of one degree of freedom defined on $T^*\T$, the cohomology class of each periodic orbit $\gamma$ is given by $\oint_\gamma y\,dx$. In our case, the restricted Hamiltonian system on the NHIC foliated by periodic orbits has one degree of freedom, so we get the cohomology class by integrating the Poincar\'e-Cartan form $y_1dx_1+y_2dx_2$ along the periodic orbit. Restricted to the NHIC, the Hamiltonian system is integrable whose $\al$-function is known to be the same as the Hamiltonian.

Finally, to see the system $\tilde h(I)$ is reversible, we notice that the reversibility of the system $H(x,y)$ implies that the Mather sets with rotation vectors $\nu g$ and $-\nu g,\ \nu>0$ are supported on the same periodic orbit with reversed time. Since the Legrendre transform of an even function is also even, so we get the Lagrangian $L(\dot x,x)$ is even with respect to $\dot x,$ hence $p=\frac{\partial L}{\partial \dot x}$ get a negative sign when we reverse the time. The two periodic orbits lie on the same energy level and their corresponding action variables are opposite to each other from the formula $I=\frac{1}{2\pi}\oint_\gamma y\, dx$. The proof is now complete. 
\end{proof}
\section{Dynamics around triple resonances}\label{STriple}
In this section, we describe the second step of reduction of order. We will construct NHICs homeomorphic to $T^*\T^{n-2}$ and build generalized transition chains connecting the NHICs crossing the triple resonance. This section gives the major part of the proof of Theorem \ref{ThmMainNHIC} in the  case of $n=4$.

We fix $\bk'$ and choose a $\Pi_{\bk'}P\in\mathcal O_1$ and  $\Pi_{\bk',\bk^o}P\in \mathcal O_3$ for finitely many strong second resonances $\bk^o$. This determines $\dt_1,\dt_2,\dt_3$ by Proposition \ref{NHICSingle}, \ref{NHICDouble} and \ref{PropAlpha}. We also fix a $\dt(<\min\{\dt_1,\dt_2,\dt_3\})$ so that Proposition \ref{NHICSingle}, \ref{NHICDouble} and \ref{PropAlpha} are applicable. In this section, we will construct further open-dense sets of admissible perturbations in $C^{r}(T^*\T^n)/(\Pi_{\bk'}C^{r}(T^*\T^n))$ and $C^{r}(T^*\T^n)/(\Pi_{\bk',\bk^o}C^{r}(\T^n))$. Here the space $C^{r}(T^*\T^n)/(\Pi_{\bk'}C^{r}(T^*\T^n))$ is defined each $P\in C^{r}(T^*\T^n)$ admits a decomposition $ P=\Pi_{\bk'} P+(P-\Pi_{\bk'} P)$ respecting the decomposition $C^{r}(T^*\T^n)=\Pi_{\bk'}C^{r}(T^*\T^n)\oplus C^{r}(T^*\T^n)/(\Pi_{\bk'}C^{r}(T^*\T^n))$. The space $C^{r}(T^*\T^n)/(\Pi_{\bk'}C^{r}(T^*\T^n))$ inherits the $C^r$ norm of $C^r(T^*\T^n)$. Similarly for $C^{r}(T^*\T^n)/(\Pi_{\bk',\bk^o}C^{r}(\T^n))$.
\subsection{Frequency refinement}\label{SSTracking}

 We have been working in a $\mu$-neighborhood of the frequency segment
$
\omega_a=\rho_a\Big(a,\frac {P}{Q}\omega_2^*,\frac{p}{q}\omega_2^*,\hat\omega^*_{n-3}\Big),\ a\in[\omega^{*i}_1-\varrho,\omega^{*f}_1+\varrho].
$ Note that $\mu$ is determined by $\dt$ through $K$. 

We pick a rational number denoted by $\frac{\bar p}{\bar q}$ satisfying
\begin{equation}\label{frequencyapproximation}
\Big|\dfrac{\bar p}{\bar q}\omega_2^*-\omega_4^*\Big|<\mu,\quad \mathrm{g.c.d.}(\bar p qQ,\bar q)=1,\quad \mathrm{g.c.d.}(\bar qp,\bar p q)=1,
\end{equation}
and obtain a new segment of frequency $\bar{\omega}_a:=\bar\rho_a(a,\frac {P}{Q}\omega_2^*,\frac{p}{q}\omega_2^*,\frac{\bar p}{\bar q}\omega_2^*,\hat\omega_{n-4}^*)$. 

Besides $\bk'$, the frequency $\bar{\omega}_a$ now admits a new resonant integer vector denoted by $\bk''$ for all $a$. For $\mu$ sufficiently small, the rational number $\bar p/\bar q$ necessarily has large denominator bounded from below by $O(\mu^{-1})$. So we get that $|\bk''|$ is bounded from below by $O(\mu^{-1})$. Thus $|\bk''|\gg |\bk'|$ if $\mu$ is small enough.

The transformed frequency segment is $M'\bar \omega_a=\bar\rho_a(a, 0,\frac{\bar P}{\bar Q}\omega_2^*,\frac{\bar p}{\bar q}\omega_2^*,\hat\omega_{n-4}^* )$ where $\frac{\bar P}{\bar Q}=\frac{1}{qQ}$ with $\bar P=1$ and $\bar Q=qQ$. 

When restricted to the wNHIC in Proposition \ref{NHICSingle}, we remove the zero entry in $M'\bar\omega$.
Now we are in a situation completely parallel to Section \ref{ssct:Number}. Again we encounter the situation of single and double resonances. The new resonant integer vector can be determined from the equation $ \bk''(M')^{-1}=(0,0, \bar Q\bar p,-\bar q\bar P,\hat 0_{n-4})$ where $\mathrm{g.c.d.}(\bar q \bar P,\bar p \bar Q)=\mathrm{g.c.d.}(\bar p qQ,\bar q)=1$.

As we vary $a$ in an interval, a third resonance may appear. We fix $K=(\dt/3)^{-1/2}$ as in Lemma \ref{Lm:nbd} by fixing $\dt$. Parallel to Lemma \ref{Lm:nbd}, we have the following.

\blm\label{Lm:nbd3}
Let $\omega_a,\mu, \bar\omega_a,\, K,\, \bk',\, \bk''$ be as above. For any $\bar K>\max\{K,|\bk''|\}$,

 let $\bk^o_{a^o_i}$, $i=1,\ldots,m$, be the collection of all the irreducible integer vectors in $\Z_{\bar K}^n\setminus\mathrm{span}\{\bk',\bk''\}$ satisfying $\langle \bk^o_{a^o_i},\bar\omega_{a_i^o}\rangle=0,$ and let $(\bk^o_{a^o_i})^\perp$ be the $(n-1)$-dimensional space orthogonal to the vector $\bk^o_{a^o_i}$ where $a^o_i\in [\omega_1^{*i}-\varrho, \omega_1^{*f}+\varrho]$, $i=1,\ldots,m$. Then there exists $\bar\mu=\bar\mu(\bar K)$ such that $
B(\bar\omega_a,\bar\mu)\subset
B(\omega_a,\mu)$ and
\ben
\item
for any small $\eps$ and all $\omega$ in the neighborhood
$
B(\bar\omega_a,\bar\mu)\setminus \bigcup_{i}B\l(\bar\omega_{a_i^o}+(\bk^o_{a^o_i})^\perp,\eps^{1/3}\r),
$
 we have
$|
\langle \bk,\omega\rangle|> \eps^{{1/3}},\quad \forall\ \bk\in \Z_{\bar K}^n\setminus\mathrm{span}_\Z\{\bk',\bk''\}.
$
\item for all $\omega$ in $B(\bar\omega_a,\bar\mu)\bigcap B\l(\bar\omega_{a_i^o}+(\bk^o_{a^o_i})^\perp,\eps^{1/3}\r)$, for each $i$ and for all
    $
    \bk\in \Z_{\bar K}^n\setminus\mathrm{span}_\Z\left\{ \bk' ,\bk_{a_i^o}^o,\bk'' \right\},
    $
we have
\begin{equation}\label{EqDiop1/2,3}
|\langle\bk,\omega\rangle|\geq n\bar K\bar \mu.
\end{equation}
\een
\elm
\begin{proof}
We reduce the proof to Lemma \ref{Lm:nbd}.
The transformed frequency segment $M'\bar \omega_a=\bar\rho_a(a, 0,\frac{\bar P}{\bar Q}\omega_2^*,\frac{\bar p}{\bar q}\omega_2^*,\hat\omega_{n-4}^* )$ admits resonant integer vectors $$ \bk'(M')^{-1}=(0,1,0,\ldots,0),\quad  \bk''(M')^{-1}=(0,0, \bar Q\bar p,-\bar q\bar P,\hat 0_{n-4}).$$ If $a=a^o_{i}$, it also admits the integer vector $ \bk^o_{a^o_i}(M')^{-1}$. Now, remove the zero entry in $M'\bar \omega_a$, then the resulting vector has the form of $\omega_a$ but one dimension less. The integer vectors $\pi_{ -2}(\bk''(M')^{-1})$ and $\pi_{ -2}(\bk^o_{a^o_i}(M')^{-1})$ play the role of $\bk'$ and $\bk^o_{a^o_i}$ in Lemma \ref{Lm:nbd} respectively, where $\pi_{ -2}:\ \R^n\to \R^{n-1}$ means to remove the second entry.  Now this lemma follows from Lemma \ref{Lm:nbd} up to a linear transform.\end{proof}
\subsection{The nondegeneracy condition}
Similar to Proposition \ref{NHICSingle}, we have the following result.


\begin{pro}\label{PropGenericity2}


Given $\Pi_{\bk'}P\in \mathcal O_1(\bk')\subset \Pi_{\bk'}C^r(T^*\T^n),\ r\geq 4$, we choose $\dt,\mu,$ $\bk''$ and $\bar\omega_a$ as above. Then there exists an open-dense subset $\mathcal O_{1,2}=\mathcal O_{1,2}(\bk',\bk'') $ in the unit ball of $\Pi_{\bk',\bk''}C^r(T^*\T^n)/\Pi_{\bk'}C^r(T^*\T^n)$ such that 
each $\Pi_{\bk',\bk''}P$ with $\Pi_{\bk'} \Pi_{\bk',\bk''}P=\Pi_{\bk'}P$ and $\Pi_{\bk',\bk''}P-\Pi_{\bk'}P\in\mathcal O_{1,2}$, has a unique nondegenerate global max along the segment $y\in \omega^{-1}(\bar\omega_a)$, up to finitely many bifurcations, where there are two nondegenerate global max. Moreover,  the curves $\{\mathrm{Argmax} \Pi_{\bk',\bk''}P(y,\cdot ),\ y\in  \omega^{-1}(\bar\omega_a)\}$, when projected to the set $\{\langle \bk',x\rangle, x\in \T^n\}\times\R^n$, is within $O(\mu)$ Hausdorff distance of  the curves $\{\mathrm{Argmax} \Pi_{\bk'}P(y,\cdot ),\ y\in  \omega^{-1}(\omega_a)\}$.
\end{pro}
\begin{proof}


The statement (without the ``Moreover" part) can be obtained directly by applying the main theorem of \cite{CZ2} which is a higher dimensional generalization of Proposition \ref{PropCZ}. Here we give an argument using only Proposition \ref{PropCZ}. Since we have $\Pi_{\bk'} P\in \mathcal O_1$ so $\Pi_{\bk'} P$ 
has a nondegenerate global max up to finitely many bifurcations where there are two nondegenerate global max. Moreover $\Pi_{\bk'} P$ determines $\dt,\mu,\bar\omega_a$ and $\bk''$. We next decompose $\Pi_{\bk',\bk''}P(y,x)=\Pi_{\bk'}P(y,x)+\bar P(y,\langle \bk', x\rangle,\langle \bk'', x\rangle)$ induced by the decomposition $ \Pi_{\bk',\bk''}C^r(T^*\T^n)=\Pi_{\bk'}C^r(T^*\T^n)\oplus \Pi_{\bk',\bk''}C^r(T^*\T^n)/\Pi_{\bk'}C^r(T^*\T^n)$. 
So we get $|\bar P|_{C^2}\leq C\frac{1}{|\bk''|^2}\leq C\mu^2$ since $|\bk''|\geq C\mu^{-1}$. We next make a linear coordinate change in $x$ so that the $Z_2(y, \sx_1,\sx_2)=Z(y, \sx_1) +\bar P(y, \sx_1,\sx_2)$, where $\langle \bk',x\rangle:=\sx_1$, $\langle \bk'',x\rangle:=\sx_2$, $Z(y, \sx_1)=\Pi_{\bk'}P(y,x)$, and $Z_2(y, \sx_1,\sx_2)=\Pi_{\bk',\bk''}P(y,x)$. 

By the choice of  $\Pi_{\bk'} P\in \mathcal O_1$, for each $y$, we have $\max_x Z(y,\sx_1)$ is nondegenerate and attained at $\sx_1^*(y)$. Then by the implicit function theorem, for small enough $\mu$, the global max of $Z_2$ is attained at a point $(\bar\sx_1^*,\bar \sx_2^*)(y)$ with $|\bar\sx_1^*(y)-\sx_1^*(y)|\leq C\mu^2$.  To see the nondegeneracy of the global max for $Z_2$, we consider for each $y$ and $\sx_2$, the function $Z_2(y, \cdot, \sx_2)$ attains the global max at a point $\bar{\bar \sx}^*_1(y,\sx_2)$ that is within $\mu^2$-distance from $\sx_1^*$ by the implicit function theorem. Now the function $Z_2(y,\bar{\bar \sx}^*_1(y,\sx_2),\sx_2 )$ becomes a function of $y$ and $\sx_2$. We then apply Proposition \ref{PropCZ} to $Z_2$ to get an open-dense set $\tilde{\mathcal O}_{1,2}(\Pi_{\bk'}P)$ such that $Z_2$ has nondegenerate global max along $\omega^{-1}(\omega_a)$. The nondegeneracy can be achieved by adding a function $f\in C^r(\T)$ of $\sx_2$ only. This induces an open-dense set $\mathcal O_{1,2}(\Pi_{\bk'}P)$ in the unit ball of $ \Pi_{\bk',\bk''}C^r(T^* \T^n)/\Pi_{\bk'}C^r(T^* \T^n)$. 
\end{proof}

In the proposition, each $\mathcal O_{1,2}$ depends on $\Pi_{\bk'}P\in \mathcal O_1$, so we denote $\mathcal O_{1,2}=\mathcal O_{1,2}(\Pi_{\bk'}P)$. 
\begin{lem} \label{LmResidualKU}The  union $\bigcup_{\Pi_{\bk'}P\in \mathcal O_1}(\mathcal O_{1,2}(\Pi_{\bk'}P)\times (C^r(T^*\T^n)/ \Pi_{\bk',\bk''} C^r(T^*\T^n)))$ intersects the unit ball of $C^r(T^*\T^n)$ in an open-dense subset of the latter.
\end{lem}
\begin{proof}
We first decompose $C^r(T^*\T^n)=\Pi_{\bk'}C^r(T^*\T^n)\oplus (C^r(T^*\T^n)/\Pi_{\bk'}C^r(T^*\T^n))$ for each irreducible $\bk'\in\Z^n$. 
Applying the following Karatowski-Ulam Theorem \ref{ThmKU}, we get that the union $\bigcup_{\Pi_{\bk'}P\in \mathcal O_1}\mathcal O_{1,2}(\Pi_{\bk'}P)$ is a set of second category in $C^r(T^*\T^n)$. Indeed, we first divide $\mathcal O_1 =\mathcal O_1(\bk')$ into union of the form $\mathcal O_1 =\bigcup_{\bk''} \mathcal O_{1,\bk''}$ such that each $\Pi_{\bk'}P\in \mathcal O_{1,\bk''}$ admits the frequency segment $\bar\omega_a$ having a second resonance $\bk''$ (see Section \ref{SSTracking}, note that $\Pi_{\bk'}P$ determines $\dt$ hence $\mu$). Each $\mathcal O_{1,\bk''}$ is open (may be empty). We use the notation $\mathfrak B_1(E)$ to denote the unit ball of a Banach space $E$. 

Now each $\Pi_{\bk'}P\in \mathcal O_{1,\bk''}$ determines an open-dense subset $\mathcal O_{1,2}(\Pi_{\bk'}P)$ in\\
 $ \mathfrak B_1(\Pi_{\bk',\bk''}C^r(T^* \T^n)/\Pi_{\bk'}C^r(T^* \T^n))$ by Proposition \ref{PropGenericity2}. So by the following Karatowski-Ulam Theorem \ref{ThmKU}, the union $\bigcup_{\Pi_{\bk'}P\in \mathcal O_{1,\bk''}}\mathcal O_{1,2}(\Pi_{\bk'}P)$ is second category in the product space $\mathcal O_{1,\bk''}\times  \mathfrak B_1(\Pi_{\bk',\bk''}C^r(T^* \T^n)/\Pi_{\bk'}C^r(T^* \T^n))$. 
Next since each $\mathcal O_{1,2}(\Pi_{\bk'}P)$ is open in $\Pi_{\bk',\bk''}C^r(T^*\T^n)/\Pi_{\bk'}C^r(T^*\T^n)$, we get that the union $\bigcup_{\Pi_{\bk'}P\in \mathcal O_{1,\bk''}}\mathcal O_{1,2}(\Pi_{\bk'}P)$ is also open in $\Pi_{\bk',\bk''}C^r(T^*\T^n)/\Pi_{\bk'}C^r(T^*\T^n)$. So we get that $\bigcup_{\Pi_{\bk'}P\in \mathcal O_{1,\bk''}}(\mathcal O_{1,2}(\Pi_{\bk'}P)\times \mathfrak B_1(C^r(T^*\T^n)/ \Pi_{\bk',\bk''} C^r(T^*\T^n)))$ is open dense in $ \mathcal O_{1,\bk''}\times\mathfrak B_1(C^r(T^*\T^n)/ \Pi_{\bk'} C^r(T^*\T^n))$. 

Taking union over all the $\bk''$, we get the statement in the Proposition.

\end{proof}
\begin{theo}[Karatowski-Ulam, Theorem 15.1 of \cite{Ox}]\label{ThmKU} Let $X,Y$ be two topological spaces where $Y$ has a countable bases.
If $E\subset X\times Y$ is a set of first category, then $E\cap \{x\}\times Y$ is first category in $Y$ for all $x$ except a set of first category.
\end{theo}

\subsection{The KAM normal forms}\label{SSKAM+}

\begin{defi} Given three irreducible integer vectors $\bk^o,\bk',\bk''$, we define the triple resonance sub manifold as $$\Sigma(\bk^o,\bk',\bk''):=\{y\ |\ \langle\bk',\omega(y)\rangle=\langle\bk'',\omega(y)\rangle=\langle\bk^o,\omega(y)\rangle=0\}.$$
\end{defi}

Lemma \ref{Lm:nbd3} allows us to apply Proposition \ref{prop: codim1} in its two cases to obtain the following normal forms.
\begin{lem} \label{LmNormalForm1+}
Let $\bar\dt$ be a small number satisfying $\bar\dt<\min\{3(|\bk''|)^{-2},\dt\}$ and let $\bar K=(\bar\dt/3)^{-1/2}$. Then there exists $\bar\eps_1=\bar\eps_1(\bar\dt,\Lambda)$ such that for all $\eps<\bar\eps_1$, the following holds. 
Suppose $\omega^\star\in B(\bar\omega_a,\bar\mu(\bar K))\setminus \bigcup_{i}B\l(\bar\omega_{a_i^o}+(\bk^o_{a^o_i})^\perp,\eps^{1/3}\r)$ as in case $(1)$ of Lemma \ref{Lm:nbd3}, then there exists a symplectic transform $\bar\phi$ defined on $B(0,\Lambda)\times \T^n$ that is $o_{\eps\to 0}(1)$ close to identity in the $C^{r}$ norm, such that 

\begin{equation}\label{EqNormalForm1+}
\begin{aligned}
\sH\circ \bar \phi(x,Y)=&\dfrac{1}{\sqrt{\eps}}\langle \omega^\star, Y\rangle+\dfrac{1}{2}\langle \sA Y,Y\rangle+ \Pi_{\bk',\bk''}\sV+\bar\dt\bar \sR(x,Y),
\end{aligned}
\end{equation}
where
\begin{enumerate}
\item $ \Pi_{\bk',\bk''}\sV=V(\langle \bk',x\rangle)+\dt\bar V(\langle \bk',x\rangle,\langle \bk'',x\rangle)$ with $\sA$, $V$ the same as that in Lemma \ref{LmNormalForm1}, and $|\bar V(\langle \bk',x\rangle,\langle \bk'',x\rangle)|_{r-2}\leq 1$.
\item $\bar\sR(x,Y)=\bar\sR_I(x)+\bar\sR_{II}(x,Y)$, where $\bar\sR_I$ consists of Fourier modes of $\sV$ not in the set $\mathrm{span}_\Z\{\bk',\bk''\}\cup\Z^n_{\bar K}$, and we have $|\bar\sR_I|_{{r-2}}\leq 1$ and $|\bar\sR_{II}|_{{r-5}}\leq 1$.
\end{enumerate}
\end{lem}
\begin{lem}\label{LmNormalForm2+}
Let $\bar\dt$ and $\bar K$ be as in the previous lemma. Then there exists $\bar\eps_2=\bar\eps_2(\bar\dt,\Lambda)$ such that for all $\eps<\bar\eps_2$, the following holds. 
Suppose  $\omega^\star$ is in the set $ B(\bar\omega_a,\bar\mu(\bar K))\bigcap B\l(\bar\omega_{a_i^o}+(\bk^o_{a^o_i})^\perp,\eps^{1/3}\r)$ as in case $(2)$ of Lemma \ref{Lm:nbd3}, then there exists a symplectic transform $\bar\phi$ defined on $B(0,\Lambda)\times \T^n$ that is $o_{\eps\to 0}(1)$ close to identity in the $C^{r}$ norm, such that
\begin{equation}\label{EqNormalForm2+}
\begin{aligned}
\sH\circ \bar \phi(x,Y)=&\dfrac{1}{\sqrt{\eps}}\langle \omega^\star, Y\rangle+\dfrac{1}{2}\langle \sA Y,Y\rangle+ \Pi_{\bk',\bk'',\bk^o}\sV+\bar\dt\bar \sR(x,Y),
\end{aligned}
\end{equation}
where
\begin{enumerate}
\item 

\begin{enumerate}
\item if $|\bk^o|<K$, we have $$\Pi_{\bk',\bk'',\bk^o}\sV=V\l(\l\langle  \bk', x\r\rangle,\l\langle \bk^o, x\r\rangle\r)+ \dt \bar V\l(\l\langle  \bk', x\r\rangle,\l\langle \bk'', x\r\rangle,\l\langle \bk^o, x\r\rangle\r)$$ with $\sA$, $V$ the same as that in Lemma \ref{LmNormalForm2} or
\item if $|\bk^o|\geq K$, we have
$$\Pi_{\bk',\bk'',\bk^o}\sV=V\l(\l\langle  \bk', x\r\rangle\r)+ \dt \bar V\l(\l\langle  \bk', x\r\rangle,\l\langle \bk'', x\r\rangle,\l\langle \bk^o, x\r\rangle\r)$$ with $\sA$, $V$ the same as that in Lemma \ref{LmNormalForm1}. 
\end{enumerate}
In both cases, we have $|\bar V\l(\l\langle  \bk', x\r\rangle,\l\langle \bk'', x\r\rangle,\l\langle \bk^o, x\r\rangle)\r|_{r-2}\leq 1$
\item $\bar\sR(x,Y)=\bar\sR_I(x)+\bar\sR_{II}(x,Y)$, where $\bar\sR_I$ consists of Fourier modes of $\sV$ not in the set $\mathrm{span}_\Z\{\bk^o,\bk',\bk''\}\cup \Z^n_{\bar K}$, and we have $|\bar\sR_I|_{r-2}\leq 1$ and $|\bar\sR_{II}|_{r-5}\leq 1$.
\end{enumerate}
\end{lem}

Now there are several sub cases. We assume that $\langle\omega^\star,\bk'\rangle=\langle\omega^\star,\bk''\rangle=0$.

\begin{enumerate}
\item  $\omega^\star$ is as in Lemma \ref{LmNormalForm1+}. The same argument as Proposition \ref{NHICSingle} gives that there is a $C^r$ wNHIC homeomorphic to $T^*\T^{n-2}$ if $\bar \dt$ is sufficiently small. The normal hyperbolicity is independent of $\eps$ or $\bar\dt$, but may depend on $\dt$. This wNHIC is a subset of the wNHIC in Proposition \ref{NHICSingle}.

\item $\omega^\star$ is as in item (1.b) of  Lemma \ref{LmNormalForm2+}. This case occurs when $|\bk^o|\geq K$. We first apply Proposition \ref{NHICSingle} to reduce the Hamiltonian system to a system defined on $T^*\T^{n-1}$. The restricted system to the wNHIC would depend on $x$ through $\langle \bk^o,x\rangle$ and $\langle \bk'',x\rangle$ up to a $\bar\dt$ perturbation. That means that the restricted system is at double resonance. If the double resonance is weak, then it is treated as a single resonance given by $\bk''$. Otherwise, we apply Proposition \ref{NHICDouble} to find a wNHIC homeomorphic to $T^*\T^{n-2}$ and a Proposition \ref{PropAlpha} to find a generalized transition chain connecting two neighboring wNHICs. 
\item $\omega^\star$ is as in item (1.a) of  Lemma \ref{LmNormalForm2+}.  This case occurs when $|\bk^o|<K$, i.e. the vectors $\bk',\bk^o$ gives rise to a strong double resonance for the first step of reduction of order. We call this case a strong triple resonance. 

    \end{enumerate}

{\it In the following, without loss of generality, we focus on the third case to explain how to introduce the extra resonance $\bk''$.}
The other two cases can be reduced to Proposition \ref{NHICSingle}, \ref{NHICDouble} and \ref{PropAlpha}.

\begin{Not}
We denote by $\Sigma_!(\bk^o,\bk',\bk'')$ the triple resonant submanifold determined by \emph{strong} triple resonances as in case $(1.a)$ of Lemma \ref{LmNormalForm2+}.
\end{Not}
\subsection{Reduction of order around triple resonance}\label{SSReduction+}
In this section, we perform the reduction of order around the triple resonance. We will find wNHICs getting close to the triple resonance. 
 
We assume $\omega^\star\in \Sigma(\bk',\bk'')$ and within $\eps^{1/3}$ distance of $\Sigma_{!}(\bk^o,\bk',\bk'')$. Again there are two subcases depending on if $\omega^\star$ is within $\Lambda\eps^{1/2}$ distance of $\Sigma_{!}(\bk^o,\bk',\bk'')$ or not. The case of dist$(\omega^\star, \Sigma_{!}(\bk^o,\bk',\bk''))>\Lambda\eps^{1/2}$ can be treated in the same way as Theorem \ref{LmHighEnergy} and Proposition \ref{NHICHighEnergy}, which is essentially reduced to the case of Lemma \ref{LmNormalForm1+}, so we skip this case and focus on the case of dist$(\omega^\star, \Sigma_{!}(\bk^o,\bk',\bk''))\leq\Lambda\eps^{1/2}$. Without loss of generality, we assume $y^\star\in \Sigma_!(\bk^o,\bk',\bk'')$ so that $\omega^\star=\omega(y^\star)$ is perpendicular to $\bk',\bk^o,\bk''.$

\subsubsection{The shear transformation}\label{SSSShear}
Similar to Lemma \ref{LmM''}, there exists a matrix $M'''\in \mathrm{SL}(n,\Z)$ whose first three rows are $\bk^o,\bk'$ and $\bk''$ respectively.
The matrix $M'''$ induces a symplectic transformation $$\mathfrak M''':\ T^*\T^n\to T^*\T^n,\quad (x,Y)\mapsto(M'''x, M'''^{-t}Y),\quad \quad A=M'''\sA M'''^t,$$
We denote $\bs\om=M'''\omega^\star$ which has 0 as the first three entries since $y^\star\in \Sigma(\bk^o,\bk',\bk'')$.
By the symplectic transformation $\mathfrak M'''$, one obtains the  Hamiltonian
\begin{equation}\label{EqHamReduced-1bar}
\begin{aligned}
H:=&{\mathfrak M'''}^{-1*}(\sH\bar \phi)=\dfrac{1}{2}\langle  A Y,Y\rangle+ V\l(x_1,x_2\r)+\dt \bar V\l(x_1,x_2,x_3\r)\\
&+\dfrac{1}{\sqrt{\eps}}\langle\widehat{ \bs\omega}_{n-3}, \hat Y_{n-3}\rangle+\bar\dt \bar R(x,Y),\end{aligned}
\end{equation}
where  $\bar R={\mathfrak M'''}^{-1*}\bar\sR$. The matrix $M'''$ depends on $\dt$ through $\bk''$ but is independent of $\bar \dt$.

We next introduce the shear transformation as we did in Lemma \ref{shear} to block diagonalize $A$. Let $A,S'''\in \mathrm{SL}(n,\R)$ be defined as follows
\begin{equation}\label{EqS'''}
A=\bmt{cc}\tilde{A}_3&\breve{A}_3\\
\breve{A}_3^t&\hat{A}_3 \emt,
\qquad
S'''=\bmt{cc}\mathrm{id}_3&0\\
-\breve{A}_3^{t}\tilde{A}_3^{-t}&\mathrm{id}_{n-3}\emt
\end{equation}
where $\tilde{A}_3,\breve{A}_3,\hat{A}_3$ are $3\times 3,3\times(n-3)$ and $(n-3)\times (n-3)$ respectively. With the shear matrix we introduce a symplectic transform
$$
\mathfrak S''':T^*\T^n\to T^*\T^n_{S'''},\ (x,Y)\mapsto (S'''x,S'''^{-t}Y):=(\sx,\sy),\quad \omega_{S'''}:= S''' \bs\om,
$$
which transforms the Hamiltonian into the following form defined on $T^*\T^n_{S'''}$
 \begin{equation}\label{EqHamReduced-1bar}
\begin{aligned}
\sH_{S'''}:=&(\mathfrak S'''{\mathfrak M'''})^{-1*}(\sH\circ \phi)
=\l[\dfrac{1}{2}\langle \tilde{A}_3 \tilde\sy_3,\tilde\sy_3\rangle+ V\l(\tilde\sx\r)+\dt \bar V\l(\tilde \sx_3\r)\r]\\
&+\dfrac{1}{\sqrt{\eps}}\langle\hat\omega_{S''',n-3}, \hat \sy_{n-3}\rangle+\dfrac{1}{2}\langle B_3\hat \sy_{n-3},\hat \sy_{n-3}\rangle+ \mathfrak S'''^{-1*}\l(\bar \dt \bar R(\sx,\sy)\r),\end{aligned}
\end{equation}
where we denote $B_3=(\hat{ A}_3-\breve{A}_3^t\tilde{ A}_3^{-1}\breve{A}_3)$ and $\tilde \sx_3=(\sx_1,\sx_2,\sx_3),\, \tilde \sy_3=(\sy_1,\sy_2,\sy_3)$, $\tilde \sx=(\sx_1,\sx_2).$ The norms of the matrices $B_3$ and $S'''$ depend on $\dt$ but not on $\bar \dt$.

\subsubsection{The existence of wNHICs}
To understand the full system $\sH_{S'''}$, we first need to understand its bracketed subsystem in \eqref{EqHamReduced-1bar}. The next lemma shows the existence of NHIC of dimensional 2 in the subsystem. 
\begin{lem}\label{Prop3DOF}
For any $\lambda>0$, there exists an open dense subset $\tilde{\mathcal O}_2\subset C^r(\T^2)/\R$, $r\geq 5$,  such that for each $V\in \tilde{\mathcal O}_2$ normalized by $\max V=0$, there exist $\tilde\dt_2=\tilde\dt_2(V)$ and an open-dense subset $\tilde{\mathcal O}_{2,*}$ in the $\tilde\dt_2$-ball of $ C^{r}(\T^3)/C^{r}(\T^2)$, such that for each each $\dt\bar V\in \tilde{\mathcal O}_{2,*}$, the subsystem
\begin{equation}\label{EqtG3}\sG_{3,\dt}:=\frac{1}{2}\langle \tilde{A}_3 \tilde\sy_3,\tilde\sy_3\rangle+ V\l(\tilde\sx\r)+\dt \bar V\l(\tilde \sx_3\r),\quad T^*\T^3\to \R\end{equation}
\begin{enumerate}
\item admits a $C^r$ NHIC homeomorphic to $T^*\T$, up to finitely many bifurcations;
\item the NHIC is foliated by hyperbolic periodic orbits as Mather sets with rotation vectors $\nu(1,0,0)$ and $|\nu|\geq \lambda$;
\item the absolute values of the normal Lyapunov exponents are bounded away from zero by $C\sqrt\dt$ for some constant $C>0$.
\end{enumerate}
\end{lem}
The next proposition establishes the existence of wNHICs in the full system. 

\begin{pro}\label{NHICTriple}  \begin{enumerate}
\item Given irreducible $\bk',\bk^o\in \Z_K^n$, let $y'^\star\in \Sigma_!(\bk^o,\bk')$, 
\item let $\lambda$, $\Pi_{\bk',\bk^o}P(y'^\star,x)\in \mathcal O_2$ and $\dt$ be as in the assumption of Proposition \ref{NHICDouble},
\item let $\bk''$ be the third resonance given in Section \ref{SSTracking} and consider $y''^\star\in \Sigma(\bk^o,\bk',\bk'')$ that is $\mu$-close to $y'^\star$.
\end{enumerate}
 Then there exists an open-dense set $\mathcal O_{2,*}=\mathcal O_{2,*}(\Pi_{\bk',\bk^o}P(y'^\star,\cdot),\bk'')$ in the unit ball of $\Pi_{\bk^o,\bk',\bk''} C^r(\T^n)/\Pi_{\bk',\bk^o} C^r(\T^n),\ r\geq 7,$ such that for each $\Pi_{\bk^o,\bk',\bk''}P(y''^\star,x)$ with $$\Pi_{\bk^o,\bk'}\Pi_{\bk^o,\bk',\bk''}P(y''^\star,x)=\Pi_{\bk',\bk^o}P(y''^\star,x),\ \Pi_{\bk^o,\bk',\bk''}P(y''^\star,x)-\Pi_{\bk',\bk^o}P(y''^\star,x)\in \mathcal O_{2,*},$$ there exists 
$\bar\dt_1=\bar\dt_1(\Pi_{\bk^o,\bk',\bk''}P(y''^\star,x),\lambda,\dt)>0$ such that for all $0<\bar\dt\leq \bar\dt_1$ and all $0<\eps<\bar\eps_2$,  
\begin{enumerate}
\item the Hamiltonian system \eqref{EqNormalForm2+} admits
a $C^r$ wNHIC $\cC(\bk',\bk'')$ homeomorphic to $T^*\T^{n-2}$ up to finitely many bifurcations.  The normal hyperbolicity is independent of $\eps$ or $\bar\dt$, but may depend on $\dt.$
\item Mather sets in the region $B(0,\Lambda)\times \T^n$ with rotation vectors orthogonal to both $\bk'$ and $\bk''$ and of distance $\lambda$-away from $\eps^{-1/2}\omega(y''^\star)+(\bk^o)^\perp$  lie inside $\cC(\bk',\bk'')$. 


\end{enumerate}

\end{pro}
\begin{proof}
The proof is similar to that of Proposition \ref{NHICSingle} and \ref{NHICDouble}. After the linear transform induced by $S'''M'''$, the problem of finding NHIC is reduced to Lemma \ref{Prop3DOF}. The NHIC persists if the $\bar\dt$ perturbation is sufficiently small. 
Here we only explain two points. First, here we choose $\mathcal O_{2,*}$ to be in the unit ball of $\Pi_{\bk^o,\bk',\bk''} C^r(\T^n)/\Pi_{\bk',\bk^o} C^r(\T^n)$ rather than in a $\tilde \dt_2$-ball as in Lemma \ref{Prop3DOF}. The reason is that $\tilde\dt_2$ is determined by the persistence of the NHIC in the subsystem $\tilde\sG$ of $\sG_{3,\dt}$. The theorem of NHIC requires only $C^1$ smallness of the perturbation to the Hamiltonian flow and we have that every function in the unit ball of $\Pi_{\bk^o,\bk',\bk''} C^r(\T^n)/\Pi_{\bk',\bk^o} C^r(\T^n)$ has $C^{r-2}$ norm less than $|\bk''|^{-2}<\dt_2$ in Propostion \ref{NHICDouble}.

Next, we explain the difference of $y'^\star$ and $y''^\star$. For each $\Pi_{\bk',\bk^o}P(y'^\star,x)\in \mathcal O_2$, there exists a wNHIC $\mathcal C(\bk')$ that is $\lambda$-away from the double resonance by Proposition \ref{NHICDouble}. If we perform the homogenization based at the point $y''^\star$ that is $\mu$-close to $y'^\star$ the resulting $\tilde \sG$'s differ by $O(\mu)$ in the $C^2$ topology. Since the normal hyperbolicity of the wNHIC $\mathcal C(\bk')$ is independent of $\dt$ and $\mu=o(\dt)$, we see that for small enough $\dt$, Proposition \ref{NHICDouble} that is stated for any $y^\star=y'^\star\in \Sigma_!(\bk',\bk^o)$, remains to hold for another $y''^\star\in \Sigma_!(\bk^o,\bk',\bk'')$ that is $\mu$-close to $y'^\star$. 
\end{proof}
The remaining part of this subsection is devoted to the proof of Lemma \ref{Prop3DOF}.

\begin{proof}[Proof of Lemma \ref{Prop3DOF}]
Applying Theorem \ref{ThmCZ}, we get an open dense subset $\tilde{\mathcal O}_2\subset C^r(\T^2)/\R$ such that for each $V\in \tilde{\mathcal O}_2$, the system $\tilde \sG:\ T^*\T^2\to \R$ admits a NHIC foliated by periodic orbits with rotation vectors $\nu(1,0), |\nu|>\lambda$. Let us now fix such a $V\in \tilde{\mathcal O}_2$. 

We next block diagonalize the quadratic form $\langle \tilde{A}_3 \tilde\sy_3,\tilde\sy_3\rangle$ by introducing one more shear transformation
$S_3=\bmt{ccc}
\mathrm{id}_2&0\\
-\sa_3 \tilde A^{-1}&1\\
\emt:=\bmt{ccc}
1&0&0\\
0&1&0\\
s_1&s_2&1\\
\emt,$
 where $\sa_3=(a_{31},a_{32})\in \R^2$ is the vector formed by the entries of $A$ on the third row to the left of the diagonal.
We can verify that \beq\label{Eqb3}
S_3
\tilde A_3
S_3^t=\bmt{ccc}
\tilde A&0\\
0&b_3
\emt \eeq
where $b_3=a_{33}-\sa_3 \tilde A^{-1}\sa_3^t$.
This linear transform induces a linear transform $\mathfrak S_3:\ (\tilde \sx_3,\tilde\sy_3)\mapsto (S_3 \tilde \sx_3, S_3^{-t}\tilde \sy_3):=(\tilde x_3,\tilde y_3)$ and transforms the Hamiltonian $ \sG_{3,\dt}$ to the following system $\mathfrak S_3^*\sG_{3,\dt}:=\mathcal G_{3,\dt}$ of the form
\begin{equation}\label{EqSsG3} \mathcal G_{3,\dt}=\frac{1}{2}\langle \tilde{A} \tilde y,\tilde y\rangle+ V\l(\tilde x\r)+\frac{b_3}{2} y_3^2+\dt \bar V(S_3^{-1}\tilde x_3),\quad T^*\T^3_{S_3}\to \R.\end{equation}

In the above system $\mathcal G_{3,\dt}$, we apply Theorem \ref{ThmCZ} with homology class $g=(1,0)$ and find NHIC in the subsystem $\tilde\sG:=\frac{1}{2}\langle \tilde{A} \tilde y,\tilde y\rangle+ V\l(\tilde x\r)$. Restricted to the NHIC, the subsystem $\tilde\sG$ is reduced to a system of one degree of freedom denoted by $\tilde h(I)$ in action-angle coordinates (Proposition \ref{PropSymp}). We restrict to the region $|h'(I)|>\lambda$. In the case of $\dt=0$, restricted to the NHIC, the system $\mathcal G_{3,0}$ becomes $\overline{\mathcal G_{3,0}}:=\tilde h(I)+\frac{b_3}{2}y_3^2$ defined on $T^*\T^2_{\bar S}$ where $\bar S=\bmt{ccc}
1&0\\
s_1&1
\emt\in \mathrm{SL}(2,\R)$ and $\T^2_{\bar S}=\T^3_{S_3}/\T^1$. 

When the $\dt$-perturbation in $\mathcal G_{3,\dt}$ is turned on, we apply the theorem of NHIM to get that
$\mathcal G_{3,\dt}$ admits a NHIC homeomorphic to $T^*\T^2_{\bar S}$ for sufficiently small $\dt$ and for any $\lambda>0$, the bound $\tilde\dt_2$ is determined in the same way as the proof of Proposition \ref{NHICDouble}(2). The restriction of $\mathcal G_{3,\dt}$ to the NHIC has the form 
\begin{equation}\label{EqBarSsG3}
\bar{\mathcal G}_{3,\dt}:=\tilde h(I)+\frac{b_3}{2} y_3^2+\dt\bar Z(I,\varphi,x_3,y_3),\ T^*\T^2_{\bar S}\to \R,\end{equation} where we have $\bar Z=\bar V(\tilde x(I,\varphi), x_3-(s_1,s_2)\cdot \tilde x(I,\varphi))+O(\dt)$. Indeed, the leading term in $\bar Z$ is obtained by evaluating $\bar V(S_3^{-1}\tilde x_3)$ restricted to the unperturbed NHIC with $\tilde x=\tilde x(I,\varphi)$. The $O(\dt)$ error is created by the deformation of the NHIC under the perturbation. 

Finally, going back to the original system $\sG_{3,\dt}$, we obtain an expression for the restricted system to the NHIC which is homeomorphic to $T^*\T^2$. We introduce the following undo-shear transformation
\beqa
\bar{\mathfrak S}:\ (\varphi,x_3;I,y_3)&\mapsto(\bar S(\varphi,x_3);\bar S^{-t}(I,y_3) )=(\varphi, s_1 \varphi+x_3;I-s_1 y_3, y_3)
:=(\varphi, \sx_3; J, \sy_3),
\eeqa
under which, we get the restriction of $\sG_{3,\dt}$ to the NHIC
\begin{equation}\label{EqDouble+}
\begin{aligned}
\bar\sG_{3,\dt}:=&\Big[\tilde h(J+s_1\sy_3)+\frac{1}{2}b_3\sy^2_3+\dt U(J,\sy_3,\varphi, \sx_3)\Big]:\quad T^*\T^2\to \R,
\end{aligned}
\end{equation}
where $U(J,\sy_3,\varphi, \sx_3)=\bar V(\tilde \sx(I,\varphi), \sx_3+(s_1,s_2)\cdot\tilde\sx(I,\varphi)-s_1 \varphi)+O(\dt)$ with $I=J+s_1 \sy_3$. Moreover, the $O(\dt)$ part depends on the angular variables $\sx_3,\varphi$ in the same way as the leading term. To see that $\sx_3$ is defined on $\T^1$,  we lift a periodic orbit $\tilde\sx$ with homology class $g=(1,0)$ to the universal cover, as $\varphi\mapsto \varphi+1$ we get $\tilde\sx\mapsto \tilde\sx+(1,0)$ and after the shear and undo-shear transformations $(s_1,s_2)\cdot\tilde\sx(I,\varphi)-s_1 \varphi\mapsto (s_1,s_2)\cdot\tilde\sx(I,\varphi)-s_1 \varphi$. 

We will apply the procedure of order reduction to the system $\bar\sG_{3,\dt}: \ T^*\T^2\to\R$. Namely, we want to apply Theorem \ref{ThmCZ} with homology class $g=(1,0)$ to get a NHIC and restrict the system to the NHIC to get a system of one degree of freedom. It is known that all its Mather sets with rotation vectors $\nu(1,0),\ |\nu|>\lambda,$ are supported on periodic orbits due to the two-dimensionality. Going back to the system $\sG_{3,\dt}$ of three degrees of freedom, we obtain that all its Mather sets with rotation vectors $\nu(1,0,0),\ |\nu|>\lambda,$ are supported on periodic orbits. It remains to show the nondegeneracy and hyperbolicity of the periodic orbits if $\dt\bar V$ is chosen in an open-dense subset $\tilde{\mathcal O}_{2,*}$ of the $\tilde\dt_2$-ball of the quotient $C^{r}(\T^3)/C^{r}(\T^2)$. The proof is essentially the same as the proof of Theorem \ref{ThmCZ} in \cite{CZ1}, but there is a subtle point: here we are only allowed to perturb the potential of the system $\sG_{3,\dt}$ of three degrees of freedom but cannot perturb $\bar\sG_{3,\dt}$ of two degrees of freedom directly.





We next show how to adapt the proof of \cite{CZ1} to our setting. Let us briefly recall the perturbation argument of \cite{CZ1}. For a Tonelli Lagrangian system $L(x,\dot x):\ T\T^2\to \R$. 
\begin{enumerate}
\item We first pick a section $\{x_1=0\}$ and reduce it to a nonautonomous system defined on $T\T\times \T\to \R$ then introduce the action functional $F(E,x_2):\ I\times \T\to \R$ where $E$ is the energy level and $I$ is the energy interval, by evaluating the action along the orbit on energy level $E$ starting and ending at the same point $x_2\in\T$. 
\item We then choose perturbation of the form $A_\ell \cos\ell x_2+B_\ell \sin\ell x_2,\ A_\ell,B_\ell\in [\epsilon,2\epsilon], \ell=1,2$. By the construction in Section 3 of \cite{CZ1}, such a perturbation to the Lagrangian becomes a perturbation of the same form to the action functional. 
\item Show that an open dense subset of the perturbation can make the global min of $F$ nondegenerate uniformly for $E\in I$ (Proposition \ref{PropCZ} here and Theorem 3.1 of \cite{CZ1}). 
\item Show that nondegenerate periodic orbits are hyperbolic. 
\end{enumerate}
Now we show that the above argument applies to the subsystem $\bar\sG_{3,\dt}$ of two degrees of freedom by perturbing the system $\sG_{3,\dt}$ of three degrees of freedom. In place of the above step (1), we pick the section $\{\varphi=0\}$ in the subsystem $\bar\sG_{3,\dt}$. 
Next consider a perturbation to the system $\sG_{3,\dt}$ depending only on $\sx_3$ of the form  $A_\ell \cos\ell \sx_3+B_\ell \sin\ell \sx_3,$ $\ A_\ell,B_\ell\in [\epsilon\dt,2\epsilon\dt], \ell=1,2$ as above item (2). Restricted to the section $\{\varphi=0\}$ in the subsystem $\bar \sG_{3,\dt}$ we get a perturbation of the same form up to a horizontal translation by a constant (see the expression of $U$ above). Then item (3) and (4) go through without any change. Since the system $\bar\sG_{3,\dt}$ is already restricted to a NHIC, its hyperbolic periodic orbit is also hyperbolic in the system $\sG_{3,\dt}$. 


\end{proof}

\begin{lem}\label{Lmb3s}We have the following estimates for the constant $b_3$ and $s_1$ appearing in equation \eqref{EqDouble+}
$$ b_3=\mathrm{const}_{b_3} |\bk''|^2,\quad s_1=\mathrm{const}_{s_1}|\bk''|,$$
where the constants are independent of $\dt$, $\mathrm{const}_{b_3}>0$ and $\mathrm{const}_{s_1}\in \R$.
\end{lem}
\begin{proof}Recall the definitions of $b_3$ (see \eqref{Eqb3})
$b_3=a_{33}-\sa_3 \tilde A^{-1}\sa_3^t$
and $s$ is the first entry of $\sa_3 \tilde A^{-1}$.
The $(i,j)$-th entry of $A=M'''\sA M'''^t$ is $m_i \sA m_j^t$ where $m_i,m_j$ are the $i$-th and $j$-th rows of $M'''$ respectively. Since the first three rows of $M'''$ are $\bk^o,\bk',\bk''$ respectively, we get that
 \[b_3=\bk'' \sA(\bk'')^t-(\bk'' \sA \mathbf K^t)(\bK \sA\bK^t)^{-1} (\bK \sA(\bk'')^t),\] and $s$ is the first entry of $\bk''\sA \bK^t(\bK \sA\bK^t)^{-1}$,
where we denote by $\mathbf K$ the matrix of $2\times n$ whose two rows are $\bk^o$ and $\bk'$ respectively.
Now $s_1$ is estimated easily as const.$|\bk''|$ since $\sA \bK^t(\bK \sA\bK^t)^{-1}$ does not depend on $\dt$.

We focus on $b_3$ in the following.
Since $\sA$ is positive definite, we decompose $\sA=CC^t$ for some $C\in \mathrm{GL}(n,\R)$ and denote $\bk''C:=\sk,\ \bK C=\sK$. This gives us
\[b_3=\sk (\sk^t- \sK^t (\sK\sK^t)^{-1} \sK\sk^t).\]
Now, we recall the Gauss least square method. The equation $\sK^t\sx=\sk^t$, though in general not solvable for $\sx\in \R^2$, we can seek for a least square solution given by $\sx_{ls}=(\sK\sK^t)^{-1} \sK\sk^t$, which has geometric interpretation as follows. The vector $\sK \sx_{ls}=\sK(\sK\sK^t)^{-1} \sK\sk^t$ is the projection of $\sk$ to the linear space spanned by the column vectors of $\sK^t$.  Hence $(\sk^t- \sK^t (\sK\sK^t)^{-1} \sK\sk^t)$ is the projection of $\sk$ to the orthogonal complement of the linear space spanned by the column vectors of $\sK^t$. We see from the construction of the vectors $\bk^o,\bk',\bk''$ that $\bk''$ forms a nonzero angle with the plane span$\{\bk',\bk^o\}$ independent of $\dt$, since as $\mu\to 0$ one has
$$
\bar\bk'=(0,\bar Q\bar p,-\bar q\bar P,\hat 0_{n-4})\parallel (0,1,-\frac{\bar q\bar P}{\bar Q\bar p},\hat 0_{n-3})\to(0,1,-\frac{\omega_2^*\bar P}{\omega_4^*\bar Q},\hat 0_{n-3}),
$$
which is obtained from $\bk''$ by removing the second entry of $\bk'' M'$, and the matrices $\bK,M',\sA$ do not depend on $\dt$. This linear independence relation is preserved by the linear transformation $C$. Hence we get that $b_3=c|\bk''|^2$ for some constant $c>0$ and independent of $\dt$.\end{proof}

\subsection{Description of the $\al$-function}\label{SSAlpha3}
Applying Theorem \ref{flatthm1} to the system $\sG_{3,\dt}$, we get a three dimensional flat on the energy level $\min\al_{\sG_{3,\dt}}$. Next, applying Lemma \ref{LmFlatPer} twice (since we have applied Theorem \ref{ThmCZ} twice in the proof of Lemma \ref{Prop3DOF}) we see that the NHICs in Lemma \ref{Prop3DOF} correspond to two channels $$\mathbb C_\pm=\{\partial\beta_{\sG_{3,\dt}}(\nu(1,0,0)),\ |\ \pm\nu>\lambda\}\subset H^1(\T^3,\R).$$ For each $c\in \mathbb C_\pm$, the corresponding Mather set $\widetilde\cM(c)$ lies in the NHIC with $\pm\nu>\lambda>0$. Lemma \ref{LmFlatPer} implies that the Mather set $\widetilde\cM(c)$ remains the same for $c$ in a two dimensional rectangle. Taking union over all the energy levels, we see that each $\mathbb C_\pm$ is a three dimensional rectangular prism. Moreover, the channels $\mathbb C_+$ and $\mathbb C_-$ are centrally symmetric to each other since $\sG_{3,\dt}$ is reversible.

In the following, since the rationality and irrationality of the rotation vectors do not play a role, for simplicity of notations, we will work with the system $\mathcal G_{3,\dt}:=\frak S_3^*\sG_{3,\dt}:\ T^*\T^3_{S_3}\to \R$ (equation \eqref{EqSsG3}), which is related to the system $\sG_{3,\dt}:\ T^*\T^3\to \R$ (equation \eqref{EqtG3}) by the symplectic transformation induced by $S_3$. Similarly we will work with  $\bar{\cG}_{3,\dt}: \ T^*\T^2_{\bar{S}}\to \R$ (equation \eqref{EqBarSsG3}) instead of the system $\bar\sG_{3,\dt}:\ T^*\T^2\to\R$ (equation \eqref{EqDouble+}) for the system restricted to the NHICs.
We first have the following description of the $\al$-functions.
\begin{lem}\label{lemAlpha}
\begin{enumerate}
\item We have the estimate for the $\al$-function of $\cG_{3,\dt}$:
$\|\al_{ \cG_{3,\dt}}-\al_{ \cG_{3,0}}\|_{C^0}\leq \dt $ with $\al_{ \cG_{3,0}}(c)=\al_{\tilde \sG}(\tilde c)+\frac{b_3}{2}c_3^2.$
\item For the $\al$-functions of the Hamiltonian $\bar\cG_{3,\dt}$ restricted to the NHIC, we have the estimate
$ \|\al_{ \bar\cG_{3,\dt}}-\al_{\bar \cG_{3,0}}\|_{C^0}\leq \dt$ with
$\al_{\bar\cG_{3,0}}(c_1,c_3)=\tilde h(c_1)+\frac{b_3}{2} c_3^2.$
\end{enumerate}

\end{lem}
The proof of this sublemma is the same as that of Lemma \ref{LmAlphaH''} so we skip it. 

\begin{pro}\label{pizza}
Under the assumption of Lemma \ref{Prop3DOF}, the flat $\mathbb F_0=\{c\ |\ \al_{\cG_{3,\dt}}(c)=\min \al_{\cG_{3,\dt}}\}$ is a three dimensional convex set lying in a $O(\sqrt{\delta/b_3})$-neighborhood of the disk $\tilde{\mathbb{F}}_0\times\{c_3=\hat 0\}$, where $\tilde{\mathbb{F}}_0=\mathrm{Argmin}\al_{\tilde\sG}$. 

\end{pro}

\begin{proof} The fact that the flat is three-dimensional is given by Theorem \ref{flatthm1}.
Since we have $|\sG_{3,\dt}-\sG_{3,0}|_{C^0}<\delta$, we have (c.f. Lemma \ref{LmAlphaH''})
\begin{equation*}
|\alpha_{\sG_{3,\dt}}(c)-\alpha_{\sG_{3,0}}(c)|\le\delta, \qquad \forall\ c\in H^1(\mathbb{T}^{3},\mathbb{R}).
\end{equation*}
After the same linear transformation $S_3^t$, this gives
\beq\label{deviation}|\al_{\cG_{3,\dt}}-\al_{\cG_{3,0}}|_{C^0}\le\delta.\eeq
Since we have
$\al_{\cG_{3,0}}(c)=\al_{\tilde \sG}(\tilde c)+\frac{b_3}{2}c_3^2,$ we get that $$\al_{\cG_{3,0}}(c)>2\delta,\ \mathrm{ if\ } |c_3|>2\sqrt{\delta/b_3},\quad \tilde c\in \tilde{\mathbb F}_0.$$ As $\alpha_{\tilde\sG}$ is non-negative, it follows from Formula \eqref{deviation} that $\al_{\cG_{3,\dt}}(c)>\delta$. Also due to Formula (\ref{deviation}), we have $\min\al_{\cG_{3,0}}\le\delta$. Therefore, $$\al_{\cG_{3,\dt}}(c)>\min\al_{\cG_{3,\dt}},\quad \mathrm{if\ }|c_3|>2\sqrt{\delta/b_3}.$$ This completes the proof for the $O(\sqrt{\dt/b_3})$ estimate. 

\end{proof}
Therefore, the flat looks like a pizza, horizontal in the direction of $\tilde c$ with small thickness of order $O(\sqrt{\delta/b_3})$ (see Figure \ref{fig1}). 

\subsection{Construction of the ladder}\label{SSLadder}
The generalized transition chain built by the application of the $c$-equivalence mechanism (Proposition \ref{PropManeBroken}) does not connect the channels $\mathbb C_\pm$ mainly due to the misalignment in the $c_3$ component. In this section, we show how the misalignment appears and how to overcome it to build a transition chain connecting $\mathbb C_\pm$, which is called a ladder (see Figure \ref{fig1}). 

The next result gives the existence of generalized transition chain in the subsystem $\sG_{3,\dt}.$
\begin{pro}\label{PropCross3Res}
Let $V\in \hat{\mathcal O}_3\cap \tilde{\mathcal O}_2\subset C^r(\T^2)/\R, \ r\geq 5,$ normalized by $\max V=0$ $($see Lemma \ref{LmOverlap} for $\hat{\mathcal O}_3$ and Lemma \ref{Prop3DOF} for $\tilde{\mathcal O}_2),$ and $\lb$ be as in Lemma \ref{LmOverlap}.  Then there exist $\tilde \dt_3(\leq \tilde\dt_2)$ and an open-dense subset $\hat{\mathcal O}_{3,*}(\subset \tilde{\mathcal O}_{2,*})$ in the $\tilde\dt_3$-ball of $C^r(\T^3)/C^r(\T^2)$ such that for each $
\dt\bar V\in \hat{\mathcal O}_{3,*}$ the following holds. For any point $\sC^*=(\tilde\sC^*,\sC^*_3)\in \mathbb C_+$ satisfying $$\al_{\tilde\sG}(\partial\beta_{\tilde\sG}(\lambda(1,0)))<\al_{\tilde \sG}(\tilde \sC^*+(s_1,s_2)\sC_3^*)< 2\al_{\tilde\sG}(\partial\beta_{\tilde\sG}(\lambda(1,0))),$$  on the energy level of $\al_{\sG_{3,\dt}}(\sC^*)$, there exists a generalized transition chain connecting $\sC^*\in \mathbb C_+$ to $-\sC^*\in \mathbb C_-$.
 
\end{pro}
Similar to Proposition \ref{PropAlpha}, we have the following result extending the generalized transition chain of the system $\sG_{3,\dt}$ to the full system.
\begin{pro}  \label{PropChainTriple}For each $\Pi_{\bk',\bk^o}P\in \mathcal O_3\subset\Pi_{\bk^o,\bk'}C^r(\T^n)/\R,\ r\geq 7, $ as in Proposition \ref{PropAlpha} and let $ \dt_3$ be as in  Proposition \ref{PropAlpha}. For any $\dt<\dt_3$ and $|\bk''|>K=(\dt/3)^{-1/2}$, there exists an open-dense subset $\mathcal O_{3,*}$ in the unit ball of $ \Pi_{\bk^o,\bk',\bk''}C^r(\T^n)/\Pi_{\bk',\bk^o}C^r(\T^n)$ such that for each $\Pi_{\bk^o,\bk',\bk''}P(y^\star,\cdot)$ satisfying $$\Pi_{\bk',\bk^o}(\Pi_{\bk^o,\bk',\bk''}P)=\Pi_{\bk',\bk^o}P\ \mathrm{and\ } \Pi_{\bk^o,\bk',\bk''}P-\Pi_{\bk',\bk^o}P\in \mathcal O_{3,*},$$ there exists $\bar\dt_2=\bar\dt_2(\Pi_{\bk^o,\bk',\bk''}P)>0$ such that for all $0<\bar\dt\leq \bar\dt_2$
and each $\hat c^*\in \R^{n-3}$ satisfying $\|\hat c^*\|<\Lambda$, there exists a generalized transition chain of  the Hamiltonian system \eqref{EqNormalForm2+} connecting the two channels corresponding to the NHICs $\cC(\bk',\bk'')$. 
\end{pro}


We next prove Proposition \ref{PropCross3Res}, which is reduced to the following two lemmas.

\begin{lem}\label{LmCeqTriple} For each $V\in \hat{\mathcal O}_3\subset C^r(\T^2)/\R,\ r\geq 5, $ as in Lemma \ref{LmOverlap} normalized by $\max V=0$, and let $\lb$ be as in Lemma \ref{LmOverlap}.  
Then there exists $\tilde\dt_3(<\tilde \dt_2$ in Lemma \ref{Prop3DOF}) such that for any $\dt\bar V$ in the $\tilde\dt_3$-ball of $C^r(\T^3)/C^r(\T^2)$ and any $c^*$ as in Proposition \ref{PropCross3Res},
there is a generalized transition chain of the system $\sG_{3,\dt}$ connecting the point $\sC^*=S_3^t c^*=S_3^t(\tilde c^*, c_3^*)$ to a point $S_3^t(\tilde c^\sharp, c_3^*)$, where $\tilde c^\sharp$ is $\dt$-close to $-\tilde c^*$ and satisfies $$\al_{\cG_{3,\dt}}(\tilde c^\sharp, c_3^*)=\al_{\cG_{3,\dt}}( c^*)=\al_{\sG_{3,\dt}}(\sC^*).$$
\end{lem}
\begin{proof}

Given $c^*$, we fix $c_3=c_3^*$ and define the path
$$\Gamma_\dt(c_3^*)=\{\tilde c\in \R^2\ |\ \al_{\cG_{3,\dt}}(\tilde c,c_3^*)=\al_{\cG_{3,\dt}}(c^*)\}.$$
In the case of $\dt=0$, this path lies on the constant energy level of $\al_{\tilde\sG}$ and in the small positive $\dt$ case, the path undergoes a $\dt$-perturbation as proved in Proposition \ref{PropAlpha}.

On $\Gamma_\dt(c_3^*)$, we find a point that is closest to $(-\tilde c^*,c_3^*)$ and denote it by $(\tilde c^\sharp,c_3^*)$ where $|\tilde c^\sharp+\tilde c^*|\leq C\dt$.

The fact that the path $\Gamma_\dt(c_3^*)$ is a generalized transition chain follows from Proposition \ref{PropManeBroken} (see also Proposition \ref{PropAlpha}) and the upper-semi-continuity of the Ma\~n\'e set. 
\end{proof}
\begin{figure}
        \centering
    \begin{subfigure}[b]{0.5\textwidth}
                \includegraphics[width=0.8\textwidth]{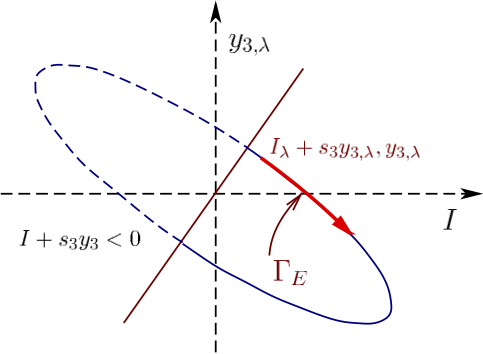}
                \caption{The ladder for $\al_{\bar\sG_{3,\dt}}$
                 \
                 }  \label{figLadder1}
       \end{subfigure}%
             \begin{subfigure}[b]{0.5\textwidth}
                \includegraphics[width=0.8\textwidth]{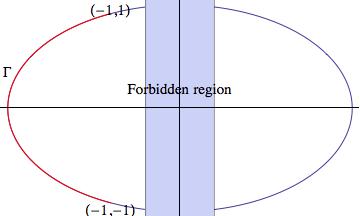}
                \caption{The ladder for $\al_{\bar{\mathcal G}_{3,\dt}}$}
                \label{figLadder2}
       \end{subfigure}
       \caption{The ladder construction}\label{fig:cone}
\end{figure}

\begin{lem}\label{ladder}
Let $\lb, \tilde{\mathcal O}_2\subset C^r(\T^2)/\R,\ r\geq 5,$ and $\tilde \dt_2$ be as in Lemma \ref{Prop3DOF} and $\hat{\mathcal O}_3$ be as in Lemma \ref{LmOverlap}, and let $V\in \tilde{\mathcal O}_2\cap \hat{\mathcal O}_3$. Then there exists an open-dense subset $\hat{\mathcal O}_{3,*}$ in the $\tilde\dt_3$-ball of $C^r(\T^3)/C^r(\T^2)$ such that for each $\dt\bar V\in\hat{\mathcal O}_{3,*}$, letting $\sC^*$ and $\tilde c^\sharp$ be as in Lemma \ref{LmCeqTriple}, 
there is a generalized transition chain of the system of $\sG_{3,\dt}$ lying on the energy level $\al_{\sG_{3,\dt}}(\sC^*)$ connecting $S_3^t(\tilde c^\sharp,c_3^*)$ to $-\sC^*\in\mathbb C_-$.
\end{lem}
\begin{proof}
We fix the energy level $E=\al_{\sG_{3,\dt}}(\sC^*)$ in the system $\sG_{3,\dt}$, on which there exists a NHIM restricted to which the Hamiltonian system is $\bar \sG_{3,\dt}$.
We define (see Figure \ref{figLadder2})$$\bar{\mathbb L}_\dt:= \{ (c_1,c_3)\ |\ \al_{ \bar\cG_{3,\dt}}(c_1,c_3)=E,\quad |\tilde h'(c_1)|> \lambda\}.$$
The variable $c_2$ does not appear due to Proposition \ref{PropSymp}(2).

To see the path $\bar{\mathbb L}_\dt$ clearly, we introduce the following coordinate change $\cR: (c_1,c_3)\mapsto(c_1,\frac{1}{s_1}c_3)$. In the new coordinates, the $\al$-function for the restricted system has the form
$$\mathcal R^*\al_{ \bar\cG_{3,0}}(c_1,c_3):=\tilde h(c_1)+\frac{b_3}{2s_1^2}c_3^2.$$
From Lemma \ref{Lmb3s}, we see that $b_3/s_1^2$ is of order one as $\dt\to 0.$ The frequency $\nu(1,0)$ for the system $\bar\sG_{3,0}$ is transformed to $\nu(1,1)$ under the linear transform $\bar S$ followed by $\mathcal R$.  Its Legendre transform of $\nu(1,1)$ is solved from the equation $(\tilde h'( c_1),\frac{b_3}{s_1^2}c_3)=\nu(1,1)$ for $\mathcal R^*\al_{ \bar\cG_{3,0}}$. For the $\al$-function  $\mathcal R^*\al_{ \bar\cG_{3,\dt}}$, we get that the projection of $\mathbb C_\pm$ to the $(c_1,c_3)$ plane is $\dt$-close to the set $\{(\tilde h'(c_1),\frac{b_3}{s_1^2}c_3)=\nu(1,1)\}.$

The path $\bar{\mathbb L}_\dt$ is $\dt$-close to $\bar{\mathbb L}_0$ which is  the path on the level set $\al_{ \bar\cG_{3,0}}(c_1,c_3)=E$ connecting the points $(-c_1^*,c_3^*)$ to $(-c_1^*,-c_3^*)$ symmetric around the $c_1$-axis. If we choose $c_3=c_3^*$, we get a unique point on $\bar{\mathbb L}_\dt$ near $-c_1^*.$
In the full system $\sG_{3,\dt}$, adding back the $c_2$ variable by Proposition \ref{PropSymp}(2), we obtain a two dimensional channel $\mathbb L_\dt$ in $ \al_{\sG_{3,\dt}}^{-1}(E).$ By the definition of the point $\tilde c^\sharp$ in the proof of the previous lemma, we get that the point $(\tilde c^\sharp,c_3^*)\in \Gamma_\dt(c^*_3)\cap \mathbb L_\dt.$ This shows that $\mathbb L_\dt\cap \Gamma_\dt(c^*_3)\neq\emptyset$ and $\mathbb L_\dt\cap \mathbb C_-\neq\emptyset$. 

We claim that {\it for $\dt\bar V$ chosen in an open-dense subset $\hat{\mathcal O}_{3,*}$ of the $\tilde\dt_3$-ball of \\ $C^r(\T^3)/C^r(\T^2)$, any continuous curve in the interior of $\mathbb L_\dt$ is a generalized transition chain.} 

We introduce a subset $\Delta\subset \mathbb L_\dt$ in the following way: $\sigma\in \Delta$ if and only if the weak KAM $u_{\sigma}$ of the restricted system $\bar \sG_{3,\dt}$ on the NHIM is $C^1$ (must be $C^{1,1}$ also \cite{Be4}), i.e. the Ma\~n\'e set is an invariant 2-torus. For $\sigma\notin\Delta$, certain section $\Sigma_{\sigma}$ of 2-torus exists such that $\mathcal{N}_{\sigma}\cap\Sigma_{\sigma}$ is shrinkable, so that (H2) of Definition \ref{chaindef1} can be verified. To prove that $\mathbb L_\dt$ is a generalized transition chain it remains to prove the following in order to verify (H1) of Definition \ref{chaindef1}:  

{\it for $\dt \bar V$ in $\tilde{\mathcal O}_{2,\star}$ and for all $\sigma\in\Delta$, each connected component of $\mathrm{Argmin}\{B_{\sigma},\Sigma_{0,\sigma}\backslash\cup_mN_m\}$ is contained in certain disk $O_m\subset\Sigma_{0,\sigma}$},

\noindent where $B_\sigma$ is the barrier function of the system $\sG_{3,\dt}$, $\Sigma_{0,\sigma}$ is a $2$-dimensional section of $\mathbb{T}^3$ which is transversal to $\sigma$-semi static curves, $\cup_mN_m$ denotes a neighborhood of the Aubry set in the finite covering space, $\mathrm{Argmin}\{B_{\sigma},\Sigma_0\backslash\cup_mN_m\}$ denotes the set of minimal points of $B_{\sigma}$ which fall into the set $\Sigma_{0,\sigma}\backslash\cup_mN_m$.

This is given by Theorem \ref{fundamental} in Appendix \ref{AppGenericity} (the autonomous case and the Ma\~n\'e perturbation case). The argument is similar to the case of {\it a priori} unstable systems (\cite{CY1,CY2}). 
\end{proof}
In the next remark, we explain our mechanism of ladder climbing.
\brk[The diffusion mechanism for the ladder climbing]\label{RkLadder}{\rm
Here we employ a variant of Arnold's mechanism \eqref{EqArnold}. Consider Hamiltonian system of three degrees of freedom of the form
\[H=\dfrac{y_1^2}{2}+\dfrac{y_2^2}{2}+\dfrac{y_3^2}{2}+(\cos x_3-1)(1+\eps (\cos x_1+\sin x_2)).\]
In this system, there exists a diffusing orbit for each $E>0$ such that $(y_1,y_2)$ stays close to the circle $\{y_1^2+y_2^2=2E\}$ and $\arctan\frac{y_1}{y_2}$ achieves any value in $[0,2\pi)$. Loosely speaking, $(y_1,y_2)$ moves along the circle $\{y_1^2+y_2^2=2E\}$. 

 This can be considered as a system of {\it a priori} unstable type. One can compute the Melnikov integral as in the Arnold's example to verify that Arnold diffusion exists. In our case, the system $\sG_{3,\dt}$ plays the role of $H$ here and the system $\bar \sG_{3,\dt}$ plays the role of $y_1^2+y_2^2$.}\erk
\subsection{The generalized transition chain in the frequency space}\label{SSChainFreq3}
Let us look at the generalized transition chain of Proposition \ref{PropChainTriple} in the frequency space. 

Now by Proposition \ref{NHICTriple}, we need to connect two frequencies $\omega^i,\omega^f\in \bk'^\perp\cap \bk''^\perp\cap \partial\al(\al^{-1}(E))$ turning around $(\bk^o)^\perp\cap \bk'^\perp\cap \bk''^\perp\cap \partial\al(\al^{-1}(E))$. Let $\omega^\star(\bk^o)^\perp\bk'^\perp\cap \bk''^\perp\cap \partial\al(\al^{-1}(E))$ be the triple resonance. By Proposition \ref{PropAlpha}, the $c$-equivalent mechanism (Proposition \ref{PropManeBroken}) gives a small loop $\omega^\star+\ell( \bk^o,\bk')$ around $(\bk^o)^\perp\cap \bk'^\perp\cap \partial\al(\al^{-1}(E))$ intersecting $\bk'^\perp\cap \partial\al(\al^{-1}(E))$ at two points (See Section \ref{SSChainFreq2}). Let us suppose one of the intersection points is $\omega^i$ and denote the other intersecting point by $\omega^\sharp$. The curve on $\omega^\star+\ell( \bk^o,\bk')$ connecting $\omega^i$ to $\omega^\sharp$ corresponds to the path in Lemma \ref{LmCeqTriple}. It remains to move $\omega^\sharp$ to some $\omega^f$, which is given by Lemma \ref{ladder}. 

In the coordinates system transformed by $S'''M'''$, a frequency $\omega\in \bk'^\perp\cap \bk''^\perp\cap \partial\al(\al^{-1}(E))$ has vanishing second and third entries. Since both $\omega^\sharp$ and $\omega^f$ are in $\bk'^\perp\cap \partial\al(\al^{-1}(E))$, the vanishing second entry means that they are frequencies of the restricted system to the NHIC $\cC(\bk')$. 
Since we have separated the subsystem $\sG_{3,\dt}$ from the full system, we focus only on the first three entries. 
Since $\omega^i\in  \bk'^\perp\cap \bk''^\perp\cap \partial\al(\al^{-1}(E))$ and $\lambda$ is small, we get that  $\omega^\sharp\in \bk'^\perp\cap \partial\al(\al^{-1}(E))$ has vanishing second entry and $\lambda$-small third entry.  The path in Lemma \ref{ladder} is constructed from the subsystem $\bar\sG_{3,\dt}$ which is the restriction of $\sG_{3,\dt}$ to the NHIC homeomorphic to $T^*\T^2$. Going back to the original system, this corresponds to a path for the restricted system to $\cC(\bk')$. So the path in Lemma \ref{ladder} when viewed in the frequency space is a curve fixing the second entry to be zero (restrict to the NHIC $\cC(\bk')$)  and moving the first entry and the third entry along  a convex curve that is the Legendre transform of a energy level curve of the subsystem $\bar\sG_{3,\dt} $. Moving along such a curve allows us to kill the $\lambda$-small third entry of $\omega^\sharp$ to arrive at $\omega^f$. 

\section{Induction and dynamics around complete resonances}\label{SHierarchy}
The main construction in this paper was done in the previous section for the $n=4$ case. In this section, we perform induction. We will find a frequency path $\bs\omega^\sharp_a$ admitting $n-2$ linear independent resonant relations $\bk',\bk'',\ldots,\bk^{(n-2)}$ for all $a$ in an interval and there is one more resonant integer vector $\bk^o$ for some $a$.
The goals are:
\begin{enumerate}
\item to show that away from the complete resonance, there are wNHICs homeomorphic to $T^*\T^2$, restricted to which the time-1 map of the Hamiltonian system is a twist map;
\item to show that near the complete resonance there is a generalized transition chain connecting nearby pieces of wNHICs.
\end{enumerate}

The frequency path $\bs\omega^\sharp_a$ is constructed to shadow the frequency segment $\Omega(a)=\rho_{a}(a,\omega_2^{*i},\omega_3^{*i},\ldots, \omega_n^{*i})$, $a\in [\omega^{*i}_{1}-\varrho,\omega^{*f}_{1}+\varrho]$, as in Section \ref{SOutline} where $\omega^{*i}=(\omega_1^{*i},\ldots, \omega_n^{*i})$ is Diophantine and $i$ means ``initial'' and $f$ means ``final". Let us recall the construction in the previous sections.
\begin{enumerate}
\item We first modify $\omega_2^{*i}$ and $\omega_3^{*i}$ to two rational numbers $\frac{p_2}{q_2}\omega_2^{*i}$ and $\frac{p_3}{q_3}\omega_2^{*i}$ respectively with $\mathrm{g.c.d.}(q_2p_3,q_3p_2)=1$ to obtain the frequency vector $$\omega^{(1)}(a)=\rho^{(1)}_{a}\left(a,\frac{p_2}{q_2}\omega_2^{*i},\frac{p_3}{q_3}\omega_2^{*i},\omega_4^{*i},\ldots, \omega_n^{*i}\right),$$ which admits a resonant vector $\bk'=(0,q_2p_3,-q_3p_2,\hat 0_{n-3})$ for all $a$ and one more $\bk^o$ for some $a^o$.
\item In a $\mu$-neighborhood of the frequency segment $\omega^{(1)}(a)$, we distinguish the frequencies therein into two cases, single resonance and near double resonance, according to Lemma \ref{Lm:nbd}. Accordingly, we have two Hamiltonian normal forms Lemma \ref{LmNormalForm1}  and Lemma \ref{LmNormalForm2}.
\item In the case of single resonance, we apply Theorem \ref{PropCZ} and the theorem of NHIM to find wNHIC $\cC(\bk')$ homeomorphic to $T^*\T^{n-1}$ as in Proposition \ref{NHICSingle}.
\item In the case of double resonance, we first apply a shear transformation (Lemma \ref{shear}) to separate a mechanical system of two degrees of freedom, which was well-understood in  \cite{CZ1,C17a,C17b}. The main results established in \cite{CZ1,C17a,C17b} enable us to find wNHIC $\cC(\bk')$ approaching the double resonance (Proposition \ref{NHICDouble}) and a path of $c$-equivalence connecting neighboring wNHICs (Proposition \ref{PropAlpha}).
\end{enumerate}
When this is done, we restrict the Hamiltonian system to the wNHICs to get a new system of one less degrees of freedom and repeat the above procedure. We need to refine the frequency segment to $$\omega^{(2)}(a)=\rho^{(2)}_{a}\left(a,\frac{p_2}{q_2}\omega_2^{*i},\frac{p_3}{q_3}\omega_2^{*i},\frac{p_4}{q_4}\omega_2^{*i},\omega_5^{*i},\ldots, \omega_n^{*i}\right),$$ along which we introduce a new resonant integer vector $\bk''$ and find wNHICs $\cC(\bk',\bk'')$ homeomorphic to $T^*\T^{n-2}$ away from triple resonances (Proposition \ref{NHICTriple}) and build a generalized transition chain (Proposition \ref{PropCross3Res}) connecting neighboring wNHICs near triple resonances. Details were carried out in Section \ref{STriple}. This part is parallel to the $\omega^{(1)}(a)$ case with the only difference being to construct a new piece of generalized transition chain near triple resonance in addition to that constructed via the mechanism of cohomology equivalence.

We repeat the above procedure of reduction inductively (see below in this section) for $n-2$ steps to get a frequency segment of the form $$\omega^{(n-2)}(a)=\rho^{(n-2)}_{a}\omega_2^{*i}\left(\frac{a}{\omega_2^{*i}},\frac{p_2}{q_2},\frac{p_3}{q_3},\ldots,\frac{p_n}{q_n}\right),$$ along which we have $n-2$ resonant integer vectors $\bk',\bk'',\ldots,\bk^{(n-2)}$ for all $a$ and $n-1$ for some $a$. We will get wNHICs $\cC(\bk',\bk'',\ldots,\bk^{(n-2)})$ homeomorphic to $T^*\T^2$ away from complete resonances and generalized transition chains connecting neighboring wNHICs. Diffusing orbits can be constructed (Theorem \ref{ThmMainNHIC}) along the wNHIC and generalized transition chain so that the first entry $a$ moves from $\omega_{1}^{*i}$ to $\omega_{1}^{*f}$.

This completes the construction of diffusing orbits whose projection in the frequency space shadows $\Omega(a)$. There are two things to do next. The first is to switch from shadowing $\Omega(a)$ to shadowing $\Omega'(b)=\rho'_{b}(\omega_1^{*f},b,\omega_3^{*i},\ldots, \omega_n^{*i})$, which will be done in Proposition \ref{PropSwitch} below. The second is to construct diffusing orbit shadowing $\Omega'(b)$. However, there is a problem. We should start with a frequency vector of the form $\rho'_{\frac{p_2}{q_2}\omega_2^{*i}}\omega_2^{*i}(\frac{\omega_{1}^{*f}}{\omega_2^{*i}},\frac{p_2}{q_2},\frac{p_3}{q_3},\ldots,\frac{p_n}{q_n})$ obtained after all the above procedure along $\Omega(a)$ and try to move the second entry from $\frac{p_2}{q_2}$ to $\frac{\omega_{2}^{*f}}{\omega_2^{*i}}$. Note that the above procedure of reduction relies crucially on the choice that $|\frac{p_{i+1}}{q_{i+1}}\omega_2^{*i}-\omega^{*i}_{i+1}|\ll|\frac{p_{i}}{q_{i}}\omega_2^{*i}-\omega^{*i}_{i}|$ when refining $\omega^{(i-2)}$ to $\omega^{(i-1)}$, so that the newly introduced resonance $\bk^{(i-1)}$ due to $\frac{p_{i+1}}{q_{i+1}}$ is much weaker (has a much longer length) than the previous resonances $\bk',\ldots,\bk^{(i-2)}$, so that the Fourier modes in span$\{\bk^{(i-1)}\}$ are treated as a small perturbation which does not destroy the wNHICs constructed along $\omega^{(i-2)}$. Moreover, $\frac{p_{i+1}}{q_{i+1}}$ can only be determined when the wNHIC along $\omega^{(i-2)}$ is determined. The rational numbers in $\omega^{(n-2)}$ cannot be determined simultaneously ahead of time. So in addition to the frequency refinement procedure describe above along one frequency segment, in Section \ref{SSFrequency} we will design a procedure to refine all the frequency segments $\Omega_{i,[j]},\ j=1,\ldots,n,\ i=1,\ldots,M-1$ in Section \ref{SSFrequency0} simultaneously into resonant frequency segments of multiplicity at least $n-2$.
\subsection{Frequency refinement: general strategy}\label{SSFrequency}
Let $\bs\omega_i^*,\ i=1,\ldots,M$ be the frequencies and $\Omega_{i,[j]},\ j=1,\ldots,n,\ i=1,\ldots,M-1$ be the frequency segments defined in Section \ref{SSFrequency0}. Up to permutation of entries and a scalar multiple, we can list the frequency segments as follows:
$$\begin{array}{c}
\Omega_{1,[1]}(a)=\\
\Omega_{1,[2]}(a)=\\
\cdots\\
\Omega_{1, [n]}(a)=\\
\Omega_{2, [1]}(a)=\\
\cdots\\
\end{array}\quad\begin{array}{ccccccccccccc}
a&\omega_{1,2}^{*}&\omega_{1,3}^{*}&\ldots& \omega_{1,n}^{*}\\
\ &a&\omega_{1,3}^{*}&\ldots&  \omega_{1,n}^{*}& \omega_{2,1}^{*}\\
\ &\ &\ &\cdots &\ &\cdots&\ &\cdots\\
\ &\ &\ &\ &a&\omega_{2,1}^{*}&\omega_{2,2}^{*}&\ldots& \omega_{2, n-1}^{*}\\
\ &\ &\ &\ &\ &a&\omega_{2,2}^{*}&\ldots& \omega_{2, n-1}^{*}&\omega_{2, n}^{*}\\
\ &\ &\ &\ &\ &\cdots&\ &\cdots\\
\end{array}.$$
The rules are as follows. 
\begin{enumerate}
\item In $\Omega_{i,[j]}(a)$, the $j$-th entry is $a\in [\omega^{*}_{i,[j]}-\varrho, \omega^{*}_{i+1,[j]}+\varrho]$. The entries with subscripts $<j$ coincide with that of $\bs\omega^*_{i+1}$ and entries with subscripts $>j$ coincide with that of $\bs\omega^*_i$. 
\item We permute entries of $\Omega_{i,[j]}$ in such a way that $a$ is the leading entry and the entries with subscripts $<j$ are placed after its last entry. 
\item The vectors $\Omega_{i,[j]}$, $i=1,\ldots,M-1$ and $j=1,\ldots,n$ are arrayed in a parallelogram such that $\Omega_{i,[j]}$ is placed on the $((i-1)n+j)$-th row with the leading entry $a$ placed at the $((i-1)n+j)$-th column. 
\end{enumerate}We will inductively refine the frequency segment such that after $Mn-3$ steps, all the above $\omega^*_{i,j}$s become a rational multiple of $\omega^*_{1,2}$. Denoting the resulting vector by $\bs\omega^\sharp_{i,[j]}(a)$, we have $|\bs\omega^\sharp_{i,[j]}(a)-\Omega_{i, [j]}(a)|<\varrho$.
$$\begin{array}{c}
\bs\omega^\sharp_{1,[1]}(a)=\\
\bs\omega^\sharp_{1,[2]}(a)=\\
\cdots\\
\bs\omega^\sharp_{1,[n]}(a)=\\
\bs\omega^\sharp_{2,[1]}(a)=\\
\cdots\\
\end{array} \begin{array}{ccccccccccccc}
\frac{a}{\omega_{1,2}^{*}}&\frac{p_{2}}{q_{2}}&\frac{p_{3}}{q_{3}}&\ldots& \frac{p_{n}}{q_{n}}\\
\ &\frac{a}{\omega_{1,2}^{*}}&\frac{p_{3}}{q_{3}}&\ldots& \frac{p_{n}}{q_{n}}&\frac{p_{n+1}}{q_{n+1}}\\
\ &\ &\ &\cdots &\ &\cdots&\ &\cdots\\
\ &\ &\ &\ &\frac{a}{\omega_{1,2}^{*}}&\frac{p_{n+1}}{q_{n+1}}&\frac{p_{n+2}}{q_{n+2}}&\ldots& \frac{p_{2n-1}}{q_{2n-1}}\\
\ &\ &\ &\ &\ &\frac{a}{\omega_{1,2}^{*}}&\frac{p_{n+2}}{q_{n+2}}&\ldots& \frac{p_{2n-1}}{q_{2n-1}}& \frac{p_{2n}}{q_{2n}}\\
\ &\ &\ &\ &\ &\cdots&\ &\cdots\\
\end{array}.$$

The frequency refinements are done inductively as follows. 
We introduce the superscript $(\ell)$ with $1\leq \ell\leq Mn-3$ counting the step of refinement. During the $\ell$-th step of order reduction, we modify the Diophantine number in the $(\ell+2)$-nd column into $\frac{p_{\ell+1}}{q_{\ell+1}}\omega_{1,2}^*$, where the number $\frac{p_{\ell+1}}{q_{\ell+1}}$ is to be determined. 
\begin{Not}
\begin{enumerate}
\item For each $1\leq \ell\leq Mn-3$, there is an index set $\mathcal I(\ell)$ such that for each $(i,[j])\in \mathcal I(\ell)$, the $\Omega_{i,[j]}$ intersects the $(\ell+2)$-th column of the table not at the $a$ entry. If $(i,[j])\in \mathcal I(\ell)$, we denote the frequency vector by $\omega^{(\ell)}_{i,[j]}$. 
\item If the frequency vector with subscript  $(i,[j])$ lies entirely to the left of the $(\ell+2)$-nd column without intersecting it at step $(\ell)$, this vector has completed its refinement and has all entries rational multiples of $\omega^*_{1,2}$ except the leading $a$, so we denote it by $\bs\omega^\sharp_{i,[j]}.$ 
\item The frequency vector $\Omega_{i,[j]}$ lying to the right of the $\ell+2$-nd column without intersecting it will maintain its notation $\Omega_{i,[j]}$. 
\end{enumerate}
\end{Not}
For example, in the case of $\ell\leq n-2,$ we have that $\omega^{(\ell)}_{1,[1]}(a),
\omega^{(\ell)}_{1,[2]}(a),
\cdots
,\omega^{(\ell)}_{1,[\ell]}(a)$ are the following respectively: 
\begin{equation*}
\begin{aligned}
\begin{array}{cccccccccccccccccc}
a&\frac{p_{1,2}}{q_{1,2}}\omega_{1,2}^{*}&\frac{p_{1,3}}{q_{1,3}}\omega_{1,2}^{*}&\ldots&\frac{p_{1,\ell+1}}{q_{1,\ell+1}}\omega_{1,2}^{*}& \frac{p_{1,\ell+2}}{q_{1,\ell+2}}\omega_{1,2}^{*}&\omega^{*}_{1,\ell+3}&\ldots&\omega^{*}_{1,n}\\
\ &a&\frac{p_{1,3}}{q_{1,3}}\omega_{1,2}^{*}&\ldots&\frac{p_{1,\ell+1}}{q_{1,\ell+1}}\omega_{1,2}^{*}& \frac{p_{1,\ell+2}}{q_{1,\ell+2}}\omega_{1,2}^{*}&\omega^{*}_{1,\ell+3}&\ldots&\omega^{*}_{1,n}&\omega^{*}_{2,1}\\
\ &\ &\ &\cdots &\ &\cdots&\ &\cdots\\
\ &\ &\ &\ &a&\frac{p_{1,\ell+2}}{q_{1,\ell+2}}\omega_{1,2}^{*}&\omega^{*}_{1,\ell+3}&\ldots&\omega^{*}_{1,n}&\omega^{*}_{2,1}&\ldots&\omega_{2,\ell-3}^{*}\\
\end{array}.\end{aligned}\end{equation*}

The choice of the rational number during the step $\ell$, as we have seen from Section \ref{STriple}, relies crucially on the dynamics determined by $P$ along the frequencies $\omega^{(\ell-1)}_{i,[j]}$. In this section, we will show how to make the refinement going from step $\ell$ to step $\ell+1$ and study the dynamics along the frequencies $\omega^{(\ell)}_{i,[j]}$ and $\omega^{(\ell+1)}_{i,[j]}$.

Note that the frequency segments $\omega^{(\ell)}_{i,[j]}(a), \ (i,[j])\in \mathcal I(\ell)$ has distinct numbers of independent resonances holding for all $a$.  
\begin{Not}
To simplify the notation and for clarity,  instead of using double subscript $(i,[j])\in \mathcal I(\ell)$ we introduce a single subscript $\kappa=0,1,\ldots,\sharp \mathcal I (\ell)-1$ ($\kappa\leq n-2$) counting the number of independent irreducible resonant integer vectors for each $\omega_{\kappa}^{(\ell)}(a)$ for all $a$. 
\end{Not}
\subsection{Two types of resonances and normal forms}

Suppose we have completed step $\ell$ and are about to work on the $(\ell+1)$-st step of the induction. At step $\ell$, we are handed with the following data:
\begin{enumerate}
\item 
for each  $\kappa=0,\ldots,\sharp\mathcal I(\ell)-1$, we have a frequency segment $\omega^{(\ell)}_{\kappa}(a)$,
\item a number $\mu^{(\ell)}$: the size of the neighborhood of $\omega^{(\ell)}_{\kappa}(a)$ for all $\kappa$;
\item associated to each $\omega^{(\ell)}_{\kappa}(a)$ for all $a$,  a collection of irreducible resonant integer vectors $\bK_{\kappa}^{(\ell)}=\{\bk'^{(\ell)}_\kappa,\ldots, \bk^{(\kappa),(\ell)}_\kappa\}$, and for some $a$, there is one more denoted by $\bk^{o,(\ell)}_{\kappa}$. We denote $\bK^{o,(\ell)}_{\kappa}=\bK^{(\ell)}_{\kappa}\cup\{\bk^{o,(\ell)}_{\kappa}\}.$ By definition, we have $\sharp \bK^{(\ell)}_{\kappa}=\kappa$ and $\sharp \bK^{o,(\ell)}_{\kappa}=\kappa+1$. 
\end{enumerate}
We next pick a rational number $\frac{p_{\ell+1}}{ q_{\ell+1}}$ such that $\frac{p_{\ell+1}}{ q_{\ell+1}}\omega_{1,2}^{*}$ is within $\mu^{(\ell)}$-distance of the irrational number on the $(\ell+3)$-rd column of the table. 

This introduces $\mathcal I(\ell+1), \omega^{(\ell+1)}_{\kappa+1}(a), \bK^{(\ell+1)}_{\kappa+1}, \bK^{o,(\ell+1)}_{\kappa+1}$
as before.  We have $\bK_{\kappa}^{(\ell)}\subset \bK^{(\ell+1)}_{\kappa+1} $ and $\bK_{\kappa}^{o,(\ell)}\subset \bK^{o,(\ell+1)}_{\kappa+1}$.

When we update from $(\ell)$ to $(\ell+1)$, the subscript $(i,[j])$ remains unchanged, but the subscript $\kappa$ associated to each $(i,[j])$ will also update to $\kappa+1$. The $\kappa=0$ case is handled in Section \ref{sct:NormalForm}, \ref{SSingle} and \ref{SDouble}, and the $\kappa=1$ case was done in Section \ref{STriple}. The $\kappa=n-2$ case means that the frequency segment $\omega^{(\ell)}_{i,[j]}(a)$ has completed the reduction of order procedure so it becomes $\bs\omega^\sharp_{i,[j]}$ and it will be treated in Section \ref{SSCompleteRes}. So in remaining part of this Subection till Section \ref{SSCompleteRes}, we will consider the range $\ell=0,\ldots, Mn-3$ and $\kappa=0,1,\ldots,\min\{\sharp \mathcal I(\ell)-1, n-3\}$.


\subsubsection{Two types of resonances}
The following lemma is an analogue of Lemma \ref{Lm:nbd} and Lemma \ref{Lm:nbd3}. 

\blm\label{Lm:nbdell}
Let $\omega^{(\ell)}_{\kappa}(a),$ $\mu^{(\ell)}, $ and $\omega^{(\ell+1)}_{\kappa+1}(a),\, \bK^{(\ell+1)}_{\kappa+1},\ \kappa=0,1,\ldots,\min\{\sharp \mathcal I(\ell)-1, n-3\}$ be as above.
For any $ K^{(\ell+1)}> \max_\kappa |\bK^{(\ell)}_{\kappa}|$,  let $\bk^o_{\kappa+1, a^o_i}$, $i=1,\ldots,m_{\kappa}$, be the collection of all the integer vectors in $\Z_{K^{(\ell+1)}}^n\setminus\mathrm{span}\bK^{(\ell+1)}_{\kappa+1}$ satisfying $\langle \bk^o_{\kappa+1,a^o_i},\omega^{(\ell+1)}_{\kappa+1}(a_i^o)\rangle=0$ for some $a_i^o$, and let $(\bk^o_{\kappa+1,a^o_i})^\perp$ be the $(n-1)$-dimensional space orthogonal to the vector $\bk^o_{\kappa+1, a^o_i}$. Then there exists $\mu^{(\ell+1)}=\mu^{(\ell+1)}(K^{(\ell+1)})$ with $
B(\omega^{(\ell+1)}_{\kappa+1}(a),\mu^{(\ell+1)})\subset
B(\omega^{(\ell)}_{\kappa}(a),\mu^{(\ell)})$ and
\ben
\item
for all $\omega\in
B\left(\omega^{(\ell+1)}_{\kappa+1}(a),\mu^{(\ell+1)}\right)\setminus \bigcup_{i}B\l(\omega^{(\ell+1)}_{\kappa+1}(a^o_i)+(\bk^o_{\kappa+1,a^o_i})^\perp,\eps^{1/3}\r),
$
and for sufficiently small $\eps$ we have
$$|
\langle \bk,\omega\rangle|> \eps^{{1/3}},\quad \forall\ \bk\in \Z_{ K^{(\ell+1)}}^n\setminus\mathrm{span}_\Z\bK^{(\ell+1)}_{\kappa+1}.
$$
\item for all $\omega \in B\left(\omega^{(\ell+1)}_{\kappa+1}(a),\mu^{(\ell+1)}\right)\bigcap B\l(\omega^{(\ell+1)}_{\kappa+1}(a^o_i)+(\bk^o_{\kappa+1,a^o_i})^\perp,\eps^{1/3}\r)$, for each $i$ and for all
    $
    \bk\in \Z_{ K^{(\ell+1)}}^n\setminus\mathrm{span}_\Z\bK^{o,(\ell+1)}_{\kappa+1},
    $
we have
\begin{equation}\label{EqDiop1/2,3}
|\langle\bk,\omega\rangle|\geq nK^{(\ell+1)}\mu^{(\ell+1)}.
\end{equation}
\een
\elm
Note that in the lemma, our choice of $\mu^{(\ell+1)}$ and $K^{(\ell+1)}$ are independent of the subscript $\kappa$. We will next introduce a small parameter $\dt^{(\ell+1)}$, independent of $\kappa$, to determine $K^{(\ell+1)}$ hence $\mu^{(\ell+1)}$. 
\subsubsection{The KAM normal forms}
Now we determine the resonance sub manifold as $$\Sigma(\bK^{o,(\ell)}_{\kappa}):=\{y\ |\ \langle\bk_{\kappa}^{o,(\ell)},\omega(y)\rangle=\langle\bk_{\kappa}'^{(\ell)},\omega(y)\rangle=\cdots=\langle\bk_{\kappa}^{(\kappa),(\ell)},\omega(y)\rangle=0\}.$$

Lemma \ref{Lm:nbdell} allows us to apply Proposition \ref{prop: codim1} in the two cases in Lemma \ref{Lm:nbdell} to obtain the following normal forms.
\begin{lem} \label{LmNormalFormell}
Let $\dt^{(\ell+1)}$ be a small number satisfying $\dt^{(\ell+1)}<\min_\kappa\{3(|\bK^{(\ell)}_{\kappa}|)^{-2}\}$ and denote  $K^{(\ell+1)}=(\dt^{(\ell+1)}/3)^{-1/2}$. Then there exists $\eps^{(\ell+1)}_1=\eps^{(\ell+1)}_1(\dt^{(\ell+1)},\Lambda)$ such that for all $\eps<\eps^{(\ell+1)}_1$, the following holds. 
Suppose $\omega^\star$ is in case $(1)$ in Lemma \ref{Lm:nbdell},  then there exists a symplectic transform $\phi_{\kappa+1}^{(\ell+1)}$ defined on $B(0,\Lambda)\times \T^n$ that is $o_{\eps\to 0}(1)$ close to identity in the $C^{r}$ norm, such that \begin{equation}\label{EqNormalFormell}
\begin{aligned}
\sH^{(\ell+1)}_{\kappa+1,\dt^{(\ell+1)}}:&=\sH\circ \phi_{\kappa+1}^{(\ell+1)}(x,Y)\\
&=\dfrac{1}{\sqrt{\eps}}\langle \omega^\star, Y\rangle+\dfrac{1}{2}\langle \sA Y,Y\rangle+ \Pi_{\bK^{(\ell+1)}_{\kappa+1}}\sV+\dt^{(\ell+1)}\sR^{(\ell+1)}_{\kappa+1}(x,Y),
\end{aligned}
\end{equation}
where
\begin{enumerate}
\item  $\sA$ and $\sV$ are the same as that in Lemma \ref{LmNormalForm1}.
\item $\sR^{(\ell+1)}_{\kappa+1}(x,Y)=\sR^{(\ell+1)}_{\kappa+1,I}(x)+\sR^{(\ell+1)}_{\kappa+1,II}(x,Y)$, where $\sR^{(\ell+1)}_{\kappa+1,I}$ consists of Fourier modes of $\sV$ not in  $\Z^n_{K^{(\ell+1)}}\cup\mathrm{span}_\Z\bK^{(\ell+1)}_{\kappa+1}$, and we have $|\sR^{(\ell+1)}_{\kappa+1,I}|_{{r-2}}\leq 1$ and $|\sR^{(\ell+1)}_{\kappa+1,II}|_{{r-5}}\leq 1$.
\end{enumerate}
\end{lem}
\begin{lem}\label{LmNormalFormell+1}
Let $\dt^{(\ell+1)}$ and $K^{(\ell+1)}$ be as in the previous lemma,  then there exists $\eps^{(\ell+1)}_2=\eps^{(\ell+1)}_2(\dt^{(\ell+1)},\Lambda)$ such that for all $\eps<\eps^{(\ell+1)}_2$ and any $y^\star$ such that $\omega^\star=\omega(y^\star)$ is as in case $(2)$ in Lemma \ref{Lm:nbdell}, there exists a symplectic transform $\phi_{\kappa+1}^{(\ell+1)}$ defined on $B(0,\Lambda)\times \T^n$ that is $o_{\eps\to 0}(1)$ close to identity in the $C^{r}$ norm, such that
\begin{equation}\label{EqNormalFormell+1}
\begin{aligned}
\sH^{(\ell+1)}_{\kappa+1,\dt^{(\ell+1)}}:&=\sH\circ  \phi_{\kappa+1}^{(\ell+1)}(x,Y)\\
&=\dfrac{1}{\sqrt{\eps}}\langle \omega^\star, Y\rangle+\dfrac{1}{2}\langle \sA Y,Y\rangle+ \Pi_{\bK^{o,(\ell+1)}_{\kappa+1}}\sV+\dt^{(\ell+1)} \sR^{(\ell+1)}_{\kappa+1}(x,Y),
\end{aligned}
\end{equation}
where
\begin{enumerate}
\item  $\sA$ and $\sV$ are the same as that in Lemma \ref{LmNormalForm1}.
\item $\sR^{(\ell+1)}_{\kappa+1}(x,Y)=\sR^{(\ell+1)}_{\kappa+1,I}(x)+\sR^{(\ell+1)}_{\kappa+1,II}(x,Y)$, where $\sR^{(\ell+1)}_{\kappa+1,I}$ consists of Fourier modes of $\sV$ not in $\Z^n_{K^{(\ell+1)}}\cup \mathrm{span}_\Z\bK^{o,(\ell+1)}_{\kappa+1}$, and we have $|\sR^{(\ell+1)}_{\kappa+1,I}|_{{r-2}}\leq 1$ and $|\sR^{(\ell+1)}_{\kappa+1,II}|_{{r-5}}\leq 1$.
\end{enumerate}
\end{lem}
\subsection{NHICs away from strong resonances}

The following result is an analogue of Proposition \ref{PropGenericity2}, which will be used to establish the existence of NHIC.
\begin{pro}\label{ProNDGInduction}
Suppose there exists an open-dense set $\mathcal O_{\kappa}^{(\ell)}\subset \Pi_{\bK_{\kappa}^{(\ell)}}C^{r}(T^*\T^n),\ r\geq 5$ such that each $\Pi_{\bK_{\kappa}^{(\ell)}}P(y,\cdot )\in \mathcal O_{\kappa}^{(\ell)}$ has a unique nondegenerate global max along the segment $y\in \omega^{-1}(\omega^{(\ell)}_{\kappa}(a))$, up to finitely many bifurcations, where there are two nondegenerate global max. 

 Then for each $\Pi_{\bK_{\kappa}^{(\ell)}}P \in \mathcal O_{\kappa}^{(\ell)}$ there exists $\dt^{(\ell)}_{0,\kappa}=\dt^{(\ell)}_{0,\kappa}(\Pi_{\bK_{\kappa}^{(\ell)}}P)$, such that defining $\dt^{(\ell)}_{0}=\min_\kappa \dt^{(\ell)}_{0,\kappa}$, for any $\dt^{(\ell)}<\dt^{(\ell)}_{0}$,  $K^{(\ell)}=(\dt^{(\ell)}/3)^{1/2}$ and $\mu^{(\ell)}=\mu^{(\ell)}(K^{(\ell)})$ as in \eqref{EqKmu}, and choosing $\omega^{(\ell+1)}_{\kappa+1}(a)\subset B(\omega^{(\ell)}_{\kappa}(a),\mu^{(\ell)})$ associated to irreducible $\bK_{\kappa+1}^{(\ell+1)}$, we have an open-dense set $\mathcal O_{\kappa+1}^{(\ell+1)}= \mathcal O_{\kappa+1}^{(\ell+1)}(\Pi_{\bK_{\kappa}^{(\ell)}}P)$ in  the unit ball of \\ $\Pi_{\bK_{\kappa+1}^{(\ell+1)}}C^{r}(T^*\T^n)/\Pi_{\bK_{\kappa}^{(\ell)}}C^{r}(T^*\T^n),\ r\geq 5,$ such that for each $\Pi_{\bK_{\kappa+1}^{(\ell+1)}}P$ with $$\Pi_{\bK_{\kappa}^{(\ell)}}(\Pi_{\bK_{\kappa+1}^{(\ell+1)}}P)=\Pi_{\bK_{\kappa}^{(\ell)}}P, \mathrm{\ and\ }\Pi_{\bK_{\kappa+1}^{(\ell+1)}}P-\Pi_{\bK_{\kappa}^{(\ell)}}P\in \mathcal O_{\kappa+1}^{(\ell+1)},$$ we have that  $\Pi_{\bK_{\kappa+1}^{(\ell+1)}}P(y,\cdot )$ has a unique nondegenerate global max along the segment $y\in \omega^{-1}(\omega^{(\ell+1)}_{\kappa+1}(a))$, up to finitely many bifurcations, where there are two nondegenerate global max.
\end{pro}
Similar to Lemma \ref{LmResidualKU}, by an application of the Karatowski-Ulam Theorem \ref{ThmKU}, we get that $\cup_{\Pi_{\bK_{\kappa}^{(\ell)}}P \in \mathcal O_{\kappa}^{(\ell)}}\mathcal O_{\kappa+1}^{(\ell+1)}(\Pi_{\bK_{\kappa}^{(\ell)}}P)$ intersects the unit ball of $C^r(T^*\T^n)$ in an open-dense set of the latter. 

In the case of Lemma \ref{Lm:nbdell}(1) and Lemma \ref{LmNormalFormell}, we can repeat the argument of Proposition \ref{NHICSingle} to find wNHICs $\cC(\bK_{\kappa+1}^{(\ell+1)})$ homeomorphic to $T^*\T^{n-\kappa-1}$ along the frequency $\omega_{\kappa+1}^{(\ell+1)}(a)$ with the help of Proposition \ref{ProNDGInduction}. In the case of Lemma \ref{Lm:nbdell}(2) and Lemma \ref{LmNormalFormell+1} in the presence of an extra resonance $\bk^{o,(\ell+1)}_{\kappa+1}$, the wNHICs may or may not exist. When the wNHICs do not exist, we denote the corresponding resonant submanifolds by $\Sigma_!(\bK_{\kappa+1}^{o,(\ell+1)})$. We have the following result. 

\begin{pro}\label{PropSingleell}
Let $\Pi_{\bK_{\kappa+1}^{(\ell+1)}}P$ be as in the conclusion of Proposition \ref{ProNDGInduction} with $r\geq 7$. Then for any $\lambda>0$, there exists $\dt_{1}^{(\ell+1)}$ such that in the system $\sH^{(\ell+1)}_{\kappa+1,\dt^{(\ell+1)}}$, for all $0<\dt^{(\ell+1)}<\dt^{(\ell+1)}_1$ we have the following,
\begin{enumerate}
\item there exists a $C^r$ wNHIC $\cC(\bK^{(\ell+1)}_{\kappa+1})$ homeomorphic to $T^*\T^{n-\kappa-1}$ up to finitely many bifurcations;
\item Mather set lying in $B(0,\Lambda)\times \T^n$ and with rotation vector orthogonal to $\bK^{(\ell+1)}_{\kappa+1}$ lies inside $\cC(\bK^{(\ell+1)}_{\kappa+1})$, provided the rotation vector does not intersect the $\lambda\sqrt\eps$-neighborhood of $\partial h(\Sigma_!(\bK_{\kappa+1}^{o,(\ell+1)}))$;
\item the normal hyperbolicity is independent of $\eps$ or $\dt^{(\ell+1)}$.
\end{enumerate}
\end{pro}
We next focus on the case (2) of Lemma \ref{Lm:nbdell}.
\subsection{Induction around strong resonances}\label{SSInductionDouble}
The material in this section is a higher dimensional generalization of that in Section \ref{SSReduction+}, \ref{SSAlpha3} and \ref{SSLadder}.
In this section, we perform the reduction of order around a resonance as in the case of Lemma \ref{Lm:nbdell}(2) and Lemma \ref{LmNormalFormell+1} in the presence of an extra resonance $\bk^{o,(\ell+1)}_{\kappa+1}$. For given $(i,[j])$, the extra resonance may appear during the $\kappa$-th step of reduction of order. Without loss of generality, we assume we encounter the extra resonance point during the $\kappa=0$ step of reduction of order. In this case $\bk'^{(\ell+1)}_{\kappa+1}$ and $\bk_{\kappa+1}^{o,(\ell+1)}$ have comparable lengths and are much shorter than other vectors in $\bK_{\kappa+1}^{o,(\ell+1)}$.
\subsubsection{The linear symplectic transform and Hamiltonian normal form}
We construct a matrix $M^{(\ell+1)}_{\kappa+1}\in \mathrm{SL}(n,\Z)$ ,$\kappa=0,\ldots,\min\{\sharp\mathcal I(\ell), n-3\}$, whose first $\kappa+2$ rows are exactly the vectors in $\bK^{o,(\ell+1)}_{\kappa+1}$ ordered as $\bk_{\kappa+1}^{o,(\ell+1)},\bk'^{(\ell+1)}_{\kappa+1},\ldots,\bk_{\kappa+1}^{(\kappa+1),(\ell+1)} $. This is always possible by applying Lemma \ref{LmM''} repeatedly. The matrix $M^{(\ell+1)}_{\kappa+1}$ induces a symplectic transformation
\[ \mathfrak M^{(\ell+1)}_{\kappa+1}: (x,Y)\mapsto( M^{(\ell+1)}_{\kappa+1} x, (M^{(\ell+1)}_{\kappa+1})^{-t}Y).\]
We denote $\sA^{(\ell+1)}_{\kappa+1}= M^{(\ell+1)}_{\kappa+1} \sA  (M^{(\ell+1)}_{\kappa+1})^t$. Then the $(i,j)$-th entry of $\sA^{(\ell+1)}_{\kappa+1}$ is given by $\bk_{\kappa+1}^{(i-1),(\ell+1)}\sA(\bk_{\kappa+1}^{(j-1),(\ell+1)})^t$, $i,j=1,\ldots,\kappa+2,$ and we count $o$ as $0$.

We choose the base point $y^\star$ such that the frequency vector $\omega^\star=\omega(y^\star)\in \Sigma(\bK^{o,(\ell+1)}_{\kappa+1})$, then we get the transformed frequency vector $M^{(\ell+1)}_{\kappa+1}\omega^\star$ has zero as the first $\kappa+2$ entries. We denote  $M^{(\ell+1)}_{\kappa+1}\omega^\star=(0,\hat\omega^{(\ell+1)}_{\kappa+1})\in \R^n$ for some vector $\hat \omega^{(\ell+1)}_{\kappa+1}\in \R^{n-\kappa-2}.$

The Hamiltonian \eqref{EqNormalFormell+1} under the transformation becomes
\begin{equation}\label{Eqsympj}
\begin{aligned}
&(\mathfrak M^{(\ell+1)}_{\kappa+1})^{-1*}\sH^{(\ell+1)}_{\kappa+1,\dt^{(\ell+1)}}\\
=&\dfrac{1}{\sqrt\eps }\langle \hat \omega^{(\ell+1)}_{\kappa+1},\hat Y\rangle+\dfrac{1}{2} \langle \sA^{(\ell+1)}_{\kappa+1} Y ,Y \rangle+V_{\kappa+1}^{(\ell+1)}(x_1,\ldots,x_{\kappa+2})+\dt^{(\ell+1)}R^{(\ell+1)}_{\kappa+1}(x,Y),
\end{aligned}
\end{equation}
where $V_{\kappa+1}^{(\ell+1)}= (\mathfrak M^{(\ell+1)}_{\kappa+1})^{-1}\Pi_{\bK^{o,(\ell+1)}_{\kappa+1}}\sV$ and $R^{(\ell+1)}_{\kappa+1}=(\mathfrak M^{(\ell+1)}_{\kappa+1})^{-1*}\sR^{(\ell+1)}_{\kappa+1}$.

We denote by $A^{(\ell+1)}_{\kappa+1}$ the first $(\kappa+2)\times (\kappa+2)$ block of $\sA^{(\ell+1)}_{\kappa+1}$ and by $A^{(\ell)}_{\kappa}$ the first $(\kappa+1)\times (\kappa+1)$ block of $\sA^{(\ell+1)}_{\kappa+1}$. Note that $A^{(\ell)}_{\kappa}$ depends only on $\sA$ and $\bK^{o,(\ell)}_{\kappa}$ but does not depend on $\bk^{(\kappa+1),(\ell+1)}_{\kappa+1}$.

Next we introduce two subsystems \begin{equation}\begin{aligned}\sG_{\kappa}^{(\ell)}&=\dfrac{1}{2} \langle \sA^{(\ell)}_{\kappa} Y_{\kappa}^{(\ell)} ,Y_{\kappa}^{(\ell)} \rangle+V_{\kappa}^{(\ell)}(x_{\kappa}^{(\ell)}),\quad T^*\T^{\kappa+1}\to \R,\\
\sG_{\kappa+1}^{(\ell+1)}&=\dfrac{1}{2} \langle \sA^{(\ell+1)}_{\kappa+1} Y_{\kappa+1}^{(\ell+1)} ,Y_{\kappa+1}^{(\ell+1)} \rangle+V_{\kappa+1}^{(\ell+1)}(x_{\kappa+1}^{(\ell+1)}),\ T^*\T^{\kappa+2}\to \R.
\end{aligned}
\end{equation}
Defining $\dt^{(\ell)}\bar V_{\kappa+1}^{(\ell+1)}(x_{\kappa+1}^{(\ell+1)}):=V_{\kappa+1}^{(\ell+1)}-V_{\kappa}^{(\ell)}$, we have 
$$\|\dt^{(\ell)}\bar V_{\kappa+1}^{(\ell+1)}\|_{C^{r-2}}\leq \frac{1}{|\bk^{(\kappa+1),(\ell+1)}_{\kappa+1}|^2}\leq \dt^{(\ell)}/3,$$ since the difference comes from the Fourier modes in $\sV$ containing $\bk^{(\kappa+1),(\ell+1)}_{\kappa+1}$ whose length is greater than $K^{(\ell)}=(\dt^{(\ell)}/3)^{1/2}$.


\subsubsection{The induction}
In the following, without loss of generality, {\it we fix $\lb$ such that $\al_{\tilde\sG}(\partial\beta_{\tilde\sG}(\lambda(1,0)))< \tilde\Delta_0$, where $\tilde\Delta_0$ and $\tilde\sG$ $($see equation \eqref{EqtsG}$)$depend only on $\Pi_{\bk',
\bk^o}\sV$ but not on other resonant integer vectors $($Proposition \ref{PropManeBroken}$).$ This is assumed in  Proposition \ref{PropAlpha} and Proposition \ref{PropCross3Res}.}

We make the following inductive hypothesis.

\noindent {\bf The Inductive Hypothesis:}

{\it  There exists an open-dense set $\mathcal O^{(\ell)}_{\kappa}$ in the unit ball of $\Pi_{\bK^{o,(\ell)}_{\kappa}}C^r(\T^{n}), \ r\geq 7,$ such that for each $\sV$ with $\Pi_{\bK^{o,(\ell)}_{\kappa}} V\in \mathcal O^{(\ell)}_{\kappa}$, we have the following for the system $\sG^{(\ell)}_{\kappa}$, $\kappa=0,1,\ldots,\min\{\sharp \mathcal I(\ell)-1, n-3\}$,

\begin{enumerate}
\item up to finitely many bifurcations, there exists a NHIC homeomorphic to $T^*\T^1$ foliated by Mather sets of rotation vector $\nu(1,0,\ldots,0)\in H_1(\T^{\kappa+1},\R)$, $|\nu|>\lb$. Each Mather set is a hyperbolic periodic orbit and at each bifurcation point, the Mather set consists of two periodic orbits;
\item the normal hyperbolicity is independent of $\dt^{(\ell)}$;
\item there is a generalized transition chain connecting the channels \\ $\C^{(\ell)}_{\kappa,\pm}:=\{\partial \beta_{ \sG_{\kappa}^{(\ell)}}(\nu(1,0\ldots,0))\ |\ \pm\nu>\lambda\}\subset H^1(\T^{\kappa+1},\R)$.
\end{enumerate}}

With the hypothesis, we have the following result

\begin{pro}\label{PropDoubleell}
Assume the Inductive Hypothesis above, then for each $\sV$ with $\Pi_{\bK^{o,(\ell)}_{\kappa}}\sV\in \mathcal O^{(\ell)}_{\kappa}$, there exists $\dt^{(\ell)}_{2}=\dt^{(\ell)}_{2}(\Pi_{\bK^{o,(\ell)}_{\kappa}}\sV),$ such that for all $0<\dt^{(\ell)}\leq \dt^{(\ell)}_{2}$, $K^{(\ell)}=(\dt^{(\ell)}/3)^{1/2}$ and any $\bk^{(\kappa+1),(\ell+1)}_{\kappa+1}$ with $|\bk^{(\kappa+1),(\ell+1)}_{\kappa+1}|> K^{(\ell)}$, there exists 
an open-dense set $\mathcal O^{(\ell+1)}_{\kappa+1}=\mathcal O^{(\ell+1)}_{\kappa+1}(\Pi_{\bK^{o,(\ell)}_{\kappa}}\sV)$  in the unit ball of $\Pi_{\bK^{o,(\ell+1)}_{\kappa+1}}C^r(\T^{n})/\Pi_{\bK^{o,(\ell)}_{\kappa}}C^r(\T^{n}),$  $r\geq7$, such that for each $\Pi_{\bK^{o,(\ell+1)}_{\kappa+1}}\sV\in \mathcal O^{(\ell+1)}_{\kappa+1}$ with $\Pi_{\bK^{o,(\ell)}_{\kappa}}(\Pi_{\bK^{o,(\ell+1)}_{\kappa+1}}\sV)=\Pi_{\bK^{o,(\ell)}_{\kappa}}\sV $ and $\Pi_{\bK^{o,(\ell+1)}_{\kappa+1}}\sV-\Pi_{\bK^{o,(\ell)}_{\kappa}}\sV\in \mathcal O^{(\ell+1)}_{\kappa+1}$, the system $\sG_{\kappa+1}^{(\ell+1)}$, $\kappa=0,1,\ldots,\min\{\sharp \mathcal I(\ell)-1, n-3\}$, satisfies the following
\begin{enumerate}
\item up to finitely many bifurcations, there exists a NHIC homeomorphic to $T^*\T^1$ foliated by Mather sets of rotation vector $\nu(1,0,\ldots,0)\in H_1(\T^{\kappa+2},\R),$ $|\nu|>\lambda$. Each Mather set is a periodic orbit, and at each bifurcation point, the Mather set consists of two periodic orbits;
\item the normal hyperbolicity is independent of $\dt^{(\ell+1)}$;
\item there is a generalized transition chain connecting the channels
$\C^{(\ell+1)}_{\kappa+1,\pm}:=\{\partial \beta_{ \sG_{\kappa+1}^{(\ell+1)}}(\nu(1,0\ldots,0))\ |\ \pm\nu>\lambda\}\subset H^1(\T^{\kappa+2},\R)$.
\end{enumerate}
\end{pro}

\begin{proof}
This proposition is a generalization of Lemma \ref{Prop3DOF} and Proposition \ref{PropCross3Res}. We consider the case of passage from the system $\sG^{(\ell)}_{\kappa}$ at resonance $\bK^{o,(\ell)}_{\kappa}$ to the system  $\sG^{(\ell+1)}_{\kappa+1}$ at resonance $\bK^{o,(\ell+1)}_{\kappa+1}$.


We next form matrices
\beqa\label{EqST}
S^{(\ell+1)}_{\kappa+1}&=\bmt{cc}
\mathrm{id}_{\kappa+1 }&0\\
-\sa^{(\ell+1)}_{\kappa+1} (A^{(\ell)}_{\kappa})^{-1}&1 \\
\emt\in \mathrm{SL}(\kappa+2,\R),
\eeqa
where we denote by $\sa^{(\ell+1)}_{\kappa+1}\in \R^{\kappa+1}$ is obtained by removing the last entry in the last row of the matrix $A^{(\ell+1)}_{\kappa+1}\in \R^{(\kappa+2)^2}$.

We next denote
$$
b^{(\ell+1)}_{\kappa+1}=a^{(\ell+1)}_{\kappa+2,\kappa+2}-\sa^{(\ell+1)}_{\kappa+1} (A^{(\ell)}_{\kappa})^{-1} (\sa^{(\ell+1)}_{\kappa+1})^t
$$
where $a^{(\ell+1)}_{\kappa+2,\kappa+2}$ is the last diagonal entry of $A^{(\ell+1)}_{\kappa+1}$, and
 denote the first entry of the vector $-\sa^{(\ell+1)}_{\kappa+1} (A^{(\ell)}_{\kappa})^{-1}$ by $s^{(\ell+1)}_{\kappa+1}$.

The same argument as Lemma \ref{Lmb3s} gives us
\beq\label{Eqb3s}
b^{(\ell+1)}_{\kappa+1}=\mathrm{const}_{b} |\bk_{\kappa+1}^{(\kappa+1),(\ell+1)}|^2,\quad s^{(\ell+1)}_{\kappa+1}=\mathrm{const}_{s}|\bk^{(\kappa+1),(\ell+1)}_{\kappa+1}|,\eeq where const$_{b}(>0)$ and const$_{s}$ do not depend on $\dt^{(\ell+1)}$.

The matrix $S_{\kappa+1}^{(\ell+1)}$ induces a symplectic map $$\mathfrak S^{(\ell+1)}_{\kappa+1}: T^*\T^{\kappa+2}\to T^*\T_{S^{(\ell+1)}_{\kappa+1}}^{\kappa+2},\quad (x,Y)\mapsto (S^{(\ell+1)}_{\kappa+1}x,(S^{(\ell+1)}_{\kappa+1})^{-t}Y):=(\sx,\sy)$$ such that $(\mathfrak S^{(\ell+1)}_{\kappa+1})^*\sG^{(\ell+1)}_{\kappa+1}:=\mathcal G^{(\ell+1)}_{\kappa+1}$
$$\mathcal G^{(\ell+1)}_{\kappa+1}(\sx,\sy)=\sG^{(\ell)}_{\kappa}(\tilde \sx,\tilde \sy)+\dt^{(\ell)}((S^{(\ell+1)}_{\kappa+1})^{-1})^*\bar V^{(\ell+1)}_{\kappa+1}(\sx)+\frac{b^{(\ell+1)}_{\kappa+1}}{2}\hat \sy^2:\ T^*\T^{\kappa+2}_{S^{(\ell+1)}_{\kappa+1}}\to \R$$
where we use notation $\tilde \sx$, $\tilde \sy$ to denote the first $\kappa+1$ entries of $\sx$ and $\sy$ respectively and $\hat \sx$ and $\hat \sy$ to denote the last entry of $\sx$ and $\sy$ respectively.

 By assumption, $\sG^{(\ell)}_{\kappa}$ admits a NHIC homeomorphic to $T^*\T^1$ restricted to which we can introduce action-angle coordinates $(I,\theta)$ to write the Hamiltonian $\sG^{(\ell)}_{\kappa}$ as $\tilde h^{(\ell)}_{\kappa}(I)$ which is convex (Proposition \ref{PropSymp}). Correspondingly, the system $\sG^{(\ell+1)}_{\kappa+1}$ becomes the system $\bar\sG^{(\ell+1)}_{\kappa+1}:\ T^*\T^2\to \R$ of the form (see the proof of Proposition \ref{Prop3DOF})
$$\bar \sG^{(\ell+1)}_{\kappa+1}(I,\theta, \hat \sx, \hat \sy)=\tilde h^{(\ell)}_{\kappa}(I+s^{(\ell+1)}_{\kappa+1}\hat \sy)+\frac{b^{(\ell+1)}_{\kappa+1}}{2}\hat \sy^2+\dt^{(\ell)} U^{(\ell)}_{\kappa+1}(I,\hat \sy,\theta,\hat\sx).$$

As a Hamiltonian system of two degrees of freedom, we apply Theorem \ref{ThmCZ} to the system $\bar \sG^{(\ell+1)}_{\kappa+1}$ to find a NHIC foliated by periodic orbits in the homology class $(1,0)$. The proof of the hyperbolicity is also contained in the proof of Proposition \ref{Prop3DOF}.

It remains to prove item (3). By the inductive hypothesis, we get that there is a generalized transition chain $\Gamma^{(\ell)}_{\kappa}\subset H^1(\T^{\kappa+1}, \R)$ for the system $\sG^{(\ell)}_{\kappa}$ connecting sending a point $\tilde\sC\in \mathbb C^{(\ell)}_{\kappa,+}$ to $-\tilde\sC\in \mathbb C^{(\ell)}_{\kappa,-}$, where the centrally symmetric channels $\mathbb C^{(\ell)}_{\kappa,\pm}\subset H^1(\T^{\kappa+1},\R)$ correspond to two neighboring pieces of NHICs. Our goal is then to send some point $(\tilde\sC,\hat\sC)\in \mathbb C^{(\ell+1)}_{\kappa+1,+}$ to $-(\tilde\sC,\hat\sC)\in \mathbb C^{(\ell+1)}_{\kappa+1,-}$. Applying the transformation $\mathfrak S^{(\ell+1)}_{\kappa+1}$, we working with the system $\mathcal G^{(\ell+1)}_{\kappa+1}$. 

The path $\Gamma^{(\ell)}_{\kappa}$ determines a new path in the system $\sG^{(\ell+1)}_{\kappa+1}$ as follows. By definition of a generalized transition chain, for each fixed $\hat \sC$ entry such that $(\Gamma^{(\ell)}_{\kappa},\hat \sC)\cap (S^{(\ell+1)}_{\kappa+1})^t\mathbb C^{(\ell+1)}_{\kappa+1,+}\neq \emptyset$, the new path $(\Gamma^{(\ell)}_{\kappa},\hat \sC)$ is a generalized transition chain for the system $\mathcal G^{(\ell+1)}_{\kappa+1}$ with $\dt^{(\ell)}=0$ (see Proposition \ref{PropAlpha}). By the upper-semi-continuity of the Ma\~n\'e set, we get that for sufficiently small $\dt^{(\ell)}$, there exists a generalized transition chain $(\Gamma^{(\ell)}_{\kappa,\dt^{(\ell)}},\hat \sC)\subset H^1(\T^{\kappa+2},\R)$, lying in a $\dt^{(\ell)}$ neighborhood of $(\Gamma^{(\ell)}_{\kappa},\hat \sC)$ and on a fixed level set of $\al_{\mathcal G^{(\ell+1)}_{\kappa+1}}$, such that $\Gamma^{(\ell)}_{\kappa,\dt^{(\ell)}}$ also connects $\mathbb C^{(\ell)}_{\kappa,\pm}$.

However, this path $\Gamma^{(\ell)}_{\kappa,\dt^{(\ell)}}\times \{\hat \sC\}$ does not connect $(S^{(\ell+1)}_{\kappa+1})^t\mathbb C^{(\ell+1)}_{\kappa+1,+}$ to $(S^{(\ell+1)}_{\kappa+1})^t\mathbb C^{(\ell+1)}_{\kappa+1,-}$ and it remains to move $\hat \sC\to -\hat \sC$ by the central symmetry of the channels $\mathbb C^{(\ell+1)}_{\kappa+1,\pm}$. Now we are in the same situation as Lemma \ref{ladder} with $\bar\sG^{(\ell+1)}_{\kappa+1}$ playing the role of $\bar\sG_{3,\dt}$. By the same argument, we construct a generalized transition chain connecting $\mathbb C^{(\ell+1)}_{\kappa+1,\pm}$.

The open-dense set $\mathcal O^{(\ell+1)}_{\kappa+1}$ is constructed in the same way as the proof of Lemma \ref{Prop3DOF} and Lemma \ref{ladder}.
This completes the proof. 

\end{proof}

\subsection{Dynamics around complete resonances}\label{SSCompleteRes}
Suppose for frequency segment with the subscript $(i,[j])$ we have completed all the reduction of orders hence it becomes the frequency $\bs\omega^\sharp_{i,[j]}(a)$, for which there are $(n-2)$ resonant integer vectors $\bk'_{i,[j]},\ldots, \bk^{(n-2)}_{i,[j]}$ for all $a$, and for finitely many $a$'s, there is one more resonant integer vector $\bk_{i,[j]}^o$. We assume each vector is irreducible. In the above Proposition \ref{PropDoubleell}, we take $\kappa+1=n-2$. 

The complete resonance on the energy level $E>\min h$
$$\Sigma_{!}(\bK^{o}_{i,[j]})=\{ y\in h^{-1}(E)\ |\ \langle\bk^o_{i,[j]},\omega(y)\rangle=\langle\bk^1_{i,[j]},\omega(y)\rangle=\ldots=\langle\bk^{(n-2)}_{i,[j]},\omega(y)\rangle=0\}$$
is a point. We choose $y^\star\in \Sigma_{!}(\bK^{o}_{i,[j]})$ so that $\omega^\star=\omega(y^\star)$ is such a complete resonant point. In the remaining part of this subsection, we omit the subscript $(i,[j])$ for simplicity. 

We introduce a matrix $M^\sharp\in \mathrm{SL}(n,\Z)$ whose first $n-1$ rows are $\bk^o,\bk,\ldots, \bk^{(n-2)}$.

We first apply Proposition \ref{LmNormalFormell+1} to get a Hamiltonian normal form.  We next introduce a linear symplectic transformation $$\mathfrak M^\sharp:\ (x,Y)\mapsto (M^\sharp x,(M^\sharp)^{-t}Y)$$ We denote $\sA^\sharp=M^\sharp\sA (M^\sharp)^t$.
The transformed frequency has the form $M^\sharp\omega^\star=(0,\ldots, 0,\omega_{n}),\ \omega_{n}\neq 0$. Applying the symplectic transformation $\mathfrak M^\sharp$ to the normal form, one obtains a Hamiltonian of the following form
\begin{equation}\label{Eqsympj}
\begin{aligned}
H(x,Y)&=\dfrac{1}{\sqrt\eps }\omega_{n}Y_{n}+\dfrac{1}{2} \langle \sA^\sharp Y ,Y \rangle+V(x_1, x_2,\ldots,x_{n-1})+\dt^\sharp R(x,Y)
\end{aligned}
\end{equation}
defined on $T^*\T^{n}$ where the remainder $R(x,Y)$ is bounded in $C^2$ and $V=(\mathfrak M^\sharp)^{-1*} \Pi_{\bK^{o}_{i,[j]}} \sV,$ $\sV(\cdot)=P(y^\star,\cdot)$.

Next we perform a standard energetic reduction to reduce it to a system of $n-1/2$ degrees of freedom.
We update the notation $x=(x_1,\ldots,x_{n-1})$, $y=(Y_1,\ldots,Y_{n-1})$. Removing the last row and column of $\sA^\sharp$ we get a matrix $A^\sharp\in \mathrm{GL}(n-1,\R)$. As $\omega_{n}\neq 0$ and $\eps>0$ is very small, one has the function $Y_n(x,x_n,y)$ as the solution of the equation
$$
H_{}(x,x_n,y,Y_n(x,x_n,y))=E>\min\alpha_{H},
$$
which takes the form $Y_n=-Y_{\dt^\sharp }\frac{\sqrt{\eps}}{\omega_{n}}$, where
\begin{equation}\label{nonautonomous}
Y_{\dt^\sharp}=\frac 12\langle A^\sharp y, y\rangle+V(x_1,\ldots,x_{n-1})+\delta^\sharp\hat R\left(x,-\frac{x_n \omega_{n}}{\sqrt\eps},y\right),\end{equation}
is defined on $T^*\T^{n-1}\times \T$ and the remainder $\hat R\left(x,\tau,y\right)$ is bounded in $C^2$.

Applying Proposition \ref{PropSingleell} and \ref{PropDoubleell} inductively, we get the following result

\begin{pro}\label{PropComplete}
There exists an open-dense set $\mathcal O$ in the unit ball of $C^r(\T^n),\ r\geq 7$, such that for each $\sV(\cdot)=P(y^\star,\cdot)\in \mathcal O$, there exists $\dt^\sharp _0$ such that for all $0<\dt^\sharp <\dt^\sharp _0$ there exists $\eps_0^\sharp=\eps_0^\sharp(\dt^\sharp )$, such that for all $0<\eps<\eps_0^\sharp$ we have the following for the Hamiltonian system $Y_{\dt^\sharp }$:
\begin{enumerate}
\item There exist a collection of wNHICs homeomorphic to $T^*\T^1$, restricted to which the time-1 map of the system $Y_{\dt^\sharp }$ is a twist map. Any Mather set with rotation vectors $\omega^\sharp$ lie on the wNHICs, if the rotation vector does not lie in the $\lambda\sqrt\eps$-neighborhood of $\Sigma_{!}(\bK^{o}_{i,[j]})$. 
\item The normal hyperbolicity is independent of $\eps$ or $\dt^\sharp$. 
\item There exists a generalized transition chain connecting the two channels $\mathbb C^\sharp_\pm:=\{\partial\beta_{Y_{\dt^\sharp }}(\nu(1,0,\ldots,0))\ |\ \pm\nu>\lambda\}\subset H^1(\T^{n-1},\R)$, corresponding two neighboring wNHICs.
\end{enumerate}
\end{pro}
\begin{proof}
We take intersection of the open-dense sets obtained in Proposition \ref{ProNDGInduction} and \ref{PropDoubleell} to get the open-dense set $\mathcal O$.

Note that the system $Y_0$ is exactly the system $\sG_{\kappa+1}^{(\ell+1)}$ in Proposition \ref{PropDoubleell} with $\kappa+1=n-2$.
By induction using Proposition \ref{PropSingleell} and \ref{PropDoubleell}, the system $Y_{0}$ admits finitely many disjoint wNHICs and generalized transition chain $\Gamma^\sharp$ connecting two channels $\mathbb C_\pm^\sharp$ and lying on the constant energy level of $\al_{Y_{0}}$. For sufficiently small $\dt^\sharp $, the wNHICs persists in the system $Y_{\dt^\sharp }$. Next we have that $\|\al_{Y_{0}}-\al_{Y_{\dt^\sharp }}\|_{C^0}\leq \dt^\sharp $. For sufficiently small $\dt^\sharp $, there exists a generalized transition chain $\Gamma^\sharp_{\dt^\sharp }$ on a constant level set of $\al_{Y_{\dt^\sharp }}$ and lying in the $\dt^\sharp $ neighborhood of the chain $\Gamma^\sharp$, by the upper-semi-continuity of the Ma\~n\'e sets and the Definition \ref{chaindef1}.

\end{proof}

\section{Switching frequency segments and Proof of Theorem \ref{ThmMainNHIC}}\label{SSwitchMain}
In this section, we show how to switch from one frequency segment to another and complete the proof of Theorem \ref{ThmMainNHIC}.
\subsection{Switching from one frequency line to another}\label{SSwitch}
All the previous works are about how to move along one frequency segment. In this section, we explain how to move from one frequency segment to the next. There is a extra difficulty which does not exist when moving along a single frequency segment or in the case of three degrees of freedom \cite{C17b}. Simply put, we are required to move from a frequency segment to another one with much weaker resonances. 

The difficulty is as follows. From the construction,  our frequency segments have a hierarchy structure. To see the difficulty clearly, we consider the switch from $\bs\omega^\sharp_{1,[1]}$ to $\bs\omega^\sharp_{1,[2]}$. For simplicity, we use the subscript $[i]$ instead of $(1,[i])$ for $i=1,2$. We need to switch from
$$\bs\omega^\sharp_{[1]}(a)=\rho^\sharp_{[1],a}\omega_{1,2}^{*i}\left(\frac{a}{\omega_{1,2}^{*}},\frac{p_2}{q_2},\ldots,\frac{p_n}{q_n}\right)\mathrm{\ to\ }\bs\omega^\sharp_{[2]}(b)=\rho^\sharp_{[2],b}\omega_{1,2}^{*}\left(\frac{p_{n+1}}{q_{n+1}}, \frac{b}{\omega_{1,2}^{*}},\frac{p_3}{q_3},\ldots,\frac{p_n}{q_n}\right).$$
The switch occurs near the complete resonances $\bs\omega^\sharp_{[1]}\cap\bs \omega^\sharp_{[2]}=(\frac{p_{n+1}}{q_{n+1}},\frac{p_2}{q_2},\ldots,\frac{p_n}{q_n})$ up to a positive multiple. When moving $a$ through $\frac{p_{n+1}}{q_{n+1}}\omega^*_{1,2}$, since $\frac{p_{n+1}}{q_{n+1}}$ is much closer to a Diophantine number than other rational numbers, the new resonance introduced by $\frac{p_{n+1}}{q_{n+1}}$ is a weak resonance and the NHIC $\cC(\bk'_{[1]},\ldots,\bk^{(n-2)}_{[1]})$ (homeomorphic to $T^*\T^2$) exists. So moving $a$ through $\frac{p_{n+1}}{q_{n+1}}\omega^*_{1,2}$ is standard as in {\it a priori} unstable systems. However, it is not clear if it is possible to move $b$ through $\frac{p_{2}}{q_{2}}\omega^*_{1,2}$ along $\bs\omega^\sharp_{[2]}$, since $\frac{p_{2}}{q_{2}}$ introduces a new strong resonance $\bk^o_{[2]}$ so NHIC $\cC(\bk'_{[2]},\ldots,\bk^{(n-2)}_{[2]})$ does not exist near $\bs\omega^\sharp_{[1]}\cap \bs\omega^\sharp_{[2]}$. 

In the next proposition, we solve the problem by combining and applying repeatedly the $c$-equivalence mechanism (Proposition \ref{PropManeBroken}) and the new mechanism (Lemma \ref{ladder}). 
\begin{pro}\label{PropSwitch}
Under the assumption of Proposition \ref{PropComplete}, there exists a generalized transition chain connecting the two channels $\mathbb C^\sharp_{i,[j]}(a):=\partial \beta_{H}(\bs\omega^\sharp_{i,[j]}(a))$ and $\mathbb C^\sharp_{i',[j']}(b):=\partial \beta_{H}(\bs\omega^\sharp_{i',[j']}(b))$ near the complete resonance  $\bs\omega^\sharp_{i,[j]}\cap \bs\omega^\sharp_{i',[j']},$ $(i',[j'])=(i,[j+1])$ for $j=1,\ldots, n-1, $or $(i',[j'])=(i+1,[0]),\ j=n$, and $i=1,\ldots,M-1$. 
\end{pro}
\begin{proof} Without loss of generality, we study only the case of switching from $\bs\omega^\sharp_{[1]}(a)$ to $\bs\omega^\sharp_{[2]}(b)$ as above. All other cases are similar. 
By the construction in the previous section, there exists a NHIC $\cC(\bK_{[1]})$ with $\bK_{[1]}=\{\bk'_{[1]},\ldots,\bk^{(n-2)}_{[1]}\}$ along the frequency segment $\bs\omega^\sharp_{[1]}(a)$, since by the choice of $p_{n+1}/q_{n+1}$, the point $\bs\omega^\sharp_{[1]}(a)$ with $a=p_{n+1}/q_{n+1}\omega_{1,2}^*$ is always a point of weak resonance during each reduction of order along the segment $\bs\omega^\sharp_{[1]}(a)$.

When viewed along the frequency segment $\bs\omega^\sharp_{[2]}(b)$,  the complete resonant point $\bs\omega^\dagger:=\bs\omega^\sharp_{[1]}\cap \bs\omega^\sharp_{[2]}$ admits an extra resonance $\bk^o_{[2]}$ which has shorter length than any of $\bk^{(i)}_{[2]}$. So the NHIC $\cC(\bK_{[2]})$ with $\bK_{[2]}=\{\bk'_{[2]},\ldots,\bk^{(n-2)}_{[2]}\}$ may not exists near the complete resonance $\Sigma(\bk^o_{[2]},\bK_{[2]})$, and Mather set with rotation vector $\bs\omega^\dagger$ does not lie on $\cC(\bK_{[2]})$. 

We want to move a point on $\bs\omega^\sharp_{[2]}$ to $\bs\omega^\sharp_{[1]}$. 
The argument goes as follows.  We first move along $\cC(\bK_{[2]})$ to arrive at a point $\omega^i\in \bs\omega^\sharp_{[2]}$ with dist$(\omega^i, \omega(\Sigma(\bk^o_{[2]}, \bk'_{[2]})))<\lambda.$ By Proposition \ref{PropAlpha}  and Section \ref{SSChainFreq2}, we get a convex loop $\bs\omega^\dagger+\ell(\bk^o_{[2]}, \bk'_{[2]})$ enclosing 0 on the plane $\bs\omega^\dagger+(SM''_{[2]})^{-1}\mathrm{span}\{e_1,e_2\}$ whose Legendre transform is a generalized transition chain of Proposition \ref{PropAlpha} (essentially due to Proposition \ref{PropManeBroken}). We first find a point $\omega'$ on $\bs\omega^\dagger+\ell(\bk^o_{[2]}, \bk'_{[2]})\in (\bk'_{[1]})^\perp\cap(\partial\al(\al^{-1}(E)))$ by the argument in Section \ref{SSChainFreq2}.

Complementing to $\bK_{[1]}=\{\bk'_{[1]},\ldots,\bk^{(n-2)}_{[1]}\}$, the rational number $p_{n+1}/q_{n+1}$ introduces one more resonant integer vector denoted by $\bk^o_{[1]}$ whose lengths are much longer than any one in $\bK_{[1]}$. We introduce a normal form \eqref{Eqsympj} at this complete resonance $\bs\omega^\dagger$ as in Section \ref{SSCompleteRes}. Here the $n-1$ rows of the matrix $M^\sharp\in \mathrm{SL}(n,\Z)$ are ordered as $\bk'_{[1]},\ldots,\bk^{(n-2)}_{[1]}, \bk^o_{[1]}$. We permute the variables to $x=(x_2,x_3,\ldots,x_n,x_1),\ y=(y_2,y_3,\ldots,y_n,y_1)$. In this new coordinates system the frequency $\omega'$ has the form $(0,O(\lambda),\ldots,O(\lambda),O(\eps^{-1/2}))$ since $\omega'\in (\bk'_{[1]})^\perp\cap(\partial\al(\al^{-1}(E)))$. As in Section \ref{SSSShear} after a shear transform $S'''$ in equation \eqref{EqS'''}, we separate a subsystem $\sG_{3,\dt}$ (equation \eqref{EqtG3})of three degrees of freedom (corresponding to the first three coordinates) from the full system.  We want to kill the second  entry $O(\lambda)$. Note that the system $\sG_{3,\dt}$ admits a NHIC which is due to the NHIC $\cC(\bk'_{[1]})$ in the original system. Restricted to the NHIC we get a system $\bar\sG_{3,\dt}$ (equation \eqref{EqDouble+}) of two degrees of freedom. We remark that the NHIC here is not near strong double resonance. 
 By Lemma \ref{ladder} and Remark \ref{RkLadder}, under generic perturbation, all the cohomology classes on a level set of $\al_{\bar\sG_{3,\dt}}$ lies in a generalized transition chain, along which the frequency vector moves on a convex curve enclosing 0 on the plane span$\{e_2,e_3\}$. In this way, we kill the second entry $O(\lambda)$ of $\omega'$. Denote the resulting frequency $\omega''$. Now $\omega''$ lies on $(\bk'_{[1]})^\perp\cap(\bk''_{[1]})^\perp \cap(\partial\al(\al^{-1}(E)))$. We next perform a shear transform to separate a subsystem of four degrees of freedom and restricted to its NHIC $\mathcal C(\bk'_{[1]},\bk''_{[1]})$, we again get a subsystem of two degrees of freedom of the form $\bar\sG$ above. We then kill the next $O(\lambda)$ entry using again Lemma \ref{ladder} and Remark \ref{RkLadder}. This procedure can be done repeatedly to obtain a resulting frequency vector having the first $n-2$ entries vanished. In the original coordinates, this means that the frequency is orthogonal to $\bK_{[1]}$ so it lies on $\bs\omega^\sharp_{[1]}$. The proof is now complete. 
\end{proof}
\subsection{Proof of Theorem \ref{ThmMainNHIC}}\label{SSProofMain}
In this section, we complete the proof of Theorem \ref{ThmMainNHIC}.

\begin{proof}[Proof of Theorem \ref{ThmMainNHIC}]
When the induction in Section \ref{SHierarchy} is complete, we obtain a collections of frequency segments $\bs\om^\sharp_{i,[j]},\ i=1,\ldots,M-1,\ j=1,\ldots,n$ which concatenate into a connected curve $\omega(t):\ [0,M]\to (\partial\al_H)(\al^{-1}_H(E))$ lying in the $\varrho$-neighborhood of the union of $\Omega_{i,[j]}$.  

Next, the existence of wNHIC (part (2) and part (3.a) of Theorem \ref{ThmMainNHIC}) is given by Proposition \ref{PropComplete}(1). Neighboring wNHICs near a complete resonance are connected by a generalized transition chain by Proposition \ref{PropComplete}(2), which proves Theorem \ref{ThmMainNHIC} part (3.c).

Next, the existence of transition chain switching from one frequency line segment to the next is done by Proposition \ref{PropSwitch}.

It remains to prove the existence of the generalized transition chain along the NHIC (part (3.b)) and the cusp-residual genericity. We have finitely many genericity conditions from Proposition \ref{ProNDGInduction} and Proposition \ref{PropSingleell} and \ref{PropDoubleell}. Denote by $\mathcal O\subset \mathfrak B_1$ the open-dense set obtained by taking intersection of the finitely many open-dense sets. We choose a $P\in\mathcal O$ such that the finitely many conditions are satisfied. This $P$ determines $\eps_P$ such that for $\eps<\eps_P$ Proposition \ref{PropComplete} holds (note that $\dt^{\sharp}_0$ therein depends on $\eps$). It remains to prove part (3.b) of the theorem.
 For this purpose, we fix an $\eps$ and apply the argument in Section 7 of \cite{CY1}, which gives us an $\eps'=\eps'(\eps P)$ and an open-dense set  $\mathfrak R_{\eps'}(\eps P)\subset\mathfrak B_{\eps'}$ such that part (3.b) holds for $\eps (P+ P')$ for any $\eps P'\in \mathfrak R_{\eps'}(\eps P)$.
The details of the argument are given in Appendix \ref{AppGenericity}. 

 Finally, applying Karatowski-Ulam Theorem \ref{ThmKU} to the set $\cup_{P\in \mathcal O}\cup_{\eps<\eps_P}\mathfrak R_{\eps'}(\eps P)$, we get that there exists an open-dense set $\mathfrak R\subset \mathfrak S_1$ such that for each $P\in \mathfrak R$ there exist $\eps_P$ and a residual set $R_P\subset (0,\eps_P)$ such that the theorem holds for all $\eps P$ for $\eps\in R_P$ and $P\in \mathfrak R$. This completes the proof.
\end{proof}

\appendix
\section{A brief introduction to Mather theory and weak KAM}\label{AppMather}

In this appendix, we give a brief introduction to the Mather theory and weak KAM theory.
\subsection{Minimizing measure and $\alpha$,$\beta$ function}\label{minimalmeasure}
The theory is established for Tonelli Lagrangian.

\begin{defi}
Let $M$ be a closed manifold. A $C^2$-function $L$: $TM\times\mathbb{T}\to\mathbb{R}$ is called \emph{Tonelli Lagrangian} if it satisfies the following conditions:
\begin{enumerate}
    \item \emph{Positive definiteness}. For each $(x,t)\in M\times\mathbb{T}$, the Lagrangian function is strictly convex in velocity: the Hessian $\partial_{\dot x\dot x}L$ is positive definite.
    \item \emph{Super-linear growth}. We assume that $L$ has fiber-wise superlinear growth: for each $(x,t)\in M\times\mathbb{T}$, we have $L/\|\dot x\|\to\infty$ as $\|\dot x\|\to \infty$.
    \item \emph{Completeness}. All solutions of the Lagrangian equations are well defined for the whole $t\in\mathbb{R}$.
\end{enumerate}
\end{defi}

For autonomous systems, the completeness is automatically satisfied, as each orbit entirely stays in certain compact energy level set.

Given a closed 1-form $\langle\eta_c(x),dx\rangle$ with first cohomology class $[\langle\eta_c(x),dx\rangle]=c$, we introduce a Lagrange multiplier $\eta_c=\langle\eta_c(x),\dot x\rangle$. Without danger of confusion, we also call it a closed 1-form.

For each $C^1$ curve $\gamma$: $\mathbb{R}\to M$ with period $k$, there is a unique probability measure $\mu_{\gamma}$ on $TM\times\mathbb{T}$ so that the following holds
$$
\int_{TM\times\mathbb{T}}f\,d\mu_{\gamma}=\frac 1k\int_0^kf(d\gamma(s),s)\,ds\qquad
$$
for each $f\in C^0(TM\times\mathbb{T},\mathbb{R})$, where we use the notation $d\gamma=(\gamma,\dot\gamma)$. Let
$$
\mathfrak{H}^*=\{\mu_{\gamma}\ |\ \gamma\in C^1(\mathbb{R},M)\ \text{\rm is periodic of}\ k\}.
$$
The set $\mathfrak{H}$ of \emph{holonomic probability measures} is the closure of $\mathfrak{H}^*$ in the vector space of continuous linear functionals. One see that $\mathfrak{H}$ is convex.

For each $\nu\in\mathfrak{H}$ the action $A_c(\nu)$ is defined as
$
A_c(\nu)=\int (L-\eta_c)\,d\nu.
$
It is proved in \cite{M91,Me} that for each co-homology class $c$ there exists at least one invariant probability measure $\mu_c$ minimizing the action over $\mathfrak{H}$
$$
A_c(\mu_c)=\inf_{\nu\in\mathfrak{H}}\int (L-\eta_c)\,d\nu,
$$
called \emph{$c$-minimal measure}. 
\begin{defi}
\begin{enumerate}
\item Let $\mathfrak{H}_c\subset\mathfrak{H}$ be the set of $c$-minimal measures, the \emph{Mather set} $\widetilde{\mathcal {M}}(c)$ is defined as
$$
\widetilde{\mathcal{M}}(c)=\bigcup_{\mu_c\in\mathfrak{H}_c}\text{\rm supp} \mu_c.
$$
\item The \emph{$\alpha$-function} is defined as $\alpha(c)=-A_c(\mu_c): H^1(M,\mathbb{R})\to\mathbb{R}$. It is convex, finite everywhere with super-linear growth. 
\item Its Legendre transformation $\beta: H_1(M,\mathbb{R})\to\mathbb{R}$ is called \emph{$\beta$-function}
$$
\beta(\omega)=\max_c (\langle\omega,c\rangle -\alpha(c)).
$$
It is also convex, finite everywhere with super-linear growth (see \cite{M91}).
\end{enumerate}
\end{defi}
Note that $\int\lambda d\mu_{\gamma}=0$ holds for each exact 1-form $\lambda$ and each $\mu_{\gamma} \in\mathfrak{H}^*$. Therefore, 
\begin{defi}
for each measure $\mu\in\mathfrak{H}$ one can define its \emph{rotation vector}
$\omega(\mu)\in H_1(M,\mathbb{R})$ such that
$$
\langle[\lambda],\omega(\mu)\rangle=\int\lambda\, d\mu,
$$
holds for every closed 1-form $\lambda$ on $M$. 
\end{defi}
We have the following relation
$$
c\in\partial{\beta}(\rho)\ \ \iff \ \ \alpha(c)+\beta(\rho)=\langle c,\rho\rangle.
$$

\subsection{(Semi)-static curves, the Aubry set and the Ma\~n\'e set}
The concept of semi-static curves is introduced by Mather and Ma\~n\'e (cf. \cite{M93,Me}). 
\begin{defi}A curve
$\gamma$: $\mathbb{R}\to M$ is called \emph{$c$-semi-static} if 

\begin{enumerate}
\item in the time-1-periodic case we have
$$
[A_c(\gamma)|_{[t,t']}]=F_c((\gamma(t),t),(\gamma(t'),t'))
$$
where
\begin{equation}\label{introeq2}
[A_c(\gamma)|_{[t,t']}]=\int_{t}^{t'}\Big (L(d\gamma(t),t)-\eta_c(d\gamma(t))\Big )\, dt+\alpha(c)(t'-t),\notag
\end{equation}
$$
F_c((x,t),(x',t'))=\inf_{\stackrel {\tau=t\ \text{\rm mod}\, 1}{\scriptscriptstyle \tau '=t'\text{\rm mod}\,1}} h_c((x,\tau),(x',\tau')),
$$
in which
\begin{equation}\label{introeq3}
h_c((x,\tau),(x',\tau'))=\inf_{\stackrel{\stackrel {\xi\in C^1}{\scriptscriptstyle \xi(\tau)=x}}{\scriptscriptstyle \xi(\tau')=x'}}[A_c(\xi)|_{[\tau,\tau']}].\notag
\end{equation}
\item In autonomous case, the period is considered as any positive number. Consequently, the notation of semi-static curve in this case is simpler
$$
[A_c(\gamma)|_{(t,t')}]=F_c(\gamma(t),\gamma(t')),
$$
where
\begin{equation}\label{introeq4}
F_c(x,x')=\inf_{\tau>0}h_c((x,0),(x',\tau)).\notag
\end{equation}
\end{enumerate}
\end{defi}
\noindent{\bf Convention}: Let $I\subseteq\mathbb{R}$ be an interval $($either bounded or unbounded$)$. A continuous map $\gamma$: $I\to M$ is called a \emph{curve}. If it is differentiable, the map $d\gamma=(\gamma,\dot\gamma)$: $I\to TM$ is called an \emph{orbit}. When the implication is clear without danger of confusion, we use the same symbol to denote the graph, $\gamma:=\cup_{t\in I}(\gamma(t),t)$ is called a curve and $d\gamma:=\cup_{t\in I}(\gamma(t),\dot\gamma(t),t)$ is called an orbit. In autonomous systems, the terminology also applies to the image: $\gamma:=\cup_{t\in I}\gamma(t)$ is called a curve and $d\gamma:=\cup_{t\in I}(\gamma(t),\dot\gamma(t))$ is called an orbit.

\begin{defi}
\begin{enumerate}
\item
A semi-static curve $\gamma\in C^1(\mathbb{R},M)$ is called \emph{$c$-static} if, in addition, the relation
$$
[A_c(\gamma)|_{(t,t')}]=-F_c((\gamma(t'),\tau'),(\gamma(t),\tau))
$$
holds in time-1-periodic case and
$$
[A_c(\gamma)|_{(t,t')}]=-F_c(\gamma(t'),\gamma(t))
$$
holds in autonomous case. 
\item An orbit $X(t)=(d\gamma(t), t\, \text{\rm mod}\ 2\pi)$ is called $c$-static $($semi-static$)$ if $\gamma$ is $c$-static $($semi-static$)$. 
\end{enumerate}
\end{defi}
\begin{defi}
We call the \emph{Ma\~n\'e set} $\widetilde{\mathcal {N}}(c)$ the union of $c$-semi-static orbits
$$
\widetilde{\mathcal{N}}(c)=\bigcup\{d\gamma:\gamma\ \text{\rm is}\ c\text{\rm -semi static}\}
$$
and call the \emph{Aubry set} $\widetilde{\mathcal {A}}(c)$ the union of $c$-static orbits
$$
\widetilde{\mathcal{A}}(c)=\bigcup\{d\gamma :\gamma\ \text{\rm is}\ c\text{\rm -static}\}.
$$
\end{defi}
\begin{Not}
We use $\mathcal {M}(c)$, $\mathcal {A}(c)$ and $\mathcal {N}(c)$ to denote the standard projection of $\widetilde{\mathcal {M}}(c)$, $\widetilde{\mathcal {A}}(c)$ and $\widetilde{\mathcal {N}}(c)$ from $TM\times\mathbb{T}$ to $M\times\mathbb{T}$ respectively. 
\end{Not}
They satisfy the inclusion relation
$$
\widetilde{\mathcal{M}}(c)\subseteq\widetilde{\mathcal{A}}(c)\subseteq\widetilde{\mathcal{N}}(c).
$$
It is showed in \cite{M91,M93} that the inverse of the projection is Lipschitz when it is restricted to $\mathcal {A}(c)$ as well as to $\mathcal {M}(c)$. By adding subscript $s$ to $\mathcal{N}$, i.e. $\mathcal{N}_s$ we denote its time-$s$-section. This principle also applies to $\widetilde{\mathcal{N}}(c)$, $\widetilde{\mathcal{A}}(c)$, $\widetilde{\mathcal{M}}(c)$, $\mathcal{A}(c)$ and $\mathcal{M}(c)$ to denote their time-$s$-section respectively. For autonomous systems, these sets are defined without the time component.

On the time-1-section of Aubry set a \emph{pseudo-metric $d_c$} is introduced by Mather in \cite{M93}, whose definition relies on the quantity $h_c^{\infty}$. Define
\begin{equation}\label{introeq5}
h_c^{\infty}((x,s),(x',s'))=\liminf_{\stackrel{\stackrel{s=t\ \text{\rm mod}\ 1}{\scriptscriptstyle t'=s'\ \text{\rm mod}\ 1}}{\scriptscriptstyle t'-t\to\infty}}h_c((x,t),(x',t')).\notag
\end{equation}
For autonomous system
\begin{equation}\label{introeq6}
h^{\infty}_c(x,x')=\liminf_{\tau\to\infty}h_c((x,0),(x',\tau)).\notag
\end{equation}
The pseudo-metric $d_c$ on Aubry set is defined as
$$
d_c((x,t),(x',t'))=h_c^{\infty}((x,t),(x',t'))+h_c^{\infty}((x',t'),(x,t)).
$$
With the pseudo-metric $d_c$ one defines equivalence classes in an Aubry set. The equivalence $(x,t)\sim (x',t')$ implies $d_c((x,t),(x',t'))=0$, with which one can define \emph{quotient Aubry set} $\mathcal{A}(c)/\sim$. Its element is called \emph{Aubry class}, denoted by $\mathcal{A}_i(c)$ or $\mathcal{A}_{c,i}$, whose lift to $TM\times\mathbb{T}$ is denoted by $\widetilde{\mathcal{A}}_i(c)$. Thus,
$$
\mathcal{A}(c)= \bigcup_{i\in\Lambda}\mathcal{A}_i(c), \qquad \widetilde{\mathcal{A}}(c)=\bigcup_{i\in\Lambda}\widetilde{\mathcal{A}}_i(c).
$$
Although Mather constructed an example with a quotient Aubry set homeomorphic to an interval, it is generic that each $c$-minimal measure contains not more than $n+1$ ergodic components if the system has $n$ degrees of freedom \cite{BC}. In this case, each Aubry set contains at most $n+1$ classes.
\subsection{A brief introduction to Weak KAM theory}\label{SSWKAM}
The concept of $c$-semi-static curves can be extended to the curves only defined on $\mathbb{R}^{\pm}$, which are called \emph{forward (backward) $c$-semi-static curves}, denoted by $\gamma^{\pm}_c$ respectively. A curve $\gamma_c^-$ ($\gamma_c^+$) produces a backward (forward) semi-static orbit orbit $(\gamma_c^-,\dot\gamma_c^-)$ ($(\gamma_c^+,\dot\gamma_c^+)$).
\begin{pro}\label{weakpro1}If the Lagrangian $L$ is of Tonelli type, for each point $(x,\tau)\in M\times\mathbb{T}$, there is at least one $\gamma^{\pm}_c(t,x,\tau)$ which is a forward $($backward$)$ semi-static curve.
\end{pro}

Since both the $\omega$-limit set of $d\gamma_c^+$ and the $\alpha$-limit set of $d\gamma_c^-$ are in the Aubry set one define
$$
W^{\pm}_c=\bigcup_{(x,\tau)\in M\times\mathbb{T}}
\left\{x,\tau,\frac{d\gamma^{\pm}_c(\tau,x,\tau)}{dt}\right\},
$$
and call $W^+_c$ the stable set, $W^-_c$ the unstable set of the $c$-minimal measure respectively. If  $\dot\gamma^-(\tau,x,\tau)=\dot\gamma^+(\tau,x,\tau)$ holds for some $(x,\tau)\in M\times\mathbb{T}$, passing through the point $(x,\tau,\dot\gamma^-_c(\tau,x,\tau))$ the orbit is either in the Aubry set or homoclinic to this Aubry set.

If the Aubry set consists of one class, the stable as well as the unstable set has its own generating function $u_c^{\pm}$ such that $W^{\pm}_c=\text{\rm Graph}(du_c^{\pm})$ holds almost everywhere \cite{F}. These functions are \emph{weak KAM solutions}. We use $u_c^{\pm}$ to denote the weak KAM solution for the Lagrangian $L-\eta_c$, where $\eta_c$ is a closed form with $[\eta_c]=c$. These functions are Lipschitz, thus differentiable almost everywhere. At each differentiable point $(x,\tau)$, $(x,\tau,\partial_xu^-(x,\tau))$ uniquely determines backward $c$-semi static curve $\gamma^-_x$: $(-\infty, \tau]\to M$ such that $\gamma^-_x(\tau)=x$, $\dot\gamma^-_x(\tau)=\partial_yH(x,\tau,\partial_xu^-(x,\tau))$. Similarly, $(x,\tau,\partial_xu^+(x,\tau))$ uniquely determines forward $c$-semi static curve $\gamma^-_x$: $[\tau,\infty)\to M$ such that $\gamma^+_x(\tau)=x$, $\dot\gamma^+_x(\tau)=\partial_yH(x,\tau,\partial_xu^+(x,\tau))$.

\section{Proof of Lemma \ref{LmAllDiop}}\label{AppDiop}
We prove Lemma \ref{LmAllDiop} in this appendix. 
Fix $\varrho>0,\tau>n$. We prove the lemma by induction from $j+1$ to $j$. First, for $j=n$, it is easy to find two Diophantine numbers $\omega_n^i,\ \omega_n^f$. Suppose we already have that
$$
\l(\omega_{j+1}^{*i},\ldots,\omega_{n}^{*i}\r)\in \mathrm{DC}(n-j,\al,\tau).
$$
We claim that {\it given $\omega_j^i$ and $\omega_j^f$, there are numbers $\om_j^{*i}$ and $\om_j^{*f}$ satisfying $$|\om_j^{*i}-\om_j^{i}|<\varrho,\ |\om_j^{*f}-\om_j^{f}|<\varrho,\ \mathrm{and\ }
(\om_j^{*i,f},\omega_{j+1}^{*i},\ldots,\omega_{n}^{*i})\in \mathrm{DC}(n-j+1,\al,\tau)
$$
for sufficiently small $\al>0$. }

Indeed, by assumption we already have
$$
\Big|\Big\langle \hat\om_{n-j}^{*i},\hat \bk_{n-j}\Big\rangle\Big|\geq \dfrac{\al}{|\hat \bk_{n-j}|^\tau},\quad \forall\ \hat \bk_{n-j}\in \Z^{n-j}\setminus\{0\}.
$$
We want to show that all those $\omega_j\in\R$ which satisfy the condition
\begin{equation}\label{EqDiop+1}
|\langle (\omega_j,\hat \om_{n-j}^{*i}),\hat \bk_{n-j+1}\rangle|\geq \dfrac{\al}{|\hat \bk_{n-j+1}|^\tau},\quad \forall\ \hat \bk_{n-j+1}\in \Z^{n-j+1}\setminus\{0\}
\end{equation}
form a $\varrho$-dense set provided $\al$ is small enough. Given $\hat \bk_{n-j}$, we consider all $k_j$ and $\omega_j^\dagger$ satisfying
$$
k_j\omega^\da_j+\langle \hat\om_{n-j}^{*i},\hat \bk_{n-j}\rangle=0.
$$
Formula \eqref{EqDiop+1} is satisfied automatically for $\hat\bk_{n-j+1}=(k_j,\hat\bk_{n-j})$ when $k_j=0$, so we assume $k_j\neq 0.$
In order to guarantee \eqref{EqDiop+1} we need to remove an interval of measure $\frac{2\al}{k_j(|\hat \bk_{n-j}|+|k_j|)^\tau}$ centered at $\omega_j^\da$ so that \eqref{EqDiop+1} is satisfied for all $\omega_j$ in the complement for this $k_j$. The total measure of these intervals removed when $k_j$ ranges over $\Z\setminus\{0\}$ is
\[\sum_{k_j}\dfrac{2\al}{|k_j|(|\hat\bk_{n-j}|+|k_j|)^\tau}\leq 2\int_1^\infty\dfrac{2\al}{x(|\hat\bk_{n-j}|+x)^\tau}\,dx.\]
Next the total measure of these intervals when $\hat \bk_{n-j}$ ranges over $\Z^{n-j}\setminus\{0\}$ is
\beqa\nonumber
\sum_{\hat\bk_{n-j}}\sum_{k_j}\dfrac{4\al}{|k_j|(|\hat\bk_{n-j}|+|k_j|)^\tau}&\leq \sum_{\hat\bk_{n-j}}\int_1^\infty\dfrac{4\al}{x(|\hat\bk_{n-j}|+x)^\tau}\,dx\\
&\leq \int_{\mathbb S^{n-j-1}}\int_1^\infty\int_1^\infty\dfrac{4\al}{x(r+x)^\tau}\,dx\, r^{n-j-1}\,dr d\mathbb S^{n-j-1}\\
&\xlongequal{y=x/r}
   4\al C\int_1^\infty r^{n-j-\tau-1}\int_{1/r}^\infty\dfrac{1}{y(1+y)^\tau}\,dy\,dr
\eeqa
where the constant $C=\frac{2\pi^{(n-j-1)/2}}{\Gamma((n-j-1)/2)}$ is the area of the sphere $\mathbb S^{n-j-1}$. The inner integral converges for large $y$ and has the asymptote $\log r$ for $r$ large and $y$ close to $1/r$. Hence the iterated integral can be estimated as
$$
\int_1^\infty r^{n-j-\tau-1}\int_{1/r}^\infty\dfrac{1}{y(1+y)^\tau}\,dy\,dr \leq 2 \int_1^\infty r^{n-j-\tau-1}(\log r+\mathrm{const})\,dr
$$
where the right-hand-side is convergent since $\tau>n$. The assertion above is proven if $\al>0$ is chosen small enough.

\section{A special normally hyperbolic invariant manifold theorem}\label{SNHICSpecial}
In this paper, we need a special version of the theorem of NHIM. Here we present its detailed proof using the graph transform method. The statement given below is adapted to the setting needed in the paper and we do not pursue generality. 
\begin{theo}\label{NHICSpecial}
Let $N=(\R^m/S\Z^{m})\times \R^{m'}$, $S\in \mathrm{SL}(m,\R),$ be a submanifold of a $($non compact$)$ manifold $M^{m+m'+k}$. Given $\Lambda>0$, we denote $N_\Lambda=(\R^m/S\Z^{m})\times B(0,\Lambda)$ and choose a small neighborhood $U(\subset M)$ of $N_\Lambda$.  
Let $V=(\dot q, \dot p,\dot n)$ be a $C^2$ vector field compactly supported in $U$ satisfying the following properties:
\begin{enumerate}
\item $\begin{cases}
&\dot q=\eps^{-1/2}\omega^\star+a(p)\\
&\dot p=0
\end{cases}$ where $(q,p)\in N_\Lambda$, $\omega^\star\in \R^{m}$ is constant and $a(p)\in C^2$, $\eps>0$;
\item restricted on the normal bundle $\cup_{x\in N_\Lambda}E^s_x\oplus E^u_x,$ we have $\dot n=A n$, where $A\in \R^{k\times k}$ is a constant matrix all of whose eigenvalues lies off the imaginary axis. 
\end{enumerate}
Then there exists $\gamma_0$ such that any vector field $V_{\gamma,\eps}$ compactly supported in $U$ and satisfying that $\|V_{\gamma,\eps}-V\|_{C^1}\leq \gamma_0,$ admits 
a NHIC that is a graph over $N_\Lambda$.

\end{theo}
\begin{Rk}
 We will see in the proof that $\eps$ does not play any role, since the large term $\eps^{-1/2}\omega^\star$ is constant and does not appear in the derivative of the time-1 map of the flow, on the other hand, only derivative information matters in the proof (see \eqref{EqNHBounds}). The vector field $V_{\gamma,\eps}$ is also allowed to depend on $t$ periodically, even with fast oscillation for instance it depends on $t/\eps^\al\in \T$, any $\al>0$. Note that the $\|\cdot \|_{C^1}$ norm does not include the derivative with respect to $t$. 
\end{Rk}

\begin{proof}
In the proof, for clarity of the ideas, we consider first the contracting case, namely, $E^u=0$ in the splitting of $T_xM$ (see Definition \ref{DefNHIM}), i.e. all the eigenvalues of $A$ has negative real parts. 

We denote by $f$ (resp. $f_\gamma$) the time-1 map generated by the vector field $V$ (resp. $V_{\gamma,\eps}$). We now introduce coordinates. We cover a neighborhood $U_d,\ d>0,$ of the center manifold $N_\Lambda$ by balls of the form $B(p_i,2d)$ with $p_i\in N_\Lambda$ using any preferred Riemannian metric. In each of the ball $B(p_i,2d)$, we choose a local coordinates given by $\exp_{p_i}: T_{p_i}N_\Lambda\oplus E^s_{p_i}\to B(p_i,2d)$ with \begin{equation}\label{EqExpMap}\exp_{p_i}(x,0)\in B(p_i,2d)\cap N_\Lambda,\ \mathrm{and\ }\exp_{p_i}(0,0)=p_i.\end{equation}
In coordinates, the map $f^n$ can be written as  $$F_{j,i}:=\exp_{p_j}^{-1}\circ f^n\circ\exp_{p_i}:\ T_{p_i}N\oplus E^s_{p_i}\to T_{p_j}N\oplus E^s_{p_j}$$ if $p=f^{-n}(p')$ for $p\in B(p_i,2d)$ and $p'\in B(p_j,2d),$ where the number of iterates $n$ will be determined later.  We suppress the subscripts $i,j$ for simplicity and denote $F(x,y)=(X(x,y),Y(x,y))$ where $Y(x,0)=0$. We denote $$dF=\left(\begin{array}{cc}
\partial_x X&\partial_y X\\
\partial_x Y&\partial_y Y
\end{array}\right):=\left(\begin{array}{cc}
A&B\\
C&D\end{array}\right).$$
We have by definition that
$$
\begin{aligned}
C(x,0)&=\partial_x Y(x,0)=0,\\
D(x,0)&=\partial_y Y(x,0)= dF|_{E^s},\\
A(x,0)&=\partial_x X(x,0)=dF|_{E^c}.
\end{aligned}
$$
Now the normal hyperbolicity assumption implies the following important bounds
\begin{equation}\label{EqNHBounds}\Vert D\Vert_{C^0}\Vert A^{-1}\Vert_{C^0}< 1/2,\quad \Vert D\Vert_{C^0}< 1/2,\quad \Vert C\Vert_{C^0}< \eta\ll 1 \end{equation}
by choosing $n$ large and the neighborhood $U_d$ small enough. The derivative $dF$ is obtained by integrating the variational equation derived from the ODE of $V.$ Note that the term $\eps^{-1/2}\omega^\star$ does not appear in the variational equation since $\omega^\star$ is a constant. Since the map $f_\gamma$ is $\gamma$-close to $f$ in the $C^1$ norm, we define $F_\gamma$ from $f_\gamma^n$ in the same way as $F$ from $f^n$. For small enough $\gamma$, the above bounds \eqref{EqNHBounds} also holds for $F_\gamma$ in the domain $U_d$. In the following, we suppress the subscript $\gamma$ and work exclusively with $F_\gamma$ instead of $F$.
\subsubsection{The graph transform}\label{SSSGraph}
Define first the set $\mathcal S$ of Lipschitz sections $S: T_{p_i}N_\Lambda\to T_{p_i}N_\Lambda\oplus E^s_{p_i}$. Next we define $$\mathcal S_\dt:=\{S\in \mathcal S\ |\ \mathrm{Lip}(S)\leq \dt\}.$$ 

The graph transform is defined to be
\begin{equation}\label{EqGraphTrans}
G:\ \mathcal S_\dt\to\mathcal  S,\quad (G(S))(X(x,S(x)))=Y(x,S(x)).\end{equation}
\begin{lem}\label{LmSdt}For sufficiently small $\eta$, $\dt$, the image of the graph transform $G$ lies in $\mathcal S_\dt$, i.e. $G:\ \mathcal S_\dt\to\mathcal  S_\dt.$\end{lem}
\begin{proof}
Suppose $\xi=X(x,S(x))$ and $\xi'=X(x',S(x'))$ are sufficiently close. The injectivity of $X(\cdot,S(\cdot))$ will be shown below. Then we have
\begin{equation}\label{EqGs}
\begin{aligned}
\Vert (G(S))(\xi)-(G(S))(\xi')\Vert&=\Vert Y(x,S(x))-Y(x',S(x'))\Vert\\
&\leq \Vert C\Vert_{C^0} \Vert x-x'\Vert+\dt\Vert D\Vert_{C^0} \Vert x-x'\Vert.
\end{aligned}
\end{equation}
Next we bound $\Vert x-x'\Vert$ using $\Vert \xi-\xi'\Vert$.
\begin{equation*}
\begin{aligned}
\Vert \xi-\xi'\Vert&=\Vert X(x,S(x))-X(x',S(x'))\Vert\\
&\geq \Vert X(x,S(x))-X(x',S(x))\Vert-\Vert X(x',S(x))-X(x',S(x'))\Vert \\
&\geq \Vert A^{-1}\Vert^{-1}_{C^0} \Vert x-x'\Vert-\Vert B\Vert_{C^0} \Vert S(x)-S(x')\Vert\\
&\geq (\Vert A^{-1}\Vert_{C^0}^{-1}-\dt \Vert B\Vert_{C^0})\Vert x-x'\Vert.
\end{aligned}
\end{equation*}
Let $c=\frac{\Vert C \Vert_{C^0}+\dt\Vert D\Vert_{C^0}}{\Vert A^{-1}\Vert_{C^0}^{-1}-\dt \Vert B\Vert_{C^0}}$.  Combined with \eqref{EqGs}, we get
$$
\Vert (G(S))(\xi)-(G(S))(\xi')\Vert\leq c\Vert \xi-\xi'\Vert
$$
We can make $\Vert C\Vert_{C^0}$ as small as we wish by choose $\eta$ small, hence for small $\dt$, the leading term in $c$ is given by $\dt \Vert D\Vert_{C^0}\Vert A^{-1}\Vert_{C^0}\leq \dt/2$.
\end{proof}

\begin{lem}\label{LmContraction}The graph transform $G:\ \mathcal S_\dt\to\mathcal  S_\dt$ is a contraction in the $C^0$ norm, i.e. $\Vert G(S)-G(S')\Vert_{C^0}\leq \lb\Vert S-S'
\Vert_{C^0} $ for some $0<\lb<1$. \end{lem}
\begin{proof}
For $S,S'\in \mathcal S_\dt$, choosing $x$ and $x'$ with $\xi=X(x,S(x))=X(x',S'(x'))$, we get
\begin{equation*}
\begin{aligned}
&\Vert (G(S))(\xi)-(G(S'))(\xi)\Vert=\Vert Y(x,S(x))-Y(x',S'(x')) \Vert\\
&\leq \Vert C\Vert_{C^0}\Vert x-x'\Vert+\Vert D\Vert_{C^0}(\Vert S(x)-S'(x)\Vert+\Vert S'(x)-S'(x')\Vert)\\
&\leq (\Vert C\Vert_{C^0}+\dt\Vert D\Vert_{C^0} )\Vert x-x'\Vert+\Vert D\Vert_{C^0} \Vert S-S'\Vert_{C^0}.
\end{aligned}
\end{equation*}
Since $\Vert C\Vert_{C^0}<\eta$ can be as small as we wish, and $\Vert D\Vert_{C^0} <1/2$ due to the contraction. The proof will be complete if we can show $\Vert x-x'\Vert\leq c\Vert S-S'\Vert_{C^0}$ for some constant $c$. We have
\begin{equation*}
\begin{aligned}
&\Vert X(x,S(x))-X(x',S(x))\Vert\geq \Vert A\Vert_{C^0}\Vert x-x'\Vert
\end{aligned}
\end{equation*}
and
\begin{equation*}
\begin{aligned}
&\Vert X(x',S'(x'))-X(x',S(x))\Vert\leq \Vert B\Vert_{C^0} (\dt\Vert x-x'\Vert+\Vert S-S'\Vert_{C_0}).
\end{aligned}
\end{equation*}
Since we have $\xi=X(x,S(x))=X(x',S'(x'))$, combining the two estimates we get $\Vert x-x'\Vert\leq c\Vert S-S'\Vert_{C^0}$ for some constant $c$. This completes the proof.
\end{proof}
By the contracting mapping theorem, there exists a unique $\dt$-Lipschitz solution $S$ to the graph transform, $S=G(S).$ By the uniqueness of the fixed point of $G$, we get that $\exp S$ is invariant under $f_\gamma$.

For hyperbolic splitting, i.e. the matrix $A$ has both positive and negative eigenvalues, we introduce coordinates respecting the splitting and write the map $f_\gamma^n$ in coordinates as before\begin{equation}\label{EqFt} F(x,y,z)=(X(x,y,z), Y(x,y,z),Z(x,y,z))\in E_{p'}^c\oplus E_{p'}^u\oplus E_{p'}^s,\end{equation}
where $(x,y,z)\in E_p^c\oplus E_p^u\oplus E_p^s$ and $f_\gamma^{-n}(p')=p$ with the derivative control for sufficiently small $\eta$, and in a sufficiently small neighborhood $U_d$
$$
\Vert \partial_x X^{-1}\Vert_{C^0}^k\Vert \partial_z Z\Vert_{C^0}<1/2,\quad \Vert \partial_x X\Vert_{C^0}^k\Vert (\partial_z Y)^{-1}\Vert_{C^0}<1/2,
$$
$$\Vert \partial_x Z\Vert_{C^0},\Vert \partial_x Y\Vert_{C^0}, \Vert \partial_y Z\Vert_{C^0}, \Vert \partial_z Y\Vert_{C^0}< \eta.$$
The graph transform is defined to be for $S(x)=(S^u(x),S^s(x))$, a section in $E_{p'}^c\to E_{p'}^c\oplus E_{p'}^u\oplus E_{p'}^s$, we assign $S'=(S'^u(x),S'^s(x))=G(S)$, where
$
S'^s(X(x,S^u(x),S^s(x)))=Z(x,S^u(x), S^s(x))$ and $S'^u(x)$ is solved implicitly from $$S^u(X(x,S'^u(x),S^s(x)))=Y(x,S'^u(x),S^s(x)).$$ The solution exists since $Y(x,0,0)=0$ and $\partial_y Y\neq 0.$ One can verify that the graph transform $G$ is a contraction from $\mathcal S_\dt\to\mathcal  S_\dt$, hence there is unique solution $(S^u,S^v)$ satisfying
$$
\begin{aligned}
&S^s(X(x,S^u(x),S^s(x)))=Z(x,S^u(x),S^s(x)),\\
&S^u(X(x,S^u(x),S^s(x)))=Y(x,S^u(x),S^s(x)).
\end{aligned}
$$

Here we only show how to prove the existence of the NHIC. We see from the above proof that the $\eps^{-1/2}\omega^\star$ term does not play a role since it disappears in the derivative of the map. It turns out the conclusion of the standard normally hyperbolic invariant manifold theorem holds in our setting. For more information such as the regularity of the center manifold, the existence and regularity of stable and unstable manifolds, we refer the readers to \cite{Fe}. 

\end{proof}
\section{Variational construction of global diffusion orbits}\label{SVariation}
\setcounter{equation}{0}
Global diffusion orbits are constructed shadowing a sequence of local connecting orbits end to end. There are two types of local connecting orbit, one is called \emph{type-{\it h}} as which looks like a ``heteroclinic" orbit, another one is called \emph{type-{\it c}} as it is constructed by using ``cohomology equivalence".

\subsection{Local connecting orbits of type-{\it h} with incomplete intersections}\label{incompleteconnection} For an Aubry set, if its stable set ``intersects"  its unstable set transversally, this Aubry set is connected to any other Aubry set nearby by local minimal orbits. It can be thought as a variational version of Arnold's mechanism, the condition of geometric transversality is replaced by the total disconnectedness of minimal points of the barrier function.

However, this condition is not always satisfied for the problem we encountered here. The stable set may intersect the unstable set on a set with nontrivial first homology, i.e. {\it incomplete intersection}.  In this section, we design a new method to handle this problem. Let us first formulate a version for time-periodic dependent Lagrangian.

Recall the definition of the function $h_c^{\infty}$ introduced in \cite{M93}
$$
h_c^{\infty}(x,x')=\liminf_{k\to\infty}\inf_{\stackrel{\gamma(-k)=x}{\scriptscriptstyle \gamma(k)=x'}}\int_{-k}^{k} \Big(L(\gamma(t),\dot\gamma(t),t)-\langle c,\dot\gamma\rangle+\alpha(c)\Big)dt.
$$
This function is closely related to weak KAM. Indeed, for $x\in \mathcal{A}_{c,i}|_{t=0}$ (the time-1-section of the Aubry class $\mathcal{A}_{c,i}\subset\mathcal{A}(c)$) we have
$$
h_c^{\infty}(x,x')=u_{c,i}^-(x')-u_{c,i}^+(x),
$$
where both $u_{c,i}^-$ and $u_{c,i}^+$ are the time-1-section of backward and forward elementary weak KAM respectively (see the Appendix A.3 for details). It inspired us to introduce a barrier function for two Aubry classes $\mathcal{A}_{c,i}$ and $\mathcal{A}_{c,j}$
$$
B_{c,i,j}(x)=u^-_{c,j}(x)-u^+_{c,i}(x).
$$
Passing through its minimal point there is a semi-static curve connecting these two classes, provided this point does not lie in the Aubry set.

If the Aubry set contains only one class, we work in certain finite covering space so that there are two classes. For example, if the configuration space is $\mathbb{T}^{j+k+\ell}$ and the time-1-section of the Aubry set stays in a neighbourhood of certain lower dimensional torus, $\mathcal{A}_0(c)\subset \mathbb{T}^{j+\ell}+\delta$, we introduce a covering space $\mathbb{T}^{j+\ell}\times\mathbb{T}^{k-1}\times 2\mathbb{T}$. With respect to this covering space the Aubry set contains two classes.

We introduce some notation and conventions. 
\begin{Not}
\begin{enumerate}
 \item For the product space $\mathbb{T}^{j+k+\ell}$ we use $\mathbb{T}^{j+\ell}=\{x\in\mathbb{T}^{j+k+\ell}:x_i=0,\ \forall\ i=j+1,\cdots,j+k\}$.
     \item   Given a set $S$, a point $x$ and a number $\delta$, $S+x$ denotes the translation of $S$ by $x$, i.e. $S+x=\{x'+x:x'\in S\}$ and $S+\delta$ denotes $\delta$-neighborhood of $S$, i.e. $S+\delta=\{x: d(x,S)\le\delta\}$.
          \item A set $N$ is called \emph{neighborhood of $(j,\ell)$-torus} if it is homeomorphic to an open neighborhood of $(j+\ell)$-dimensional torus whose first homology group is generated by $\{e_i:i=1,\cdots,j,j+k+1,\cdots,j+k+\ell\}$.
              \item Given a function $B$, we use $\mathrm{Argmin}\{B,S\}=\{x\in S:B(x)=\min B\}$ to denote the set of those minimal points of $B$ which are contained in the set $S$.
\end{enumerate}
\end{Not}
\begin{theo}\label{ThmTypeht}
For a time-periodic $C^2$-Lagrangian $L:T\mathbb{T}^{j+k+\ell}\times\mathbb{T}\to\mathbb{R}$ and a first cohomology class $c\in H^1(\mathbb{T}^{j+k+\ell},\mathbb{R})$ we assume the conditions as follows:
\begin{enumerate}
    \item the Aubry set $\mathcal{A}(c)$ contains two classes $\{\mathcal{A}_{c,i},\mathcal{A}_{c,i'}\}$ which lie in a neighbourhood of $(j,\ell)$ torus $\mathcal{A}_{c,i}|_{t=0}\subset N_i$ and $\mathcal{A}_{c,i'}|_{t=0}\subset N_{i'}$. These neighborhoods are separated, i.e. $\bar N_{i}\cap\bar N_{i'}=\varnothing$;
    \item there exist topological balls $\{O_m\subset\mathbb{T}^{j+k}\}$ with $\bar O_m\cap\bar O_{m'}=\varnothing$ for $m\neq m'$, each connected component of $$
        \mathrm{Argmin}\{B_{c,i,i'},\mathbb{T}^{j+k+\ell}\backslash N_i\cup N_{i'}\}
        $$
        is contained in certain $O_m\times\mathbb{T}^{\ell}$;
\end{enumerate}
Then, for $c'\in H^1(\mathbb{T}^{j+k+\ell},\mathbb{R})$ satisfying following conditions
\begin{enumerate}
    \item $\langle c'-c,g\rangle =0$ holds $\forall$ $g\in H_1(\mathbb{T}^{j+k+\ell},\mathbb{T}^{j+k}, \mathbb{Z})$ and $|c'-c|\ll 1$;
    \item the Aubry set $\mathcal{A}(c')\subset N_i\cup N_{i'}$;
\end{enumerate}
there exists an orbit $(\gamma,\dot\gamma)$ of $\phi_L^t$ which connects $\widetilde{\mathcal{A}}(c)$ to $\widetilde{\mathcal{A}}(c')$ in the following sense, the $\alpha$-limit set of $(\gamma,\dot\gamma)$ is contained in $\widetilde{\mathcal{A}}(c)$, the $\omega$-limit set of $(\gamma,\dot\gamma)$ is contained in $\widetilde{\mathcal{A}}(c')$ or vice versa.
\end{theo}
\brk{\rm If $\ell=0$, the set $\mathrm{Argmin}\{B_{c,i,i'},\mathbb{T}^{j+k+\ell} \backslash N_i\cup N_{i'}\}$ is topologically trivial, it implies the stable set intersects the unstable set topologically transversely. Therefore, it turns out to be a variational version of Arnold's mechanism. The case of $\ell>0$ is a generalization of Arnold's mechanism in which case we allow that stable and unstable sets to intersect non transversely. Geometrically, this allows the separatrix to remain non splitting on the $\T^\ell$ component. }
\erk
\brk{\rm If the Aubry set consists of one Aubry class, we study this problem in certain covering space so that the Aubry set consists of two classes. The second condition for $c$ can be weakened so that the result becomes sharper, but the condition here is easier to verify and good enough for our purpose. Because of the upper semi-continuity of Ma\~n\'e set in the first cohomology class, the Aubry set $\mathcal{A}(c')$ is also contained in neighborhoods of these lower dimensional tori.}
\erk
\begin{proof}
It is proved by exploiting the upper semi-continuity of Ma\~n\'e set with respect to perturbation on the Lagrangian. As $\mathcal{A}(c')\subset N_i\cup N_{i'}$, without lose of generality we assume $\mathcal{A}(c')\cap N_{i'}\ne\varnothing$.

Given a ball $O_m$ there exists small $\epsilon$ such that $O_m+\epsilon$ does not touch other balls. Let $\tau_{1}$: $\mathbb{R}\to [0,\eps]$ be a smooth function such that $\tau_1(t)=0$ for $t\in(-\infty,0]\cup[1,\infty)$, $\tau_1(t)\ge0$ for $t\in[0,1]$ and $\max\tau_1=1$, Let $\tau_2$: $\mathbb{R}\to [0,1]$ be a smooth function such that $\tau_2(t)=0$ for $t\le 0$ and $\tau_2(t)=1$ for $t\ge 1$. Let $v$: $\mathbb{T}^{j+k+\ell}\to [0,\eps]$ so that $v(x)=0$ if $x\notin (O_m+\epsilon) \times\mathbb{T}^{\ell}$ and $v(x)=\eps$ if $x\in O_m\times\mathbb{T}^{\ell}$. As $\langle c'-c,g\rangle=0$ for each $g\in H_1(\mathbb{T}^{j+k+\ell}, \mathbb{T}^{j+k},\mathbb{Z})$, $\exists$ smooth function $u\in\mathbb{T}^{j+k+\ell}\to\mathbb{R}$ so that $\partial u=c'-c$ when it is restricted in $(O_m+\epsilon)\times\mathbb{T}^{\ell}$ and $\partial u=0$ if $x\notin (O_m+2\epsilon)\times\mathbb{T}^{\ell}$.

We introduce a modified Lagrangian
$$
L_{c,v,u}(\dot x,x,t)=L(\dot x,x,t)-\langle c,\dot x\rangle-\tau_1(t) v(x)-\tau_2(t)\langle c'-c-\partial u,\dot x\rangle
$$
and consider the minimizer $\gamma_{k^-,k^+}$: $[-k^-,k^+]\to\bar M$ of the action
$$
h^{k^-,k^+}_{c,v,u}(x^-,x^+)=\inf_{\stackrel{\gamma(-k^-)=x^-}{\scriptscriptstyle \gamma(k^+)=x^+}}\int_{-k^-}^{k^+} L_{c,v,u}(\gamma(t),\dot\gamma(t),t)dt+k^-\alpha(c)+k^+\alpha(c')
$$
where $x\in\mathcal{A}_{c,i}|_{t=0}$ and $x'\in\mathcal{A}_{c',i'}|_{t=0}$. As the Lagrangian is Tonelli, for any large $T$, the set of the curves $\{\gamma_k|_{[-T,T]}:k^-,k^+\ge T\}$ is $C^2$-bounded, therefore it is $C^1$-compact. Let $T\to\infty$, by diagonal extraction argument, we can find a subsequence of $\gamma_{k_i}$ which converges $C^1$-uniformly on each compact interval to a $C^1$-curve $\gamma$: $\mathbb{R}\to\bar M$ which is a minimizer of $L_{c,v,u}$ on any compact interval of $\mathbb{R}$.

Let $\mathscr{C}(L_{c,v,u})$ denote the set of minimal curves of $L_{c,v,u}$, it follows from the above argument that the set  $\mathscr{C}(L_{c,v,u})$ is non-empty. Restricted on $(-\infty,0]$ as well as on $[1,\infty)$, each curve in $\mathscr{C}(L_{c,v,u})$ solves the Euler-Lagrange equation for $L$ since $\tau_1=0$ and $\langle c'-c-\partial u,\dot x\rangle$ is closed. We are going to show that it also solves the equation for $t\in[0,1]$.

If both $\tau_1$ and $\tau_2$ vanish, each curve in the set $\mathscr{C}(L_{c,v,u})$ is nothing else but a $c$-semi static curve of $L$. These curves produce orbits which connect $\mathcal{A}_{c,i}$ to $\mathcal{A}_{c,i'}$. Consider all semi-static curves which intersect $O_m\times\mathbb{T}^{\ell}$ at $t=0$. As $O_m\times\mathbb{T}^{\ell}$ is open, the set of semi-static curves is closed, $\exists$ small $t_{\delta}>0$ such that these curve intersect $ O_m\times\mathbb{T}^{\ell}$ also for $t\in[0,t_{\delta}]$. If we set $\tau_1=0$ for $t\in(-\infty,0]\cup[t_{\delta},\infty)$ and set $\tau_2\equiv 0$, these semi-static curves solve the Euler-Lagrange equation produced by $L_{c,v,u}$. As a matter of fact, along these curves the function $v$ keeps constant when $\tau_1\neq 0$, the term $\tau_1v$ does not contribute to the equation. Clearly, the action of $L_{c,v,u}$ along these curves is smaller than those semi-static curves which do not pass through $O_m\times\mathbb{T}^{\ell}$ around $t=0$. Since $L_{c,v,u}$ is no longer time-periodic, a time-1-translation of its minimal curve is not necessarily minimal, i.e. $\gamma\in\mathscr{C}(L_{c,v,u})$ does not guarantee $k^*\gamma\in\mathscr{C}(L_{c,v,u})$ for $k\in\mathbb{Z}$, where $k^*$ denotes a translation operator such that $k^*\tau(t)=\tau(t+k)$.

Next, let us recover the term $\tau_2$. Because of upper semi-continuity, the minimal curve of $L_{c,v,u}$ must pass through $O_m\times\mathbb{T}^{\ell}$ if $c'$ is sufficiently close to $c$. Again, along these curves, the term $\tau_2\partial u$ does not contribute to the Euler-Lagrange equation, along these curves $\partial u=c'-c$ when $\tau_2\in (0,1)$.

Obviously, the orbit produced by each curve in the set $\mathscr{C}(L_{c,v,u})$ takes $\widetilde{\mathcal{A}}(c)$ as its $\alpha$-limit set and take $\widetilde{\mathcal{A}}(c')$ its $\omega$-limit set.
\end{proof}

The orbit $(\gamma,\dot\gamma)$ obtained in this theorem is locally minimal in the following sense:

\noindent{\bf Local minimum}: {\it There are open balls $V_i^-$, $V_{i'}^+$ and positive integers $t^-,t^+$ such that $\bar V_i^-\subset N_i\backslash\mathcal {A}_0(c)$, $\bar V_{i'}^+\subset N_{i'}\backslash\mathcal {A}_0(c')$, $\gamma(-k^-)\in V_i^-$, $\gamma(k^+)\in V_{i'}^+$ and
\begin{equation}\label{localeq1}
\begin{aligned}
&h_c^{\infty}(x^-,m_0)+h_{c,v,u}^{k^-,k^+}(m_0,m_1)+h_{c'}^{\infty}(m_1,x^+)\\
&-\liminf_{\stackrel {k^-_i\to\infty}{\scriptscriptstyle k_i^+\to\infty}}\int_{-k^-_i}^{k^+_i}
L_{c,v,u}(d\gamma(t),t)dt-k^-_i\alpha(c)-k^+_i\alpha(c')>0
\end{aligned}
\end{equation}
holds $\forall$ $(m_0,m_1)\in\partial(V_i^-\times V_{i'}^+)$, $x^-\in N_i\cap\pi_x(\alpha(d\gamma))_{t=0}$, $x^+\in N_{i'}\cap\pi_x(\omega(d\gamma))|_{t=0}$, where $k^-_i, k^+_i\in\mathbb{Z}^+$ are the sequences such that $\gamma(-k^-_i)\to x^-$ and $\gamma(k^+_i)\to x^+$.}

The set of curves starting from $V_i^-$ and reaching $V_{i'}^+$ with time $k^-+k^+$ make up a neighborhood of the curve $\gamma$ in the space of curves. If it touches the boundary of this neighborhood, the action of $L_{c,v,u}$ along a curve $\xi$ will be larger than the action along $\gamma$. The local minimality is crucial in the variational construction of global connecting orbits.

Next, we formulate the theorem for autonomous Lagrangian. As the Lagrangian is independent of time, one angle variable plays the role of time. Given a first cohomology class, some coordinate system exists $G_c^{-1}x$ such that $\omega_1(\mu)>0$ for each ergodic $c$-minimal measure $\mu$ if $\alpha(c)>\min\alpha$, where we use $\omega(\mu)=(\omega_1(\mu),\cdots,\omega_n(\mu))$ to denote the rotation vector of the invariant measure  (see \cite{Lx}). For this purpose, we work in a covering space $\bar\pi :\bar M=\mathbb{R}\times\pi_{ -1}\check M$, where $\pi_{ -1}$ denotes the operation to eliminate the first entry, $\pi_{ -1}(x_1,x_2,\cdots,x_{m})=(x_2,\cdots,x_m)$, the dimension $\mathbb{R}$ is for the coordinate $x_1$, $\check M=\mathbb{T}^{j+\ell}\times\mathbb{T}^{k-1}\times 2\mathbb{T}$ if the Aubry set consists of only one class which stays in a neighbourhood of $(j,\ell)$-torus and $\check M=\mathbb{T}^{j+k+\ell}$ if the Aubry set contains two classes.

\begin{theo}\label{connecting-lemma}
For the autonomous $C^2$-Lagrangian $L$: $T\mathbb{T}^{j+k+\ell}\to\mathbb{R}$ and the first cohomology class $c\in H^1(\mathbb{T}^{j+k+\ell},\mathbb{R})$ we assume the conditions as follows:
\begin{enumerate}
    \item $\omega_1(\mu)>0$ holds for each ergodic $c$-minimal measure.
    \item the Aubry set $\mathcal{A}(c,\check M)$ contains two classes $\{\mathcal{A}_{c,i},\mathcal{A}_{c,i'}\}$, both stay in a neighbourhood of $(j,\ell)$ torus, i.e. $\mathcal{A}_{c,i}\subset N_i$, $\mathcal{A}_{c,i'}\subset N_{i'}$. These neighborhoods are separated, i.e. $\bar N_{i}\cap\bar N_{i'}=\varnothing$. The lift of both $N_{i}$ and $N_{i'}$ to $\bar M$ is still connected and extends to $x_1=\pm\infty$;
    \item there exist topological disks $\{O_m\subset\pi_{ -1}(\mathbb{T}^{j}\times\mathbb{T}^k)\}$ with $\bar O_m\cap\bar O_{m'}=\varnothing$ for $m\neq m'$, such that each connected component of $$
        \mathrm{Argmin}\{B_{c,i,i'},\Sigma_0\backslash N_i\cup N_{i'}\}
        $$
        is contained in certain $\{x_1=0\}\times O_m\times\mathbb{T}^{\ell}$, where $\Sigma_0=\{x_1=0\}\times\pi_{ -1}\check M$ is a section of $\bar M$.
\end{enumerate}
Then, for $c'\in H^1(\mathbb{T}^{j+k+\ell},\mathbb{R})$ satisfying following conditions
\begin{enumerate}
    \item $\alpha(c')=\alpha(c)$;
    \item $\langle c'-c,g\rangle =0$ holds $\forall$ $g\in H_1(\mathbb{T}^{j+k+\ell},\mathbb{T}^{j+k}, \mathbb{Z})$ and $|c'-c|\ll 1$;
    \item the Aubry set $\mathcal{A}(c')\subset N_i\cup N_{i'}$;
\end{enumerate}
there exists an orbit $(\gamma,\dot\gamma)$ of $\phi_L^t$ which connects $\widetilde{\mathcal{A}}(c)$ to $\widetilde{\mathcal{A}}(c')$ in the following sense, the $\alpha$-limit set of $(\gamma,\dot\gamma)$ is contained in $\widetilde{\mathcal{A}}(c)$, the $\omega$-limit set of $(\gamma,\dot\gamma)$ is contained in $\widetilde{\mathcal{A}}(c')$ or vice versa.
\end{theo}
\brk{\rm For autonomous system, barrier function keeps constant along minimal curve. The intersection of minimal curves of autonomous system with the section $\Sigma_c$ is an analogy of $\mathcal{A}_0(c)$ and $\mathcal{N}_0(c)$ for time-periodic system.}
\erk

To prove this theorem and establish an analogous inequality of (\ref{localeq1}), we need some notations and definitions. A Lagrangian $L$: $T\bar M\to\mathbb{R}$ is called space-step if there exist Lagrangian $L^-,L^+\in C^2(T\mathbb{T}^{j+k+\ell},\mathbb{R})$, such that $L^-(x_1,\cdot)|_{(-\infty,-\delta)}= L(x_1,\cdot)|_{(-\infty,-\delta)}$ and $L^+(x_1,\cdot)|_{(\delta,\infty)}=L(x_1,\cdot)|_{(\delta,\infty)}$ where we treat $L^{\pm}$: $TM\to\mathbb{R}$ as its natural lift to $T\mathbb{T}^{j+k+\ell}$. We assume some conditions:
\begin{enumerate}
   \item $\omega_1(\mu^{\pm})>0$ for each ergodic minimal measure $\mu^{\pm}$ of $L^{\pm}$ respectively;
   \item $\min\beta_{L^-}=\min\beta_{L^+}$, without losing of generality, it equals zero;
   \item $|L^--L^+|\le\frac 12\min_{\omega_1=0}\{ \beta_{L^-}(\omega'),\beta_{L^+}(\omega')\}$.
\end{enumerate}
 As the minimal average action of $L^{\pm}$ is achieved on $\text{\rm supp}\mu^{\pm}$ with $\omega(\mu^{\pm})=0$, one has $\min_{\omega_1(\nu)\neq0}\int L^{\pm}d\nu>\min\int L^{\pm}d\nu$, so the third condition makes sense. To introduce the concept of minimal curve for space-step Lagrangian, we define
\begin{equation*}
h_{L}^{T}(\bar m_0,\bar m_1)=\inf_{\stackrel{\bar\gamma(-T)=\bar m_0} {\scriptscriptstyle \bar\gamma(T)=\bar m_1}}A_L(\bar\gamma|_{[-T,T]}), \qquad \forall\ \bar m_0,\bar m_1\in\bar M,
\end{equation*}
where
$$
A_L(\bar\gamma|_{[-T,T]})=\int_{-T}^{T}L(\bar\gamma(t),\dot{\bar\gamma}(t))dt.
$$

To generalize semi-static curve to space-step Lagrangian, we first define a set $\mathscr{G}(L)$ of minimal curves. We have the following lemma.
\begin{lem}\label{semicontinuitylem2}
If the rotation vector of each ergodic minimal measure has positive first component $\omega_1(\mu^{\pm})>0$, $\bar m_0\neq\bar m_1$, then
$$
\lim_{T\to 0}h_{L}^{T}(\bar m_0,\bar m_1)=\infty \ \ \ \
 and\ \ \ \ \lim_{T\to\infty}h_{L}^{T}(\bar m_0,\bar m_1)=\infty.
$$
\end{lem}
\begin{proof}
Let $\bar\gamma^{T}_{L}$: $[-T,T]\to\bar M$ be the minimizer of $h_{L}^{T}(\bar m_0,\bar m_1)$. Let $m_0=\pi \bar m_0$, $m_1=\pi \bar m_1$, $\zeta$: $[0,1]\to M$ be a smooth curve connecting $m_1$ to $m_0$, $\dot\zeta(0)=\dot{\bar\gamma}^{T}_{L}(T)$ and $\dot\zeta(1)=\dot{\bar\gamma}^{T}_{L}(-T)$. The action of $L^+$ along $\zeta$ is clearly bounded, thus for any $\epsilon>0$, one has $A_{L^+}(\zeta)\le 2T\epsilon$ provided $T$ is sufficiently large. The curve $\xi=\zeta\ast\pi \bar\gamma^{T}_{L}$ determines a holonomic probability measure $\nu^{T}_{L}\in\mathfrak{H}$ such that
$$
\int fd\nu^{T}_{L}=\frac 1{2T+1}\int_{-T}^{T+1}f(\xi(t),\dot\xi(t))dt\qquad \forall\ f\in C(TM,\mathbb{R}).
$$
Since $|\bar\gamma_L^T(T)-\bar\gamma_L^T(-T)|$ is bounded for any $T>0$, one has $\omega_1(\nu^{T}_{L})\to 0$ as $T\to\infty$. By using the third condition, we obtain
\begin{align*}
\frac 1{2T}h_{L}^{T}(\bar m_0,\bar m_1)=&\frac{2T+1}{2T}\int L^+d\nu^{T}_{L}-\frac 1{2T}
\int_0^1 L^+(\zeta(t),\dot\zeta(t))dt \\
&+\frac 1{2T}\int_{-T}^{T}(L-L^+)(\bar\gamma^{T}_{L}(t), \dot{\bar\gamma}^{T}_{L}(t))dt\\
\ge&\int L^+d\nu^{T}_{L}-\frac 12\min_{\omega_1=0}\beta_{L^+}(\omega)-\epsilon>0.
\end{align*}
It implies that $\lim_{T\to\infty}h_{L}^{T}(\bar m_0,\bar m_1)=\infty$. The case for $T\to 0$ is a consequence of the super-linear growth of $L$ in $\dot x$.
\end{proof}

Consequently, the following definition makes sense
\begin{defi}A curve $\bar\gamma:\mathbb{R}\to\bar M$ is in $\mathscr{G}(L)$ if
$$
A_L(\bar\gamma|_{[-T,T]})=\inf_{T'\in\mathbb{R}_+} h_{L}^{T'}(\bar\gamma(-T),\bar\gamma(T)).
$$
\end{defi}
The set $\mathscr{G}(L)$ is nonempty. Denote by $\bar\gamma_{L}(\cdot,\bar m_0,\bar m_1):[-T,T]\to M$ the minimizer such that $\bar\gamma_{L}(-T)=\bar m_0$, $\bar\gamma_{L}(T)=\bar m_1$ and
$$
A(\bar\gamma_{L})=\int_{-T}^{T}L(\bar\gamma_{L}(t),\dot{\bar\gamma}_{L}(t))dt
=\inf_{T'\in\mathbb{R}_+} h_{L}^{T'}(\bar m,\bar m').
$$
Because of Lemma \ref{semicontinuitylem2}, this infimum is attained for finite $T>0$ if $\bar m_0$ and $\bar m_1$ are two different points in $\bar M$. The super-linear growth of $L$ in $\dot x$ guarantees that $T\to\infty$ as $-\bar m_{01},\bar m_{11}\to\infty$, where $\bar m_{i1}$ denotes the first entry of $\bar m_i$ for $i=0,1$. Given an interval $[-T,T]$, for sufficiently large $-\bar m_{01},\bar m_{11}$, the set $\{\bar\gamma_{L}(\cdot,\bar m_0,\bar m_1)|_{[-T,T]}\}$ is
pre-compact in $C^1([-T,T],\bar M)$. Let $T\to\infty$. By diagonal extraction argument, there is a subsequence of $\{\bar\gamma_{L}(\cdot,\bar m_0,\bar m_1)\}$ which converges
$C^1$-uniformly on any compact set to a $C^1$-curve $\bar\gamma$: $\mathbb{R}\to\bar M$. \begin{pro}\label{semicontinuitypro1}
Some number $K>0$ exists so that $|h_{L}^{T}(\bar\gamma(-T), \bar\gamma(T))|\le K$ holds simultaneously for all curve $\bar\gamma\in \mathscr{G}(L)$ and all $T>0$.
\end{pro}
\begin{proof}
By the assumption, one has $\alpha_{L^{\pm}}(0)=\min\beta_{L^{\pm}}=0$. So, some $K'>0$ exists such that $|A(\gamma|_I)|\le K'$ holds for any interval $I\subset\mathbb{R}_+(\mathbb{R}_-)$ provided it is a forward (backward) semi-static curves for $L^{+}$ ($L^-$). Also, some $K''>0$ exists such that
$$
-K''\le\max_{\bar x,\bar x'\in \{x\in\bar M:|x_1|\le 1\}}\inf_{T\ge 0}h^T_L(\bar x,\bar x')\le K''.
$$
We claim that $K\le 2K'+K''$.

If there exists some $\bar\gamma\in\mathscr{G}(L)$ and some $T>0$ such that  $h_{L}^{T}(\bar\gamma(-T), \bar\gamma(T))>2K'+K''$, we join $\bar\gamma(-T)$ to $\bar\gamma(T)$ by another curve $\xi=\bar\gamma_-\ast\zeta\ast\bar\gamma_+$ where $\bar\gamma_-$ is a lift of backward semi-static curve $\gamma_-$ for $L_-$ such that $\bar\gamma(-T)=\bar\gamma_-(0)$, denote by $\bar x_-$ the intersection point of this curve with the section $\{\bar x\in\bar M:\bar x_1=-1\}$,  $\bar\gamma_+$ is a lift of forward semi-static curve $\gamma_+$ for $L_+$ such that $\bar\gamma(T)=\bar\gamma_+(0)$, denote by $\bar x_+$ the intersection point of this curve with the section $\{\bar x\in\bar M:\bar x_1=1\}$, $\zeta$ is a minimal curve of $L$ that connects the point $\bar x_-$ to $\bar x_+$. Obviously, one has $A_L(\xi)\le 2K'+K''<h_{L}^{T}(\bar\gamma(-T), \bar\gamma(T))$, but it contradicts the definition of $\mathscr{G}(L)$.
\end{proof}

Each $k\in\mathbb{Z}$ defines a Deck transformation ${\bf k}:\bar M\to\bar M$: ${\bf k}x=(x_1+k,x_2,\cdots,x_n)$. Let $\bar M^-_{\delta}=\{x\in\bar M:x_1<-\delta\}$, $\bar M^+_{\delta}=\{x\in\bar M:x_1>\delta\}$. With this notation we are able to define the set of {\it pseudo connecting curve}.
\begin{defi}[pseudo connecting curve]\label{semicontinuitydef2}
A curve $\bar\gamma\in\mathscr{G}(L)$ is called pseudo connecting curve if the following holds
$$
A_L(\bar\gamma|_{[-T,T]})=\inf_{\stackrel{\stackrel{T'\in\mathbb{R}_+} {\scriptscriptstyle {\bf k}^-\bar\gamma(-T)\in \bar M^-_{\delta}}}{\scriptscriptstyle {\bf k}^+ \bar\gamma(T)\in \bar M^+_{\delta}}} h_{L}^{T'}({\bf k}^- \bar\gamma(-T),{\bf k}^+\bar\gamma (T))
$$
for each $\bar\gamma(T)\in \bar M^-_{\delta}$ and $\bar\gamma(T)\in\bar M^+_{\delta}$. Denote by $\mathscr{C}(L)$ the set of all pseudo connecting curves.
\end{defi}

\begin{lem}\label{semicontinuitylem3}
The set $\mathscr{C}(L)$ is non-empty.
\end{lem}
\begin{proof} Let us start with a curve $\bar\gamma\in\mathscr{G}(L)$. Given $\Delta>0$, if some interval $[t^-_i,t^+_i]$  exists such that ${\bf k}^-_i \bar\gamma(t^-_i)$ can be connected to ${\bf k}^+_i\bar\gamma(t_i^+)$ by another curve $\zeta_i$ with smaller action
$$
A_L(\gamma|_{[t_i^-,t_i^+]})-A_L(\zeta_i)\ge\Delta>0,
$$
then one obtain a curve $\bar\gamma_i={\bf k}^-_i \bar\gamma|_{(-\infty,t^-_i]}\ast\zeta\ast{\bf k}^-_i \bar\gamma|_{[t^+_i,\infty)}$ by one step of such surgery.

Given any $\Delta>0$, we claim that there are finitely many intervals $[t_i^-,t_i^+]$ with $t_i^+\le t_{i+1}^-$ such that ${\bf k}^-_i \bar\gamma(t^-_i)$ can be connected to ${\bf k}^+_i\bar\gamma(t_i^+)$ by another curve $\zeta_i$ with the action $\Delta$ smaller than the original one.
Let us assume the contrary. Then, for any positive integer $m$, some large $T>0$ exists such that $[-T,T]\supset \cup_{i=1}^m [t_i^-,t_i^+]$. We can choose arbitrarily many of such intervals such that  either $t_1^->\delta$ or $t^+_m<-\delta$. In the first case, let $\bar x^-=\bar\gamma(-T)$ and $\bar x^+=\Pi_{\ell=1}^m{\bf k}^-_{\ell}{\bf k}^+_{\ell}\bar\gamma(T)$. By assumption, these two points can be connected by a curve $\zeta$ along which the action $A_L(\zeta)\le K-m\Delta$ as it follows from Proposition \ref{semicontinuitypro1} that $A_L(\bar\gamma|_{[-T,T]})\le K$. Since $m$ can be arbitrarily large, it implies the existence of a curve along which the action of $L$ approaches to minus infinity, it also contradicts Proposition \ref{semicontinuitypro1}.

Given a curve $\bar\gamma\in\mathscr{G}(L)$ and any small $\epsilon_i>0$, by finitely many steps of such surgery, we obtain a curve $\bar\gamma_i:\mathbb{R}\to\bar M$ with following properties:

1, for each small $\epsilon_i>0$, some large $T_i$ exists such that $\bar\gamma(-T_i)\in\bar M^-_{\delta}$, $\bar\gamma(T)\in\bar M^+_{\delta}$ and
$$
A_L(\bar\gamma_i|_{[-T_i,T_i]})\le\inf_{\stackrel{\stackrel{T'\in\mathbb{R}_+} {\scriptscriptstyle {\bf k}^-\bar\gamma(-T)\in \bar M^-_{\delta}}}{\scriptscriptstyle {\bf k}^+ \bar\gamma(T)\in \bar M^+_{\delta}}}h_{L}^{T'}({\bf k}^-\bar\gamma_i(-T),{\bf k}^+\bar\gamma_i(T)) +\epsilon_i.
$$

2, $\bar\gamma_i$ is smooth everywhere except for two points which fall beyond the region $\{x\in\bar M: |x_1|\le\Theta_i\}$, and $\Theta_i\to\infty$ as $\epsilon_i\to 0$.

Let $T'_i>0$ such that $\bar\gamma_{i1}(\pm T'_i)=\pm\Theta_i$. Because of Lemma
\ref{semicontinuitylem2}, we see that $T'_i\to\infty$ as $\Theta_i\to\infty$. In virtue of the argument before, for any large $T$ $\exists$ $i_0>0$ such that the set $\{\bar\gamma_i|_{[-T,T]}:i\ge i_0\}$ is pre-compact in $C^1([-T,T],\bar M)$. Let $T\to\infty$, by diagonal extraction argument, there is a subsequence of $\{\bar\gamma_i\}$ which converges $C^1$-uniformly on each compact set to a $C^1$-curve $\bar\gamma$: $\mathbb{R}\to\bar M$. Obviously, $\bar\gamma\in\mathscr{C}(L)$.
\end{proof}

\begin{theo}\label{semicontinuitythm3}
The map $L\to\mathscr{C}(L)$ is upper semi-continuous.
\end{theo}
\begin{proof}
Let $\bar\gamma_i\in\mathscr{C}(L_i)$, $L_i\to L$. If $\{\bar\gamma_i\}$ converges $C^1$-uniformly on each compact set to a $C^1$-curve $\bar\gamma$, it is obvious that $\bar\gamma\in\mathscr{C}(L)$.
\end{proof}
Obviously, if the space-step Lagrangian $L$ is periodic in $x_1$, then a curve $\bar\gamma\in\mathscr{C}(L)$ if and only if its projection $\gamma=\pi \bar\gamma$: $\mathbb{R}\to M$ is semi-static.

\begin{proof}[Proof of Theorem \ref{connecting-lemma}]
In autonomous system, $\widetilde{\mathcal{A}}(c)$ can be connected to $\widetilde{\mathcal{A}}(c')$ only if $\alpha(c)=\alpha(c')$. If $c,c'\in\alpha^{-1}(\min\alpha)$, then $\widetilde{\mathcal {A}}(c)\cap\widetilde{ \mathcal {A}}(c')\neq\varnothing$ (see \cite{Ms}), it is trivial to connect an Aubry set to itself. So, we only need to work on the energy level set $H^{-1}(E)$ with $E>\min\alpha$. Under this condition, there exists a coordinate system so that $\omega_1(\mu)>0$ holds for each ergodic minimal measure of $c$ and $c'$ if they are close to each other.

The section $\Sigma_0$ separates $\bar M$ into two parts, the upper part $\bar M^+$ extending to $\{x_1=\infty\}$ and the lower part $\bar M^-$ connected to $\{x_1=-\infty\}$. Denote a $\delta$-neighborhood of $\Sigma_0$ in $\bar M$ by $\Sigma_0+\delta$, we introduce a smooth function $\chi\in C^r(\bar M,[0,1])$ such that $\chi=0$ if $x\in\bar M^- \backslash (\Sigma_0+\delta)$, $\chi=1$ if $x\in\bar M^+\backslash (\Sigma_0+\delta)$.

For those $c'$ such that $\langle c'-c,g\rangle =0$ holds for each $g\in H_1(\mathbb{T}^{j+k+\ell},\mathbb{T}^{j+k},\mathbb{Z})$, there exists a smooth function $u:\bar M\to\mathbb{R}$ so that $\partial u=c'-c$ if $x\in\{|x_1|<\epsilon\}\times(O_m+\epsilon)\times\mathbb{T}^{\ell}$ and $\partial u=0$ if $x\notin \{|x_1|<2\epsilon\}\times(O_m+2\epsilon)\times\mathbb{T}^{\ell}$.

Without lose of generality we assume $\widetilde{\mathcal{A}}(c')\cap N_{i'}\ne\varnothing$. We consider the set of semi-static curves which generate orbits connecting the Aubry class $\widetilde{\mathcal{A}}_{c,i}$ to another class $\widetilde{\mathcal{A}}_{c,i'}$. The lift of the curves to $\bar M$ intersect the section $\Sigma_0$ in the set $\mathrm{Argmin}\{B_{c,i,i'},\Sigma_0\}$. We pick up a connected component of this set contained in certain tubular domain $S_{ii'}=\{x_1=0\}\times O_m\times\mathbb{T}^{\ell}$. Let $\bar\gamma_{ii'}(t,x)$ denote the lift of semi-static curves $\gamma_{ii'}(t,x)$ so that $\bar\gamma_{ii'}(0,x)=x\in S_{ii'}$. As all curve in $\{\gamma_{ii'}(t,x)\}$ take $\mathcal{A}_{c,i}$ as their $\alpha$-limit set and take $\mathcal{A}_{c,i'}$ as their $\omega$-limit set, $\dot{\bar\gamma}_{ii'}(t,x)$ is Lipschitz in $x\in\mathrm{Argmin}\{B_{c,i,i'},S_{ii'}\}$. We extend these curves to the whole $S_{ii'}$, so that $\dot{\bar\gamma}_{ii'}(t,x)$ is still Lipschitz in $x$ although the extended curves $\{\bar\gamma_{ii'}(t,x)\}$ do not generate orbits of $\phi_L^t$ if $x\notin\mathrm{Argmin}\{B_{c,i,i'},S_{ii'}\}$. By deforming $\Sigma_0\to\Sigma'$ we can assume that these curves pass transversally through the section $\Sigma'$.

Let $s=s(\bar\gamma_{ii'}(t,x))$ denote the arc-length of the curve from $\bar\gamma_{ii'}(0,x)$ to $\bar\gamma_{ii'}(t,x)$ in the Euclidean metric such that $s(\bar\gamma_{ii'}(0,x))=0$ and $s(\bar\gamma_{ii'}(t,x))>0$ if $t>0$. We approximate the function $s$ by a smooth function $s'$ in the tubular domain made up by the curves $\{\bar\gamma_{ii'}(t,x)\}$ with $\bar\gamma_{ii'}(0,x)\in S_{ii'}$. Let $\tau$: $\mathbb{R}\to[0,1]$ be a smooth function so that $\tau=0$ if $s\le 0$, $\tau=1$ if $s\ge s_0$ and $\dot\tau>0$ if $s\in(0,s_0)$. Let $w\in C^r(T\bar M,[0,1])$ such that $w=1$ when $(x,\dot x)$ is restricted in $\{(\bar\gamma_{ii'}(t,x),\dot{\bar\gamma}_{ii'}(t,x)):x\in S_{ii'},s\in [0,s_0]\}+\delta$ and $w=0$ when $(x,\dot x)$ does not lie in the set $\{(\bar\gamma_{ii'}(t,x),\dot{\bar\gamma}_{ii'}(t,x)):x\in S_{ii'},s\in [0,s_0]\}+2\delta$.

Next, we are going to show the curves in $\mathscr{C}(L_{c,v,u})$ produce orbits of $\phi_L^t$ connecting $\tilde{\mathcal{A}}(c)$ to $\tilde{\mathcal{A}}(c')$, where the modified Lagrangian $L_{c,v,u}$ is defined as follows
\begin{equation}\label{langrangianforh}
L_{c,v,u}=L-\langle c,\dot x\rangle-w\langle\partial(\tau\circ s'),\dot x\rangle-\chi\langle c'-c-\partial u,\dot x\rangle+\alpha(c),
\end{equation}
where $\chi=0$ if $x\in\bar M^- \backslash (\Sigma_0+\delta)$, $\chi=1$ if $x\in\bar M^+\backslash (\Sigma_0+\delta)$, the function $s'$ is extended to the whole space in any way one likes because $\pi_x\mathrm{supp}w$ is contained in the tubular domain where $s'$ is well-defined, $\pi_x$: $TM\to M$ denotes the standard projection along tangent fibers.
As the first step, let us set $\chi\equiv0$. Because of the upper semi-continuity of $L\to\mathscr{C}(L)$, each curve in $\mathscr{C}(L_{c,v,u})$ either is contained in $\{(\bar\gamma_{ii'}(t,x),\dot{\bar\gamma}_{ii'}(t,x)):x\in S_{ii'}\}+\delta$ or keeps away from the larger tubular domain $\{(\bar\gamma_{ii'}(t,x),\dot{\bar\gamma}_{ii'}(t,x)):x\in S_{ii'}\}+2\delta$. As $w\equiv1$ holds on the smaller tubular domain, the term $\langle\partial(\tau\circ s'),\dot x\rangle$ does not contribute to the Euler-Lagrange equation. By definition, along each orbit $(\gamma_{ii'}(t,x),\dot\gamma_{ii'}(t,x))|_{s\in[0,s_0]}$ one has $\langle\partial(\tau\circ s'),\dot x\rangle>0$. Therefore, in the lift of $\{\gamma_{ii'}(t,x)):x\in S_{ii'}\}$, only those curves are the member of $\mathscr{C}(L_{c,v,u})$ for $\chi=0$ if they pass through the section $S_{ii'}$. Other curves in the lift are not in $\mathscr{C}(L_{c,v,u})$ because they have larger Lagrange action.

Next, we recover the term $\chi\langle c'-c-\partial u,\dot x\rangle$ which is $C^2$-small. Due to the upper semi-continuity again, all curves in $\mathscr{C}(L_{c,v,u})$ must pass through $S_{ii'}$ if they connect $\mathcal{A}_{c,i}$ to $\mathcal{A}_{c',i'}$. As $\partial u=c'-c$ if $x\in\{|x_1|<\epsilon\}\times(O_m+\epsilon) \times\mathbb{T}^{\ell}$ and $\partial u=0$ if $x\notin\{|x_1|<2\epsilon\}\times(O_m+2\epsilon) \times\mathbb{T}^{\ell}$, one can see from the definition of $\chi$ that the term $\chi\langle c'-c-\partial u,\dot x\rangle$ does not contribute to the Euler-Lagrange equation. It implies these curves produce orbits of $\phi_L^t$ which connects $\widetilde{\mathcal{A}}(c)$ to $\widetilde{\mathcal{A}}(c')$.
\end{proof}

The orbit $(\gamma,\dot\gamma)$ obtained here is local minimal in following sense (analogous to \ref{localeq1}):

\noindent{\bf Local minimum}: {\it there exist two $(n-1)$dimensional disks $V_i^-$, $V_{i'}^+\subset\bar M$ and positive numbers $T,d>0$ such that $\bar\pi V_i^-\subset N_i\backslash\mathcal{A}(c)$, $\bar\pi V_{i'}^+\subset N_{i'}\backslash\mathcal{A}(c')$, $\gamma$ transversally passes $\bar\pi V_i^-$ and $\bar\pi V_{i'}^+$ at the time $-T$ and $T$ respectively, and
\begin{equation}\label{localmineq1}
\begin{aligned}
&h_c^{\infty}(x^-,\bar\pi\bar m_0)+h_{L_{c,v,u}}^{T'}(\bar m_0,\bar m_1)+h_{c'}^{\infty}(\bar\pi\bar m_1,x^+) \\
&-\lim_{\stackrel{t^-_i\to\infty}{\scriptscriptstyle t^+_i\to\infty}} \int_{-t^-_i}^{t^+_i}L_{c,v,u} (\gamma(t),\dot\gamma(t))dt-(t_i^-+t_i^+)\alpha(c)>0
\end{aligned}
\end{equation}
holds $\forall$ $(\bar m_0,\bar m_1,T')\in\partial(V_i^-\times V_{i'}^+\times [T-d,T+d])$, $x^-\in N_i\cap\pi_x(\alpha(d\gamma))$ and $x^+\in N_{i'}\cap\pi_x(\omega(d\gamma))$. Where $t^-_i\to\infty$ and $t^+_i\to\infty$ are the sequences such that $\gamma(-t^-_i)\to x^-$ and $\gamma(t^+_i)\to x^+$.}

\subsection{Local connecting orbits of type-$c$}\label{SStypec}
For autonomous system, if $c'$ is equivalent to $c$ with $|c-c'|\ll 1$, then $\langle c'-c,g\rangle=0$ holds for all $g\in H_1(\mathcal{N}(c)\cap\Sigma_c,\mathbb{Z})$ where $\Sigma_c$ is a section of $M$. So there is a function $u$ defined on the whole torus and $\partial u=c'-c$ holds in a small neighborhood of $\mathcal{N}(c)\cap\Sigma_c$. To connect $\widetilde{\mathcal{A}}(c)$ to $\widetilde{\mathcal{A}}(c')$, we work in a coordinate system $G_c^{-1}x$ so that $\omega_1(\mu_c)>0$ holds for each ergdic $c$-minimal measure.  The new coordinate system $G_c$ is chosen so that the lift $\Sigma_c$ to the covering manifold $\bar M$ contains infinitely many compact connected components. We fix one component, denoted by $\Sigma_{c}^0$. Other components in the lift of $\Sigma_c$ are obtained by translating this one by $2k\pi$ in the direction of $x_1$. The section $\Sigma_{c}^0$ separates $\bar M$ into two parts $\bar M^-$ and $\bar M^+$. In $\bar M^{\pm}$, the coordinate $x_1$ can be extended to $\pm\infty$. Let $\text{\rm sign}$ be a sign function defined as $\text{\rm sign}(x)=\pm 1$ if $x\in\bar M^{\pm}$.

Let $L_{c,u}$ be a space-step Lagrangian defined on the covering manifold $\bar M$
\begin{equation}\label{lagrangianforc}
L_{c,u}=L-\langle c,\dot x\rangle-\chi\langle c'-c-\partial u,\dot x\rangle+\alpha(c)
\end{equation}
where $\alpha(c)=\alpha(c')$, $\chi=0$ if $x\in\bar M^- \backslash (\Sigma_0+\delta)$, $\chi=1$ if $x\in\bar M^+\backslash (\Sigma_0+\delta)$. Obviously, for $c'=c$, we have $\bar\pi\mathscr{C}(L_{c,u})=\mathcal{N}(c)$. According to the upper semi-continuity, for sufficiently small $|c'-c|$, the image of each curve $\bar\gamma\in\mathscr{C}(L_{c,u})$ falls in a small neighborhood of $\mathcal{N}(c)$. Therefore, $\langle c'-c-\partial u,\dot x\rangle=0$ holds along this curve when it passes through a small neighborhood of $\Sigma_c^0$. It implies that the term $\chi\langle c'-c-\partial u,\dot x\rangle$ does not contribute to the Euler-Lagrange equation determined by $\bar L$. Therefore, this curve also solves the Euler-Lagrange equation for $L$. Clearly, $\bar\pi(\bar\gamma(t),\dot{\bar\gamma}(t))$ approaches the Aubry set for class $c'$ as $t\to\infty$. Therefore, we have

\begin{theo}[connecting orbits of type-{\it c}]\label{typecthm1}
Assume the cohomology class $c^*$ is $c$-equivalent to the class $c'$ through the path $\Gamma$: $[0,1]\to H^1(\mathbb{T}^n,\mathbb{R})$. For each $s\in [0,1]$, the following are assumed:
\ben
\item there exists a coordinate systems $G_s^{-1}x$ where the first component of rotation vector is positive, $\omega_1(\mu_{\Gamma(s)})>0$ for each ergodic $\Gamma(s)$-minimal measure $\mu_{\Gamma(s)}$;
\item for the covering space $\bar M_s=\mathbb{R}\times\mathbb{T}^{n-1}$ in the coordinate system the lift of non-degenerately embedded codimension-one torus $\Sigma_{\Gamma(s)}$ has infinitely many connected and compact components, each of which is also a codimension-one torus.
\een
Then there exist some classes $c^*=c_0, c_1,\cdots,c_k=c'$ on this path,  closed 1-forms $\eta_i$ and $\bar\mu_i$ on $M$ with $[\eta_i]=c_i$ and $[\bar\mu_i]=c_{i+1}-c_i$,  and smooth functions $\varrho_i$ on $\bar M$ for $i=0,1,\cdots,k-1$, such that the pseudo connecting curve set $\mathscr{C}(L_i)$ for the space-step Lagrangian
$$
L_{c_i,u_i}=L-\langle c_i,\dot x\rangle-\chi_i\langle c_{i+1}-c_i-\partial u_i,\dot x\rangle+\alpha(c_i)
$$
possesses the properties:
\ben
\item[(i)] each curve $\bar\gamma\in\mathscr{C}(L_i)$ determines an orbit $(\gamma,\dot{\gamma})$ of $\phi_L^t$;
\item[(ii)] the orbit $(\gamma,\dot{\gamma})$ connects $\widetilde{\mathcal{A}}(c_{i})$ to $\widetilde{\mathcal{A}}(c_{i+1})$, i.e., the $\alpha$-limit set $\alpha(d\gamma)\subseteq\widetilde{\mathcal{A}} (c_{i})$ and $\omega$-limit set $\omega(d\gamma)\subseteq\widetilde{\mathcal{A}}(c_{i+1})$.
\een
\end{theo}
\begin{proof}
By the definition of $c$-equivalence, there exists a path $\Gamma$: $[0,1]\to H^1(M,\mathbb{R})$ with $\Gamma(0)=c^*$, $\Gamma(1)=c'$ such that for each $c=\Gamma(s)$ ($s\in [0,1]$) on the path, there exists $\epsilon>0$ such that $\Gamma(s')-c\in \mathbb{V}_{{\Gamma}(s)}^{\bot}$ whenever $s'\in [0,1]$ and
$|s-s'|<\epsilon$. Thus, there exist a non-degenerately embedded  ($n-1$)-dimensional torus $\Sigma_c$,  a closed form $\bar\mu_{c}$ and a neighborhood $U$ of $\mathcal {N}(c)\cap \Sigma_{c}$ such that $[\bar\mu_{c}]=\Gamma(s')-c$ and $\text{\rm supp}\bar\mu_{c}\cap
U=\varnothing$.

In the new coordinates $x\to G^{-1}_{c}x$ on the torus as above, the codimension one hypersurface $\Sigma_c^0$ separates $\bar M$ into two parts, the upper part $\bar M^+$ and the lower part $\bar
M^-$. $\bar M^{\pm}$ extends to where the first coordinate $x_1\to\pm\infty$. Let $\Sigma_c^0+\delta$ denotes the $\delta$-neighborhood of $\Sigma_c^0$ in $\bar M$, we introduce a smooth function $\varrho\in C^r(\bar M,[0,1])$ such that $\varrho=0$ if $x\in\bar M^- \backslash (\Sigma_c^0+\delta)$, $\varrho=1$ if $x\in\bar M^+\backslash(\Sigma_c^0+\delta)$. Let $\eta$ and $\bar\mu$ are closed 1-forms on $M$ such that $[\eta]=c$ and $[\eta+\bar\mu]=c'$. These forms have natural lift on $\bar M$, with the same notation.

A sufficiently small $\delta>0$ can be chosen so that
$$
(\Sigma_{c}^0+\delta)\cap(\mathcal {C}(L+\eta)+2\delta)\subset U,
$$
It follows from the upper semi-continuity of $\mathcal {C}(L)$ w.r.t. $L$, we find
\begin{equation}\label{typeceq1}
(\Sigma_{c}^0+\delta)\cap(\mathcal{C}(L+\eta+\varrho\bar\mu)+\delta)\subset U,
\end{equation}
if $\varrho\bar\mu$ is $C^0$-sufficiently small. As $\bar\mu$ is carefully chosen so that its support is disjoint from $U$, each curve $\bar\gamma\in\mathscr{C}(L+\eta+\varrho\bar\mu)$ is clearly a solution of the Euler-Lagrange equation determined by $L$, the term $\varrho\bar\mu$ has no contribution to the equation along $\bar\gamma$. In other words, each curve in
$\mathscr{C}(L+\eta+\varrho\bar\mu)$ generates an orbit $d\gamma$ of $\phi_L^t$: $\mathbb{R}\to TM$.

The definition of $\mathscr{C}$ tells us that for each curve $\bar\gamma\in\mathscr{C}$, $\gamma|_{(-\infty,t_0]}$ is backward $\Gamma(s)$-semi static once $\bar\gamma|_{(-\infty,t_0]}$ falls entirely into
$\bar M^-\backslash (\Sigma_c^0+\delta)$, $\gamma|_{[t_1,\infty)}$ is forward $\Gamma(s')$-semi static once $\bar\gamma|_{[t_1,\infty)}$ falls entirely into $\bar M^+\backslash (\Sigma_c^0+\delta)$. Therefore,
$(\gamma(t),\dot\gamma(t))\to\tilde{\mathcal {A}}(\Gamma(s))$ as $t\to -\infty$ and $(\gamma(t),\dot\gamma(t))\to\tilde{\mathcal {A}}(\Gamma(s'))$ as $t\to\infty$.

Because of the compactness of $[0,1]$, there are finitely many numbers $s_0,\cdots,s_k\in [0,1]$ such that above argument applies if $s$ and $s'$ are replaced respectively by $s_i$ and $s_{i+1}$. Set $c_i=\Gamma(s_i)$.
\end{proof}

\begin{cor}\label{typeccor1} Let $c_i$, $c_{i+1}$, $\chi_i$ and $u_i$ be defined as in  Theorem \ref{typecthm1}. Let $U_i$ be a neighborhood of $\mathcal {N}(c_i)\cap \Sigma_{c_i}^0$ such that $c_{i+1}-c_i-\partial u_i|_{U_i}=0$. Then, there exist large $K_i>0$, $T_i>0$ and small $\delta>0$ such that for each $\bar m,\bar m'\in\bar M$, with $-K_i\le\bar m_1\le -K_i+2\pi$, $K_i-2\pi \le\bar m'_1\le K_i$, the quantity $h_{L_{c_i,u_i}}^{T}(\bar m,\bar m')$ reaches its minimum at some $T<T_i$ and the corresponding minimizer $\bar\gamma_{i}(t,\bar m,\bar m')$ satisfies the condition
\begin{equation}\label{typeceq2}
\text{\rm Image}(\bar\gamma_{i})\cap(\Sigma_{c_i}^0+\delta)\subset U_i.
\end{equation}
\end{cor}

There is some flexibility to choose the coordinate system and the non-degenerately embedded codimension one torus. Let $\pi_s$: $\bar M_s\to M=\mathbb{T}^n$ be a covering space such that $\bar M_s=\mathbb{R} \times\mathbb{T}^{n-1}$ in the coordinate system $G_s^{-1}x$.

\begin{defi}[admissible toral section]\label{typecdef1}
For $s\in [0,1]$, the non-degenerately embedded codimension one torus $\Sigma_s$ is called admissible for the coordinate system $G_s^{-1}x$ if the lift of $\Sigma_s$ to the covering space $\bar M_s$ consists of infinitely many connected and compact components, the first component of the rotation vector is positive $\omega_1(\mu_{\Gamma(s)})$ for each ergodic $\Gamma(s)$-minimal measure.
\end{defi}
\subsection{Global connecting orbits}\label{SGlobalConn}
In this section, we explain how to construct globally connecting orbit from local ones, i.e. Theorem \ref{constructionthm1}.
\begin{proof}[Sketch of the proof of Theorem \ref{constructionthm1}] The proof of this theorem is the same as \cite{LC}. We only sketch the idea of the proof here, readers can refer to \cite{LC} and Section 5 of \cite{CY1,CY2} for the details. Because of the condition of generalized transition chain, there is a sequence $0=s_0<s_1<\cdots<s_k=1$ such that for each $0\le j<k$, $\tilde{\mathcal{A}} (\Gamma(s_j))$ is connected to $\tilde{\mathcal{A}}(\Gamma(s_{j+1}))$ by local minimal orbit either of type-$h$ with incomplete intersection or of type-$c$. The global connecting orbits are constructed shadowing such a sequence of orbits.

Recall the construction of local connecting orbit as above, for each $i\in\{0,1,\cdots,k\}$ let $\eta_i(x,\dot x)=\langle c_i,\dot x\rangle$ and
$$
\mu_i(x,\dot x)=w_i\langle\partial(\tau_i\circ s'_i),\dot x\rangle,\qquad \psi_i(x,\dot x)=\chi_i\langle c_{i+1}-c_i-\partial u_i,\dot x\rangle
$$
in certain coordinate system $G^{-1}_ix$ (see (\ref{langrangianforh}), (\ref{lagrangianforc}) for the definition), if it is for type-$c$, we set $\mu_i=0$. For each integer $k$ we introduce a translation operator on functions $k^*f(x_1,x_2,\cdots,x_n)=f(x_1-k,x_2,\cdots,x_n)$.

Let $\tilde\pi$: $\mathbb{R}^n\to M$ be the universal covering space. For a curve $\tilde\gamma$: $[-K,K']\to\mathbb{R}^n$, let $\gamma=\tilde\pi\tilde\gamma$: $[-K,K']\to M$. Let $\vec{t}=(t_0^-,t_1^{\pm},\cdots, t_{k-1}^{\pm},t_{k}^+)$, $\vec{x}=(\tilde x_0^-,\tilde x_1^{\pm},\cdots, \tilde x_{k-1}^{\pm},\tilde x_{k}^+)$ with $t_i^+<t_i^-<t^+_{i+1}$, $t_0^-=-K$ and $t_k^+=K'$. we consider the minimal action
\begin{align*}\label{constructioneq23}
h_{L}^{K,K'}(m,m',\vec{x},\vec{t})&=\inf\sum_{i=0}^{k}\int^{t_{i}^-}_{t_i^+}(L-\eta_i) (d\tilde\gamma_i^-(t))dt \notag\\
&+\sum_{i=0}^{k-1}\int^{t_{i+1}^+}_{t_{i}^-}(L-\eta_i-(k_iG_i)^*(\mu_i+\psi_i))(d\tilde\gamma_i^+(t))dt
\end{align*}
where the infimum is taken over all absolutely continuous curves $\tilde\gamma$: $[-K,K']\to\mathbb{R}^n$ satisfying the boundary conditions $\tilde\gamma^-(t_i^-)=\tilde\gamma_i^+(t_i^-)=\tilde x_i^-$, $\tilde\gamma_i^+(t_{i+1}^+)=\tilde x_{i+1}^+$ for $i=0,1,\cdots, i_{k-1}$, $\gamma(-K)=m$, $\gamma(K')=m'$. By carefully setting boundary condition we find that the minimizer is smooth everywhere, along which the term $(k_iG_i)^*(\mu_i+\psi_i)$ does not contribute to the Euler-Lagrange equation. It is guaranteed by the local minimality of (\ref{typeceq2}) as well as (\ref{localmineq1}) and setting the translation $k_{i+1}-k_i$ sufficiently large. The condition of incomplete intersection does not cause new difficulty in verifying the smoothness of the minimizer.  Therefore, the minimizer produces an orbit $(\tilde\gamma,\dot{\tilde\gamma})$ of $\phi_L^t$ which has the properties stated in the theorem.
\end{proof}

\section{The proof of genericity}\label{AppGenericity}
\setcounter{equation}{0}
In this section, we present a proof of the genericity property of (H1) type generalized transition chain by applying the ideas and the techniques of \cite{CY1,CY2}.
\subsection{The setting and the main result}
Given a Hamiltonian $H$,  let $\Phi_H^{t,t'}$ denote the Hamiltonian flow of $H$, it maps the initial value at the time-$t$-section to the time-$t'$-section.

We consider two settings, the nonautonomous case (A) and the autonomous case (B):
\begin{enumerate}
\item[(A)] Given a Tonelli Hamiltonian $H(p,q,t):\ T^*\T^{n}\times \T\to \R$.
\begin{enumerate}
  \item there exists a normally hyperbolic and weakly invariant cylinder $\tilde\Pi$, which is a deformation of a standard cylinder $\{(p,q,t)\in T^*\T^{n}\times \T:\ (\hat p_{n-1},\hat q_{n-1})=0.\}$;
  \item there is a continuous path $\Gamma_c$: $[0,1]\to H^1(\mathbb{T}^n,\mathbb{R})$ such that for any $c\in\Gamma_c$, the Aubry set entirely lies in the cylinder $\tilde\Pi$;
 \end{enumerate}
\item[(B)] Given a Tonelli Hamiltonian $H(p,q):\  T^*\T^{n}\to \R$ and an energy level $E>\min\al_H$,
\begin{enumerate}
\item there is a subsystem $G:\ N\to \R$ where $N\subset T^*\T^n$ is a NHIM of the Hamiltonian flow of $H$. Coordinates can be given such that $G$ is a Tonelli system defined on $T^*\T^2$. 
\item there exists a continuous path $\Gamma_c$: $[0,1]\to H^1(\mathbb{T}^n,\mathbb{R})$ such that for any $c\in\Gamma_c$, the Aubry set entirely lies in the level set $\tilde\Pi:=G^{-1}(E)$.

\end{enumerate}
\end{enumerate}

\begin{Not}
\begin{enumerate}
\item Let $\check{\pi}$: $\check{M}\to\mathbb{T}^n$ be a double covering space of $\mathbb{T}^n$ such that the lift of $\tilde\Pi$ to $T^*\check{M}\times\mathbb{T}$ consists two copies, denoted by $\tilde\Pi_\ell$ and $\tilde\Pi_r$. For $c\in\Gamma_c$, if the Aubry set $\tilde{\mathcal{A}}(c)$ is an invariant torus $\tilde\Upsilon_c\subset\tilde\Pi$, its lift also consists of two components, $\tilde\Upsilon_{c,\ell}\subset\tilde\Pi_\ell$ and $\tilde\Upsilon_{c,r}\subset\tilde\Pi_r$. 
\item Let $\tilde\Pi_0$, $\tilde\Pi_{\ell,0}$, $\tilde\Pi_{r,0}$, $\tilde\Upsilon_{c,0}$, $\tilde\Upsilon_{c,\ell,0}$ and $\tilde\Upsilon_{c,r,0}$ denote the time-0-section of $\tilde\Pi$, $\tilde\Pi_{\ell}$, $\tilde\Pi_{r}$, $\tilde\Upsilon_{c}$, $\tilde\Upsilon_{c,\ell}$ and $\tilde\Upsilon_{c,r}$ respectively. 
\item Denote by $\pi$ the projection such that $\pi(p,q,t)=(q,t)$, let $\Upsilon=\pi\tilde\Upsilon$. 
\item Let $\Gamma_c^*\subset\Gamma_c$ such that
$$
\Gamma_c^*=\{c\in\Gamma_c:\,\tilde{\mathcal{A}}(c)\, \text{\rm is an invariant torus}\}.
$$
\end{enumerate}
\end{Not}
We allow two types of perturbations: \begin{enumerate}
\item[(a)]perturbations depending on all the variables and
\item[(b)] perturbations depending only on the angular variables. 
\end{enumerate}The latter case is also called the Ma\~n\'e perturbation. 

Let $B_D\in\mathbb{R}^n$ denotes a ball about the origin of radius $D$. We assume that $D>0$ is suitably large, such that for all $c\in\Gamma_c$ the $c$-minimal orbits of $H$ entirely stay in $B_D\times\mathbb{T}^{n+1}$. Let $\mathfrak{B}_{\epsilon}\subset C^r(B_D\times\mathbb{T}^{n+1},\mathbb{R})$ (or $\subset\mathfrak{B}_{\epsilon}\subset C^r(\mathbb{T}^n$) in the Ma\~n\'e perturbation case) denote a ball about the origin of radius $\epsilon>0$. In the autonomous case, we define $\mathfrak{B}_{\epsilon}$ similarly as subsets in $ C^r(B_D\times\mathbb{T}^{n})$ or $C^r(\mathbb{T}^n)$.

\begin{theo}\label{fundamental} Let $H$ be a $C^r$ Tonelli Hamiltonian $r\geq 2$ as the above case $(A)$ or $(B)$. Then there exists $\epsilon_0=\epsilon_0(H)$ such that for all $\epsilon<\epsilon_0$ and any small $d_1>0$, there exists a set $\mathfrak{O}$ open-dense in $\mathfrak{B}_{\epsilon}$ such that for each $H_\delta\in\mathfrak{O}$, it holds for $H+H_{\delta}$ and simultaneously for all $c\in\Gamma^*_c$ that the diameter of each connected component of the set
$$
\mathcal{N}(c,\check M)|_{t=0}\backslash(\mathcal{A}(c,\check M)+\delta)|_{t=0}\ne\varnothing
$$
is not larger than $d_1$.
\end{theo}

The following subsections are devoted to the proof of this theorem.  In Section \ref{SSBarrier}, we review the definition and basic properties of barrier functions. In the main body of the proof, i.e. Section \ref{SSRegularity}, \ref{SSPerturbation} and \ref{SSFundamental}, we work on the nonautonomous case (A) and perturbations of type (a).
 In Section \ref{SSRegularity}, we parametrize the barrier functions into a H\"older family. In Section \ref{SSPerturbation} we show how to perturb the barrier function through perturbing the Hamiltonian. In Section \ref{SSFundamental} we give the proof of Theorem \ref{fundamental}.
 Finally, in the last Section \ref{SSModification}, we explain how to modify the argument to include the autonomous case (B) as well as Ma\~n\'e perturbations (b). We remark that only Ma\~n\'e perturbations are allowed in Proposition \ref{PropCross3Res} and its decedents Proposition \ref{PropDoubleell} and \ref{PropComplete} on the dynamics crossing triple and higher resonances.

\subsection{Barrier function and semi-static curves}\label{SSBarrier}

Given an Aubry class for $c\in\Gamma_c$ we can define its elementary weak KAM solution. In the covering space $\check{M}$, there are two Aubry classes for $c\in\Gamma_c$, $\tilde\Upsilon_{c,\ell}$ and $\tilde\Upsilon_{c,r}$. To define the elementary weak KAM solution $u^\pm_{c,\ell}$ for $\tilde\Upsilon_{c,\ell}$, 

 We consider a Tonelli Hamiltonian system $H:\ T^*\T^{n}\times \T\to \R$.
If two or more Aubry classes exist, there are infinitely many weak KAM solutions, among which we are interested in so-called {\it elementary weak KAM solution}, obtained from the function $h^{\infty}_c$. Indeed, treated as the function of $(x,t)$, the function $h^{\infty}_c((x,t),(x',t'))$ is a weak KAM solution that determines orbits approaching the Aubry set as the time approaches infinity, treated as the function of $(x',t')$, the function $h^{\infty}_c((x,t),(x',t'))$ is a weak KAM solution that determines orbits approaching the Aubry set as the time approaches minus infinity. Let $(x,t)$ range over an Aubry class, denoted by $\mathcal{A}_{c,i}$ one has a decomposition
$$
h^{\infty}_c((x,t),(x',t'))=u^-_{c,i}(x',t')-u^+_{c,i}(x,t), \qquad \forall\ (x',t')\in \T^n\times\mathbb{T},
$$
where $u^+_{c,i}$ is a constant, and $u^-_{c,i}$ is called elementary weak KAM solution with respect to $\mathcal{A}_{c,i}$. Similarly, let $(x',t')$ range over an Aubry class, one obtains an elementary weak KAM solution $u^-_{c,i}$. Again, for autonomous system, one skips the time component.

For almost every point $(q,t)\in\check{M}\times\mathbb{T}\backslash\Upsilon_{c,\ell}$ the initial condition $(\partial_pu_{c,r}^\pm(q,t)+c,q,t)$ determines a forward (backward) $c$-minimal orbit that approaches $\tilde\Upsilon_{c,r}$ as $t\to\pm\infty$.
For points $(q,t)\in\check{M}\times\mathbb{T}\backslash\Upsilon_{c,r}$, $u_{c,\ell}^\pm$ determines a $c$-minimal orbit approaching $\tilde\Upsilon_{c,\ell}$.

\begin{defi}
The barrier functions for $c\in\Gamma_c$ are defined as follows
$$
B^\ell_c(q,t)=u^-_{c,\ell}(q,t)-u^+_{c,r}(q,t), \qquad B^r_c(q,t)=u^-_{c,r}(q,t)-u^+_{c,\ell}(q,t).
$$
\end{defi}
In the following, we only study $B_c^\ell$. The arguments for $B_c^r$ are the same.
Since the backward weak KAM is semi-concave and the forward weak KAM is semi-convex, the barrier function is semi-concave. Therefore,
\begin{lem}
At each minimal point of $B^\ell_c$, both $u^-_{c,r}$ and $u^+_{c,\ell}$ are differentiable.
\end{lem}
\begin{proof}
By the definition, semi-concave function admits a local decomposition as the sum of a smooth function and a concave function. For a concave function $u$, one can define its sup-derivative $D^+u(x)$ at a point $x$ such that
$u(x+x')-u(x)\le\langle p,x'\rangle$ holds for any $p\in D^+u(x)$ which is a convex set. The function $u$ is differentiable at $x$ iff $D^+u(x)$ is a singleton.

Since $B^\ell_c$ is a sum of two semi-concave functions, its sup-derivative is the sum of the sup-derivatives of $u^-_{c,\ell}$ and $-u^+_{c,r}$. Therefore, $D^+B^\ell_c$ is a single point iff both $D^+u^-_{c,\ell}$ and $D^+(-u^+_{c,r})$ are singleton \cite{CaC}.
\end{proof}

\begin{lem}
If $(q,t)\in\check{M}\times\mathbb{T}\backslash((\Upsilon_{c,\ell}\cup\Upsilon_{c,r})+\delta)$ is a global minimal point of $B^\ell_c$, then $(q,t)\subset\mathcal{N}(c,\check{M})$, namely, passing through the point $(q,t)$ there is a $c$-semi-static curve in the covering space $\check{M}\times\mathbb{T}$.
\end{lem}
\begin{proof}
By the definition, $\partial u^-_{c,\ell}=\partial u^+_{c,r}$ holds at a global minimal point of $B^\ell_c$, denoted by $x=(q,t)$. Therefore, the backward minimal curve $\gamma^-_{c,x}$ is joined smoothly to the forward minimal curve $\gamma^+_{c,x}$. They make up a $c$-semi-static curve for $\check{M}$.
\end{proof}

For a class $c\in\Gamma^*_c$, the covering space $\check{M}\times\mathbb{T}$ is divided into two annuli $\mathbb{A}_{c,r}$ and $\mathbb{A}_{c,\ell}$, bounded by $\Upsilon_{c,\ell}$ and $\Upsilon_{c,r}$. Clearly, one has $\check{\pi}\mathbb{A}_{c,r}=\check{\pi}\mathbb{A}_{c,\ell}$.
The set $\mathcal{N}(c,\check{M})\backslash\mathcal{A}(c,\check{M})$ contains $c$-minimal curves which cross the annulus from one side to another side or vice versa. Each of the curves produces a homoclinic orbit to the torus $\tilde\Upsilon_c$.

\begin{lem}\label{A3}
There is a finite partition of $\Gamma_c$: $\Gamma_c=\cup I_k$, each $I_k$ is a segment of $\Gamma_c$. For each $I_k$ there is an annulus $N_k\subset\mathbb{A}_{c,r}|_{t=0}$, two numbers $\delta>0$ and $d>0$ such that for each $c\in I_k\cap\Gamma^*_c$
\begin{enumerate}
  \item $\mathrm{dist}(N_k,\Upsilon_{c,\ell}\cup\Upsilon_{c,r})\ge\delta$;
  \item each curve $(\gamma(t),t)$ lying in $(\mathcal{N}(c,\check{M})\backslash\mathcal{A} (c,\check{M}))\cap\mathbb{A}_{c,r}$ passes through $N_k$;
  \item for each backward (forward) $c$-minimal curve $\gamma$, let $\{q_i=\gamma(2i\pi)\in N_k\}$, then $|q_i-q_j|\ge d$ if $i\ne j$.
\end{enumerate}
\end{lem}
\begin{proof}
Because $\Gamma_c$ is compact, the speed of each $c$-minimal orbit is uniformly upper bounded for all $c\in\Gamma_c^*$. Given an integer $m>0$, there will be small $\delta_c>0$ such that the period for each $c$-minimal curve to cross the annulus $N_c=\mathbb{A}_{c,r}\backslash((\Upsilon_{c,\ell}\cup\Upsilon_{c,r})+\delta_c)$ is not shorter than $4m\pi$. Because of the upper semi-continuity of Ma\~n\'e set in $c$, there exists some $\delta'_c>0$ such that $\Upsilon_{c',\ell}\cup\Upsilon_{c',r}$ does not touch $N_c$ and the period for each $c'$-minimal curve to cross the annulus $N_c$ is not shorter than $2m\pi$ provided $|c-c'|\le\delta'_c$ and $c'\in\Gamma^*_c$. The first two items are then proved if we notice $\Gamma^*_c$ is compact.

For the third one, we notice that the condition $\gamma(2i\pi)=\gamma(2j\pi)$ for $i\ne j$ implies that $\gamma$ is a curve in the Aubry set. It contradicts the assumption. Since both $N_k$ and $I_k$ are compact, such a constant $d>0$ exists.
\end{proof}
\subsection{The regularity of barrier functions}\label{SSRegularity}
The next lemma on the regular dependence on certain parameter of the invariant circles of the twist map is the key observation to establish the genericity.
\begin{lem}\label{Lm1/2Holder}
There exist a constant $C_L$ and a parametrization $\sigma \mapsto c(\sigma)\in I_k\cap \Gamma_c^*$ such that the invariant curves $\tilde\Upsilon_{c(\sigma),0}(q)$ on the NHIC forms a $1/2$-H\"older family in the $C^0$ norm with respect to the parameter $\sigma$:
\begin{equation}\label{Eq1/2Holder}
\max_q|\tilde\Upsilon_{c(\sigma),0}(q)-\tilde\Upsilon_{c(\sigma'),0}(q)|\le\sqrt{2C_L|\sigma-\sigma'|}.
\end{equation}
\end{lem}
\begin{proof}By the definition, the Aubry set $\tilde{\mathcal{A}}(c)$ is an invariant torus if $c\in\Gamma^*_c$. Its time-$2\pi$-section is an invariant circle lying in the cylinder. Fix one of the circles, we are able to parameterize other circle by the algebraic area bounded by the circles. Let us consider the twist map on the standard cylinder first. It is well-known that all invariant circles are Lipschitz with the constant $C_L$ which depends on the twist condition only. Treating each circle as the graph of some periodic function and fixing one as $\tilde\Upsilon_{0,0}$ one can parameterize another circle by the algebraic area bounded by these two circles. The annulus bounded by the circle $\tilde\Upsilon_{\sigma,0}$ and $\tilde\Upsilon_{\sigma',0}$ contains a diamond, the height of the vertical diagonals is $\max_q|\tilde\Upsilon_{\sigma,0}(q)-\tilde\Upsilon_{\sigma',0}(q)|$ and the length of the horizontal diagonal is not shorter than $\frac 1{C_L}\max_q|\tilde\Upsilon_{\sigma,0}(q)-\tilde\Upsilon_{\sigma',0}(q)|$. So, one has \eqref{Eq1/2Holder}.
A non-standard cylinder can be regarded as the image of the standard cylinder under a symplectic diffeomorphism, so the $\frac 12$-H\"older continuity still holds.
\end{proof}

Each invariant circle corresponds to a unique $c\in\Gamma_c$ such that the Aubry set is the circle. The parameter $\sigma$ is usually defined on a Cantor set, denoted by $\Sigma$. We next use the normal hyperbolicity of the cylinder to extend the H\"older estimate to barrier functions defined on $\T^{n}$.

\begin{lem}\label{moduluscontinuity}
For $\sigma,\sigma'\in\Sigma$, let $c=c(\sigma)$, $c'=c(\sigma')$. If $c,c'\in I_k$ and $q\in N_k$, then
$$
|B^\ell_{c(\sigma)}(q,0)-B^\ell_{c(\sigma')}(q,0)|\le C(\sqrt{|\sigma-\sigma'|}+|c-c'|).
$$
\end{lem}
\begin{proof}
For $c=c(\sigma)$ with $\sigma\in\Sigma$, the minimal measure is uniquely ergodic. There is only one pair of weak KAM solutions $u^{\pm}_c$ for the configuration space $\mathbb{T}^2$. With respect to the covering space $\check{M}$, we have introduced the elementary weak KAM solutions $u^{\pm}_{c,\ell}$ and $u^{\pm}_{c,r}$. Since the projection $\check\pi$ is an injection when it is restricted in the neighborhood $\Upsilon_{c,\imath}+\delta$ for $\imath=\ell,r$ respectively, for $(q,t)\in\Upsilon_{c}+\delta$ one has
\begin{equation}\label{A10}
u^{\pm}_{c,\ell}(\check\pi^{-1}(q,t)\cap(\Upsilon_{c,\ell}+\delta)) =u^{\pm}_{c,r}(\check\pi^{-1}(q,t)\cap(\Upsilon_{c,r}+\delta))
=u^\pm_c(q,t).
\end{equation}

By the definition of weak KAM solutions, for any $t'<t$ one has
$$
u^-_{c,\ell}(\gamma(t),t)-u^-_{c,\ell}(\gamma(t'),t')\le\int_{t'}^t(L(\dot\gamma(s),\gamma(s),s)-\langle c,\dot\gamma(s)\rangle)ds+(t-t')\alpha(c)
$$
which becomes an equality when $\gamma$ is a backward $c$-semi static curve. Assume $\gamma^-_{c,q}$ is a backward $c$-minimal curve such that $\gamma^-_{c,q}(0)=q$, we have
$$
\begin{aligned}
u^-_{c,\ell}(q,0)-u^-_{c,\ell}(\gamma^-_{c,q}(-2K\pi),0)=&\int_{-2K\pi}^0(L(\dot\gamma^-_{c,q}(s), \gamma^-_{c,q}(s),s)-\langle c,\dot\gamma^-_{c,q}(s)\rangle)ds\\
&+2K\pi\alpha(c),\\
u^-_{c',\ell}(q,0)-u^-_{c',\ell}(\gamma^-_{c,q}(-2K\pi),0)\le&\int_{-2K\pi}^0(L(\dot\gamma^-_{c,q}(s), \gamma^-_{c,q}(s),s) -\langle c',\dot\gamma^-_{c,q}(s)\rangle)ds\\
&+2K\pi\alpha(c').
\end{aligned}
$$
Since $N_k$ keeps away from $\Upsilon_{c,\ell}$, some $K>0$ exists such that for each $q\in N_k$, $c\in I_k$ and $q\in N_k$ one has $\gamma^-_{c,q}(-2K\pi)\in(\Upsilon_{c,\ell}+\delta)$.
Since $c$ and $c'$ are located in a compact set $\Gamma_c$, the $\alpha$ function is convex and finite everywhere, there is some constant $C_1$ such that $|\alpha(c')-\alpha(c)|\le C_1|c-c'|$. Let $\bar\gamma^-_{c,q}$ be the lift of $\gamma^-_{c,q}$ to the universal covering space, one has $|\bar\gamma^-_{c,q}(0)-\bar\gamma^-_{c,q}(-2K\pi)|\le 2C_2K\pi$.
$$
\begin{aligned}
&u^-_{c',\ell}(q,0)-u^-_{c,\ell}(q,0)-(u^-_{c',\ell}(\gamma^-_{c,q}(-2K\pi),0)-u^-_{c,\ell} (\gamma^-_{c,q}(-2K\pi),0))\\
&\le 2K\pi(C_1+C_2)|c-c'|.
\end{aligned}
$$
In the same way one can also obtain
$$
\begin{aligned}
&u^-_{c,\ell}(q,0)-u^-_{c',\ell}(q,0)-(u^-_{c,\ell}(\gamma^-_{c',q}(-2K\pi),0)-u^-_{c',\ell} (\gamma^-_{c',q}(-2K\pi),0))\\
&\le 2K\pi(C_1+C_2)|c-c'|.
\end{aligned}
$$
For $u^+_{c,r}$, $u^+_{c'r}$ we also have similar inequalities. Therefore, it follows from (\ref{A10}) that some points $(q_\ell,0),(q_r,0)\in\Upsilon_{c}+\delta$ exist such that
$$
\begin{aligned}
|B_c(q,0)-B_{c'}(q,0)|&\le 4K\pi(C_1+C_2)|c-c'|\\
&+|u^-_c(q_\ell,0)-u^-_{c'}(q_\ell,0)-u^+_c(q_r,0)+u^+_c(q_r,0)|.
\end{aligned}
$$
By the assumption, both $u_c^-$ and $u_c^+$ are $C^{1,1}$ when they are restricted in $\Upsilon_c+\delta$. Due to the normal hyperbolic property, each $(p,q)\in\tilde\Pi_{0}$ has its stable and unstable fiber which is $C^{r-1}$-smoothly depends on the point $(p,q)$. The fibers are defined by $\partial_q u^{\pm}_c+c$ and one has that
$$
|\partial _qu^\pm_{c}-\partial_qu^\pm_{c'}+c-c'|\le C_3\sqrt{|\sigma-\sigma'|}
$$
holds for some constant $C_3>0$, independent of $c,c'$. Combining above two inequalities, one finishes the proof of the lemma.
\end{proof}

\subsection{Perturbing the barrier function through perturbing the Hamiltonian}\label{SSPerturbation}
In this section, we show how to perturb the barrier function through perturbing the Hamiltonian.
We consider the $c$-minimal curves for $c\in I_k$. Because $I_k$ is compact, there exists a constant $D>0$ such that $|\dot\gamma(t)|\le D$ holds for any $c$-minimal curve with $c\in I_k$.
Let $\Omega_{\tau}=\{(q',q)\in\mathbb{R}^2\times\mathbb{R}^2:|q'-q|\le 2D\tau\ \mathrm{with}\ \tau>0\}$. We consider the action
$$
S_{-\tau}(q',q)=\min_{\stackrel {\xi(-\tau)=q'}{\scriptscriptstyle \xi(0)=q}}\int_{-t}^0 L(\dot\xi(s),\xi(s),s)ds.
$$
For suitably small $\tau>0$, there exists a unique minimal curve if $(q',q)\in\Omega_\tau$. Indeed, because $L$ is Tonelli, the second derivative of any solution $q(t)$ of the Euler-Lagrange equation is bounded by $|\ddot q|\le |\partial_{\dot q\dot q}L^{-1}(\partial_qL-\partial^2_{\dot qq}L\dot\gamma-\partial_{\dot qt}L)|$. Recall the Taylor formula
$$
q(t')=q(t)+\dot q(t)(t'-t)+\frac 12\ddot q(\lambda t+(1-\lambda)t')(t'-t)^2
$$
holds for small $|t'-t|$, where both entries of $\lambda\in\mathbb{R}^2$ takes value in $[0,1]$. Therefore, for small $|t'-t|$, there is an one to one correspondence the initial speed $\dot\gamma(t)$ and the end point $\gamma(t')$. In this case, $S_{-\tau}(q',q)$ is $C^r$-differentiable in both $q'$ and $q$. By the definition of weak KAM, for $c\in I_k$ one has
$$
u^-_c(q,0)=\min_{q'\in\mathbb{T}^2,\,|q'-q|\le 2D\tau}(S_{-\tau}(q',q)-\langle c,q-q'\rangle+u^-_c(q',-\tau))
$$
We extend $S_{-\tau}$ smoothly to the whole $\mathbb{R}^2\times\mathbb{R}^2$ such that it satisfies the twist condition. Recall the quantities defined in Lemma \ref{A3} such as the annulus $N_k$ and the number $d>0$.
\begin{lem}\label{A5} For any $\epsilon>0$ small enough, there exists $\dt$ such that if  $S_\delta(q)$ be a $C^r$-function such that $\max\{|q-q'|:q,q'\in\mathrm{supp} S_\delta\}\le d$, $\mathrm{supp} S_\delta\subset N_k$ and $\|S_\delta\|_{C^r}\leq \dt$. Then, restricted on $I_k$, there exists a perturbation $H\to H'=H+H_\delta$ with $\|H_\delta\|_{C^r}<\epsilon$ and the barrier function is subject to a translation
$$
B_c(q,0)\to B_c(q,0)+S_\delta(q) \qquad \forall\ c\in I_k,\ q\in\mathrm{supp} S_\delta.
$$
\end{lem}
\begin{proof}
The function $S_{-\tau}(q',q)$ induces a symplectic map between the time $-\tau$ section and the time-0-section $\Phi$: $(p',q')\to(p,q)$
$$
p=\frac{\partial S_{-\tau}}{\partial q}(q',q)\qquad p'=-\frac{\partial S_{-\tau}}{\partial q'}(q',q).
$$
We introduce a smooth function $\kappa$ such that $\kappa(q',q)=1$ if $|q'-q|\le K$ and $\kappa(q',q)=0$ if $|q'-q|\ge K+1$. Let $\Phi'$ be the map determined by the generating function $S_{-\tau}+\kappa S_\delta$, the symplectic diffeomorphism $\Psi= \Phi'\circ\Phi^{-1}$ is close to identity if $S_\delta$ is $C^r$-small. We choose a smooth function $\rho(s)$ with $\rho(-\tau)=0$, $\rho(0)=1$ and let $\Phi_s'$ be the symplectic map produced by $S_{-\tau}+\rho(s)\kappa S_\delta$ and let $\Psi_s=\Phi_s'\circ\Phi^{-1}$. Clearly, $\Psi_s$ defines a symplectic isotopy between the identity map and $\Psi$. Thus, there is a unique family of symplectic vector fields $X_s$: $T^*\mathbb{T}^2\to TT^*\mathbb{T}^2$ such that
$$
\frac d{ds}\Psi_s=X_s\circ\Psi_s.
$$
By the choice of perturbation, there is a simply connected and compact domain $D$ such that $\Psi_s|_{T^*\mathbb{T}^2\backslash D}=id$. It follows that there exists a Hamiltonian $H_1(p,q,s)$ such that $X_s=J\nabla H_1(p,q,s)$. Re-parametrizing $s$ by $t$, we can make $X_s$ smoothly depend on $t$ and smoothly connected to the zero vector field at $t=-\tau,0$. To show the smallness of $dH'$ we apply a theorem of Weinstein \cite{W}. A neighborhood of the identity in the symplectic diffeomorphism group of a compact symplectic manifold can be identified with a neighborhood of the zero in the vector space of closed 1-forms on the manifold. Since Hamiltomorphism is a subgroup of symplectic diffeomorphism, there is a function $H'$, sufficiently close to $H$, such that $\Phi^{-\tau,0}_{H'}=\Phi_{H_1}^{-\tau,0}\circ\Phi_{H}^{-\tau,0}$.

For all $c\in\Gamma_c$, by the assumption, any backward (forward) $c$-minimal curve will not return back to $\mathrm{supp}S_{-\tau}$ if its initial point falls into the support. Let $u^{\pm,S_\delta}_{c,\imath}$ denotes the elementary weak KAM solution for the perturbed Hamiltonian
$$
\begin{aligned}
u^{-,S_\delta}_{c,\imath}(q,0)=&\min_{|q'-q|\le 2D\tau}(S_{-\tau}(q',q)+S_\delta(q)-\langle c,q-q'\rangle+u^-_{c,\imath}(q',-\tau))\\
=&S_\delta(q)+\min_{|q'-q|\le 2D\tau}(S_{-\tau}(q',q)-\langle c,q-q'\rangle+u^-_{c,\imath}(q',-\tau))\\
=&S_\delta(q)+u^{-}_{c,\imath}(q,0).
\end{aligned}
$$
Obviously, one has $u^{+,S_\delta}_{c,\imath}(q,0)=u^{+}_{c,\imath}(q,0)$. The lemma is proved because the barrier function is the difference of the two functions.
\end{proof}
\subsection{Proof of Theorem \ref{fundamental}}\label{SSFundamental}

The proof of Theorem \ref{fundamental} is based on the following lemma.

Given $q^*\in\mathbb{T}^2$, let $\mathbb{S}_{d_1}(q^*)=\{|q-q^*|\le d_1\}$ denote a square. Given a function $B\in C^0(\mathbb{S}_{d_1}(q^*),\mathbb{R})$, let
$$
\mathrm{Argmin}(\mathbb{S}_{d_1}(q^*),B)=\{q\in\mathbb{S}_{d_1}(q^*):B(q)=\min B\}.
$$

\begin{lem}\label{LmFundamental}
For any small $\epsilon>0$, there is a set $\mathfrak{O}$ open-dense in $\mathfrak{B}_{\epsilon}$ such that for each $H_{\delta}\in\mathfrak{O}$,
letting $B^\ell_{c,\delta}$ be the barrier function for the Hamiltonian $H+H_{\delta}$ and the class $c$, it holds simultaneously for all $c\in I_k\cap\Gamma_c^*$ that the set $\mathrm{Argmin}(\mathbb{S}_{d_1}(q^*),B^\ell_{c,\delta})$ is trivial for $\mathbb{S}_{d_1}(q^*)$ provided $\mathbb{S}_{d_1}(q^*)\subset N_k$ and $d_1<d/3$ is suitably small.
\end{lem}
We first complete the proof of Theorem \ref{fundamental} assuming the lemma.

\begin{proof}[Proof of Theorem \ref{fundamental}]

Let $\pi_i$ be the projection so that $\pi_i(q_1,q_2)=q_i$ ($i=1,2$). A connected set $V$ is said to be non-trivial for $\mathbb{S}_{d_1}(q^*)$ if $\pi_iV\cap\mathbb{S}_{d_1}(q^*)=\pi_i\mathbb{S}_{d_1}(q^*)$ holds for $i=1$ or $2$. Otherwise, it is said to be trivial for $\mathbb{S}_{d_1}(q^*)$.
To finish the proof of Theorem \ref{fundamental}, we split the annulus $N_k$ equally into squares $\{\mathbb{S}_j=|q-q_j|\le\frac {d_1}5\}$. By Lemma \ref{LmFundamental}, for each $\mathbb{S}_j$, there exists an open-dense set $\mathfrak{O}_{k,j}\subset\mathfrak{B}_{\epsilon}$, for each $H_{\delta}\in\mathfrak{O}_{k,j}$ it holds simultaneously for all $c\in I_k\cap\Gamma^*_c$ that the set $\mathrm{Argmin}(\mathbb{S}_j,B^\ell_{c,\epsilon})$ is trivial for $\mathbb{S}_j$. The intersection $\cap\mathfrak{O}_{k,j}$ is still open-dense in $\mathfrak{B}_{\epsilon}$. For each $H_{\delta}\in\cap_{k,j}\mathfrak{O}_{k,j}$, it holds simultaneously for all $c\in\Gamma^*_c$ that the diameter of each connected component of the Ma\~n\'e set is not larger than $\frac 45d_1$ if it keeps away from the Aubry set.
\end{proof}
Let us now give the proof of Lemma \ref{LmFundamental}.
\begin{proof}[Proof of Lemma \ref{LmFundamental}]

The openness is obvious. To show the denseness, by Lemma \ref{A5}, we construct the perturbations $H_\delta\in\mathfrak{B}_{\epsilon}$ such that the barrier function is under a translation $B_c(q,0)\to B_c(q,0)+S_\delta(q)$ for all $c\in I_k\cap\Gamma_c^*$ and $q\in\mathrm{supp}S_\delta$. 

Recall the number $d>0$ defined in Lemma \ref{A3}. Given a square $\mathbb{S}_{d_1}(q^*)\subset N_k$ with $3d_1<d$, we consider the space of $C^r$-functions $\mathfrak{S}_1$, a function $S\in\mathfrak{S}_1$ if it satisfies the conditions that $\mathrm{supp}S\subset B_{d/2}(q^*)$ and $S$ is constant in $q_2$ when it is restricted in $\mathbb{S}_{d_1}(q^*)$. Similarly, we can define $\mathfrak{S}_2$ such that $S\in\mathfrak{S}_2$ implies that $\mathrm{supp}S\subset B_{d/2}(q^*)$ and it is constant in $q_1$ when it is restricted in $\mathbb{S}_{d_1}(q^*)$.

In $\mathfrak{S}_i$ we define an equivalent relation $\sim$, two functions $S_1\sim S_2$ implies $S_1-S_2=\mathrm{constant}$ when they are restricted on $\mathbb{S}_{d_1}(q^*)$. Obviously, $\mathfrak{S}_i/\sim$ is a linear space with infinite dimensions. For $S_1,S_2\in\mathfrak{S}_i/\sim$,
$\|S_1-S_2\|_r$ measures the $C^r$-distance if they are regarded as the functions defined on $\mathbb{S}_{d_1}(q^*)$. We also use $\mathfrak{B}_{i,\epsilon}$ to denote a ball in $\mathfrak{S}_i/\sim$, about the origin of radius $\epsilon$ in the sense of the $C^r$-topology.

We claim that there exists a set $\mathfrak{O}_{1,\epsilon}$ open-dense in $\mathfrak{B}_{1,\epsilon}$ such that for each $S_{\delta}\in\mathfrak{O}_{1,\epsilon}$ it holds simultaneously for all $c\in I_k\cap\Gamma_c^*$ that
\begin{equation}\label{Z1}
\pi_1\mathrm{Argmin}(\mathbb{S}_{d_1}(q^*),B^\ell_{c}+S_{\delta})\subsetneqq[q_1^*-d_1,q_1^*+d_1]
\end{equation}
Let $\mathfrak{F}_c=\{B^\ell_c(q,0):c\in\Gamma^*_c\}$ be the set of barrier functions. For $i=1,2$ we set
$$
\mathfrak{Z}_i=\{B\in C^0(\mathbb{S}_{d_1}(q^*),\mathbb{R}):\pi_i\mathrm{Argmin}(\mathbb{S}_{d_1}(q^*),B)=[q_i^*-d_1,q_i^*+d_1]\},
$$
where $q^*=(q^*_1,q^*_2)$.

If the denseness does not hold, there would be small $\epsilon>0$, for each $S_\delta\in\mathfrak{B}_{1,\epsilon}$, some $c\in\Gamma_c^*$ exists such that $B^\ell_c+S_\delta\in\mathfrak{Z}_1$. Let $\mathfrak{B}_{1,\epsilon}^k$ be the intersection of $\mathfrak{B}_{1,\epsilon}$ with a $k$-dimensional subspace. The box-dimension of $\mathfrak{B}_{1,\epsilon}^k$ in $C^0$-topology will not be smaller than $k$.

For any $B^\ell_c\in\mathfrak{F}_c$ there is only one $S_\delta\in\mathfrak{B}_{1,\epsilon}$ such that $B^\ell_c+S_\delta\in\mathfrak{Z}_1$. Otherwise, there would be $S_\delta'\ne S_\delta$ such that $B^\ell_c+S_\delta'\in\mathfrak{Z}_1$ also. As we have $B^\ell_c+S'_\delta=B^\ell_c+S_\delta+S'_\delta-S_\delta$ where $B^\ell_c+S_\delta\in\mathfrak{Z}_1$ and $S'_\delta\sim S_\delta$, which contradicts the definition of $\mathfrak{S}_1$. For $S_\delta\in\mathfrak{B}_{1,\epsilon}$, let $\mathfrak{S}_{S_\delta}=\{B_c^\ell\in\mathfrak{F}_c:B_c^\ell+S_\delta\in\mathfrak{Z}_1\}$. If the denseness does not hold, $\mathfrak{S}_{S_\delta}$ is non-empty. For any $S_{\delta},S'_{\delta}\in\mathfrak{B}^k_{1,\epsilon}$, each $B_c^\ell\in\mathfrak{S}_{S_\delta}$ and each $B_{c'}^\ell\in\mathfrak{S}_{S'_\delta}$ one has
\begin{equation}\label{isometric}
\begin{aligned}
d(B_c^\ell,B_{c'}^\ell)&=\max_{q\in\mathbb{S}_{d_1}(q^*)}|B_{c}^\ell(q,0)-B_{c'}^\ell(q,0)|\\
&\ge\max_{|q_1-q_1^*|\le d_1}\Big|\min_{|q_2-q_2^*|\le d_1}B_{c}^\ell(q,0)-\min_{|q_2-q_2^*|\le d_1}B_{c'}^\ell(q,0)\Big|\\
&=\max_{|q_1-q_1^*|\le d_1}|S_{\delta}(q)-S'_{\delta}(q)|=d(S_{\delta},S'_{\delta})
\end{aligned}
\end{equation}
where $q=(q_1,q_2)$ and $d(\cdot,\cdot)$ denotes the $C^0$-metric. It implies that the box-dimension of the set $\mathfrak{F}_c$ is not smaller than the box-dimension of $\mathfrak{B}_{1,\epsilon}^k$ in $C^0$-topology. Guaranteed by the modulus continuity of Lemma \ref{moduluscontinuity}, the box dimension of the set $\mathfrak{F}_c$ is not larger than 3. Therefore, we will obtain an absurdity if we choose $k\ge 4$.

In the same way, we can show that there exists a set $\mathfrak{O}_{2,\epsilon}$ open-dense in $\mathfrak{B}_{2,\epsilon}$ such that for each $S_{\delta}\in\mathfrak{O}_{2,\epsilon}$ it holds simultaneously for all $c\in I_k\cap\Gamma_c^*$ that
\begin{equation}\label{Z2}
\pi_2\mathrm{Argmin}(\mathbb{S}_{d_1}(q^*),B^\ell_{c}+S_{\delta})\varsubsetneq[q_2^*-d_1,q_2^*+d_1].
\end{equation}
Therefore, $\exists$ arbitrarily small $S_{i,\delta}\in\mathfrak{B}_{i,\epsilon}$ such that  $\pi_i\mathrm{Argmin}(\mathbb{S}_{d_1}(q^*),B^\ell_c+S_{1,\delta}+S_{2,\delta})$ is trivial for $\mathbb{S}_{d_1}(q^*)$ and for all $c\in I_k\cap\Gamma_c^*$. Due to Lemma \ref{A5} we obtain the density.
\end{proof}
\subsection{The autonomous case and the Ma\~n\'e perturbation case}\label{SSModification}
We have thus completed the proof of Theorem \ref{fundamental} in the nonautonomous case (A) and for perturbations depending on all the variables. To generalize the argument to the autonomous case (B) and the Ma\~n\'e perturbation, we first review the argument. The argument relies on the following three ingredients
\begin{enumerate}
\item one can perturb the barrier function through perturbing the Hamiltonian (Lemma \ref{A5});
\item the barrier functions associated to invariant curves in the NHIC can be parametrized into a H\"older family (Lemma \ref{Lm1/2Holder} and \ref{moduluscontinuity});
\item arbitrarily small perturbations to the Hamiltonian can make simultaneously all the barrier functions with $c\in \Gamma^*_c$ nonconstant (Lemma \ref{LmFundamental}).
\end{enumerate}
\subsubsection{The autonomous case}
To prove Theorem \ref{fundamental} in the autonomous case (B) and for perturbations depending on all the variables, we need the following two modifications. First, the nonautonomous case (A) has a natural section of the Hamiltonian flow given by $\{t=0\}$, restricted to which we study the regularity and perturbation of the barrier functions. In the autonomous case (B), in place of $\{t=0\}$, we need to pick a torus $\mathsf T$ that is homologous to $\T^{n-1}$ and transverse to orbits in the projected Ma\~n\'e set, and consider the restriction of to $\mathsf T$ of the barrier function (see Section 4.2 of \cite{C17b}). Second, the regularity results Lemma \ref{Lm1/2Holder} and Lemma \ref{moduluscontinuity} should be replaced by the corresponding versions restricted to an energy level, i.e. the following Theorem \ref{Thm1/3Holder} and Theorem \ref{extension}.

With these modifications, one can verify that the proofs of Lemma \ref{LmFundamental} and Theorem \ref{fundamental} goes through.

The next result is the main theorem of \cite{CX}.
\begin{theo}[Theorem 1.1 of \cite{CX}]\label{Thm1/3Holder}
Let $G:\ T^*\T^2\to \R$ be a Tonelli Hamiltonian, the set $\mathcal E_E$ be the set of extremal points of the convex set
$\cup_{E'\leq E}\{\alpha_G^{-1}(E')\}$, $E>\min\al_G$, and $u_c^\pm:\ \R^2\to \R,\ c\in \mathcal E_E$, be lifted elementary weak KAM solutions to $\R^2$ normalized by $u_c^\pm(0)=0$. For given bounded domain $\Omega\subset \R^2$, there exists a constant $C(\Omega,G)$ depending only on $\Omega$ and $G$, and a one-to-one parametrization of the elementary weak KAM solutions of cohomology classes in $\mathcal E_E$ by a number $\sigma\in\Sigma\subset [0,1],$ such that we have the following H\"older regularity: $\forall\ \sigma\in \Sigma,\forall\ c\in c(\Sigma)=\mathcal E_E,$\footnote{In the main theorem of \cite{CX}, the half sentence ``where $C$ is a constant depending only on the Hamiltonian $H$" is redundent and should be removed. }
$$
\|u^{\pm}_{c(\sigma)}-u^{\pm}_{c(\sigma')}\|_{C^0(\Omega)}\le C(\Omega,G)(\|c(\sigma)-c(\sigma')\|+|\sigma-\sigma'|^{\frac 13}).
$$
\end{theo}
In the presence of a NHIC, we have the following regularity result in the higher dimensional case. The proof is identical to Lemma \ref{moduluscontinuity}.
\begin{theo}[Theorem 6.1 of \cite{CX}]\label{extension}
Let $\mathbb{T}^k\times\mathbb{R}^k(\subset\mathbb{T}^n\times\mathbb{R}^n),\ k<n,$ be a normally hyperbolic invariant manifold for the Hamiltonian flow with $\ell\geq 2$ and let $u^{\pm}_{c(\sigma)}$ be elementary weak KAMs defined on $\mathbb{T}^n$ for $c(\cdot):\Sigma\to H^1(\T^k,\R)$ continuous and one-to-one, where $\Sigma$ is a compact subset of $\R^k$. If $\bar u^{\pm}_{c(\sigma)}:=\bar u^{\pm}_{c(\sigma)}|_{\mathbb{T}^k}$ is $\nu$-H\"older continuous in $\sigma$, then the weak KAM solutions $u^{\pm}_{c(\sigma)}$ satisfy the following estimate
$$\|u^\pm_{c(\sigma)}- u^\pm_{c(\sigma')}\|_{C^0(\T^n)}\leq C(\|\sigma-\sigma'\|^{\nu}+ \|c(\sigma)-c(\sigma')\|).$$
for some constant $C$.
 \end{theo}
\subsubsection{The Ma\~n\'e perturbation}
To prove Theorem \ref{fundamental} for Ma\~n\'e perturbations, we have to show that items (1) and (3) at the beginning of the subsection can be done using only Ma\~n\'e perturbations. Note that this is the case that we have to consider in order to prove Proposition \ref{PropCross3Res} and its decedents Proposition \ref{PropDoubleell} and \ref{PropComplete}. In fact, Lemma \ref{LmFundamental} was proved in Section 4.2 of \cite{C17b} as Theorem 4.2 where Lemma \ref{A5} is replaced by another argument. We refer readers to \cite{C17b} for details.

\noindent{\bf Acknowledgment.} The first author is supported by NNSF of China (No.11790272
and No.11631006). The second author is supported by NNSF of China (No.11790273),
NSF grant DMS-1500897 and One Thousand Young Talents Program in China.

\end{document}